\documentclass[12pt]{amsart}
\usepackage{amssymb,latexsym,amscd}
\usepackage[colorlinks]{hyperref}
\hypersetup{
bookmarksnumbered,
pdfstartview={FitH},
breaklinks=true=
linkcolor=blue,
urlcolor=blue,
citecolor=blue,
bookmarksdepth=2
}

\usepackage{mathrsfs}
\usepackage{amsfonts}
\usepackage[centertags]{amsmath}
\usepackage{amssymb}
\usepackage{amsthm}
\usepackage{graphicx}
\usepackage{float}
\usepackage[all]{xy}
\usepackage{tikz-cd}
\usepackage{caption}

\numberwithin{equation}{section}

\newtheorem{theorem}{Theorem}[section]
\newtheorem{proposition}[theorem]{Proposition}
\newtheorem{conjecture}[theorem]{Conjecture}
\newtheorem{corollary}[theorem]{Corollary}
\newtheorem{lemma}[theorem]{Lemma}

\theoremstyle{definition}

\newtheorem{remark}[theorem]{Remark}
\newtheorem{example}[theorem]{Example}
\newtheorem{definition}[theorem]{Definition}
\newtheorem{problem}[theorem]{Problem}

\def\endproof{\hfill$\square$\medskip}

\newcommand{\act}{\mathop{\triangleright}}

\def\AA{\mathbb{A}}

\def\FF{\mathcal{F}}
\def\ZZ{\mathbb{Z}}

\def\QQ{\mathbb{Q}}
\def\CC{\mathbb{C}}
\def\RR{\mathbb{R}}

\def\gg{\mathfrak{g}}
\def\nn{\mathfrak{n}}
\def\hh{\mathfrak{h}}
\def\ii{\mathbf{i}}
\def\jj{\mathbf{j}}

\def\ZZ{\mathbb{Z}}

\def\kk{\Bbbk}

\begin{document}

\title{Valuations, bijections, and bases}

\author{Arkady Berenstein}
    %Address of record for the research reported here
\address{Department of Mathematics, University of Oregon,
Eugene, OR 97403}
\email{arkady@math.uoregon.edu}

\author{Dima Grigoriev}
\address{\noindent CNRS, Math\'ematiques, Universit\'e de Lille, Villeneuve d'Ascq, 59655, France}
\email{dmitry.grigoryev@univ-lille.fr}
 
\thanks{This work was partially supported by the Simons Foundation Collaboration Grant for Mathematicians no.~636972 (AB)}

\begin{abstract} The aim of this paper is to build a theory of commutative and noncommutative {\it injective} valuations of various algebras (including algebras with zero divisors). The targets of our valuations are (well-)ordered commutative and noncommutative (partial and entire) semigroups including any sub-semigroups of the free monoid $F_n$ on $n$ generators and various quotients. When the range of a valuation of an algebra $A$ is a finitely generated  (partial) semigroup, we construct a generalization of the standard monomial bases in $A$, which seems to be new in noncommutative case. Quite remarkably, for any pair of  well-ordered valuations one has a canonical bijection between the valuation semigroups, which serves as an analog of the celebrated Jordan-H\"older correspondences and these bijections are ``almost" homomorphisms of the involved  semigroups. 
A spectacular demonstration of this remarkable property of JH-bijections for quantum Schubert cells $A=U_q(w)$ results in  mysterious "symplectomorphisms" of involved skew symmetric forms.
    
\end{abstract}

\maketitle

{\bf keywords}: valuations on algebras with zero divisors, injective valuations, adapted bases


\tableofcontents

\section{Introduction}

The aim of this paper is twofold:

$\bullet$ To initiate and systematically study {\it injective} valuations of algebras into (partial) semigroups.

$\bullet$ For pairs of such injective valuations of a given algebra establish and study a bijection between their images, which we refer to as {\it Jordan-H\"older} (JH) bijection. 

In this paper,  valuation $\nu$ on a $\kk$-vector space $V$ is a map $V\setminus \{0\}\to P$ where $P$ is a totally ordered set with an order $\preceq$ such that $\nu(a v)=\nu(v)$ for all $v\in V$, $a\in \kk^\times$ and
\begin{equation}
\label{eq:max}
\nu(u+v)\preceq \max(\nu(u),\nu(v))
\end{equation}
whenever $u+v \neq 0$. 
It is immediate that \eqref{eq:max} becomes an equality if $\nu(u)\ne \nu(v)$. This slightly differs with the convention (see e.g., \cite{Waerden}~Ch. 18) that  maximum is replaced by minimum in \eqref{eq:max} as well as the inequality is opposite. This discrepancy can be explained by such classical examples of valuations as the $p$-adic ones and the lowest degree valuation of Laurent series, which use the opposite convention to \eqref{eq:max}, while our valuations stem from maximal degrees of polynomials.

In addition, if  $V$ is a $\kk$-algebra, we require $P$ to be a (partial) semigroup (cf. e.g., \cite{Arbieto}) with the operation $\circ$ (see Section~\ref{four} for details)
and 
$$\nu(uv)=\nu(u)\circ \nu(v)$$
whenever $uv\ne 0$ and $\nu(u)\circ \nu(v)$ is defined in $P$.

In particular, if $(P,\circ)$ is an entire (rather than partial) semigroup then the algebra $V$ is necessarily a domain, and conversely if $V$ is not a domain, then $(P,\circ)$ is necessarily a partial semigroup.

Also we impose the following condition on the order in $P$: if $c,c',d,d' \in P$ satisfy inequalities $c\preceq d, c'\preceq d'$
then 
\begin{equation}
\label{eq:axiom order}
c\circ c'\preceq d\circ d'    
\end{equation}
 provided that $c\circ c', d\circ d' \in P$. 

If $P$ was an entire  semigroup
then the axiom \eqref{eq:axiom order} would follow from weaker ones: $a\preceq b$ implies $c\circ a\preceq c\circ b$ and $a\circ b\preceq b\circ c$ (for ordered entire semigroups one can look in \cite{Blyth}). However, in partial semigroups,  \eqref{eq:axiom order} is not always derived (see, e.g., Example~\ref{monoid_matrix}).

 Note that this axiomatic is rather strong: if $P$ is an (entire) ordered monoid then for any non-unital invertible element $c$, the unit is strictly between $c$ and $c^{-1}$.

Additionally, we say that the order $\prec$ on a partial semigroup $P$ satisfies the {\it strict property} if $c\prec d$ or $c'\prec d'$ implies that
$c\circ c'\prec d\circ d'$ in \eqref{eq:axiom order}.

One can show (see e.g., Remark~\ref{image_valuation}) that for any valuation $\nu$ of $V$ the image $P_\nu:=\nu(V\setminus \{0\})$ is always a (partial) subsemigroup of $P$.

Following \cite{Kaveh-Khovanskii,Ku} we say that a valuation $\nu$ on (a vector space or an algebra) $V$ is {\it injective} if there exists a basis ${\bf B}$ of $V$ such that $\nu|_{\bf B}$ is an injective map ${\bf B}\hookrightarrow P$ (we refer to such a basis as {\it adapted} to $\nu$).

 Note also that for some finitely generated commutative  domains $A$ there is no injective valuation to an entire semigroup  (see Lemma~\ref{slope_injective}, Theorem~\ref{curve_injective}).

For contrast, we claim that injective valuations to reasonable partial semigroups always exist (in Section~\ref{four} we construct a class of {\it coideal} partial semigroups, see Definition~\ref{monoid_partial}, which will provide such reasonable valuations).

\begin{theorem} [Theorem~\ref{graded_partial}]
\label{th:commutative valued} Any finitely generated commutative algebra $A$ admits an injective valuation onto a coideal partial subsemigroup of $\ZZ_{\ge 0}^m$. Thereby, the order in the coideal partial subsemigroup satisfies the strict property.
\end{theorem}

In view of above, this means that  finite-dimensional algebras are also ``domains." 

Our next result extends Theorem~\ref{th:commutative valued} to the noncommutative case.

\begin{theorem} [Theorem~\ref{graded_partial}]
\label{th:noncommutative valued}
Any finitely generated algebra $A$ admits an injective valuation onto a coideal partial subsemigroup of $F_m$ 
(the free monoid on $m$ generators) with respect to (strict) deglex order (Lemma \ref{deglex} and Remark \ref{rem:deglex}).
    
\end{theorem}

We can refine Theorem \ref{th:commutative valued} by employing Gr\"obner basis-like approach (cf.  \cite{Kaveh-Manon}). Namely,  consider a commutative algebra $A$ generated by a finite set $X$ and a finite set ${\bf S}$ of monomials in $X$. We say that a monomial $b$ in $X$ is {\it standard} if $b$ does not contain  elements of ${\bf S}$ as submonomials. If the set ${\bf B}$ of all standard monomials is a basis of $A$, it is referred as a {\it standard monomial} one.

The following is a restatement of a well-known result (asserting that any finitely generated commutative algebra admits a  standard 
monomial basis) in terms of injective valuation into partial semigroups.

\begin{theorem} 
\label{th:adapted standard monomials}
Let $A$ be a finitely generated commutative algebra. Then 

(a) (Proposition~\ref{monomial_partial}) Any standard monomial basis ${\bf B}$ of $A$ defines a structure of an ordered partial semigroup on $\ZZ_{\ge 0}^X$ and an injective  valuation $\nu:A\setminus \{0\}\to \ZZ_{\ge 0}^X$ such that ${\bf B}$ is adapted to $\nu$. 

(b) (Theorem~\ref{monomial_basis_partial}).  Conversely, let  $\nu:A\setminus \{0\}\to \ZZ_{\ge 0}^m$ be an injective valuation such that $\nu(A\setminus \{0\})=\ZZ_{\ge 0}^m\setminus  [{\bf S}]$, where $[{\bf S}]:=\bigcup\limits_{s\in {\bf S}} (s+\ZZ_{\ge 0}^m)$ for some finite subset ${\bf S}$ of $\ZZ_{\ge 0}^m$.
Then there exists 
a standard 
monomial basis ${\bf B}$ of $A$ adapted to $\nu$.

\end{theorem}

A notable example is the (generalized) 
Stanley-Reisner algebra $A_{\bf S}$ defined for any finite subset ${\bf S}\subset \ZZ_{\ge 0}^m$ as follows. $A_{\bf S}$ is generated by $X=\{x_1,\ldots,x_m\}$  and has a presentation $x^s=0$ for all $s\in {\bf S}$. Clearly, the set $x^c$, $c\in \ZZ_{\ge 0}^n\setminus  [{\bf S}]$ form a standard monomial basis ${\bf B}$.

The complement $\ZZ_{\ge 0}^n \setminus [{\bf S}]$ has a structure of a partial semigroup via: for $c,c'\in \ZZ_{\ge 0}^n \setminus [{\bf S}]$ their sum $c+c'$ is defined unless it belongs to $[{\bf S}]$, which is an ideal of $\ZZ_{\ge 0}^n$.
Then the lexicographic (or deglex) order on $\ZZ_{\ge 0}^n$ gives rise to a (unique) injective valuation $\nu_{\bf S}$ on $A_{\bf S}$ via
$$\nu_{\bf S}(x^a)=a$$
for all $a\in \ZZ_{\ge 0}^n\setminus [{\bf S}]$. Clearly, ${\bf B}$ is adapted to $\nu_{\bf S}$.

A theory generalizing
 Buchberger's algorithm producing  
Gr\"obner bases of  ideals  was developed for certain classes of noncommutative algebras (see e.g. \cite{Li,Mora}). 
In our version we construct bases 
for injectively valued (noncommutative) algebras  as follows (some  difference in approach is that we focus mainly on bases of quotient algebras rather than of ideals). 

First, we will replace $\ZZ_{\ge 0}^n\setminus [{\bf S}]$ by a (not necessarily commutative) partial semigroup $(P,\circ)$, second, fix an injective valuation $\nu:A\setminus\{0\}\to P$ (and denote by $P_\nu$ the image of $\nu$, which is automatically a partial sub-semigroup of $P$), and, third, construct a linearly independent set ${\bf B}:=\{b_c,c\in P\}$ of certain monomials in $A$ which is {\it adapted} to $\nu$, i.e., $\nu({\bf B})=P_\nu$. It is critical for the construction that ${\bf B}$ is a basis of $A$ which is guaranteed for any adapted set whenever $P$ is well-ordered, see Corollary \ref{adapted_arbitrarily} (the assumption of well-orderness seems to be indispensable, see Remark \ref{rem:adapted is not basis}, this is yet another reason for using $\max$ in \eqref{eq:max}).

More precisely, let $P$ be a partial semigroup  and $P_0$ be a generating set of $P$.
A {\it factorization}  of $c\in P$     is any sequence $\vec{c}=(c_1,\ldots,c_\ell)\in P_0^\ell$ such that 
$$c=c_1\circ \cdots \circ c_\ell$$
in $P$; denote by $F(c)$ the set of all factorizations of $c$ of the shortest length (to be denoted $\ell(c)$). Now suppose that $P$ is well-ordered.
We refer to $\vec c\in R(c)$ as {\it standard} if it is the 
smallest in the lexicographic order on $R(c)$ and denote it by $\vec c^{\,st}$.

Then fix an injective valuation $\nu:A\setminus \{0\}\to P$, a generating set $(P_\nu)_0$ of $P_\nu$,  and choose  $x_{c_0}\in A\setminus \{0\}$ for all $c_0\in (P_\nu)_0$ such that $\nu(x_{c_0})=c_0$ for all $c_0\in (P_\nu)_0$. For any $c\in P_\nu$ and any $\vec c=(c_1,\ldots,c_\ell)\in D(c)$ denote $x_{\vec c}:=x_{c_1}\cdots x_{c_\ell}$. 

\begin{theorem} 
[Theorem~\ref{monomial_basis_partial}]
\label{th:adapted standard monomials NC} In assumptions as above, the set of all $x_c:=x_{{\vec c}^{st}}$, $c\in P$ is a basis of $A$ adapted to $\nu$.

\end{theorem}

This is a generalization of Theorem \ref{th:adapted standard monomials}(b). Note that  unlike in the commutative case, $P$ can be a small category, in particular, with countably many arrows. In that case, its generating set $P_0$ is a quiver (possibly with relations) so that $P$ is the set of all paths in $P_0$ (cf. Examples~\ref{monoid_ideal}, ~\ref{paths}).

A large class of injective valuations on $A$ can be constructed by restricting  a given injective valuation $\hat \nu$ on a larger algebra $B$ containing $A$ 
with some care. We say that a valuation $\nu:V\setminus\{0\}\to P$ is {\it well-ordered} if its range $P_\nu$ is a well-ordered set.

\begin{lemma} 
[Proposition~\ref{pr:injective subalgebra}]
\label{le:subalgebra valued} 
 Suppose that $\hat \nu$ is an injective well-ordered valuation of an algebra $B$. Then the restriction $\hat \nu$ to any  subalgebra $A$ is also a well-ordered injective valuation.  
\end{lemma}

The well-orderness of $P$ seems to be indispensable (Remark \ref{rem: subalgebra injective}).

In fact, for any ordered partial semigroup $P$ we can construct a ``default" invective valuation onto $P$ as follows.

First, recall that for any partial semigroup $P$ and any field $\kk$  we equip the linear span $\kk P=\bigoplus\limits_{c\in P} \kk\cdot [c]$ with an associative algebra structure via
\begin{equation}
\label{eq:kP}
[c][c']=\begin{cases}
[c\circ c'] &\text{if $c\circ c'$ is defined in $P$}\\
0 &\text{otherwise}
\end{cases}
\end{equation}
for all $c,c'\in P$. In particular, if $P$ is a group then $\kk P$ is the group algebra of $P$. In the case when $P=\ZZ_{\ge 0}^n\setminus[{\bf S}]$ in notation as above, then $\kk P$ is the Stanley-Reisner algebra $A_{\bf S}$.

Second, suppose that $P$ is ordered and $\kk$ is of characteristic $0$. Clearly, the assignments $[c]\mapsto c$, $c\in P$ define an injective valuation 
\begin{equation}
\label{eq:nu_P}
\nu_P:\kk P\setminus\{0\}\to P 
\end{equation}
(see e.g., Remark~\ref{monoidal_partial}). 
We refer to $\nu_P$ as the {\it tautological valuation} of $\kk P$. For instance, if $P=\ZZ_{\ge 0}^n$, i.e., $\kk P=\kk[x_1,\ldots,x_n]$ and the order is lexicographic, then $\nu_P$ is just the usual leading degree valuation of $\kk[x_1,\ldots,x_n]$. This also works if lex is replaced by deglex. Similarly, for any finite set $X$ denote  by $Ord(X)$ the set of all total orderings of $X$  and let $\kk<X>$ be the free algebra freely generated by $X$. Since $\kk<X>=\kk P_X$, where $P_X$ is the free monoid freely generated by $X$ and in view of Lemma \ref{deglex} which guarantees  a unique deglex order $\prec_\sigma$ on $P_X$ for any   $\sigma\in Ord(X)$, one obtains tautological valuations $\nu_\sigma:\kk<X>\setminus \{0\}\to P_X$. 

Using such tautological valuations, we refine Theorems \ref{th:commutative valued}, \ref{th:noncommutative valued} to the following important case of Theorem \ref{graded_partial}.

\begin{theorem}
\label{graded_partial intro}
Let $P$ be a well-ordered partial semigroup and  $I$ be a two-sided ideal of $\kk P$. 

i) Then $J:=\nu_P(I \setminus \{0\})$ is an ideal in $P$ (i.e., it is invariant under left and right compositions), thus $P\setminus J$ is a well-ordered (coideal) partial semigroup. Moreover,  the assignments 
$\nu(a+I):= \min\limits_{j\in I} \nu_P(a+j)$    
define an injective well-ordered valuation $\nu:\kk P/I\setminus \{0\}\twoheadrightarrow P\setminus J$.

ii)  Then ${\bf B}:= \{[c]+I\, :\, c\in P\setminus J\}$ is a standard monomial basis of $\kk P/I$ with respect to $\nu$.
\end{theorem}

Thus, combining Theorem \ref{graded_partial intro} with Lemma \ref{le:subalgebra valued} and \eqref{eq:nu_P}, we obtain a large class of injective valuations of commutative and noncommutative algebras into various entire and partial semigroups.

In particular, we obtain a rather spectacular consequence  for any finitely generated monoid or group $M$. Namely, if $X$ is a finite generating set of $M$, then one has a surjective homomorphism of monoids $P_X\twoheadrightarrow M$,  where $P_X$ is the monoid freely generated by $X$ as above. In turn, this induces a surjective homomorphism of algebras $\kk P_X\twoheadrightarrow \kk M$; denote by $I$ its kernel. Then for any $\sigma\in Ord(X)$ Theorem \ref{graded_partial intro} applied to $P=P_X$ and $I$ guarantees an embedding of $M$ into $P_X$ together with a new partial multiplication and entire order $M_\sigma$ on $M$ (even for finite groups $M$!). As a free bonus, one obtains the tautological valuation $\nu_\sigma:\kk M_\sigma \setminus \{0\}\to M_\sigma$ (see e.g., Theorem~\ref{rem:coideal of projection} and Remark~\ref{rem:order_min}, Remark~\ref{valuation_min}).

In example \ref{hecke} we construct a partial monoid $M(W)$ for any Coxeter group $W=<s_i,i\in I:s_i^2=1,(s_is_j)^{m_{ij}}=1\}$. Namely we require that $s_i\circ s_j$ is undefined iff $i=j$ or $m_{ij}=2$. If Coxeter graph is linear, $M(W)$ is frequently ordered unlike $W$.

Another surprising example (Example \ref{monoid_matrix}) is a partial semigroup $M_n$, the complete groupoid on $\{1,\ldots,n\}$, i.e., it consists of all pairs $(i,j)$ with $(i,j)\circ (j,k)=(i,k)$. By the construction, $\kk M_n=Mat_n(\kk)$.
It turns out that $M_n$ has many orders satisfying the strict property, the most notable of which is
\begin{equation}
\label{eq:order M_n}    
(1,n)\prec \ldots \prec(1,1)\prec(2,n)\prec\ldots \prec(2,1)\prec\cdots \prec(n,n)\prec\ldots \prec(n,1)\ .
\end{equation}
Thus, according to \eqref{eq:nu_P} we obtain an injective tautological valuation (Example \ref{monoid_matrix_algebra})
$$\nu_n:Mat_n(\kk)\setminus\{0\} \to M_n\ .$$
Therefore, restricting $\nu_n$ to any subalgebra $A\subset M_n$ we obtain an injective valuation $A\to M_n$. In particular, for any faithful representation $\rho:G\to GL_n(\kk)$ of a finite group $G$, we obtain an ``injective valuation" $\nu:G\to M_n$ which takes $1$ to $(n,n)$ under the order \eqref{eq:order M_n}. Under an appropriate partial semigroup structure on $G$, this $\nu$ becomes an injective homomorphism of partial monoids.

Returning to general partial semigroups, we can construct new valuations from   $P$-filtered algebras and vice versa. We say that $A$ is filtered by an ordered (partial) semigroup $P$ if $A=\sum\limits_{c\in P} A_{\preceq c}$, $A_{\preceq c'}\subset A_{\preceq c}$ whenever $c'\preceq c$, and  $A_{\preceq c} A_{\preceq c'}\subset A_{\preceq c\circ c'}$ whenever $c\circ c'$ is defined in $P$. Proposition \ref{partial_filtration} asserts that this is essentially a valuation $A\setminus \{0\}\to P$.
In particular, if $P$ is well-ordered, by the standard procedure $gr$ this defines a $P$-graded algebra (recall that $A$ is graded by a (partial) semigroup $P$ if $A=\bigoplus\limits_{c\in P} A_c$ so that $A_c A_{c'}\subset A_{c\circ c'}$ whenever $c\circ c'$ is defined in $P$). 
Thus, Proposition \ref{partial_filtration} applied to $A:=\kk P$ recovers the tautological valuation $\nu_P:\kk P\setminus\{0\}\to P$.

It turns out that having a pair of injective valuations $\nu, \nu'$ of an algebra (or even a vector space) $A$ to partial semigroups $P$ and $P'$ gives an interesting information about both the algebra and the pair $P_\nu=\nu(A\setminus\{0\}), P'_{\nu'}=\nu'(A\setminus\{0\})$. 

\begin{theorem}
[Theorem \ref{th:generalized JH}] 
\label{th:JH}
Suppose that $\nu:A\setminus \{0\}\to P$ and $\nu':A\setminus \{0\}\to P'$ are injective valuations and $P$ and $P'$ are well-ordered. Then the assignments $a\mapsto \min\{\nu'(\nu^{-1}(a))\}$ define a bijection ${\bf K}_{\nu',\nu}:P_\nu\widetilde \to P'_{\nu'}$. Moreover, ${\bf K}_{\nu',\nu}^{-1}={\bf K}_{\nu,\nu'}$.
\end{theorem}

We call ${\bf K}_{\nu',\nu}$
a {\it  Jordan-H\"older
bijection (JH-bijection)}. It can be reformulated in terms of the generalized Jordan-H\"older correspondence
on matroids developed by Abels in 1991 \cite{BGW}, cf. also Remark~\ref{flag}.

In addition, under the same assumptions there is a common adapted basis for both valuations.

\begin{theorem}
[Theorem~\ref{th:generalized JH}]
\label{intro:JH_basis}
Under assumptions of Theorem~\ref{th:JH},  there exists a basis ${\bf B}_{\nu,\nu'}$ of $A$ adapted to both $\nu$ and $\nu'$ and such that
${\bf K}_{\nu',\nu}(\nu(b))=\nu'(b)$ for all $b\in {\bf B}_{\nu,\nu'}$.
\end{theorem}

We sometimes refer to such a  basis as an {\it JH-basis} of $A$. 

The following result asserts that any JH-bijection is almost a homomorphism of partial semigroups.

\begin{theorem}
[Proposition~\ref{sub-multiplicative}] 
\label{sub-multiplicative intro}Under assumptions of Theorem \ref{th:JH}
the JH-bijection ${\bf K}:={\bf K}_{\nu',\nu}$ is   sub-multiplicative in the following sense:
$${\bf K}(c\circ c')\preceq {\bf K}(c)\circ {\bf K}(c')$$
whenever $c\circ c'$ and ${\bf K}(c)\circ {\bf K}(c')$ are defined in $P$ and $P'$, respectively.
\end{theorem}

This  implies that ${\bf K}_{\nu,\nu'}={\bf K}_{\nu',\nu}^{-1}$ is also sub-multiplicative, which will, in particular, allow to stratify both $P_\nu$ and $P'_{\nu'}$ into ``multiplicativity domains" (see Examples~\ref{ex:mod 2},
and ~\ref{2.13 non-commutative}). 

In fact, this happens frequently when $P$ and $P'$ are quantum cones, i.e,  central extensions of cones (being finitely generated submonoids of $\ZZ_{\ge 0}^n$), see Example \ref{ex:quantum plane} for details. 

We conclude with a spectacular application of our JH theory to quantum Schubert cells $U_q(w)$ (for any symmetrizable Kac-Moody algebra $\gg$) viewed as algebras over $\AA=\ZZ[q^{\frac{1}{2}},q^{\frac{1}{2}}]$, which we abbreviate as $U^\AA(w)$ (Section \ref{3.13}). Namely, using injectivity of Feigin's homomorphism $U_q(w)\to A_\ii$ for any reduced word $\ii=(i_1,\ldots,i_m)$ of $w$ (Theorem \ref{th:Feigin}) we obtain an injective valuation $\nu_\ii$ from $U^\AA(w)$ to a quantum cone $P_\ii$ determined by a skew-symmetric bilinear form $\Lambda_\ii$ (Corollary \ref{cor:nu_ii}). Replacing $\ii$ by another reduced word $\ii'$ for $w$, we obtain  JH-bijections ${\bf K}^-_{\ii',\ii}: C_\ii\to C_{\ii'}$ of quantum cones, in the notation of \eqref{eq:ii,iip JHq}. 

Passing to underlying string valuations $\underline \nu_\ii$, $\underline \nu_{\ii'}$, we recover the underlying bijection of ``thin" string cones (Remark \ref{rem:thin thick}) $\underline{\bf K}^-_{\ii',\ii}: \underline C_\ii\to \underline C_{\ii'}$, in the notation of \eqref{eq:ii,iip JH}. Based on the formula \eqref{eq:underline K-} in Example \ref{ex:quantum A2} we expect that $\underline{\bf K}^-_{\ii',\ii}$ generalize string transition maps from \cite[Theorem 2.2]{bz}. It turns out that each linearity domain of (hopefully, piecewise linear) bijection $\underline{\bf K}^-_{\ii',\ii}$ becomes a ``symplectomorphism" between $\Lambda_\ii$ and $\Lambda_{\ii'}$
on their linearity domains (Corollary \ref{cor:Lambda-equivariange}(a)). In fact, this linear map $\underline{\bf K}^-_{\ii',\ii}$ becomes an "honest" symplectomorphism of $\Lambda_\ii$ in view of Remark \ref{rem:beru}. 

Likewise, we construct another class of injective valuations $\nu^\ii$ from $U^\AA$ to a quantum octant $P^\ii$ determined by a skew-symmetric bilinear form $\Lambda^\ii$ together with JH-bijections ${\bf K}^+_{\ii',\ii}:P^{\ii}\to P^{\ii'}$
together with their underlying counterparts $\underline {\bf K}^+_{\ii',\ii}:\ZZ_{\ge 0}^m\to \ZZ_{\ge 0}^m$ (Theorem \ref{th:Qii} and formulas \eqref{eq:ii,iip JH}, \eqref{eq:ii,iip JHq}). Based on the formula \eqref{eq:underline K+} in Example \ref{ex:quantum A2} we expect that $\underline{\bf K}^+_{\ii',\ii}$ generalize celebrated Lusztig's transition maps from \cite[Section 2]{lu0}. And  $\underline{\bf K}_{\ii',\ii}^+$ also becomes a``symplectomorphism" between $\Lambda^\ii$ and $\Lambda^{\ii'}$ on their linearity domains (Corollary \ref{cor:Lambda-equivariange}(b)).

Finally, and most importantly, we obtain JH-bijection ${\bf K}_{\ii',\ii}:P^\ii\to C_{\ii'}$ which interpolates between valuations $\nu^\ii:U^\AA(w)\to P^\ii$ and $\nu_{\ii'}:U^\AA(w)\to P_{\ii'}$ together with its underlying version $\underline {\bf K}_{\ii',\ii}:\ZZ_{\ge 0}^m\to C_{\ii'}$. Based on the formula \eqref{eq:underline K} from Example \ref{ex:A3}, we expect that the latter one is closely related to the twist map  between the reduced double Bruhat cells (see e.g., \cite[Section 4.3]{BZm}). And we should expect that the quantum JH-bijection ${\bf K}_{\ii',\ii}$ is related to the quantum twist (\cite{KO}). And, of course,  $\underline{\bf K}_{\ii',\ii}$ becomes a``symplectomorphism" between $\Lambda^\ii$ and $\Lambda_{\ii'}$ on their linearity domains (Corollary \ref{cor:Lambda-equivariange}(c)).

\subsection*{Acknowledgments}
 Part of this work was done during visits by A.B. to Heidelberg University, Max Planck Institute for Mathematics in the Sciences, IHES, and University of Geneva. He thanks these institutions for hospitality and Anna Wienhard, Maxim Kontsevich, and Anton Alekseev for fruitful discussions.

\section{Injective valuations on algebras with zero divisors}\label{four}

In this section we extend the concept of valuations to algebras with (possibly) zero divisors.

\subsection{Partial semigroups}

\begin{definition}
\label{monoid_partial}
We say that $(P,\circ)$ is a (not necessary commutative) {\it partial semigroup} if for some elements $c,d\in P$ their composition $c\circ d\in P$ is defined (in this case we say that $c, d$ are composable) satisfying the following property (of associativity). If two elements $c\circ c', (c\circ c')\circ c''\in P$ are defined then $c'\circ c'', c\circ (c'\circ c'')\in P$ are also defined and $(c\circ c')\circ c''=c\circ (c'\circ c'')$. Vice versa also holds (we refer to such a triple as {\it associative}). 

We say that a subset $J\subset P$ is an ideal in $P$ if for any elements $c\in P, d\in J$ it holds $c\circ d\in J$, provided that $c\circ d\in P$, and similarly, $d\circ c\in J$, provided that $d\circ c \in P$. Then $P$ becomes a partial semigroup via $c\circ_J d$ defined iff $c\circ d$ is defined and belongs to $P\setminus J$. Clearly,   $P\setminus J$ is a partial semigroup of $P$ under $circ_J$. If $M$ is a semigroup and $J\subset M$ is an ideal we call $M\setminus J
$ a {\it coideal partial semigroup}.

 We assume that $P$ is endowed with a linear order $\prec$ satisfying the following property. For $c,c',d,d'\in P$ the inequalities $c\preceq d, c'\preceq d$ 
imply that $c\circ c'\preceq d\circ d'$ (sometimes we consider partial semigroups without apriori linear order which we introduce afterwards).

We say that the order  fulfills a {\it strict property} if for any elements $a,b,c,d \in P$ such that $a\prec b, c\preceq d$ it holds $a\circ c \prec b\circ d$ (respectively, $c\circ a \prec d\circ b$), provided that $a\circ c, b\circ d \in P$ (respectively, provided that $c\circ a, d\circ b \in P$).
  \end{definition}
  
\begin{remark} 
\label{rem:absorbent}
A partial semigroup $P$ can contain a {\it left absorbing} element $\bf 0$, i.e., such that ${\bf 0}\circ P=\{{\bf 0}\}$.
Existence of a left (or right) absorbing element implies that $P$ is entire because if ${\bf 0}\circ c,\ {\bf 0}\circ d$ are defined then both $c\circ d,\ d\circ c$ are defined since $({\bf 0}\circ c)\circ d={\bf 0}\circ (c\circ d)$ due to Definition~\ref{monoid_partial}. Therefore, if $P$ is not entire, then it {\bf cannot} contain absorbing elements.

If $P$ is entire with a left absorbing element ${\bf 0}$ and $P$ is strictly ordered, then it is easy to see that  ${\bf 0}$ is minimal in $P$.
In some cases this results in a trivial composition in $P$. For instance, consider a coideal partial monoid $P':=\{x^iy^j\ :\ 0\le i,j,\ ij=0\}$ with deglex order and natural composition (see Remark~\ref{monoid_ideal}~i)). We claim that it cannot be extended to a semigroup let $P=P'\sqcup \{\bf 0\}$ with the condition that ${\bf 0} \circ P=\{\bf 0\}$. Indeed, 
$x^i \circ x^j =\bf 0$ in $P$ for all $i,j$ because $x^i\circ x^j \prec x^i \circ y^{j+1}=\bf 0$. Note, however, that if one replaces deglex with lex with all $y^i\prec x$, then such an ordered $P$ exists if the composition of elements of $P'$ was given by $x^i\circ x^j=x^{i+j}$ and all others were undefined.

Note, however, that if $P$ is ordered and contains a (left) absorbing element ${\bf 0}$, this imposes severe restrictions on the order. For instance, if ${\bf 0}$ is minimal in $P$ and $a\circ a={\bf 0}$ for some $a\in P$, then  $b\circ c={\bf 0}$ for all $b\prec a$, $c\prec a$, in particular, if $a$ was maximal, then $P\circ P=\{{\bf 0}\}$.

We can always convert a partial semigroup $P$ by adjoining to it 
a {\it two-sided} absorbing element, i.e., both left- and right-absorbing. Namely, we define a composition $\bullet$ on 
$P\sqcup \{\bf 0\}$ via
$c\bullet c'=\begin{cases} 
c\circ c' &\text{if $c,c'\in P$ and they are composable} \\
{\bf 0} &\text{otherwise} \\    
\end{cases}$. Clearly, $(P\sqcup \{\bf 0\},\bullet)$ is an entire semigroup iff $P$ is a (partial) semigroup.

\end{remark}

\begin{definition}
 \label{def:entire_homomorphism}   

We say that a
map $f:P\to Q$ of partial semigroups $P,Q$ is a homomorphism if 
for any $c,d\in P$ it holds that  $c\circ d$ is defined iff $f(c)\circ f(d)$ is defined as well, and in this case  the equality
$f(c\circ d)=f(c)\circ f(d)$ is true.
\end{definition}

\begin{example} Let $f:C:=\QQ^n \to D:=\QQ^n$ be a piecewise linear bijection (cf. \cite{GR}). This induces a partition $C=\sqcup_i P_i$ into (open in their closures) polyhedra such that a mapping $f|_{P_i}$ is linear for each $i$. One can define an order $\prec_C$ on $C$ by means of a vector $u\in \RR^n$ with $\QQ$-linear independent coordinates:  $v_1\prec_C v_2,\ v_1, v_2 \in C$ iff $(u, v_1-v_2)<0$. This induces an order on $D$ by $v_1\prec_C v_2$ iff $f(v_1)\prec_D f(v_2)$.

Endow a structure of a partial monoid $(C, \circ_C)$ as follows: for $v_1, v_2 \in P_i$ the composition $v_1\circ_C v_2:= v_1+v_2$ is defined iff $v_1+v_2\in P_i$. Then the partition $D=\sqcup_i f(P_i)$ with the induced order makes $(D, \circ _D)$ a partial monoid: $f(v_1)\circ_D f(v_2)= f(v_1)+f(v_2)$ is defined for $v_1, v_2 \in P_i$ iff $v_1+v_2\in P_i$. In the latter case it holds $f(v_1)+f(v_2)=f(v_1+v_2)\in f(P_i)$.

Thus, $f$ is an isomorphism of the ordered partial monoids $C, D$.

More generally, if $f:C\to D$ is an isomorphism of partial semigroups, and $C$ is ordered then $f$ induces an order on $D$ such that $f$ is an isomorphism of ordered partial semigroups.

\end{example}

\begin{remark}
i) For a partial semigroup $P$ we call $P_0\subset P$ a subsemigroup of $P$ if for any $c,d\in P_0$ it holds $c\circ d\in P_0$ whenever $c\circ d\in P$. Any subset $R\subset P$ generates the uniquely defined minimal subsemigroup ${\overline R}\subset P$ such that $R\subset \overline R$. 

ii) If $f:P\to Q$ is a homomorphism of partial semigroups then the image $f(P)$ is a subsemigroup of $Q$.
\end{remark}

The following are immediate applications of two-sided absorbing elements in notation of Remark \ref{rem:absorbent}.

\begin{lemma} The following are equivalent for a given map $f:P\to Q$ one has

(i) $f:P\to Q$ is a homomorphism of partial semigroups

(ii) the extension $\hat f:P\sqcup \{{\bf 0}\}\to Q\sqcup \{{\bf 0}\}$ of $f$ via $\hat f({\bf 0})={\bf 0}$ is a homomorphism of entire semigroups.
    
\end{lemma}

\begin{lemma}
\label{le:absorbent}
Let $P$ be an (entire) semigroup $P$ with a two-sided absorbing element ${\bf 0}$ and let $J$ be an ideal $P$. Then 

(a) the assignments  $c\mapsto 
\begin{cases}
{\bf 0} &\text{if $c\in J$}\\
c &\text{otherwise}
\end{cases}$ 
define an epimorphism $\pi_J:P\twoheadrightarrow (P\setminus (J\setminus\{{\bf 0}\})$ of entire semigroups.

(b) Suppose additionally that $P$ is ordered so that ${\bf 0}$ is minimal (which is automatic if $\prec$ is strict). Then $\pi_J$ is a homomorphism of ordered 
semigroups iff $J$ is an interval, i.e. for any $c\in J$ the set $\{d\in P:d\preceq c\}$ is contained in $J$.
\end{lemma}

As it happens in other aspects of algebra, $\pi_J$ plays a role of the quotient homomorphism, e.g., any homomorphism $f:P\to Q$ of semigroups with two-sided absorbing elements factors through $P\to (P\setminus (J\setminus\{{\bf 0}\})\to Q$, where $J=Ker~f=f^{-1}({\bf 0})$.

For any partial semigroup $P$ we say that a subset $S\subset P\times P$ is {\it admissible} if it defines a partial semigroup on $P$. The following is immediate.

\begin{lemma} The intersection of any family of admissible subsets of $P$ is admissible
    
\end{lemma}

This defines an admissible closure of any $X\subset P\times P$ to be the intersection of all admissible subsets containing $X$. This means that any $X\subset P\times P$ defines a canonical partial semigroup on $P$ so that pairs $(a,b)\in X$ are composable which we denote by $P_X$.

\begin{definition} We say that a mapping $f:P\to Q$ of partial semigroups $P,Q$  is a {\it partial} homomorphism if the set $S_f$  of all pairs $(a,b)\in P\times P$ such that $a,b$ are composable in $P$ and  $f(a),f(b)$ are composable in $Q$ is admissible,  
in addition we require that $f(a\circ b)=f(a)\circ f(b)$ for $(a,b)\in S_f$.
    \end{definition}

\begin{remark}\label{admissible}
If $f:P \to Q$ is a partial homomorphism of partial semigroups then $f:P_{S_f}\to Q$ is a homomorphism (see Definition~\ref{def:entire_homomorphism}).    
\end{remark}

\begin{remark}
    If $\prec$ is a linear order on a semigroup $M$ then it induces a linear order on a coideal partial semigroup $P\subset M$.
\end{remark}

\begin{proposition}\label{monoid_epimorphism}
    Let $P$ be  an ordered (partial) semigroup and $Q$ be a partial semigroup. Suppose that there is an order on $Q$  viewed as a set and let  $f$ be a surjective homomorphism $P\twoheadrightarrow Q$ of partial semigroups and an order preserving map at the same time. Then $Q$ is an ordered partial semigroup and $f$ is an ordered homomorphism.
 \end{proposition}

  Given ordered partial semigroups $P$ and $Q$, we say that a map $f:P\to Q$ is sub-multiplicative if $f(c\circ c')\preceq f(c)\circ f(c')$ whenever $c,c'$ are composable in $P$  and $f(c),f(c')$ are composable in $Q$.

  \begin{proposition} 
  \label{pr:strict partial homomorphism}
  Let $P,Q$ be ordered partial semigroups and $f:P\to Q$ be sub-multiplicative.
  Suppose that the order on $Q$ has a strict property. Then $f$ is a partial homomorphism.

  \end{proposition}

\begin{proof} It suffices to show that that $S_f$ is admissible. Indeed, suppose that $(c,c')\in S_f$ and $(c\circ c', c'')\in S_f$. Then 
$$f(c\circ c')=f(c)\circ f(c'),\ f((c\circ c')\circ c'')=f(c\circ c')\circ f(c'')=f(c)\circ f(c')\circ f(c'').$$
\noindent Thus, $c', c''$ and $c, c'\circ c''$ are composable, as well as $f(c'), f(c'')$ and $f(c), f(c')\circ f(c'')$ are composable. Therefore
$$f((c\circ c')\circ c'')=f(c\circ (c'\circ c''))\preceq f(c)\circ f(c'\circ c'') \preceq f(c)\circ f(c')\circ f(c'').$$
\noindent Since both the latter inequalities are actually, equalities, we get that $(c, c'\circ c'')\in S_f$, finally the strict property of $Q$ implies that $f(c'\circ c'')=f(c')\circ f(c'')$, thus $(c',c'')\in S_f$. The admissibility of $S_f$ is established.
\end{proof}
  
We say that a partial semigroup $P$ is free if the set $S_P$ of indecomposable elements of $P$
generates $P$ freely.  In this case there is a canonical embedding (which is {\bf not} a homomorphism, cf. Definition~\ref{def:entire_homomorphism}) $\iota_P:P\hookrightarrow <S_P>$, where $<S>$ is the free monoid freely generated by a set $S$. This defines the set $M_P$ of monomials in $S_P$ forbidden in $P$.
Note that $\iota_P$ turns $P$ into a coideal subsemigroup of $<S_P>$ under $\circ_{J_P}$ where $J_P$ is the ideal of $<S_P>$ generated by $M_P$ (cf. Definition \ref{monoid_partial}). Conversely, any set $M$ of monomials (with no singletons) in a given set $S$ defines a unique free partial semigroup $P$ with $S_P=S$.

Note, however, that there is no homomorphism from  $<S_P>$ to $P$ in contrast with the
canonical homomorphism $\kk<S_P>\twoheadrightarrow \kk P$ (see \eqref{eq:kP})
whose kernel is the two-sided ideal generated by $M_P$.

The following is immediate. 

\begin{lemma}\label{free_partial_semigroup} In the assumptions as above, any total order $\prec'$ on  the free monoid $<S_P>$ induces a unique order on $P$ via $c\prec d$ iff $\iota_P(c)\prec' \iota_P(d)$.
\end{lemma}

\begin{problem}
Can any given order on $P$  be obtained in this way? 
\end{problem}

\subsection{Quivers and homomorphisms of ordered partial semigroups}

Given a quiver $\Gamma$ with the vertex set $S$, denote by $\hat P_\Gamma$ the partial semigroup whose only relations are as follows: $(c,d)$ is composable iff $(c,d)$ is an edge of $\Gamma$. We sometimes refer to $\hat P_\Gamma$ as a $\Gamma$-free (partial) semigroup.

The following is immediate.

\begin{lemma}\label{path_realization} [Path realization]

For any quiver $\Gamma$ the $\Gamma$-free semigroup $\hat P_\Gamma$ is exactly the semigroup (under concatenation) of all paths in $\Gamma$ i.e., sequences of elements of $S$. Furthermore, any total ordering of $S$ defines a unique order on the (partial) semigroup $\hat P_\Gamma$. In particular,  $\hat P_\Gamma$ is finite iff $\Gamma$ is acyclic.

\end{lemma}

For any partial semigroup $P$ its generating set $S$ defines a simple quiver $\Gamma(P)$ (possibly with loops) whose vertex set is $S$ and edges are composable pairs of $S$.

Furthermore, we say that an element $(c_1,c_2,\ldots,c_k)\in \hat P_{\Gamma(P)}$ is $P$-admissible if $c_1\circ c_2\circ \cdots \circ c_k$ is defined in $P$; denote by $\hat P=(\hat P,S)$ the set of all $P$-admissible elements in  $\hat P_{\Gamma(P)}$. 

\begin{theorem}
\label{th:universal covering semigroup}  For any generating set of any (partial) semigroup and $P$ one has:

(a) $\hat P$ is naturally a partial semigroup generated by $S$.

    (b) (universal cover) The assignments $s\mapsto s$, $s\in S$ define a surjective homomorphism of partial semigroups $\pi_P:\hat P\twoheadrightarrow P$.
\end{theorem}

\begin{proof}
Note that a composition $s_1 \circ_{\hat P} \cdots \circ_{\hat P} s_l$ is defined in $\hat P$, where $s_1,\dots,s_l\in S$ iff $\pi_P(s_1)\circ \cdots \circ \pi_P(s_l)$ is defined in $P$.   
\end{proof}

The following result provides a converse statement to Proposition~\ref{monoid_epimorphism} under the strict property of the order.

 \begin{theorem}\label{semigroup_epimorphism} Let $P$ be any ordered partial semigroup whose order $\prec$ satisfies the strict property. Then in the notation of Theorem \ref{th:universal covering semigroup}, $\hat P$ has an order such that $\pi_P:\hat P\twoheadrightarrow P$ is an ordered homomorphism.
 \end{theorem}

 \begin{proof} 
  
Introduce an order $\triangleleft$ on $\hat P$ as follows. We say that $\hat P\ni v:=s_{i_1}\circ_{\hat P} \cdots \circ_{\hat P} s_{i_m} \triangleleft w:= s_{j_1}\circ_{\hat P} \cdots \circ_{\hat P} s_{j_p}\in \hat P,\ s_{\cdot}\in S$ iff either

 $\bullet$ $\pi_P(v)\prec \pi_P(w)$, either

 $\bullet$ $\pi_P(v)=\pi_P(w)$ and $m<p$, or

 $\bullet$ $\pi_P(v)=\pi_P(w), m=p$ and the word $v$ is less than $w$ in the lexicographical order (defined on $u_1,\dots,u_k$ in an arbitrary way). Denote the length $l(v):=m$.

 We claim that the epimorphism $\pi_P$ fulfills Definition~\ref{def:entire_homomorphism}. Indeed, let $v,w,v_1,w_1\in \hat P; v\trianglelefteq w, v_1\trianglelefteq w_1; v\circ_{\hat P} v_1, w\circ_{\hat P} w_1\in \hat P$. If either $\pi_P(v)\prec \pi_P(w)$ or $\pi_P(v_1)\prec \pi_P(w_1)$ then $\pi_P(v\circ_{\hat P} v_1)=\pi_P(v)\circ \pi_P(v_1) \prec \pi_P(w)\circ \pi_P(w_1)=\pi_P(w\circ_{\hat P} w_1)$ due to the assumption in the theorem. If $\pi_P(v)=\pi_P(w), \pi_P(v_1)=\pi_P(w_1)$ and either $l(v)<l(w)$ or $l(v_1)<l(w_1)$ then $l(v\circ_{\hat P} v_1)<l(w\circ_{\hat P} w_1)$. Finally, if $l(v)=l(w), l(v_1)=l(w_1)$ then the word $v\circ_{\hat P} v_1$ is less than $w\circ_{\hat P} w_1$ in the lexicographical order, unless $v=w, v_1=w_1$.
 \end{proof}

Using an argument similar to that of the proof of Theorem \ref{semigroup_epimorphism}, we establish the following.

\begin{lemma}\label{semigroups_product} Let $P_1$ and $P_2$ be any partial semigroups. Then 

(a) their direct product $P_1\times P_2$ is also a (partial) semigroup. 
Moreover, if $P_1$ and $P_2$ are ordered so that the ordering on $P_1$ fulfills the strict property
then 
$P_1\times P_2$ is ordered as well via $(p_1,p_2)\preceq (p'_1,p'_2)$ iff either $p_1\prec p'_1$ or $p_1=p'_1$ and $p_2\preceq p'_2$. 

(b) Moreover, if $P_2$ also fulfills the strict property then $P_1\times P_2$ fulfills the strict property as well.
\end{lemma}

Now we study an inverse issue of inducing an order on the image of an ordered partial semigroup under an epimorphism.

\begin{theorem}\label{intervals_ordered}  Let $P$ be an ordered (partial) semigroup and $Q$ be a partial semigroup. Suppose that there is an order on $P$  viewed as a set and let  $f$ be a surjective homomorphism $P\twoheadrightarrow Q$.  of partial semigroups. 

Then the following are equivalent

 (a) $Q$ admits an order so that $f$ is a homomorphism of ordered (partial) semigroups.

 (b) Any fiber $f^{-1}(q)$ is an interval in the order of $P$.
    
\end{theorem}

The proof relies on the following easy lemma.

 \begin{lemma} 
 \label{le:ordered projections} Let $P$ be an ordered set, $Q$ be  a set, and $f:P\twoheadrightarrow Q$. 

 Then the following are equivalent.

 (a) $Q$ admits an order so that $f$ is an ordered map.

 (b) Any fiber $f^{-1}(q)$ is an interval in the order of $P$.
     
 \end{lemma}

 \begin{theorem}
Let $(P, \circ)$ be an entire ordered semigroup, and $f:P \twoheadrightarrow Q$ be an epimorphism onto an entire semigroup $Q$. Then there exist  unique ordered entire semigroup $P_f$, epimorphism $f_1: P_f \twoheadrightarrow Q$ and an epimorphism of ordered semigroups $f_0: P\twoheadrightarrow P_f$ such that $f=f_1 \circ f_0$, and $P_f$ is the minimal possible with these properties. Namely, for any ordered semigroup $R$, epimorphism $g_1: R \twoheadrightarrow Q$ and an epimorphism of ordered semigroups $g_0 : P\twoheadrightarrow R$ such that $f=g_1 \circ g_0$, there exist epimorphisms of ordered semigroups $h_1: R\twoheadrightarrow P_f,\ h_0: P\twoheadrightarrow R$ for which it holds $f_0=h_1 \circ h_0$.
 \end{theorem}

 \begin{proof} For any element $q\in Q$ consider its preimage $f^{-1} (q) \subset P$, and partition it into intervals (with respect to the order in $P$). Further we consecutively refine this (initial) partition of $P$ into intervals by transfinite induction. For inductive step we suppose that a (current) partition of $P$ into intervals is given.

 Fix for the time being an element $a\in P$. Consider an element $I\subset P$ of the current partition. Note that $f(I)\in Q$ is a singleton. Denote the map $\varphi_a :P\to P$ by $\varphi_a (b):=a\circ b$. For each interval $J\subset P$ from the current partition take the partition into intervals of the set $I\cap \varphi_a^{-1} (\varphi_a (I)\cap J)$. Performing the described procedure for all pairs $I,J$, results in a refinement of the current partition of $P$ into intervals. Thus, we have performed a step of the induction.

 Also, we perform another similar step of induction corresponding to the map $\psi_a:P \to P,\ \psi_a(b):= b\circ a$. The steps of induction are performed for all possible $a\in P$.

 The result of the induction (when no further steps produce proper refinements of partitions) is a final partition of $P$ into intervals. One can verify that for any pair of intervals $I,J\subset P$ from the final partition their composition $I\circ J$ is contained in a suitable interval of the final partition. Therefore, the intervals of the final partition  form a semigroup $P_f$, being ordered due to
Theorem~\ref{intervals_ordered}. The required epimorphisms $f_0, f_1$ are naturally induced.

If an ordered semigroup $R$ satisfies the conditions of the Theorem, then for any its element $r\in R$ the preimage $g_0^{-1}(r)\subset P$ is an interval (due to Theorem~\ref{intervals_ordered}). Moreover, $g_0^{-1}(r)\subset I$ for an appropriate interval $I$ of the final partition of $P$ (see above). This implies existence of the required epimorphisms $h_0, h_1$.  
 \end{proof}

\begin{remark}\label{order_image}
Let $P,Q$ be partial semigroups, and $\prec$ be an order on $P$. Assume that $f:P\to Q$ is an epimorphism of partial semigroups. We say that a pair of elements $u,v \in Q$ is {\it colliding} if there exist $a_1, a_2, b_1, b_2 \in P$ such that $f(a_1)=f(a_2)=u,\ f(b_1)=f(b_2)=v$ and $a_1\prec b_1,\ a_2\succ b_2$. Let $I\subset Q$ be a subset such that any pair of elements $u,v\notin I$ is not colliding. Denote by $\overline{I}\subset Q$  an ideal generated by $I$. Then the co-ideal partial semigroup $Q\setminus \overline I$ is ordered as follows: for $u,v\in Q,\ a,b\in P,\ f(a)=u,\ f(b)=v$ we define $u\triangleright v$ iff $a\prec b$.
\end{remark}

Now we suggest an alternative (to Remark~\ref{order_image}) way of inducing an order on the image under a homomorphism of an ordered partial semigroup.

\begin{theorem}\label{rem:coideal of projection}

Given an ordered partial semigroup $P$, for any homomorphism $f:P\to Q$ of partial semigroups denote by $P(f)$ the set of all $c\in P$ such that $c\preceq f^{-1}(f(c))$, i.e., $c$ is minimal in the fiber $f^{-1}(f(c))$.

(a) If $P$ is entire and the order satisfies the strict property, then the operation $c\circ_f d:=cd$ whenever $c,d,c d\in P(f)$ defines an ordered partial semigroup structure on $P(f)$. Moreover, the restriction of $f$ to $P(f)$ is an injective map $P(f)\hookrightarrow Q$ which is almost a homomorphism in a sense that  $f(c\circ_f d)=f(c)f(d)$ whenever $c,d,cd\in P(f)$. 

(b) The map $P(f)\hookrightarrow Q$ is a homomorphism iff $f(P)$ is ordered and $f:P\to f(P)$ is an ordered homomorphism.

\end{theorem}

\begin{proof} It suffices to prove (a). For $b\in Q$ denote $b_P:= \min _{\prec} f^{-1}(b)\in P(f)$, provided that $b_P$ is defined. Note that for $b,d\in Q$ it holds $(b\circ_Q d)_P \preceq b_P \circ d_P$, and $b\circ _f d$ is defined iff $(b\circ_Q d)_P = b_P \circ d_P$. 

It suffices to verify associativity of $\circ _f$. Suppose the contrary. Then there exist 
$b,d,e \in Q$ such that the compositions $b\circ _f d$ and $(b\circ _f d)\circ _f e$ are defined, while $d\circ _f e$ is not defined. Taking into account that $\prec$ satisfies the strict property, we get that
$$b_P\circ  (d\circ _Q e)_P \prec b_P \circ (d_P \circ e_P) = (b_P \circ d_P)\circ e_P = (b\circ _Q d)_P \circ e_P =((b\circ _Q d)\circ _Q e)_P,$$
\noindent which leads to a contradiction. Therefore, $d\circ _f e$ is defined, and $b_P \circ (d\circ _Q e)_P = (b\circ _Q (d\circ _Q e))_P$, hence $b\circ _f (d\circ _f e)$ is also defined and
$$b\circ _f (d\circ_f e)= b_P \circ (d_P \circ e_P)= (b_P \circ d_P)\circ e_P = (b\circ _f d)\circ _f e.$$ 
The theorem is proved. \end{proof}

\begin{example}
We follow the notations from Theorem~\ref{rem:coideal of projection}.
Let $Q:=D_n=\langle s,t\ |\ s^2=t^2=(st)^n=1\rangle$ be the dihedral group, $P:=F_2=\langle s,t\rangle$ be the free monoid with deglex ordering $\prec$ in which $t\prec s$, and $f:P\twoheadrightarrow Q$ be the natural epimorphism. Denote
$$a_k:=\underbrace{stst\dots}_k,\ c_k:=\underbrace{tsts\dots}_k \in P.$$
\noindent Then $P(f)= \{a_k\ :\ 0\le k<n\} \sqcup \{c_k\ :\ 1\le k\le n\}$. It holds in $P(f)$:
$$a_{2k}\circ a_l=a_{2k+1}\circ c_{l-1}=a_{2k+l}\ \text{when}\ 2k+l<n;$$
$$c_{2k}\circ c_l=c_{2k+1}\circ a_{l-1}=c_{2k+l}\ \text{when}\ 2k+l\le n.$$
\noindent All other compositions in $P(f)$ are undefined.
\end{example}

The following is immediate.

\begin{lemma}\label{basis_partial_monoid} Let $f:P\twoheadrightarrow Q$ be an (ordered) epimorphism of (ordered) partial semigroups. Suppose that the fibers of $f$ are well-ordered. Then in the notation of \eqref{eq:kP} the assignments $x\mapsto \min\{f^{-1}(x)\}$ define a section $f^*:Q\hookrightarrow P$ of $f$. In turn, this defines a vector space decomposition $\kk P=\kk f^*(Q)\oplus I$ where $I$ is the kernel of the canonical homomorphism $\kk P\to \kk Q$. Also, all elements $y-f^*(x)$, $y\in f^{-1}(x)$, $y\ne f^*(x)$ form a basis ${\bf B}_I$ of $I$. 
    \end{lemma}

We say that a function $\ell:P\to \ZZ_{>0}$ is {\it length} if $\ell(c\circ c')=\ell(c)+\ell(c')$ for all composable $c,c'\in P$. We say that $(P,\ell)$ is a {\it graded} partial semigroup if $\ell$ is a length on $P$ (sometimes we omit $\ell$).

We say that an order $\prec$ on a graded partial semigroup is {\it length compatible} if $\ell(c)<\ell(c')$ implies that $c\prec c'$.

The following is immediate

\begin{lemma}\label{deglex} (Generalized deglex) Let $P$ be a free semigroup freely generated by a set $X$. Then 

(a) Any function $f:X\to \ZZ_{>0}$ defines (unique) length function on $P$ and vice versa.

(b) For any length function $\ell:P\to \ZZ_{>0}$ any total order $\prec$  of $X$ such that $\ell(x)<\ell(y)$ implies $x\prec y$, determines a unique length compatible order (fulfilling the strict property) on $P$ such that $xa\prec yb$ whenever $x,y\in X$, $\ell(xa)=\ell(yb)$ and $x\prec y$ or $x=y$ and $a\prec b$.

(c) If for any $m\in \ZZ_{>0}$ the preimage $f^{-1}(m) \subset X$ is finite then $\prec$ is a well ordering.

    \end{lemma}

\begin{remark} 
\label{rem:deglex}
If $\ell(x)=1$ for any generator of $P$ (e.g., when $P=F^n_+$ or $P=\ZZ^n_{\ge 0}$) this becomes an ordinary deglex on $P$.
    
\end{remark}

Denote by $P_1* P_2=P_2*P_1$
the free product of (partial) semigroups $P_1$ and $P_2$. 

Suppose that $P_1$ and $P_2$ are ordered. We say that that an order on the partial semigroup $P_1*P_2$ is  {\it compatible} if $p\prec p'$ implies $p*c\prec p'*c$ and $c*p\prec c*p'$ 
for any $c\in P_1*P_2$ and any $p,p'\in P_i$, $i=1,2$.

If $P_1$ and $P_2$ are entire then there are several constructions of the order on $P_1*P_2$ (see e.g. \cite{Bergman}). In particular, for an ideal $J_1$ in $P_1$ and an ideal $J_2$  of $P_2$, any such order restricts to an order on the free product $(P_1\setminus J_1)*(P_2\setminus J_2)$ of coideal partial semigroups.

The following immediate fact gives another construction of an order on free products of graded partial semigroups (making use of Lemma~\ref{deglex}). 

\begin{lemma}
    
\label{le:semigroups*product} Let $P_1$ and $P_2$ be any graded partial semigroups. Then their free product $P_1* P_2=P_2*P_1$ is also a graded (partial) semigroup. 

\end{lemma}

\begin{remark}\label{rem:order_min}
For an entire monoid $M$ generated by a finite set $X$ denote by $f: \langle X \rangle \twoheadrightarrow M$ the natural epimorphism, where $\langle X \rangle$ denotes the free monoid generated by $X$. Any order $\sigma$ on $X$ induces deglex order on $\langle X \rangle$ (see Lemma~\ref{deglex}). Then $M_{\sigma}:= (M, \circ _f)$ is a partial monoid produced in Theorem~\ref{rem:coideal of projection}. Clearly, $(M, \circ _f)$ depends on $\sigma$.

Denote by $\underline M$ a partial monoid having $M$ as the set of its elements, and defined composition of a pair of elements $m_1, m_2 \in \underline M$ iff their composition is defined in $M_\sigma$ for every order $\sigma$.
\end{remark}

\subsection{Constructions of ordered partial semigroups}

\begin{example}\label{monoid_ideal}
i) For a monoid $M:=\ZZ_{\ge 0}^n=\{x_1^{i_1}\cdots x_n^{i_n}\, :\, i_1,\dots,i_n\ge 0\}$ and a family of monomials $u_1,\dots,u_s$ in the variables $x_1,\dots,x_n$, the set of monomials not dividing any of  $u_1,\dots,u_s$, forms a coideal partial monoid $P(u_1,\dots,u_s)$. Then $P(u_1,\dots,u_s)$ coincides with the complement to the monomial ideal $J(u_1,\dots,u_s):=\bigcup_{1\le j\le s} (u_j+\ZZ_{\ge 0}^n)$. \vspace{1mm}

ii) For a free monoid $F_n:=\langle x_1,\dots, x_n\rangle$ consider an ideal $J:=\langle x_ix_jx_k\, :\, 1\le i,j,k\le n\rangle$. We define a well-ordering $\prec$ on $M$ as follows. For a pair of words $u,v \in M$ we say that $u\prec v$ if either $u$ is shorter than $v$ or they have the same length and $u$ is lower than $v$ with respect to the lexicographical order in which $x_n \prec \cdots \prec x_1$ (see Lemma~\ref{deglex}). 
Then $P:=M\setminus J$ is a finite coideal partial monoid. \vspace{1mm}

iii) For a free monoid $F_n:=\langle x_1,\dots, x_n\rangle$ consider an ideal $J_n:=\langle x_jx_i\, :\, j\ge i\rangle$. Then $P_n:=F_n\setminus J_n$ is a finite coideal partial monoid consisting of $2^n$ elements of the form $u:=x_{i_1}\cdots x_{i_k}, 1\le i_1<\cdots < i_k\le n$. For an element $v:=x_{j_1}\cdots x_{j_l}, 1\le j_1<\cdots < j_l\le n$ the composition $u\circ v\in P_n$ iff $i_k<j_1$.

\end{example}

\begin{example}\label{paths}
i) Now we modify the construction of Example~\ref{monoid_ideal}~iii) and produce a partial monoid $Q_n$ coinciding as a set with $P_n$ and equipped with the same ordering $\prec$. The composition law in $Q_n$ differs  from the one in $P_n$: let 
$u:=x_{i_1}\cdots x_{i_k}, 1\le i_1<\cdots < i_k\le n, v:=x_{j_1}\cdots x_{j_l}, 1\le j_1<\cdots < j_l\le n$ be two elements of $Q_n$, then
$$u\circ v= x_{i_1}\cdots x_{i_k} \circ x_{j_2}\cdots x_{j_l}$$
\noindent iff $i_k=j_1$; otherwise the composition is not defined.

The partial monoid $Q_n$ is isomorphic to the following partial monoid $R_n$. The generators of $R_n$ are $\{y_{i,j}\, :\, 1\le i<j\le n\}$. The composition $y_{i,j}\circ y_{k,l}$ is defined iff $j=k$. The isomorphism of $R_n$ and $Q_n$ is established by mapping of $y_{i_1,i_2}\circ y_{i_2,i_3}\circ \cdots \circ y_{i_{k-1},i_k}$ to $x_{i_1}\circ \cdots \circ x_{i_k}$.  \vspace{1mm}

ii) One can yield a family of partial submonoids of $R_n$ as follows. Consider a directed acyclic graph $G$ with $n$ vertices numbered by $\{1,\dots,n\}$ in such a way that for any arrow $(i,j)$ of $G$ it holds $i<j$. Then one can consider a partial submonoid $R_G$ of $R_n$ generated by the elements $\{y_{k,l}\}$ for which there is a path from a vertex $k$ to a vertex $l$ in $G$. \vspace{1mm}

iii) Alternatively, one can consider a partial submonoid $T_G$ of $R_G$ of paths in $G$ (cf. Lemma~\ref{path_realization}). A partial monoid $T_G$ is generated by the elements $\{z_{k,l}\}$ where $(k,l)$ is an arrow in $G$ (cf. \cite{Pin}).

More generally, one can consider a partial monoid $T_G$ of paths in an arbitrary directed graph $G$ (when $G$ contains cycles, $T_G$ is infinite). One can treat $T_G$ as a coideal partial submonoid of the free monoid $F_G$ generated by  $\{z_{k,l}\}$ where $(k,l)$ is an arrow in $G$. Then $T_G=F_G \setminus J_G$ where the ideal $J_G$ is generated by all compositions of the form $z_{k,l}\circ z_{i,j}$ where $l\neq i$.
\end{example}

\begin{example}\label{Coxeter}
Denote by $W_{\bf 0}(m), m\ge 1$ the (partial) nil-Coxeter semigroup generated by $s_1,s_2$ satisfying the following relations: 
$$s_1\circ s_1, s_2\circ s_2 \notin W_{\bf 0}(m),\, \underbrace{s_1\circ s_2\circ s_1\circ s_2\dots}_{m}=\underbrace{s_2\circ s_1\circ s_2\circ s_1\dots}_{m}.$$
\noindent Then $W_{\bf 0}(m)$ consists of $2m-1$ elements: for each $1\le k<m$ it contains two elements 
$$c_k:=\underbrace{s_1\circ s_2\circ s_1\circ s_2\dots}_{k},\, d_k:=\underbrace{s_2\circ s_1\circ s_2\circ s_1\dots}_{k}$$
\noindent of length $k$, and in addition the element $c_m=d_m$. 

The following compositions are defined in $W_{\bf 0}(m)$:
$$c_{2k}\circ c_l= c_{2k+l}, c_{2k+1}\circ d_l= d_{2k+l}, d_{2k}\circ d_l=d_{2k+l}, d_{2k+1}\circ c_l= d_{2k+l+1},$$
\noindent provided that $2k+l\le m$ or respectively, $2k+l+1\le m$. All other compositions are not defined. One can verify that $W_{\bf 0}(m)$ is a partial semigroup with an order defined by $d_k\prec c_k\prec d_{k+1}, 1\le k<m$. 

Observe that $\prec$ does not satisfy the strict property since $c_m=c_{m-1}\circ c_1=d_{m-1}\circ d_1$ when $m$ is odd, and $c_m=c_{m-2}\circ c_2=d_{m-2}\circ d_2$ in case of an even $m$.

However (cf. Theorem~\ref{semigroup_epimorphism}), one can construct a
partial semigroup $\overline{W_{\bf 0}(m)}$ generated by two elements $\overline{s_1}, \overline{s_2}$ such that $\overline{s_1}\circ \overline{s_1},\ \overline{s_2}\circ \overline{s_2}$ and $\underbrace{\overline{s_1}\circ \overline{s_2}\circ \cdots}_{m+1},\ \underbrace{\overline{s_2}\circ \overline{s_1}\circ \cdots}_{m+1}$ are not defined. Thus, $\overline{W_{\bf 0}(m)}$ consists of $2m$ elements of the form either $\overline{c_k}=\underbrace{s_1\circ s_2\circ \cdots}_k$ or $\overline{d_k}=\underbrace{s_2\circ s_1\circ \cdots}_k, 1\le k\le m$. The epimorphism $f:\overline{W_{\bf 0}(m)} \twoheadrightarrow W_{\bf 0}(m)$ sends $f(\overline{c_k})=c_k, f(\overline{d_k})=d_k, 1\le k\le m$. Thus, $f$ is not injective just on two elements: $f(\overline{c_m})=f(\overline{d_m})=c_m=d_m$. The order satisfying the strict property on $\overline{W_{\bf 0}(m)}$ is defined by $\overline{c_k} \triangleleft \overline{d_k}\triangleleft \overline{c_{k+1}}, 1\le k<m$ and in addition, $\overline{c_m} \triangleleft \overline{d_m}$. \vspace{1mm}

Note also  that in case of $W_{\bf 0}(m)$ one cannot replace the axiom from Definition~\ref{monoid_partial} by weaker axioms that $c\preceq d$ implies that $b\circ c \preceq b\circ d, c\circ b\preceq d\circ b$, provided that $b\circ c, b\circ d, c\circ b, d\circ b$ are defined.
\end{example}

\begin{remark}\label{non-strict} Note however, that there is no order on the corresponding braid monoid $Br_3^+=<T_1,T_2:T_1T_2T_1=T_2T_1T_2>$. Indeed, suppose without loss of generality that $T_1\prec T_2$. 

If  $T_1T_2 \prec T_2T_1$ then $(T_1T_2)T_1 \preceq T_2T_1^2\preceq (T_2T_1)T_2=T_1T_2T_1$, we get a contradiction.

If $T_2T_1 \prec T_1T_2$ then $ T_1(T_2T_1) \preceq T_1^2T_2\preceq T_2(T_1T_2)=T_1T_2T_1$, again we get a contradiction.

Thus, the braid monoid $Br^+_n$ and hence the braid group $Br_n$ are not orderable.

Analogously, it is easy to see that there is no order on the corresponding Garside submonoid of $Br_3$ generated by $u_{12}=T_1$, $u_{23}=T_2$, $u_{13}=T_1T_2T_1^{-1}=T_2^{-1}T_1T_2$ and thus subject to
$u_{13}u_{12}=u_{12}u_{23}=u_{23}u_{13}$.

Similarly, one can show that there is no order $\prec$ on $W_{\bf 0}(3)$ (as well as on  $W_{\bf 0}(m), m\ge 3$) satisfying the strict property (cf. Example~\ref{Coxeter}). Indeed, otherwise, let $s_1\prec s_2$ for definiteness. If $s_1\circ s_2\prec s_2\circ s_1$ then $(s_1\circ s_2)\circ s_1 \prec (s_2\circ s_1)\circ s_2$, otherwise if $s_2\circ s_1 \prec s_1\circ s_2$ then $s_1\circ (s_2\circ s_1)\prec s_2\circ (s_1\circ s_2)$, in both cases we get a contradiction.    
\end{remark}

\begin{example}\label{hecke}
i) For a Coxeter group $W$ given by generators $s_1,\dots, s_n$ satisfying the relations $s_i\circ s_i=1, s_i\circ s_j=s_j\circ s_i, |i-j|\ge 2$ and
\begin{equation}\label{braid}
\underbrace{s_i\circ s_{i+1}\circ s_i \cdots}_{m_i}= \underbrace{s_{i+1}\circ s_i\circ s_{i+1} \cdots}_{m_i},\ 1\le i<n
\end{equation}
denote by $M:=M(W)$ a partial monoid given by the generators $s_1,\dots, s_n$ for which $s_i\circ s_i, s_i\circ s_j, |i-j|\ge 2$ are undefined and relations \eqref{braid} hold. 

Then one can represent $M=M_{reg} \sqcup M_{exc}$. Here $M_{reg}$ is isomorphic to a partial monoid of paths (cf. Example~\ref{paths}) in the graph with vertices $v_1,\dots, v_n$, edges $(v_i,v_{i+1}),\ (v_{i+1},v_i),\ 1\le i<n$ such that subpaths
$$\underbrace{v_i v_{i+1} v_i \cdots}_{m_i},\ \underbrace{v_{i+1} v_i v_{i+1} \cdots}_{m_i},\ 1\le i<n$$
\noindent are avoided. The set $M_{exc}$ consists of all elements 
$$U_i:=\underbrace{s_i\circ s_{i+1}\circ s_i \cdots}_{m_i},\ 1\le i<n.$$ \vspace{1mm}

ii) Consider the case $n=3, m_1=m_2=4$. Then $M$ is infinite since it contains an arbitrary prefix of an infinite word $(s_1\circ s_2\circ s_3\circ s_2)^{\infty}$. Clearly, $M$ does not admit an order satisfying the strict property because $(s_1\circ s_2)\circ (s_1\circ s_2)= (s_2\circ s_1)\circ (s_2\circ s_1)$. 

On the other hand, one can introduce an order $\prec$ on $M$ as follows. First, if $|c|<|d|,\ c,d\in M$ then $c\prec d$. Secondly, when $|c|=|d|,\ c,d\in M_{reg}$, we impose $c\prec d$ iff $c$ is less than $d$ in lex ordering in which $s_1\prec \cdots \prec s_n$. Finally, $U_1\prec c\prec U_2$ for any word $c\in M_{reg}, |c|=4$. \vspace{1mm}

iii) Now let $m_1=\cdots= m_{n-1}=3$. This corresponds to the symmetric group $W=S_{n+1}$. In this case $M$ is finite. Indeed, $M_{reg}$ consists of the words
$$C_{kl}:=s_k\circ s_{k+1}\circ \cdots \circ s_l,\ D_{kl}:= s_l\circ s_{l-1}\circ \cdots \circ s_k,\ 1\le k\le l\le n.$$
\noindent Clearly, $C_{kk}=D_{kk}=s_k$. In particular, $|W|=n(n+1)$.

Remark~\ref{non-strict} shows that $M$ does not admit an order satisfying the strict property. However, one can introduce an order $\prec$ on $M$ as follows. First, if $|c|<|d|,\ c,d\in M$ then $c\prec d$. Secondly, it holds $s_1\prec\ \cdots \prec s_n$ and 
$$D_{kl}\succ C_{kl} \succ D_{k-1, l-1},\ 1\le k<l\le n$$
\noindent (whenever it is defined). Finally, we impose
$$C_{k, k+2}\succ U_k\succ D_{k-1, k+1},\ 1\le k<n$$
\noindent (whenever it is defined). \vspace{1mm}

iv) One can verify that $M(W)$ is infinite iff there exist $1\le k<l<n$ such that $m_k\ge 4, m_l\ge 4,\ m_p=3, k<p<l$ (cf. items ii), iii)).

\end{example}

{\bf Question}. For which Coxeter groups $W$ the partial monoid $M(W)$ admits an order? \vspace{2mm}

The following two propositions provide constructions of extending partial semigroups.

\begin{proposition}\label{glue}
Let $P$ and $Q$ be partial semigroups.
 Then $P'=P\sqcup Q$ is a partial semigroup with the composition inherited from $P$ and $Q$ and $pq=qp=q$ for all $p\in P$, $q\in Q$.
Suppose that $P$ and $Q$ are ordered such that 

i) $q\preceq qq'$ and $q\preceq q'q$ for all $q,q'\in Q$ (this property is called positive ordering, see, e.g. \cite{SA}, \cite{Pin}). Then the assignments $p\prec q$ for $p\in P$ and $q\in Q$ turn $P'$ into an ordered partial semigroup;

ii) $q\succeq qq'$ and $q\succeq q'q$ for all $q,q'\in Q$. Then the assignments $p\succ q$ for $p\in P$ and $q\in Q$ turn $P'$ into an ordered partial semigroup.
    \end{proposition}

Note that when $P,Q$ are commutative partial semigroups, the resulting $P'$ is commutative as well. The next proposition allows one to construct non-commutative partial semigroups from arbitrary (in particular, commutative) ones.

\begin{proposition}\label{buffer}
Let $P,Q$ be partial semigroups. Consider a partial semigroup $R:=Q\sqcup \{x\} \sqcup \{y\} \sqcup P$ defined as follows:
$$xz=x, yz=y, z\in \{P,Q,x,y\};\, Pz_1=y, z_1\in \{x,y,Q\};\, Qz_2=x, z_2\in \{x,y,P\}.$$
\noindent Then $R$ is a partial semigroup with an ordering $Q\prec x \prec y\prec P$.
\end{proposition}

When in Propositions~\ref{glue},~\ref{buffer} $P,Q$ are entire semigroups, the results of constructions are entire semigroups as well. In contrast, there are no entire finite semigroups satisfying the strong property. 

\begin{proposition}
There are no entire finite semigroups (with more than one element) satisfying the strong property.    
\end{proposition}

\begin{proof}
Suppose the contrary. If for some element $a$ of the semigroup it holds $a\prec a^2$ then $a^i\prec a^{i+1}, i\ge 1$, which leads to a contradiction. By the same token an assumption $a\succ a^2$ leads to a contradiction as well. Thus, $a=a^2$ for any element $a$.

For any pair of elements $a\prec b$ it holds $a=a^2\prec ab$, hence $ab\prec ab^2=ab$. The obtained contradiction completes the proof.
\end{proof}

Now we concoct a construction for extending partial monoids satisfying the strict property.

\begin{proposition}\label{strict_extend}
Let $P$ be a partial monoid with an order $\prec^0$ satisfying the strict property. For $x\notin P$ construct a partial monoid 
$$Q:= P \sqcup P\circ x \sqcup \dots \sqcup P\circ x^{\circ k}$$
\noindent such that $x\circ P, x^{\circ (k+1)}$ are not defined. We set the order $\prec$ on $Q$ as follows:
$$P\circ x^{\circ i} \prec P\circ x^{\circ (i+1)}, 0\le i<k\ \text{and}$$
\noindent
$$c\circ x^{\circ j} \prec d\circ x^{\circ j} \text{iff}\ c\prec^0 d,\ c,d\in P,\ 0\le j\le k.$$
\noindent Then $Q$ is  a partial monoid satisfying the strict property.

Alternatively, one could set the order as $P\circ x^{\circ i} \succ P\circ x^{\circ (i+1)}$.
\end{proposition}

\begin{proof}
To verify the strict property consider elements $u,v,w,t\in Q$ such that $u\preceq v, w\preceq t$ and at least one of two latter inequalities is strict. We assume that $u\circ w, v\circ t\in Q$. When  $u,v,w,t\in P$, the strict property follows from the strict property for $P$. Otherwise, $v\in P,\ t=t_0\circ x^{\circ i}$ for some $1\le i\le k$. Therefore, $u\in P, w=w_0\circ x^{\circ j}$ for suitable $0\le j\le i$. If $j<i$ then $u\circ w \prec v\circ t$. Otherwise, if $j=i$ then it holds $w_0\preceq^0 t_0$. Since one of two inequalities  $u\preceq^0 v,\ w_0\preceq^0 t_0$ is strict, we deduce from the strict property for $P$ that $u\circ w_0\prec^0 v\circ t_0$, which implies the strict property for $Q$:
$$u\circ w=u\circ w_0\circ x^{\circ i}\prec v\circ t_0\circ x^{\circ i}=v\circ t.$$

By the same token one considers an alternative order $P\circ x^{\circ i} \succ P\circ x^{\circ (i+1)}$.
\end{proof}

\begin{remark}
One can generalize Proposition~\ref{strict_extend}  to partial semigroups (rather than monoids). For a partial semigroup $P$ satisfying the strict property consider a partial semigroup $Q:= \bigsqcup_{0\le i\le k} (i,P)$ (where $(i,P)$ is a copy of $P$), in which the product is defined as $(0,p_0)\circ (i,p):= (i,p_0\circ p),\ p_0,p\in P,\ 0\le i\le k$, and $(j,P)\circ (i,P)$ is not defined when $j>0$. The order in $Q$ is lexicographical with respect to $i$ and to the order in $P$. As in the proof of Proposition~\ref{strict_extend} one can verify that $Q$ fulfills the strict property.     
\end{remark}

We conclude the section with orderings on the complete groupoid $M_k$ om $\{1,\ldots,k\}$

\begin{example}
\label{monoid_matrix}
Consider a partial semigroup $M_k:=\{(i,j)\, :\, 1\le i, j\le k\}$
where $(i,j)\circ (j,l):= (i,l)$ and $(i,j)\circ (m,l)$ is not defined when $m\neq j$. We define a linear order $\prec$ on $(i,j)\in M_k$ being lexicographical with respect to a vector

i) $(j-i,-i)$ or

ii) $(-i, j)$.

In both cases $M_k$ endowed with $\prec$ satisfies Definition~\ref{monoid_partial}. Moreover, $\prec$ satisfies the strict property (see Definition~\ref{monoid_partial}).

Note that one cannot take a vector $(i,j)$ in place of vectors from either i) or ii) since the induced ordering does not satisfy Definition~\ref{monoid_partial}. 

Observe that the axiom of the order from Definition~\ref{monoid_partial} for $M_k$ cannot be deduced from weaker axioms:  $c\preceq d$ implies that $a\circ c \preceq a\circ d$ and $c\circ a \preceq d\circ a$, provided that $a\circ c, a\circ d, c\circ a, d\circ a\in M_k$. 

It would be interesting to clarify whether $M_k$ can be represented as a coideal semigroups (with a compatible ordering). 

\end{example}

\begin{example} 
\label{ex:block matrices}
More generally, let ${\bf d}=(d_1,\ldots,d_m)\in \ZZ_{>0}^m$. This defines a new order on $M_n$, $n=d_1+\cdots+d_k$ as follows. For any $i\in [1,n]$ denote by $k_i$ the only $k$ such that $d_1+\cdots+d_{k-1}< i\le d_1+\cdots+d_k$. Also denote $r_i:=i-(d_1+\cdots+d_{k_i-1})$. Clearly, $r_i\in [1,d_{k_i}]$.

Then we write $(i,j)\prec (i',j')$ iff $(k_i,k_j)\prec_1 (k_{i'},k_{j'})$ whenever $(k_{i'},k_{j'})\ne  (k_i,k_j)$ and $(r_i,r_j)\preceq_2 (r_{i'},r_{j'})$ otherwise, where $\preceq_1, \preceq_2$ are any orders from Example \ref{monoid_matrix}.
    
\end{example}

\subsection{Valuations of algebras in partial semigroups}

\begin{definition}\label{valuation_partial}
For a (not necessarily unital) $\kk$-algebra $A$ a mapping $\nu:A\setminus \{0\} \to P$ to a partial semigroup $P$
is a {\it valuation} if for any $a,b\in A\setminus \{0\}$ it holds the following:

i) $\nu(\kk^\times a)=\nu(a)$;

ii) $\nu(a+b)\preceq \max \{\nu(a), \nu(b)\}$, provided that $a+b\neq 0$;

iii) $\nu(ab)=\nu(a)\circ \nu(b)$, provided that  $\nu(a)\circ \nu(b)\in P$ (in particular, in this case it holds $ab\neq 0$).
\end{definition}

\begin{remark}
\label{image_valuation}

Clearly, iii) asserts that $\nu$ is a homomorphism of partial semigroups where  $ A\setminus \{0\}$ is naturally a partial semigroup under multiplication. In particular, $\nu( A\setminus \{0\})$ is a (partial) sub-semigroup of $P$.
For example, if $A=\kk Q$ for some (partial) semigroup $Q$, then the restriction $\nu|_Q$ is an epimorphism of (partial) semigroups $Q\twoheadrightarrow P_\nu$. 

\end{remark}

A homomorphism $\underline f:P\to Q$  of partial semigroups induces a natural homomorphism of algebras $f:\kk P \to \kk Q$ by $f([c]):=[\underline f(c)]$ (note however, that if $\underline f$ was a partial homomorphism $P\to Q$, then $f:\kk P\to \kk Q$ is {\bf not} an algebra homomorphism). 

The following  fact is an extension of the above assertion and its converse when the codomain contains  a {\it two-sided} absorbing element ${\bf 0}$ in notation of Lemma 
\ref{le:absorbent}, (e.g., it is necessarily entire). 

\begin{lemma}
\label{le:monomial homomorphism}
Let $f:\kk P\to \kk Q$ be a homomorphism such that $f([P])\subset [Q]\sqcup\{0\}$. Then the assignments $[c]\mapsto 
\begin{cases}
{\bf 0} & \text{if $f([c])= 0$} \\
f([c]) &\text{otherwise}
\end{cases}
$ define a homomorphism of partial semigroups $\underline f:P\to Q\sqcup\{{\bf 0}\}$, where ${\bf 0}$ is the two-sided absorbing element.
Also, set $J_f=\{c:f([c])=0\}=\underline f^{-1}({\bf 0})$ 
is an ideal of $P$ and the restriction of $\underline f$ to $P\setminus J_f$ is a homomorphism of partial semigroups $P\setminus J_f\to Q$.

\end{lemma}

Note however that if $P$ is a partial semigroup and $J$ is an ideal of $P$ 
then $\kk J$ is an ideal of $\kk P$ and 
for the coideal $Q=P\setminus J$ the semigroup algebra $\kk Q$ is naturally isomorphic to the quotient $\kk P/ \kk  J$.

We can illustrate Lemma \ref{le:monomial homomorphism} by the natural projection $f_{n,k}:Mat_{n,k}(\kk)\to Mat_k(\kk)\oplus Mat_{n-k}(\kk)$ given by 
$\begin{pmatrix}
 A & C\\
 0 & B
\end{pmatrix}\mapsto \begin{pmatrix}
 A & 0\\
 0 & B
\end{pmatrix}$
for any $A\in Mat_k(\kk)$, $B\in Mat_{n-k}(\kk)$, $C\in Mat_{k\times (n-k)}$, where $Mat_{n,k}(\kk)$ stands for upper $2\times 2$ block-triangular matrices. 

Clearly, $Mat_{n,k}(\kk)=\kk M_{n,k}$ where $M_{n,k}$ is the set of $(ij)\in M_n$ with $(i,j)\in [1,n]\times [1,n]\setminus ([k+1,n]\times [1,k])$ 
in notation of Example \ref{monoid_matrix}. Clearly, this is a (partial) sub-semigroup of $M_n$. Then one has the underlying homomorphism $\underline {f_{n,k}}:M_{n,k}\to M_k\sqcup M_{n-k}\sqcup \{{\bf 0}\}$.

\begin{remark} 
If $f:A\to B$  a (not necessary injective)   homomorphism of algebras and let $\nu:B\setminus \{0\}\to Q$ be a valuation, then we can extend it to  $\nu_f:A\setminus\{0\}\to Q$ if $Q$ contains a two-sided absorbing element ${\bf 0}$ ( e.g., $Q$ is necessarily entire) by $\nu_f(a):=\begin{cases}
\nu(f(a)) & \text{if $f(a)\ne 0$}\\
{\bf 0} & \text{otherwise}
\end{cases}$.

Furthermore, applying this argument to situation of Lemma \ref{le:monomial homomorphism}, we obtain the valuation $\nu_f:\kk P\to \kk (Q\setminus \{{\bf 0}\})$ whose restriction to $[P]$ is the homomorphism
$\underline f:P\to Q$ whose kernel $\underline f^{-1}({\bf 0})$ is ideal $J_f$ of $P$.

Assume that  $P$ is ordered and $\underline f$ is order-preserving. Then Lemma \ref{le:absorbent} implies that  $J_f$ is an interval.

 We introduce an order $\prec$ on $M_{n,k}$ being a modification of the order from Example~\ref{monoid_matrix}. 

 $\bullet$ For $(i,j)\in ([1,k]\times [1,k])\sqcup ([k+1,n]\times [k+1,n])$ we stick with ii) in Example \ref{monoid_matrix}, i.e. $(i_1,j_1)\prec (i_2,j_2)$ if either $i_1<i_2$ or $i_1=i_2$ and $j_1>j_2$. 

 $\bullet$ For $(t,s)\in [1,k] \times [k+1,n]$ we also stick with ii) in Example \ref{monoid_matrix}, i.e. $(t_1,s_1)\prec (t_2,s_2)$ if either $t_1<t_2$ or $t_1=t_2$ and $s_1>s_2$.

 $\bullet$ $(t,s) \prec (i,j)$.

The latter condition states that the kernel of the above homomorphism $\underline{f_{n,k}}$ is an interval in the partial semigroup $M_{n,k}$. However, $\underline{f_{n,k}}$ is not a homomorphism of ordered partial semigroups (cf. Lemma~\ref{le:absorbent}~(b)).

More generally, in notation of Example \ref{ex:block matrices}
let $M_{\bf d}$ be a partial sub-semigroup $M_n$ of which consists of block upper triangular matrices, that is, 
$M_{\bf d}$ consists of all $(ij)$ with $k_i\le k_j$. Then any order on $M_n$ defined in Example \ref{ex:block matrices} restricts to $M_{\bf d}$. The above homomorphism $f_{n,k}$ naturally 
generalize to $f_{\bf d}:\kk M_{\bf d}\to Mat_{d_1}(\kk)\oplus \cdots \oplus Mat_{d_m}(\kk)$ 
as well as $\underline f_{n,k}$ to 
$\underline f_{\bf d}:M_{\bf d}\to (M_{d_1}\times \cdots \times M_{d_m})\cup {\bf 0}$.

\end{remark}

\begin{example}
Following Example~\ref{monoid_ideal}~i) one can consider the monoidal algebra $\kk P(u_1\dots,u_s)$. It is called a Stanley-Reisner algebra in case when all $u_1,\dots,u_s$ are square-free.

\end{example}

The following result allows to construct new (not necessarily injective) valuation.

\begin{proposition}
\label{pr:valuation after homomorphism}    
 Let $\nu:A\setminus \{0\}\to P$ be a valuation and let $f:P\to Q$ be an ordered homomorphism of partial semigroups. Then $f(\nu):A\setminus \{0\}\to Q$ given by $a\mapsto f(\nu(a))$ is also a valuation. 
    
\end{proposition}
\begin{proof} For $a_1,a_2\in A\setminus \{0\}$ assume that the composition of $f(\nu(a_1)), f(\nu(a_2))$ is defined in $Q$. Then the composition of $\nu(a_1), \nu(a_2)$ is defined in $P$, hence $\nu(a_1a_2)=\nu(a_1)\circ \nu(a_2)$ and $f(\nu(a_1a_2))=f(\nu(a_1)\circ \nu(a_2))=f(\nu(a_1))\circ f(\nu(a_2))$.

Now assume also that $a_1+a_2\neq 0$. Then $\nu(a_1+a_2)\preceq \max\{\nu(a_1), \nu(a_2)\}$. Let $\nu(a_1+a_2) \preceq \nu(a_1)$ for definiteness. Therefore, $f(\nu(a_1+a_2))\preceq f(\nu(a_1))$.

\end{proof}

\begin{proposition}  
\label{pr:partial homomorphism} Consider a partial semigroup $P$ and an ordered partial semigroup $Q$.
Let $\nu:\kk P\setminus\{0\}\to Q$ be a valuation.
Then the mapping $c\mapsto \nu([c])$ is a partial homomorphism $P\to Q$. In particular, $P$ acquires a new structure of a partial semigroup $P_{S_f}$ in notation of Remark \ref{admissible}.

    \end{proposition}

\begin{proof} 
Take $c, c', c'' \in P$ such that $\nu([c]), \nu([c'])$ are composable and that $\nu([c])\circ \nu([c']), \nu([c''])$ are also composable. Let $S_\nu$ be the set of all $(c,c')\in P\times P$ such that $\nu([c]), \nu([c'])$ are composable.
Then due to Definition~\ref{monoid_partial} it holds that $\nu([c']), \nu([c''])$ are composable and that $\nu([c]),  \nu([c'])\circ \nu([c''])$ are composable as well. Thus, $S_{\nu}$ is admissible due to Definition~\ref{valuation_partial}.

If $(c,c')\in S_{\nu}$ then $c, c'$ are composable in $P_{S_\nu}$ and $\nu([c\circ c'])=\nu([c])\circ \nu([c'])$ again due to Definition~\ref{valuation_partial}. 

\end{proof}

Applying Proposition~\ref{pr:lex tensor product}~(b) we obtain the following immediate

\begin{corollary} \label{tensor}
Let $A_i$, $i=1,2$ be algebras and $\nu_i:A_i\setminus \{0\}\to P_i$ be their valuations to respective partial semigroups. Then the assignments $a_1\otimes a_2\mapsto (\nu_1(a_1),\nu_2(a_2))$ extend to a valuation $\nu:A_1\otimes A_2 \setminus \{0\} \to P_1\times P_2$ (an order in $P_1\times P_2$ is defined in Lemma~\ref{semigroups_product}). If both valuations $\nu_1, \nu_2$ are injective then $\nu$ is injective as well.
     
 \end{corollary} 

\begin{remark} 
\label{rem:Takeuchi}

This can be generalized to Takeuchi products $A\otimes_H B$ where $H$ is a bialgebra acting on $A$ and $B$. Namely, if we denote the action by $h\act a$ and the coaction $\delta(b)=b^{(-1)}\otimes b^{(0)},\ b^{(-1)}\in H,\ b^{(0)}\in B$ (in Sweedler notation), the defining relations in $A\otimes_H B$ is
 $ba=(b^{(-1)}\act a)b^{(0)}$ for all $a\in A$, $b\in B$.

Clearly, $A\otimes B=A\otimes_H B$, where $H=\kk$ and $k\act a=ka$ for all $a\in A$, $k\in \kk$ $\delta(b)=1\otimes b$ for all $b\in B$.
Suppose additionally that there are valuations $\nu_1:A\setminus \{0\}\to P_1$ and $\nu_2:B\setminus \{0\}\to P_2$  such that

$\bullet$ For any nonzero $b\in B$ there exist unique elements $b'\in B$ and $h_b\in H$ such that $\nu_2(b')=\nu_2(b)$ and $\delta(b)-h_b\otimes b'$ belongs to $H\otimes B_{\prec \nu_2(b)}$.

$\bullet$ For any  $a\in A$ and $h\in H$ such that $h\act a\ne 0$ one has $\nu_1(h\act a)\preceq \nu_1(a)$ and $\nu_1(h_b\act a)=\nu_1(a)$ for any nonzero $b\in B$.

Then the assignments $ab\mapsto (\nu_1(a),\nu_2(b))$ define a valuation $\nu:A\otimes_H B\setminus\{0\}\to P_1\times P_2$ thus generalizing Corollary \ref{tensor}.

In particular,  The Weyl algebra $A_1$ with the presentation $yx=xy+1$ is a Takeuchi product with 
 $A=\kk[x]$, $B=\kk[y]$, 
 $H=\kk[d]$ is a Hopf algebra when viewed as the enveloping algebra of $1$-dimensional Lie algebra, 
 $\delta(y)=1\otimes y+d\otimes 1$, $H$ acts on  $\kk[x]$ by $d(p(x))=p'(x)$.

Let $C=\kk_q[x,y]$ with presentation $yx=qxy$. Then $C$ is a Takeuchi product with  $A=\kk[x]$, $B=\kk[y]$, $H=\kk[g]$ is a bialgebra with the coproduct $\Delta(g)=g\otimes g$ and the counit $\varepsilon(g)=1$, $\delta(y)=g\otimes y$, $H$ acts on  $\kk[x]$ by $g(p(x))=p(qx)$.

Taking in the above example $\nu_1,\nu_2$ tautological valuations to $\ZZ_{\ge 0}$, we obtain an injective valuation $\nu:C\setminus\{0\}\to \ZZ_{\ge0}^2$.

 \end{remark}

\subsection{Quantum cones}
In this section we define (generalized)  quantum monoids.

\begin{lemma}
\label{le:crossed product of monoids}
Let $\underline P$ be a partial semigroup,  $\Gamma$ be  a monoid, and $\chi:\underline P\times \underline P\to Z(\Gamma)$ be (a central) two-cocycle, i.e., 
$\chi(c\circ c',c'')\chi(c,c')=\chi(c,c'\circ c'') \chi(c',c'')$ for any associative triple $(c,c',c'')$
in $P$.
Then $\underline P\times \Gamma$ has a partial semigroup structure via
$(c,\gamma)\circ (c',\gamma')=(c\circ c',\chi(c,c')\gamma\gamma')$, we denote this partial semigroup by $P=\underline 
 P\times_\chi \Gamma$. Moreover, the projection to the first factor is a homomorphism of partial semigroups $P\twoheadrightarrow \underline P$. 

\end{lemma}

\begin{proof} Indeed, associativity reads

\noindent$((c,\gamma)\circ (c',\gamma')) \circ (c'',\gamma'')=(c\circ c',\chi(c,c')\gamma\gamma')\circ (c'',\gamma'')=(c\circ c'\circ c'',\chi(c\circ c',c'')\chi(c,c')\gamma\gamma'\gamma'')$,

\noindent $(c,\gamma)\circ ((c',\gamma') \circ (c'',\gamma''))=(c,\gamma)\circ(c'\circ c'',\chi(c',c'')\gamma'\gamma'')=(c\circ c'\circ c'',\chi(c,c'\circ c'')\gamma \chi(c',c'')\gamma'\gamma'')$.

These expressions are equal because 

$\chi(c\circ c',c'')\chi(c,c')\gamma\gamma'\gamma''=\chi(c,c'\circ c'')\gamma \chi(c',c'')\gamma'\gamma''=\chi(c,c'\circ c'') \chi(c',c'')\gamma\gamma'\gamma''$.

The lemma is proved.
\end{proof}

From now on suppose that  $\underline P$ is a monoid. Clearly, $(1,\Gamma)$ is a submonoid  of $\underline P\times_\chi \Gamma$ with the multiplication $(1,\gamma)\circ(1,\gamma')=(1,\chi(1,1)\gamma\gamma')$. It is also clear that the assignments $(c,\gamma)\mapsto c$ define a surjective homomorphism of (partial) semigroups $P=\underline  P\times_\chi \Gamma\to \underline  P$, that is, $P$ is a central extension of $\underline  P$.

 Also, any bi-character $\chi:\underline P\times \underline P\to Z(\Gamma)$ is automatically a two-cocycle.

\begin{lemma}
\label{le:anti-involution quantum plane} In the assumptions of Lemma \ref{le:crossed product of monoids}, suppose that $\Gamma$ is a group and fix an anti-involution $c\mapsto c^*$ on $\underline P$.  Then the assignments $\overline{(c,\gamma)}:= (c^*,\gamma^{-1})$ define an anti-involution $\overline {\cdot}$ on $P:=\underline P\times_\chi \Gamma$ iff  $\chi(d^*,c^*)=\chi(c,d)^{-1}$ for all $c,d\in \underline P$.

\end{lemma}

\begin{proof} Indeed,
$$(d^*,{\gamma'}^{-1})\circ (c^*,\gamma^{-1})=(d^*\circ c^*,\chi(d^*,c^*){\gamma'}^{-1}\gamma^{-1})
=((c\circ d)^*,(\chi(c,d)\gamma\gamma')^{-1})$$
    
This is equivalent to 
$\overline{(d,\gamma')}\circ \overline{(c,\gamma)}=\overline{(c\circ d,\chi(c,d)\gamma\gamma')}$, that is, $\overline {\cdot}$ is an anti-involution.

The lemma is proved.
\end{proof}

The following is immediate.
 \begin{proposition}
 \label{quantum_order} In the assumptions of Lemma \ref{le:crossed product of monoids}, suppose additionally that $\underline P$ and $\Gamma$ are ordered and the order on $\underline  P$ satisfies the strict property. Then $P:=\underline P\times_\chi \Gamma$ is ordered lexicographically, i.e., 
 $(c,\gamma)\prec (c',\gamma')~\text{iff $c\prec c'$ or $c=c'$, $\gamma \prec \gamma'$}$. Moreover, the homomorphism $P\twoheadrightarrow \underline P$ from Lemma \ref{le:crossed product of monoids} is ordered. 

 \end{proposition}

\begin{example}
\label{ex:quantum plane}
    
For instance, if $\underline P=\ZZ^m$, $\Gamma$ is any abelian  multiplicative group, and $\chi$ is uniquely determined by $\chi_{k\ell}:=\chi(e_k,e_\ell)\in \Gamma, 1\le k, \ell \le m$. 
Then $P:=\underline P\times_\chi \Gamma$ is a group generated by central elements $\chi_{k\ell}$, and elements $t_k=(e_k,1)$, $k=1,\ldots,m$ subject to 
\begin{equation}
\label{eq:quantum plane general}
t_\ell t_k=q_{k\ell}t_kt_\ell
\end{equation}
for $k<\ell$, 
where $q_{k\ell}:=\frac{\chi_{\ell k}}{\chi_{k\ell}}$ (its linearization is known as a quantum torus).
Note that this monoid admits an anti-involution as in Lemma \ref{le:anti-involution quantum plane} (with trivial anti-involution $*$ on $P$) 
iff $\chi_{\ell k}=\chi_{k\ell}^{-1}$, e.g., $q_{k\ell}=\chi_{\ell k}^2$. In that case, the element $(a,1)\in P\times_\chi \Gamma$ is given by
$$(a,1)=q_at_1^{a_1}\cdots t_m^{a_m}$$
where $q_a=\prod\limits_{1\le k<\ell\le m} q_{k \ell}^{\frac{1}{2}a_k a_\ell}=\prod\limits_{1\le k<\ell\le m} \chi_{\ell k}^{a_k a_\ell}$. Clearly, $\overline{(a,1)}=(a,1)$. Also
\begin{equation}
\label{eq:Lambda-multiplication}
 (a,1)\cdot (a',1)=g_{a,a'}\cdot (a+a',1)  
\end{equation}
where $g_{a,a'}=\prod\limits_{1\le k<\ell\le m} q_{k \ell}^{\frac{1}{2}(a_\ell a'_k-a_k a'_\ell)}=\prod\limits_{1\le k<\ell\le m} \chi_{\ell k}^{a_\ell a'_k-a_k a'_\ell}$ (e.g., $g_{e_\ell,e_k}=q_{k\ell}$). In particular, $g_{a',a}=g_{a,a'}^{-1}$.

For any submonoid $\underline P$ of $\ZZ^m$ we consider a submonoid of  $\ZZ^m\times_\chi \Gamma$ generated by $\underline P$ and $\gamma$. Clearly, $P=\underline P\times_{\chi'} \Gamma$ where $\chi'=\chi|_{P\times P}$. We informally refer to $P$ as a {\it quantum cone} (e.g., if $\underline P=\ZZ_{\ge 0}^m$, $P$ is a quantum octant, whose linearization is also known as a quantum plane).  

\end{example}

\begin{remark}
An order on quantum cones does not necessarily satisfy the property $\Gamma\prec t_k$ for $k=1,\ldots,m$ (cf. Proposition~\ref{quantum_order}). Namely, let $q_{12}, q_{13}, q_{23}$ generate a free abelian group $\Gamma$ and let $P$ be generated by $t_1,t_2,t_3$ and  $\Gamma$ subject to \eqref{eq:quantum plane general}. Then there is a lex order $\prec$ on $P$ with $q_{12}\prec q_{13}\prec t_1\prec q_{23}\prec t_2\prec t_3$. In fact, it is necessary and sufficient that $q_{kl}^i\prec t_k, t_l, i\ge 0$. To prove the necessity suppose on the contrary that  $t_k\prec q_{kl}^i$. Then $q_{kl}t_kt_l=t_lt_k\prec t_l q_{kl}^i$, hence $t_k\prec q_{kl}^{i-1}$, and induction on $i$ leads to a contradiction.
\end{remark}

\begin{example} 

One has an "exterior" analog $P$ of the quantum octant which is a partial semigroup generated by $t_1,\ldots,t_m$, all $q_{k\ell}^{\pm 1}$, $1\le k< \ell\le m$ subject to:

$\bullet$ $t_k\circ t_k$ is undefined for $k=1,\ldots,m$

$\bullet$ $t_\ell \circ t_k=q_{k\ell}t_k \circ t_\ell$
for all $k<\ell$.

The extension $\mathbb{K}\otimes \kk P$ of linearization $\kk P$ by any field $\mathbb{K}$ containing all $q_{k\ell}$ (which an exterior analogue of the quantum plane) has an injective $\mathbb{K}$-valuation  to the Boolean partial monoid: $I\circ J=I\cup J$ whenever $I\cap J=\emptyset$ and undefined otherwise 
    via $\nu(t^\varepsilon):=I_\varepsilon$ for any $\varepsilon\in \{0,1\}^m$ endowed with the lexicographical ordering (cf. Example~\ref{monoid_ideal}~iii).)
\end{example}

\subsection{String valuations of algebras}
Generalizing approach of \cite[Section 2]{bz} and \cite[Section 3.2]{BZm} and \cite{L}, we  construct in this section a large class of string valuations (cf. Lemma~\ref{le:nilpotent valuation}) of algebras based on their skew derivations. 

For an algebra $A$ we say that a $\kk$-linear map $\partial:A\to A$ is a {\it skew derivation} with respect to an endomorphism $\varphi$ of $A$ if the skew Leibniz rule holds
$$\partial(ab)=\partial(a)b+\varphi(a)\partial(b)$$
for all $a,b\in A$.

Furthermore, we say that the pair $(\partial,\varphi)$ of linear maps $A\to A$ is $r$-compatible if $\partial \circ \varphi=r\cdot \varphi\circ \partial$ for some $r\in \kk^\times$.

\begin{proposition}
\label{skew_quantum}
Let $\varphi$ be an monomorphism of $A$ and $\partial$ be a skew derivation of $A$ with respect to $\varphi$. 
Suppose that $\partial$ is locally nilpotent and $(\partial,\varphi)$ is $r$-compatible with a non-root of unity $r\in \kk^\times$. Then 
\begin{equation}
\label{eq:nil Leibniz rule}
\partial^{(m+n)}(ab)= \varphi^n(\partial^{(m)}(a))\partial^{(n)}(b)\ 
\end{equation}
for all $a,b\in A$  where we abbreviate $m:=\nu_\partial(a)$, $n:=\nu_\partial(b)$ in the notation of Lemma \ref{le:nilpotent valuation}, $\partial^{(k)}:=\frac{1}{[k]_r!}\partial^k$, the $k$-th divided power of $\partial$, and $[k]_r!:=\prod\limits_{\ell=1}^k \frac{1-r^\ell}{1-r}$.

\end{proposition} 

\begin{proof} We need the following 

\begin{lemma}
$\partial^k(ab)=\sum\limits_{J\subset [1,k]} \partial^J(a)\partial^{k-|J|}(b)$ for all $k\ge 0$, $a,b\in A$,   
where we abbreviate $\partial^J(a)=d_{1,J}\circ \cdots \circ d_{k,J}$ where $d_{j,J}=
\begin{cases} \partial & \text{if $j\in J$}\\
\varphi & \text{otherwise}
\end{cases}$.

\end{lemma}

In particular, 
$\partial^2(ab)=\partial^2(a)b+\varphi(\partial(a))\partial(b)  +\partial(\varphi(a))\partial(b)+\varphi^2(a)\partial^2(b)$.

The compatibility implies that $\partial^J=r^{\sum\limits_{i\notin J}|J_{<i}|}\varphi^{k-|J|}\circ \partial^{|J|}$, where $J_{<i}=J\cap [1,i-1]$.

In turn, this implies that \begin{equation}
\label{eq:r-binomial}
\partial^k(ab)=\sum\limits_{i=0}^k \binom{k}{i}_r\varphi^{k-i}(\partial^i(a))\partial^{k-i}(b)
\end{equation}
where $\binom{k}{i}_r$ is the $r$-binomial coefficient $\frac{(1-r^{i+1})\cdots (1-r^k)}{(1-r)\cdots (1-r^{k-i})}=\frac{[k]_r!}{[i]_r![k-i]_r!}$, 
in particular, it is not $0$ if $r$ is not a root of unity.

Finally, taking $k=m+n$ as in the proposition, all summands with $i\ne m$ in the right hand side of \eqref{eq:r-binomial} are $0$.

The proposition is proved. 
\end{proof}

\begin{theorem} 
\label{th:string valuations general}
Let 
$\varphi_k$, $k=1\ldots,m$
be  automorphisms of $A$ and $\partial_k$, $k=1,\ldots,m$ be  skew derivations of $A$ with respect to 
$\varphi_k$. 
Suppose that 

$\bullet$ $A$ is a domain. 

$\bullet$ Each $\partial_k$ is locally nilpotent. 

$\bullet$ 
$\partial_\ell\circ \varphi_k = q_{k \ell} \varphi_k \circ \partial_\ell$ for any $q_{k\ell}\in \kk^\times$ for all $k\le  \ell$.
such that  
$q_{ii}$ are non-roots of unity (e.g., $(\partial_k,\varphi_k)$ is $q_{kk}$-compatible).

Then the string valuation $\nu_{(\partial_1,\ldots,\partial_m)}$ from \eqref{eq:string valuation vector space} is a valuation of the algebra $A$ to the 
additive semigroup $\ZZ_{\ge 0}^m$.

\end{theorem}

\begin{proof} Iterating \eqref{eq:nil Leibniz rule}, we obtain for any $a,b\in A\setminus \{0\}$:
\begin{equation}
\label{eq:multi-Leibniz}
\partial_m^{(k_m+\ell_m)}\cdots \partial_1^{(k_1+\ell_1)}(ab)= \varphi_m^{\ell_m}\partial_m^{(k_m)}\cdots \varphi_1^{\ell_1}\partial_1^{(k_1)} (a)\cdot \partial_m^{(\ell_m)}\cdots \partial_1^{(\ell_1)}(b)
\end{equation}
where $k_1,\ldots,k_m$ and $\ell_1,\ldots,\ell_m$ are defined recursively via

$$k_{i+1}=\nu_{\partial_{i+1}}(\varphi_i^{\ell_i}\partial_i^{(k_i)}\cdots \varphi_1^{\ell_1}\partial_1^{(k_1)}(a)),~\ell_{i+1}=\nu_{\partial_{i+1}}(\partial_i^{(\ell_i)}\cdots \partial_1^{(\ell_1)}(b))\ .$$
Moreover, the third assumption of the theorem implies (once again, recursively) that
$$\varphi_m^{\ell_m}\partial_m^{(k_m)}\cdots \varphi_1^{\ell_1}\partial_1^{(k_1)}=
q\cdot \varphi \partial_m^{(k_m)}\cdots \partial_1^{(k_1)}\ ,$$
where $\varphi:=\varphi_m^{\ell_m}\cdots\varphi_1^{\ell_1}$and  $q:=\prod\limits_{1\le i<j\le m} q_{ij}^{\ell_ik_j}$.

Using this (and replacing $m$ by $i$), we see that
$k_{i+1}=\nu_{\partial_{i+1}}(\partial_i^{(k_i)}\cdots \partial_1^{(k_1)}(a))$.

Taking into account that $(k_1,\ldots,k_m)=\nu_{\partial_1,\ldots,\partial_m}(a)$ and $(\ell_1,\ldots,\ell_m)=\nu_{\partial_1,\ldots,\partial_m}(b)$ we conclude that 
$$\nu_{\partial_1,\ldots,\partial_m}(ab)=\nu_{\partial_1,\ldots,\partial_m}(a)+\nu_{\partial_1,\ldots,\partial_m}(b)\ .$$

Note that the formula \eqref{eq:multi-Leibniz} refines as well:
$$\partial_1^{(k_m+\ell_m)}\cdots \partial_1^{(k_1+\ell_1)}(ab)= q\cdot \varphi\partial_i^{(k_i)}\cdots \partial_1^{(k_1)} (a)\cdot \partial_m^{(\ell_m)}\cdots \partial_1^{(\ell_1)}(b)\ .$$

 Therefore, the leading term $\lambda_{\partial_1,\ldots,\partial_m}$ (see Corollary \ref{cor:nilpotent valuation and leading coefficient}) is almost multiplicative:
$\lambda_{\partial_1,\ldots,\partial_m}(ab)=q\cdot\varphi\lambda_{\partial_1,\ldots,\partial_m}(a)\cdot \lambda_{\partial_1,\ldots,\partial_m}(b)$.

Theorem \ref{th:string valuations general} is proved.

\end{proof}

Let ${\bf q}=(q_{k\ell},1\le k< \ell\le m)$ be an array in $\kk^\times$.
Denote by 
$A_{\bf q}$ the $\kk$-algebra generated by $t_1,\ldots,t_m$ subject to 
\eqref{eq:quantum plane general}. This algebra is usually referred to as a quantum plane.

\begin{theorem}
\label{th:generalized Feigin}
In the assumption of Theorem \ref{th:string valuations general}, suppose additionally that 
$$\partial_\ell\circ \varphi_k = q_{k \ell} \varphi_k \circ \partial_\ell$$ for all $1\le k,\ell\le m$, 
where ${\widehat {\bf q}}=(q_{k\ell})$ is an $m\times m$-matrix over $\kk^\times$ whose upper part is ${\bf q}$.
Fix a homomorphism $\varepsilon:A\to \kk$ of $\kk$-algebras such that $\varepsilon \circ \varphi_k=\varepsilon$ for $k=1,\ldots,m$.

Then the assignment
$x\mapsto \sum\limits_{a\in \ZZ_{\ge 0}^m} \varepsilon(\partial_m^{(a_m)} \cdots \partial_1^{(a_1)}(x))t_1^{a_1}\cdots t_m^{a_m}$ for any $x\in A$ define a homomorphism of algebras $\Phi_\partial:A\to A_{\bf q}$.

\end{theorem}

We prove the theorem in Section \ref{Proof of Theorem th:generalized Feigin}.

\begin{example}  Let $A=\kk[t]$.
Denote by $\varphi$ the automorphism of $A$ sends $t$ to $qt$.
Define a linear map $\partial:A\to A$ by 
$\partial(t^n)=[n]_q t^{n-1}$ for $n\ge 0$,
where $[n]_q:=1+q+\cdots+q^{n-1}$ is the $q$-number (e.g., $[0]=0$).

Then $\partial$ is a (locally nilpotent) skew derivation of $A$ with respect to $\varphi$.

Suppose that $q$ is non-root of unity. 
Then all hypotheses of Theorem \ref{th:string valuations general} hold and the assignments $t^n\mapsto n=\nu_{\partial}(t^n)$ define a (standard) valuation $\nu_\partial:A\setminus\{0\}\to \ZZ_{\ge 0}$.

More generally,  in the notation of Example \ref{ex:quantum plane}, the assignments $t_\ell\mapsto q_{k\ell}t_\ell$ define an automorphism $\varphi_k$ of $A_{\bf q}$. All these automorphisms commute.

Also for $k\in [1,m]$ define a linear map $\partial_k:A_{\bf q}\to A_{\bf q}$ by 
$$\partial_k(t_1^{a_1}\cdots t_m^{a_m}):=[a_k]_{q_{kk}}\cdot \varphi_k(t_1^{a_1}\cdots t_{k-1}^{a_{k-1}})t_k^{a_k-1}t_{k+1}^{a_{k+1}}\cdots t_m^{a_m}$$ is a skew derivation  with respect to $\varphi_k$. One can show that
$\partial_k\circ \varphi_\ell=q_{\ell k} \varphi_\ell \circ \partial_k$ for all $k,\ell\in [1,m]$. Thus, the produced $\partial_i, \varphi_i, 1\le i\le m$ fulfill the conditions of Theorem~\ref{th:string valuations general}. Therefore $\nu_{(\partial_1,\dots, \partial_m)}$ is a valuation of the algebra $A_{\bf q}$. Moreover, the valuation is injective due to Proposition~\ref{commutative_injective}. 
\end{example}

Even more generally, let $V$ be a vector space, $\Psi:V\otimes V\to V\otimes V$ be a linear map. Then the tensor algebra $T(V)=\oplus_{n\ge 0} V^{\otimes n}$ has a coproduct homomorphism $\Delta:T(V)\to T(V)\otimes_\Psi T(V)$ given by $\Delta(v)=v\otimes 1+ 1\otimes v$ for $v\in V$. For any linear function $f\in V^*$ define a linear map  $\varepsilon_f:T(V)\to \kk$ by $\varepsilon_f|_V=f$, $\varepsilon_f|_{V^{\otimes k}}=0$ if $k\ne 1$.
Then define a homogeneous of degree $-1$ linear map $\partial_f:T(V)\to T(V)$ by
$\partial_f=(\varepsilon_f\otimes Id_{T(V)})\circ \Delta$,
that is, using Sweedler notation $\Delta(x)=x_{(1)}\otimes x_{(2)}$, we obtain
$$\partial_f(x)=\varepsilon_f(x_{(1)})\cdot x_{(2)}\ .$$

\begin{remark} 
\label{rem:Psi}
Usually one assumes that $\Psi$ is {\it braiding} of $V$, i.e.,  satisfies the braid equation $\Psi_{12}\circ \Psi_{23}\circ\Psi_{12}=\Psi_{23}\circ\Psi_{12}\circ\Psi_{23}$ in $End_\kk(V\otimes V\otimes V)$, where $\Psi_{12}=\Psi\otimes Id_V$ and $\Psi_{23}=Id_V\otimes \Psi$.

In that case, $T(V)$ is a braided Hopf algebra and the common kernel $J_\Psi(V)$ of all $\partial_f$, $f\in V^*$ on $T(V)_{\ge 2}=\oplus_{n\ge 2} V^{\otimes n}$ is a graded Hopf ideal  which is maximal among all Hopf ideals $J$
of $T(V)$ such that $J\cap V=\{0\}$.
Then the algebra
 $B_\Psi(V):=T(V)/J_\Psi(V)$ (known as Nichols algebra, see e.g., \cite[Sections 2 and 3]{B}) is a (braided) Hopf algebra.
    
\end{remark}

$\Psi$ is called diagonal if there is a basis $e_i$, $i\in I$ such that $\Psi(e_i\otimes e_j)=p_{ij} e_j\otimes e_i$ for some $I\times I$ matrix ${\bf p}=(p_{ij})$ over $\kk^\times$.
The following is immediate
\begin{lemma} (a) If $\Psi=\tau$ is the permutation of factors, then $\partial_f$ is a derivation of $T(V)$

(b) If $\Psi$ is a diagonal braiding, then $\partial_i:=\partial_{e_i^*}$ is skew derivation of $T(V)$ with respect to the automorphism $\varphi_i$ of $T(V)$ given by $\varphi_i(e_j)=p_{ij}e_j$.
Moreover, 
$\partial_i\circ \varphi_j=p_{ji}  \varphi_j\circ \partial_i$.
    
\end{lemma}

Note that for any ideal $J\subset T(V)$ invariant under all $\partial_i$ and $\varphi_i$ the quotient algebra $A=T(V)/J$ inherits all $\partial_i$ and $\varphi_i$.

In particular, if $C=(c_{ij})$ is a symmetrized Cartan matrix for some Kac-Moody Lie algebra $\gg=\nn_-\oplus \hh\oplus \nn_+$, then  $J_\Psi(V)$  is the quantum Serre ideal, i.e., $T(V)/J_{max}=U_q(\nn_+)$.

The following is an immediate 

\begin{corollary} In notation of Theorem~\ref{th:string valuations general}, for any sequence $\ii=(i_1,\ldots,i_m)\in I^m$ one has a  string $\ii$-valuation of $T(V)$ (or any of its quotients by any ideal under all $\partial_i$ and $\varphi_i$) given by 
$$\nu_\ii:=\nu_{(\partial_{i_m},\ldots,\partial_{i_1})}$$  
\end{corollary}

\subsection{Proof of Theorem~\ref{th:generalized Feigin}}

\label{Proof of Theorem th:generalized Feigin}

Let $\kk$ be a field and ${\mathcal U}$ be the free algebra freely generated by $E_1,\ldots, E_m$. Clearly, ${\mathcal U}$ is
$\ZZ_{\ge 0}^m$-graded via $\deg E_k=\alpha_k$, where $\alpha_1,\ldots,\alpha_m$ is the basis of $\ZZ_{\ge 0}^m$.

Let $\widehat {\bf q}=(q_{kl})$, $1\le k,l\le m$ be a  
$m\times m$-matrix with all
$q_{k l}\in \kk^\times$. Following G. 
Lusztig (\cite{lu}) we turn 
${\mathcal U}\bigotimes
{\mathcal U}$ into an algebra via (see also Remark \ref{rem:Psi}):
\begin{equation}  
\label{(1.2)}
(a\otimes b)(c\otimes d):=Q(deg(b),deg(c))(ac\otimes bd)
\end{equation}\
for any   homogeneous elements $b,c$ of ${\mathcal
U}$ and any $a,d\in {\mathcal U}$, where  
\begin{equation}
\label{(1.1)}
Q((n_1,\ldots,n_m),(n'_1,\ldots,n'_m))=
\prod_{k,l=1}^m q_{k l}^{n_k n'_l} \ . 
\end{equation}

This makes  
${\mathcal U}\bigotimes {\mathcal U}$ into a 
$\ZZ_{\ge 0}^m$-graded associative algebra 
(with the standard grading $({\mathcal
U}\bigotimes_\kk {\mathcal U})(\gamma)=\bigoplus_{\gamma'}~ {\mathcal
U}(\gamma')\bigotimes {\mathcal U}(\gamma-\gamma')$). 
This algebra will be denoted  by ${\mathcal U}
{\bigotimes}_{\widehat {\bf q}}{\mathcal U}$ and called the 
${\widehat {\bf q}}$-{\it braided} tensor square of ${\mathcal U}$.

Denote by 
$\Delta:{\mathcal U}\to {\mathcal U} {\bigotimes}_{\widehat {\bf q}}
{\mathcal U}$ 
the algebra homomorphism determined by $\Delta(E_k)=E_k\otimes 1+1\otimes E_k$, $k=1,\ldots,m$.

By definition,
\begin{eqnarray}\label{(1.4)}
\Delta(u)=u\otimes 1+1\otimes u+\sum_n u_n\otimes u'_n\
\end{eqnarray}
for any $u\in {\mathcal U}$, 
where all $u_n,u'_n$ are homogeneous elements of nonzero degrees.

Let $\widehat {\mathcal U}$ be
the completion of ${\mathcal U}$ with respect to the grading, 
that is, the space
of all formal series $\widehat u=\sum\limits_{\gamma\in \ZZ_{\ge 0}^m} 
u_\gamma$,
where $u_\gamma\in {\mathcal U}(\gamma)$.  
Clearly, $\widehat {\mathcal U}$ is an algebra. 
The coproduct $\Delta$ uniquely extends to  
$\widehat \Delta: \widehat {\mathcal U}\to \widehat {\mathcal U}\widehat 
{\bigotimes}_{\widehat {\bf q}} 
\widehat {\mathcal U}$.

Recall that $A_{\bf q}$ a ${\bf
k}$-algebra  generated by 
$t_1,\ldots,t_m$ subject
to the  relations \eqref{eq:quantum plane general}
for all $1\le k<l\le m$. 
  
Define $\widehat {\mathcal U}_{\widehat {\bf q}}=A_{\bf q}\widehat \bigotimes_\kk 
{\mathcal U}$, the space of formal series of the  
form $\sum_\gamma t_\gamma 
\otimes u_\gamma$, 
where $t_\gamma\in A_{\bf q}$ and $u_\gamma\in {\mathcal U}(\gamma)$.
We consider $\widehat {\mathcal U}_{\widehat {\bf q}}$ with the standard 
algebra structure (so we can write $tu=ut=t\otimes u$).

Consider the completed tensor square 
${\mathcal V}_{\widehat {\bf q}}={\mathcal U}_{\bf
q}\widehat {\bigotimes\limits_{A_{\bf q}}}  {\mathcal U}_{\widehat {\bf q}}$,  
where the left factor is regarded as a 
right $A_{\bf q}$-module and the right 
factor as a left
$A_{\bf q}$-module. Note that ${\mathcal V}_{\widehat {\bf q}}$ is a 
$A_{\bf q}$-bimodule. In ${\mathcal V}_{\widehat {\bf q}}$, we can write 
$t(u\otimes v)=(tu)\otimes v=u\otimes (tv)=(u\otimes v)t$
for any $u,v\in {\mathcal U}, t\in A_{\bf q}$. 
Under the standard identification 
${\mathcal V}_{\widehat {\bf q}}\cong A_{\bf q} \widehat  \bigotimes  
\left ({\mathcal U}{\bigotimes}_{\widehat {\bf q}}  {\mathcal U}\right )$ 
this bimodule ${\mathcal V}_{\widehat {\bf q}}$ becomes an algebra.

There is a  natural morphism of $A_{\bf q}$-bimodules 
$$\widehat \Delta_{\widehat {\bf q}}:\widehat {\mathcal U}_{\widehat {\bf q}}\to {\mathcal V}_{\widehat {\bf q}}$$
which is the $A_{\bf q}$-linear extension of 
the coproduct $\widehat \Delta$ on  $\widehat {\mathcal U}$. Clearly, 
$\widehat \Delta_{\widehat {\bf q}}$ is an algebra homomorphism.

Define 
an element ${\bf e}  
\in \widehat {\mathcal U}_{\widehat {\bf q}}$ as follows:
\begin{eqnarray}\label{(1.6)}
{\bf e}={\rm exp}_{q_1}(t_1E_1)
{\rm exp}_{q_2}(t_2E_2)\cdots
{\rm exp}_{q_m}(t_mE_m)
\end{eqnarray}
where $q_k=q_{k k}$ for $k=1,\ldots,m$, 
and ${\rm exp}_q(x):=\sum\limits_{k\ge 0} \frac{1}{[k]_q!}x^k$ is 
the quantum exponential (in the notation of Proposition \ref{skew_quantum}).

\begin{theorem} \label{1.1}
The element 
${\bf e}$ is 
 {\rm grouplike} in $\widehat {\mathcal U}_{\widehat {\bf q}}$, 
i.e., $\widehat \Delta_{\widehat {\bf q}}({\bf e})=
{\bf e}\otimes {\bf
e}$.

\end{theorem}

\begin{proof} We need the following  

\begin{lemma}\label{1.2} 

(a) Each factor ${\bf e}_k=
{\rm exp}_{q_k}(t_kE_k)$ of ${\bf
e}$ is a group-like element in $\widehat {\mathcal U}_{\widehat {\bf q}}$. 

(b) $(1\otimes {\bf e}_k)({\bf e}_l
\otimes 1)=({\bf e}_l\otimes 1) (1 \otimes
{\bf e}_k)$ for any $1\le k<l\le m$.

\end{lemma}

\begin{proof}  Prove (a). Denote $E=t_kE_k$. Since 
$\Delta(E_k)=
E_k\otimes 1+1\otimes E_k$, for each $k$ we have 
$$\widehat \Delta_{\widehat {\bf q}}(E)=t_k(E_k\otimes 1+1\otimes E_k)=
E\otimes 1+1\otimes E.$$ 
Denote $x=E\otimes 1, y=1\otimes E$. Let us show that  
$yx=qxy$ where
$q:=q_k$. 
Indeed, 
$$yx=(1\otimes E)(E\otimes 1)=(1\otimes t_kE_k)(t_kE_k\otimes 1)=
t_k^2(1\otimes E_k)(E_k\otimes 1)$$
$$=Q(E_k,E_k)t_k^2(E_k\otimes E_k)=
q_k (t_kE_k\otimes t_kE_k)=qxy \ .$$ 
Further, we obtain  
$$\widehat \Delta_{\widehat {\bf q}}({\rm exp}_q(E))
={\rm exp}_q(\widehat \Delta_{\widehat {\bf q}}(E))
={\rm exp}_q(x+y)$$ and 
$${\rm exp}_q(E)\otimes
{\rm exp}_q(E)=({\rm exp}_q(E)\otimes 1)(1\otimes {\rm exp}_q(E))$$
$$=({\rm exp}_q(E\otimes 1))({\rm exp}_q(1\otimes
E))={\rm exp}_q(x) {\rm exp}_q(y) \ . $$ Then the well-known rule 
for the quantum exponentials. 
$${\rm exp}_q(x+y)={\rm exp}_q(x) {\rm exp}_q(y)$$
(provided that $yx=qxy$) implies that 
$\widehat \Delta({\rm exp}_q(E))=
{\rm exp}_q(E)\otimes {\rm exp}_q(E)$. Part (a) is proved. 

\smallskip

Prove (b).  Denote $E=t_kE_k$ and $E'=t_lE_l$. By definition of ${\mathcal
U} {\bigotimes}_{\widehat {\bf q}} {\mathcal U}$,
$$(1\otimes E)(E'\otimes 1)=t_kt_l(1\otimes E_k)(E_l\otimes
1)=q_{k l}t_kt_l(E_l\otimes E_k)$$
$$
=t_lt_k(E_l\otimes E_k)=E'\otimes E=(E'\otimes 1) (1\otimes E) $$
by 
\eqref{eq:quantum plane general}.
Therefore,  $(1\otimes f(E))(g(E')\otimes 1)=f(E')\otimes 
g(E)=(f(E')\otimes 1) (1\otimes g(E))$
for any polynomials $f$ and $g$ in one variable. Passing to the 
completion, we see that 
$f$ and $g$ can also be power series in the above formula.  
Taking  $f(E):={\bf e}_k={\rm exp}_{q_k}(E)$ and $g(E'):={\bf e}_l={\rm
exp}_{q_l}(E')$ completes the proof of  (b).  
The lemma is proved.   
\end{proof}

\medskip 

We are ready to complete the proof of Theorem~\ref{1.1} now. 
Recall that we use the
shorthand ${\bf e}_k={\rm exp}_{q_k}(t_kE_k)$ so that
${\bf e}={\bf
e}_1{\bf e}_2\cdots {\bf e}_m$. 

Using Lemma~\ref{1.2} and the fact that 
$(a\otimes 1)(1\otimes b)=
a\otimes b$ for any  $a,b\in \widehat {\mathcal U}_{\widehat {\bf q}}$,
we obtain 
$$\widehat \Delta_{\widehat {\bf q}}({\bf e})=
\widehat \Delta_{\widehat {\bf q}}({\bf e}_1{\bf
e}_2\cdots {\bf e}_m)=
\widehat \Delta_{\widehat {\bf q}}({\bf e}_1)
\widehat \Delta_{\widehat {\bf q}}({\bf e}_2)\cdots \widehat
\Delta_{\widehat {\bf q}}({\bf e}_m)=
({\bf e}_1\otimes {\bf e}_1)({\bf e}_2\otimes {\bf e}_2)\cdots 
({\bf e}_m\otimes {\bf e}_m)$$
$$=({\bf e}_1\otimes 1)(1\otimes {\bf e}_1)({\bf e}_2\otimes 1)
(1\otimes
{\bf e}_2)\cdots ({\bf e}_m\otimes 1)(1\otimes {\bf e}_m)\ . $$
Using the commutativity property  in Lemma~\ref{1.2}~(b), we obtain
$$\widehat \Delta_{\widehat {\bf q}}({\bf e})
=\big(({\bf e}_1\otimes 1)({\bf e}_2\otimes 1)\cdots
({\bf e}_m\otimes 1)\big)\big((1\otimes {\bf e}_1)(1\otimes {\bf
e}_2)\cdots 
(1\otimes {\bf e}_m)\big)$$
Finally, using the identities $(u\otimes 1)(v\otimes 1)=uv\otimes 1,
(1\otimes u)(1\otimes v)=1\otimes uv$ for any $u,v\in \widehat {\mathcal U}$,
we obtain 
$\widehat \Delta_{\widehat {\bf q}}({\bf e})=({\bf e}\otimes 1)
(1\otimes {\bf e})={\bf e}\otimes 
{\bf e}$. Theorem~\ref{1.1} is proved.\end{proof}

Denote by ${\mathcal U}'$ the subspace of  ${\mathcal U}$ spanned by all monomials $E_1^{a_1}\cdots E_m^{a_m}$. Clearly, ${\mathcal U}'$ is a sub-coalgebra of ${\mathcal U}'$. Denote  $\widehat {\mathcal U}'_{\widehat {\bf q}}:=A_{\bf q}\widehat \bigotimes_\kk 
{\mathcal U}'$. Clearly, $\widehat {\mathcal U}'_{\widehat {\bf q}}$ is a subcoalgebra of $\widehat {\mathcal U}_{\widehat {\bf q}}$ containing ${\bf e}$.

\medskip 

Now we define a pairing between $A$ and ${\mathcal U}$ by 
$$\langle a,E_{k_1} \cdots E_{k_\ell}\rangle :=\varepsilon(\partial_{k_\ell}\cdots \partial_{k_1}(a))$$
for any $k_1,\ldots,k_\ell\in [1,m]$.

\begin{proposition}

\label{pr:coalgebra homomorphism}
$\langle a b,u\rangle =\langle a\otimes b,\Delta(u)\rangle $
where $\langle a\otimes b,u_1\otimes u_2\rangle =\langle a,u_1\rangle \langle b,u_2\rangle $ for any $a,b\in A$, $u_1, u_2\in {\mathcal U}$. 
\end{proposition} 

\begin{proof} We need the following

\begin{lemma} $\langle a,E_ku\rangle=\langle \partial_k(a),u\rangle$ and  
$\langle \varphi_k(a),u\rangle=\langle a,\varphi_k^\vee(u)\rangle$ in the notation of Theorem \ref{th:generalized Feigin}, where $\varphi_k^\vee$ is an automorphism of ${\mathcal U}$ given by $\varphi_k^\vee(E_l)=q_{k l}E_l$.
    
\end{lemma}

\begin{proof} The first assertion is immediate. We prove the second assertion by induction in length of a homogeneous element $u$. Indeed,
$$\langle \varphi_k(a),E_l u'\rangle=\langle \partial_l\varphi_k(a),u'\rangle=q_{kl}\langle \varphi_k(\partial_\ell(a)),u'\rangle=q_{kl}\langle \partial_l(a),\varphi^\vee_k(u')\rangle$$
$$=q_{kl}\langle a,E_l\varphi^\vee_k(u')\rangle=\langle a,\varphi^\vee_k(E_lu')\rangle\ .$$   
\end{proof}

Now we can finish the proof of the proposition by induction in the length of a homogeneous element $u$. Indeed, 
$$\langle a b,E_ku'\rangle=\langle \partial_k(ab), u'\rangle=\langle \partial_k(a)b+\varphi_k(a)\partial_k(b), u'\rangle=\langle \partial_k(a)\otimes b+\varphi_k(a)\otimes \partial_k(b), \Delta(u')\rangle$$
$$=\langle \partial_k(a),u'_{(1)}\rangle\langle  b,u'_{(2)}\rangle +\langle \varphi_k(a),u'_{(1)}\rangle \langle\partial_k(b),u'_{(2)}\rangle $$
$$=\langle a,E_ku'_{(1)}\rangle\langle  b,u'_{(2)}\rangle +\langle a,\varphi_k^\vee(u'_{(1)})\rangle \langle b,E_k u'_{(2)}\rangle =\langle a\otimes b,E_ku'_{(1)}\otimes u'_{(2)} +\varphi_k^\vee(u'_{(1)})\otimes E_ku'_{(2)}\rangle$$
by the inductive assumption.

On the other hand, 
$$\Delta(E_ku')=\Delta(E_k)\Delta(u')=(E_k\otimes 1+1\otimes E_k)(u'_{(1)}\otimes u'_{(2)})$$
$$=E_ku'_{(1)}\otimes u'_{(2)}+(1\otimes E_k)(u'_{(1)}\otimes u'_{(2)})=E_ku'_{(1)}\otimes u'_{(2)}+\varphi_k^\vee (u'_{(1)})\otimes E_ku'_{(2)}$$

Thus, 
$\langle a b,E_ku'\rangle=\langle a\otimes b,\Delta(E_ku')\rangle$.
\end{proof}

Furthermore, we define the pairing 
$A\times \widehat {\mathcal U}'_{\widehat {\bf q}}\to A_{\bf q}$ by the formula
$\langle a,\sum_\gamma t_\gamma u_\gamma\rangle =\sum_\gamma \langle a,u_\gamma\rangle t_\gamma$
(The sum is finite by the local nilpotence of $\partial_k$.) 

We need the following general fact.
\begin{lemma} 
\label{le:grouplike homomorphism}Let $U$ be a coalgebra over an algebra $P$ (i.e., $U$ is a coalgebra in the monoidal category of $P$-bimodules) and $A$ be an algebra. Suppose that $\langle\cdot,\cdot\rangle$ is a pairing $A\otimes U\to P$ such that
$$\langle ab,u\rangle=\langle a,u_{(1)}\rangle \cdot \langle b,u_{(2)}\rangle$$
in Sweedler notation $\Delta(u)=u_{(1)}\otimes u_{(2)}$.

Then for any grouplike element ${\bf g}\in U$ the assignments $a\mapsto \langle a,{\bf g}\rangle$ define a homomorphism $\Phi:A\to P$.
    
\end{lemma}

\begin{proof} Since $\Delta({\bf g})={\bf g}\otimes {\bf g}$ it holds
$$\Phi(ab)=\langle ab,{\bf g}\rangle=\langle a,{\bf g}\rangle\langle b,{\bf g}\rangle=\Phi(a)\Phi(b)$$
for all $a,b\in A$.
\end{proof}

Define a map 
$\psi_{\widehat {\bf q}}:
A \to A_{{\widehat {\bf q}}^+}$ by 
$\psi_{\widehat {\bf q}}(x):=\langle x,{\bf e}\rangle$.
Expanding ${\bf e}$ into a power series we obtain 
\begin{eqnarray}\label{(1.7)}
\psi_{\widehat {\bf q}}(x)=\sum_{a_1,\ldots,a_m\in \ZZ_{\ge 0}}
\langle x,E_1^{(a_1)}E_2^{(a_2)}\cdots
E_m^{(a_m)}\big\rangle t_1^{a_1}t_2^{a_2}\cdots t_m^{a_m}
\end{eqnarray}
where $E_k^{(n)}=\displaystyle{{E_k^n\over [n]_{q_k}!}}$.
Note that the sum in \eqref{(1.7)} is always finite 
because ${\bf e}\in \widehat {\mathcal U}'_{\widehat {\bf q}}$. 

The following is an immediate consequence of Lemma \ref{pr:coalgebra homomorphism}  Theorem \ref{1.1}, and Lemma \ref{le:grouplike homomorphism} (taken with $P=A_{\bf q}$ and ${\bf g}={\bf e}$).

\begin{corollary}
\label{1.3} 
The map
$\psi_{{\widehat {\bf q}}}:A \to A_{{\bf q}}$ defined by \eqref{(1.7)} 
is a homomorphism of $\ZZ_{\ge 0}^m$-graded algebras.
\end{corollary}

\medskip

This finishes the proof of Theorem \ref{th:generalized Feigin} \endproof

\subsection{Injective valuations of algebras}
\begin{definition}\label{def_injective}
We say that a valuation $\nu:A\setminus \{0\} \twoheadrightarrow P$ onto a partial semigroup $P$ is {\it injective} if there exists a $\kk$-basis ${\bf B}\subset A$ of $A$ such that $\nu: {\bf B} \to P$ is a bijection. A basis fulfilling the latter property is called {\it adapted} with respect to $\nu$.
\end{definition}

\begin{remark} In the notation of the second part of Remark \ref{image_valuation},  if $\nu$ was injective and $Q$ was adapted to $\nu$, i.e., $\nu|_Q$ injective, then $\nu|_Q$ is an isomorphism. 
\end{remark}

\begin{remark}(Tautological valuation)
\label{monoidal_partial}
For a partial semigroup $P$, the tautological valuation $\nu_P:\kk P\setminus \{0\}\to P$ (see \eqref{eq:nu_P}) is injective and satisfies
$$\nu(\sum_{u\in P} \alpha_u [u]):= \max_{\alpha_u\in \kk^\times}\{u\} $$
Clearly, the basis $\{[u]\, :\, u\in P\}$ of $\kk P$ is naturally adapted to $\nu_P$ (even though the order on $P$ is not necessarily a well-order).

\end{remark}

\begin{remark}
In Proposition~\ref{pr:valuation after homomorphism}  if $\nu$ and $f$ are both injective then the valuation $f(\nu)$ is injective as well.  Indeed, if ${\bf B} \subset A$ is an adapted basis with respect to $\nu$ then $\bf B$ is an adapted basis with respect to $f(\nu)$ as well. 
\end{remark}

Now we consider behavior  of injective valuations under field extensions.

 Let $A$ be a $\kk$-algebra, and $\Gamma\subset A^\times$ be a subgroup. Suppose that $\Gamma$ is finitely generated free abelian group, denote by $\mathbb{K}$ the field of fractions of $\kk \Gamma$.

 \begin{proposition} 
 \label{pr:lifting of valuation}
 In the assumptions as above suppose that $\nu$ is a valuation $A\setminus \{0\}\to P$ for some well-ordered semigroup $P$ and 
 
 $\bullet$ $\nu(\Gamma)$ is in the center of $P$, and $\nu$ is injective on $\Gamma$, 

 $\bullet$ $\nu(\Gamma)\prec c$ for any $c\in P\setminus \nu(\Gamma)$.
 
 Then the assignments
 $a\mapsto \nu(a)\cdot\nu(\Gamma)$
 define a ($\mathbb{K}$-linear) valuation $\underline \nu:\mathbb{K}\otimes A\to \underline P=P/\nu(\Gamma)$.

 Moreover, $\nu$ is injective iff $\underline \nu$ is injective.
    
\end{proposition}

\begin{proof} Clearly, $\underline P$ is naturally ordered (see Lemma \ref{le:ordered projections}) so that the canonical homomorphism $P\twoheadrightarrow \underline P$ is ordered. Then $\underline \nu$
is a valuation, see Proposition~\ref{pr:valuation after homomorphism}.  

If $\nu$ is injective, then Proposition \ref{pr:injective} implies that $\underline \nu$ is also injective.

Conversely, suppose that $\underline \nu$ is injective. Let $\underline {\bf B}$ be any $\mathbb{K}$-basis of $\mathbb{K}\otimes A$ adapted to $\underline \nu$.
Clearly ${\bf B}=\Gamma\cdot \underline {\bf B}$ is a basis of $A$ adapted to $\nu$.

The proposition is proved.
\end{proof}

\begin{remark}
In the conditions of Lemma~\ref{basis_partial_monoid}  ${\bf B}=f^*(Q)\sqcup {\bf B}_I$ is a basis of $\kk P$ adapted to the tautological valuation of $\nu_P:\kk  P\setminus \{0\} \to P$.
  
\end{remark}
The proof of the following proposition is similar to the proof of Theorem~\ref{filtration}~ii),~iv).

\begin{proposition}\label{valuation_injective_partial}
Let $\nu:A\setminus \{0\} \twoheadrightarrow P$ be a valuation onto a partial semigroup $P$. When $P$ is well-ordered and $\dim (A_u)=1$ for any $u\in P$, the valuation $\nu$ is injective. Every set ${\bf B}\subset A$ such that the mapping $\nu:{\bf B} \to P$ is a bijection, is an adapted basis of $A$ (with respect to $\nu$). 

Vice versa, if $\nu$ is injective then $\dim (A_u)=1$ for any $u\in P$.
\end{proposition}

The following is a direct consequence of Corollary \ref{cor:new valuations}.

\begin{proposition}
\label{pr:injective subalgebra}
Let $\nu:A\setminus \{0\} \to P$ be an injective valuation of an algebra $A$ into a well-ordered (partial or entire) semigroup $P$, and $B$ be a subalgebra of $A$. Then the restriction of $\nu$ on $B\setminus \{0\}$ is also an injective valuation.    
\end{proposition}

\begin{remark} 
\label{rem: subalgebra injective}
In view of Remark \ref{rem: subspace injective}, it is interesting whether an analog of Proposition \ref{pr:injective subalgebra} holds without assumption of well-orderness of $P$.
    
\end{remark}

In the following theorem we consider different words in  generators of a partial semigroup representing the same element of the partial semigroup, among them we choose the minimal with respect to deglex (also for non-commutative partial semigroups), cf. Lemma~\ref{deglex},
and call this word canonical. The following theorem can be easily deduced from Corollary~\ref{adapted_arbitrarily}.

\begin{theorem}\label{monomial_basis_partial} Let $\nu:A\setminus \{0\} \twoheadrightarrow P_{\nu} \subset P$ be an injective well-ordered valuation into a partial semigroup $P$, generated by $P_0$,
and let $X_0$ be a generating set of $A$ such that $\nu|_{X_0}$ is a bijection $X_0\widetilde \to P_0$. 
Then the set of all monomials $x_u:=\vec {\prod\limits_{x\in X_0}} x$ corresponding to the canonical factorization of $u\in P_\nu$ is  an adapted to $\nu$ basis in $A$, and $\nu(x_u)=u$ (we will refer to the elements $x_u$ as {\it standard monomials}).
    
\end{theorem}

\begin{remark}  $X_0$ is not always a minimal generating set for $A$. The same applies to $P_0$. In principle we can require that $P_0$ is minimal by inclusion. In some cases, including submonoids of $\ZZ_{\ge 0}^m$, $P^{ind}$ of indecomposable elements of $P$ generate $P$, in which case we can choose $P_0:=P^{ind}$.  
    
\end{remark}

In the following theorem we show that given an injective valuation on an algebra, how one can define it on its quotient algebra.

\begin{theorem}\label{graded_partial}
Let $A$ be a $\kk$-algebra and $\nu_0:A\setminus \{0\} \twoheadrightarrow P$ be an injective valuation onto a well-ordered partial semigroup $P$. Let $I\subset A$ be an ideal. 

i) Then $\nu_0(I\setminus \{0\})$ is an ideal in $P$.
For $a\in (A/I)\setminus \{0\}$ the formula from Proposition  \ref{pr:surjective nu}, i.e.,
\begin{equation}\label{48}
\nu(a):= \min \nu_0(a+I)    
\end{equation}
defines an injective valuation $\nu:(A/I)\setminus \{0\} \twoheadrightarrow (P\setminus \nu_0(I\setminus \{0\}))$. If $\nu_0(a)\in P\setminus \nu_0(I\setminus \{0\})$ then $\nu(a)=\nu_0(a)$.

ii) If $u\in P\setminus \nu_0(I\setminus \{0\})$ is indecomposable then $u$ is also indecomposable in $P$.

iii) Let $\{x_u\, :\, u\in P\}$ be a standard monomial basis of $A$ with respect to $\nu_0$ (cf. Theorem~\ref{monomial_basis_partial}). Then ${\bf B}:= \{q(x_u)\, :\, u\in P\setminus \nu_0(I\setminus \{0\})\}$ is a standard monomial basis of $A/I$ with respect to $\nu$ where $q:A\twoheadrightarrow A/I$ is the natural projection.
\end{theorem}

\begin{proof} i) First, we note that if  $\nu_0(a)\in P\setminus \nu_0(I\setminus \{0\})$ then $\nu(a)=\nu_0(a)$. Indeed, suppose that on the contrary it holds $\nu_0(a+f)\prec \nu_0(a)$ (cf.  \eqref{48}). Then $\nu_0(f)=\nu_0(a)$ which contradicts the supposition.

Observe that for any $a\in (A/I)\setminus \{0\}$ it holds $\nu(a)\notin \nu_0(I\setminus \{0\})$. Indeed, otherwise $\nu(a)=\nu_0(a+f)\in \nu_0(I\setminus \{0\})$ for suitable $f\in I\setminus \{0\}$. Then there exists $g\in I\setminus \{0\}$ such that $\nu_0(g)=\nu_0(a+f)$. Due to the injectivity of $\nu_0$ there exists $\alpha \in \kk^\times$ for which holds $\nu_0(a+f+\alpha g)\prec \nu_0(a+f)$, this contradicts to the equality $\nu(a)=\nu_0(a+f)$ and to \eqref{48}.

Now let $a,b\in (A/I)\setminus \{0\}$ and $f,g\in I$ be such that $\nu(a)=\nu_0(a+f), \nu(b)=\nu_0(b+g)$ according to \eqref{48}. Then 
$$\nu(a+b)\preceq \nu_0(a+f+b+g) \preceq \max\{\nu_0(a+f), \nu_0(b+g)\} =\max\{\nu(a+f), \nu(b+g)\}$$
\noindent which justifies Definition~\ref{valuation_partial}~ii) for $\nu$. 

To verify Definition~\ref{valuation_partial}~iii) for $\nu$ assume that $\nu(a)\circ \nu(b) \in P\setminus \nu_0(I\setminus \{0\})$. Since 
$$\nu(ab)\preceq \nu_0(ab+ag+fb+fg)=\nu_0(a+f)\circ \nu_0(b+g)$$
\noindent due to \eqref{48} and to Definition~\ref{valuation_partial}~iii) for $\nu_0$, we get $\nu(ab)\preceq \nu(a)\circ \nu(b)$. Suppose that $\nu(ab)\prec \nu(a)\circ \nu(b)$. Let $\nu(ab)=\nu_0((a+f)(b+g)+f_0)$ for appropriate $f_0\in I\setminus \{0\}$ (see \eqref{48}). Hence 
$$\nu_0((a+f)(b+g)+f_0) \prec \nu(a)\circ \nu(b)=\nu_0((a+f)(b+g))$$
\noindent and thereby, $\nu_0((a+f)(b+g))=\nu_0(f_0)\in \nu_0(I\setminus\{0\})$. The obtained contradiction shows that $\nu(ab)=\nu(a)\circ \nu(b)$.

Finally, we prove that $\nu$ is injective. Let $a,b\in (A/I)\setminus \{0\}$ and $f,g\in I$ be such that $\nu_0(a+f)=\nu(a)=\nu(b)=\nu_0(b+g)$ (see  \eqref{48}). Since $\nu_0$ is injective there exists $\alpha \in \kk^\times$ such that either $\nu_0((a+f)+\alpha (b+g))\prec \nu_0(a+f)$ or $a+f+\alpha (b+g)=0$. In the former case $\nu(a+\alpha b)\preceq \nu_0((a+f)+\alpha (b+g))\prec \nu(a)$, while in the latter case $(A/I)\ni a+\alpha b=0$. 
\vspace{2mm}

ii) Suppose the contrary, then $u=u_1u_2$ for suitable $u_1, u_2\in P$. It holds $u_1, u_2 \notin \nu_0(I\setminus \{0\})$, this contradicts to that $u\in P\setminus \nu_0(I\setminus \{0\})$ is indecomposable. \vspace{2mm}

iii) Due to i) it holds $\nu(q(x_u))=\nu_0(x_u)=u$ for $x_u \in \bf B$ (cf. Theorem~\ref{monomial_basis_partial}) and $\nu((A/I)\setminus \{0\})=P\setminus \nu_0(I\setminus \{0\})=\nu(\bf B)$. Therefore, Proposition~\ref{valuation_injective_partial} implies that $\bf B$ is an adapted basis of $A/I$ with respect to $\nu$. Finally, ii) entails that $\bf B$ is a standard monomial basis. 

The theorem is proved. \end{proof}

\vspace{2mm}

\begin{example}
Let an algebra $A:=\kk[x,y]/(x^2-y^3)$. Following Theorem~\ref{quotient}  one produces an injective valuation $\nu: A\setminus \{0\} \twoheadrightarrow C$ onto a semigroup $C:=\{(i,j)\, :\, 0\le i<\infty, j=0,1\}$ where $(0,1)\circ (0,1)=(3,0)$, and $\nu(y^i)=(i,0), \nu(xy^i)=(i,1)$ (cf. Example~\ref{cusp}).

On the other hand, applying Theorem~\ref{graded_partial} one obtains an injective valuation ${\underline \nu} :A\setminus \{0\} \twoheadrightarrow P\subset \ZZ_{\ge 0}^2$ onto a partial semigroup $P$ which coincides with $C$ as a set, while $(0,1)\circ (0,1)$ is not defined in $P$. The values of ${\underline \nu}$ coincide with the corresponding values of $\nu$, i.e. ${\underline \nu}(y^i)=(i,0), {\underline \nu}(xy^i)=(i,1)$.
\end{example}

\begin{remark}\label{valuation_min}
For an entire monoid $M$ generated by a finite set $X$ denote by $f:\langle X\rangle \twoheadrightarrow M$ the natural epimorphism, where $\langle X\rangle$ denotes the free monoid generated by $X$. Then $\kk M=\kk \langle X\rangle /I$ where ideal $I$ is generated by differences $c-d,\ c,d\in \langle X\rangle,\ f(c)=f(d)$. Any order $\sigma$ on $X$ induces deglex order on $\langle X\rangle$ (see Lemma~\ref{deglex}). Denote by $\nu'_{\sigma}: \kk \langle X\rangle \setminus \{0\} \twoheadrightarrow \langle X\rangle$ the tautological valuation.  

Theorem~\ref{graded_partial} implies the valuation
$$\nu_{\sigma} : \kk M \twoheadrightarrow \langle X\rangle \setminus \nu'_{\sigma} (I\setminus \{0\}).$$
\noindent Observe that the coideal monoid $\langle X\rangle \setminus \nu'_{\sigma} (I\setminus \{0\})$ coincides with the partial monoid $M_{\sigma}:= (M, \circ _f)$ produced in Theorem~\ref{rem:coideal of projection}.
\end{remark}

\begin{remark} In the notation of Theorem~\ref{rem:coideal of projection}, let $I_f$ be the kernel of the corresponding algebra homomorphism $\kk P\twoheadrightarrow \kk Q$. Denote $J_f:=\nu_P(I_f)$. Then $P(f)= P\setminus J_f$,  provided that $P$ is well-ordered. Therefore, the restriction of the injective valuation of $\kk Q$ prescribed by Theorem \ref{graded_partial} to $Q$ is inverse of the bijection $P(f)\widetilde \to Q$.
\end{remark}

\begin{remark}
 One can study the following inverse issue to Theorem~\ref{graded_partial}. Let $J\subset A$ be an ideal in a commutative algebra $A$, and let $\nu: J\setminus \{0\} \to P$ be a valuation (not necessary injective) in a partial semigroup $P$ whose ordering $\prec$ fulfills the strict property. Assume in addition that for any element $a\in A$ it holds $aJ\neq \{0\}$ and that for any $c\in P$ there exists $d\in P$ such that $c,d$ are composable.

 When one can extend the valuation $\overline{\nu}:A\setminus \{0\} \to Q$ for a suitable partial semigroup $Q\supset P$ such that $\overline{\nu}|_{J\setminus \{0\}}=\nu$ ? To define $Q$ consider a set $P\times P$ with the component-wise composition $(c_1,c_2)\circ_Q (d_1,d_2):=(c_1\circ d_1, c_2\circ d_2)$ (provided that both $c_1, d_1$ and $c_2, d_2$ are composable) and impose the following relations (the idea is to treat $Q$ as a set of "fractions" with numerators and denominators from $P$). Firstly, we identify pairs $(c\circ d_1,d_1),\ (c\circ d_2,d_2)\in P\times P$ (provided that both $c,d_1$ and $c,d_2$ are composable). 
 Secondly, we identify in $Q$ pairs $(c_1,c_2),\ (d_1,d_2)$ if $c_1\circ d_2=c_2\circ d_1$ (provided that both $c_1,d_2$ and $c_2,d_1$ are composable in $P$). Thirdly, for any $a\in A$ if $ab_1=b_2,\ ab_3=b_4$ for non-zero $b_1,b_2,b_3,b_4\in J$, we identify in $Q$ the pairs $(\nu(b_2), \nu(b_1))$ and $(\nu(b_4), \nu(b_3))$. We define an order on $Q$ as follows: $(c_1,c_2)\prec_Q (d_1,d_2)$ iff $c_1\circ d_2 \prec c_2\circ d_1$ (provided that both  $c_1, d_2$ and $c_2,d_1$ are composable). 

 If the resulting semigroup $Q$ is ordered and contains $P$ embedded for $c\in P$ by $(c\circ d, d)\in Q$ such that $c,d$ are composable, then one can define an extension $\overline{\nu}(a):=(\nu(b_2), \nu(b_1))$. Moreover, in this case the order $\prec_Q$ fulfills the strict property.
\end{remark}

\subsection{Constructions of valuations in partial semigroups}

Denote by $A_1* A_2=A_2*A_1$
the free product of algebras $A_1$ and $A_2$.

The following is immediate.

\begin{lemma} Let $A_i$, $i=1,2$ be  algebras and let $\nu_i:A_i\setminus \{0\}\to P_i$, $i=1,2$ be a valuation of $A_i$ to a (partial) semigroup $P_i$. Suppose that $P_1*P_2$ has a compatible order (see the definition after Remark~\ref{rem:deglex}). Then 
 the free product $A_1*A_2$ has a natural valuation $\nu_1*\nu_2:A_1*A_2\setminus \{0\}\to P_1*P_2$.
    
\end{lemma}

\begin{example}\label{monoid_matrix_algebra}
In the notation of Example \ref{monoid_matrix}, clearly, $Mat_k(\kk)=\kk M_k$ and the tautological valuation $\nu:Mat_k(\kk)\setminus \{0\} \to M_k$
is injective and given by $\nu(e_{ij})=(i,j)$ from Definition~\ref{monoidal_partial}. 

Observe that the valuation of the unit of the algebra $Mat_k(\kk)$ equals $\nu(e_{1,1}+\cdots +e_{k,k})=(1,1)$ in both cases in Example~\ref{monoid_matrix}~i), ii).
\vspace{1mm}

Consider a partial semigroup $P'$ with an ordering $\triangleleft$. One can construct (see Lemma~\ref{semigroups_product}) a partial semigroup $M_k \times P'$ in which the ordering is given by the lexicographical pair $(\prec, \triangleleft)$, where $\prec$ is one of the described above orderings on $M_k$. If $\triangleleft$     fulfills the strict property then the resulting ordering fulfills the strict property as well (cf. Lemma~\ref{semigroups_product}). Thus, if an algebra $A$ admits an injective valuation onto $P'$ then the matrix algebra $Mat_k(\kk)\otimes A$ admits an injective valuation onto $M_k \times P'$, see Corollary~\ref{tensor}. \vspace{1mm} 

Denote by $T_k$ the partial monoid of paths in the complete directed graph having $k$ vertices and loops (cf. Example~\ref{paths}~iii)). So, $T_k$ is generated by the set $\{z_{i,j}\, :\, 1\le i,j\le k\}$. Following the construction in the proof of Theorem~\ref{semigroup_epimorphism} we produce an epimorphism $f:T_k \twoheadrightarrow M_k$ such that $f(z_{i,j})=(i,j)$, thus $f(z_{i_1,i_2}\circ z_{i_2,i_3}\circ \cdots \circ z_{i_{s-1},i_s})=(i_1,i_s)$. Denote by $\prec$ one of the introduced above orders on $M_k$ (say, i) or ii)). Now define an order $\vartriangleleft$ on $T_k$ as follows. We say that $z_{i_1,i_2}\circ \cdots \circ z_{i_{s-1},i_s} \vartriangleleft z_{j_1,j_2}\circ  \cdots \circ z_{j_{l-1},j_l}$ if either

$\bullet$ $f(z_{i_1,i_2}\circ \cdots \circ z_{i_{s-1},i_s}) \prec f(z_{j_1,j_2}\circ  \cdots \circ z_{j_{l-1},j_l})$, either

$\bullet$ $f(z_{i_1,i_2}\circ \cdots \circ z_{i_{s-1},i_s}) = f(z_{j_1,j_2}\circ  \cdots \circ z_{j_{l-1},j_l})$ and $s<l$, or

$\bullet$  $f(z_{i_1,i_2}\circ \cdots \circ z_{i_{s-1},i_s}) = f(z_{j_1,j_2}\circ  \cdots \circ z_{j_{l-1},j_l}),\, s=l$ and the vector $(i_1,\dots,i_s)$ is less than $(j_1,\dots,j_l)$ in the lexicographical order (in which, say, $1<\cdots < n$).

\noindent One can verify that $f$ satisfies Proposition~\ref{monoid_epimorphism}.

Note that $f$ induces a natural epimorphism of semigroup algebras $\kk T_k \twoheadrightarrow \kk M_k=Mat_k(\kk)$. \vspace{1mm}

Consider $M_\infty:=\{(i,j)\, :\, 1\le i,j <\infty\}$, this is naturally an ordered partial semigroup of infinite matrices, and inclusions $M_k\subset M_\infty$ are ordered, moreover $M_\infty$ is their injective limit. Despite $\prec$ is not a well ordering, the semigroup algebra $\kk M_\infty$ is the algebra 
$Mat_\infty(\kk)$ of infinite matrices with finite support, and the
injective valuation onto $P_\infty$ provides the tautological valuation $Mat_\infty(\kk)\setminus \{0\} \to M_\infty$
with an adapted basis $\{e_{i,j}\, :\, 1\le i,j<\infty\}$. \vspace{2mm}

Denote by $F:=\kk\langle \{e_{i,j}\ :\ 1\le i,j\le k\} \rangle$ the free algebra with the natural injective valuation $\nu_0$ onto the free semigroup $P:= < (i,j),\ 1\le i,j\le k >$. We assume that $P$ is equipped with the well ordering produced in Lemma~\ref{deglex}.
Denote by 
$$I:= \langle \{e_{i,j}e_{p,q}\ : \ j\neq p,\ 1\le i,j,p,q\le k\} \cup \{e_{i,l}-e_{i,j}e_{j,l}\ :\ 1\le i,j,l\le k\} \rangle$$
\noindent an ideal in $F$. When we apply Theorem~\ref{graded_partial} we obtain an injective valuation $\nu : (F/I)\setminus \{0\} = Mat_k(\kk)\setminus \{0\} \twoheadrightarrow P\setminus \nu_0(I\setminus \{0\})$. Observe that $P\setminus \nu_0(I\setminus \{0\})$ is a partial semigroup consisting of $k^2$ elements $\{(i,j)\ :\ 1\le i,j\le k\}$ such that no composition of them is defined since $\nu_0(e_{i,j}e_{p,q})=(i,j)\circ (p,q)$ and $\nu_0(e_{i,l}-e_{i,j}e_{j,l})=(i,j)\circ (j,l)$. Thus, $P\setminus \nu_0(I\setminus \{0\})$ differs from $M_k$.
\end{example}

\begin{remark}
Recall from Theorem \ref{semigroup_epimorphism}
that any finitely generated partial semigroup $P$ can be covered by a coideal of an entire semigroup $\widehat P$. For instance, we can take $\widehat M_n$ to be generated by all $\widehat{(ij)}$, $i,j=1,\ldots,n$ subject to $\widehat{(ij)}\circ \widehat{(jk)}=\widehat{(ik)}$. However, in a contrast with free coideal semigroup in Theorem \ref{semigroup_epimorphism}, we do not know whether (an appropriate coideal of) $\widehat M_n$ is ordered in a compatible way.  It would be interesting to classify all ordered partial semigroups $P$ which admit such a lifting to (coideals of)  ordered entire semigroups $\widehat P$.

On the other hand, one can apply Theorem~\ref{graded_partial} to the free semigroup $\widehat P$ freely generated by $(ij)$, $ 1\le i,j\le n$, tautological valuation $\nu_0$ on $\kk \widehat P$, and the ideal
$$I:=\langle \{e_{ij}e_{pq}\ :\ p\neq q\} \cup \{e_{ij}e_{jl}-e_{il}\}\rangle \subset \kk \widehat P.$$
\noindent Then one obtains an injective valuation $\nu : (\kk \widehat P/I)\setminus \{0\}=A\setminus \{0\} \twoheadrightarrow \widehat P\setminus J$, where $J=\nu_0(I\setminus\{0\})$ is the corresponding ideal of $\widehat P$. By definition,  $\widehat P\setminus J$ is a partial semigroup consisting of $n^2$ elements $\{(ij)\ :\ 1\le i,j\le n\}$ such that no composition of them is defined. Thus, $\widehat P\setminus J$ differs from $M_n$.

\end{remark}

\begin{example} Let $W$ be a finite reflection group on the space $V$, recall that its coinvariant algebra $A_W=S(V)/<S(V)^W_+>$ has dimension $|W|$. Also if $W$ is a Weyl group of a complex semisimple group $G$, then $A_W\cong H^*(G/B)$, where $B$ is the Borel subgroup of $G$. In this case, $A_W$ has a canonical Schubert basis $X_w$, $w\in W$.

If $W=<s_1,s_2| s_1^2=s_2^2=1, (s_1s_2)^n=1>$ is dihedral of order $2n$, then $A_W=\CC[z,\overline z]/<z\overline z,z^n+\overline z^n>$ because $W$ acts on $V=\CC\cdot z\oplus \CC\cdot \overline z$ by $s_1(z)=\overline z$, $s_1s_2(z)=\zeta z,s_1s_2(\overline z)=\zeta^{-1} \overline z$, where $\zeta=e^{\frac{2\pi i}{n}}$, therefore $z\overline z$ and $z^n+\overline z^n$ are basic $W$-invariants. Writing $z=x+iy$ we expect that the Schubert basis is  $\{Re~z^k=\frac{z^k+\overline z^k}{2},Im~z^k=\frac{z^k-\overline z^k}{2i},k=0,\ldots,n\}\setminus\{0\}$. Note that in this case $A_W\cong  \CC[z,\overline z]/<z\overline z,z^n-\overline z^n>=\CC P$, where $P$ is a partial additive monoid on $I_n\sqcup_{0,n} I_n$ where $I_n$ is the partial monoid on $[0,n]$ with  $a\circ b$ defined iff $a+b\le n$, in which case the composition is $a+b$ and $\sqcup_{0,n}$ stands for disjoint union with identified unit $0$ and identified $n$. Namely, the first copy of $I_n$ consists of $z^k, 0\le k\le n$, while the second copy consists of $\overline z^k, 0\le k\le n$, note that $z^{n+1}=\overline z^{n+1}=0$. \vspace{1mm}

Note also that $A_{S_3}\cong \CC[x_1,x_2,x_3]/<e_1,e_2,e_3>=\CC[x_1,x_2]/<x_1^2+x_1x_2+x_2^2,x_1x_2(x_1+x_2)>$ where $e_1=x_1+x_2+x_3$, $e_2=x_1x_2+x_1x_3+x_2x_3$, $e_3=x_1x_2x_3$. An $S_3$-equivariant isomorphism  is given by $z=x_1-\zeta x_2$, $\overline z=x_2-\zeta x_1$. The latter algebra has a Schubert basis $\{1,x_1,x_1+x_2,x_1^2,x_1x_2,x_1^2x_2\}$.

Note also that $A_{I_2(4)}\cong \CC[x_1,x_2]/<x_1^2+x_2^2,x_1^2x_2^2>$ with the action given by $s_1(x_1)=x_2$, $s_2(x_2)=-x_2$, $s_2(x_1)=x_1$. An $I_2(4)$-equivariant isomorphism  is given by $z=x_1-i x_2$, $\overline z=x_1+i x_2$. The latter algebra has a Schubert basis $\{1,x_1,x_1+x_2,x_1^2,x_1x_2,x_1^2x_2+x_1x_2^2,x_1^3,x_1^3x_2\}$. \vspace{1mm}

When $n$ is odd $P$ admits the following ordering fulfilling the strict property (see Definition~\ref{monoid_partial}):
$$1\prec z\prec \cdots \prec z^{(n-1)/2}\prec \overline z \prec \cdots \prec \overline z^{n-1}\prec z^{(n+1)/2} \prec \cdots \prec z^{n-1} \prec z^n (=\overline z^n).$$
\noindent In contrast, when $n$ is even there is no ordering of $P$ fulfilling the strict property since if $z^{n/2}\prec \overline z^{n/2}$ (or, respectively $z^{n/2}\succ \overline z^{n/2}$) then $z^n \prec \overline z^n$ (respectively, $z^n \succ \overline z^n$). On the other hand, any ordering on $P$ merging the orderings $1\prec z\prec \cdots \prec z^n$ and $1\prec \overline z\prec \cdots \prec \overline z^n$ satisfies Definition~\ref{monoid_partial}. \vspace{2mm}

One can consider another representation $A_W=\CC Q$ where $Q$ is a partial monoid
$$Q:=\{c_k:=z^k+\overline z^k\ :\ 0\le k\le n\} \sqcup \{d_k:= z^k-\overline z^k\ :\ 0< k<n\}$$
\noindent with the following composition rules:

$\bullet$ $c_k\circ c_l=d_k\circ d_l=c_{k+l}$ iff $k+l\le n$;

$\bullet$ $c_k\circ d_l=d_{k+l}$ iff $k+l<n$.

\noindent Then $Q$ satisfies Definition~\ref{monoid_partial} with an ordering
$$1\prec c_1\prec d_1\prec \cdots \prec c_{n-1}\prec d_{n-1}\prec c_n,$$
\noindent while $Q$ does not admit an ordering fulfilling the strict property. \vspace{2mm}

Now we construct a common adapted basis of $A_W$ for a pair of injective valuations $\nu_P:A_W\setminus \{0\}(=\CC P\setminus \{0\})\twoheadrightarrow P$ and $\nu_Q:A_W\setminus \{0\}(=\CC Q\setminus \{0\})\twoheadrightarrow Q$ (see Theorem~\ref{intro:JH_basis}). When $n$ is odd, the  common basis consists of
$$\{1,z^n\}\sqcup \{z^k, z^k+\overline z^k\ :\ 1\le k<n/2\}\sqcup \{\overline z^k, z^k+\overline z^k\ :\ n/2 <k<n\}.$$
\noindent When $n$ is even, fix the following ordering in $P$:
$$1\prec z\prec \overline z\prec \cdots \prec z^k\prec \overline z^k\prec \cdots \prec z^{n-1}\prec \overline z^{n-1} \prec z^n(= \overline z^n).$$
\noindent Then the common basis consists of
$$\{1,z^n\}\sqcup \{z^k, z^k+\overline z^k\ :\ 1\le k<n\}.$$ \vspace{1mm}

Consider the tautological injective valuation $\nu_0:\CC[z,\overline z]\setminus \{0\}\twoheadrightarrow \ZZ_{\ge 0}^2$ where $\ZZ_{\ge 0}^2$ is endowed with lex ordering in which $\overline z \prec z$. If we apply Theorem~\ref{graded_partial} to an ideal $I:=\langle z\overline z, z^n-\overline z^n\rangle \subset \CC[z,\overline z]$ then we obtain an injective valuation $\nu: A_W\setminus \{0\} \twoheadrightarrow P'$, where $P'=\{1,z,\dots, z^{n-1}, \overline z,\dots, \overline z^n\}$. Thus, $P'$ contains also $2n$ elements as $P$, but differs from $P$ as a partial monoid since $z^n$ is not defined in $P'$ unlike $P$. Nevertheless, $\nu$ coincides with $\nu_P$ element-wise.

\end{example}

\begin{example}\label{Hecke} Recall that the nil Hecke algebra ${\mathcal H}_{S_2}$ of $S_2$ is generated by $\alpha$, $x$ subject to 
$$x^2=0,x\alpha+\alpha x=-2$$
In particular, $s=\alpha x+1=-x\alpha-1$ is an involution. One can show that ${\mathcal H}_{S_2}$ is the Takeuchi product of $A=\kk[\alpha]$, $B=\kk[y]/(y^2)$ is the semigroup algebra of the nil-Coxeter monoid $\{1,x\}$ with $H$ is an algebra generated by $d,s$ subject to $ds=-sd$, $s^2=1$, $d^2=0$. Then $H$ is a $4$-dimensional Hopf algebra with $\Delta(s)=s\otimes s$, $\Delta(d)=d\otimes 1+s\otimes d$, i.e., the celebrated Sweedler Hopf algebra with $s\act \alpha=-\alpha$, $d$ acts on $\kk[\alpha]$ as an $s$-derivation via $d(\alpha)=-2$, $d(pq)=d(p)q+s(p)d(q)$, 
$\delta(x)=s\otimes x+d\otimes 1$. The (bullet) conditions in the end of Remark~\ref{rem:Takeuchi}  hold with the tautological valuations $\nu_1:\kk[\alpha]\setminus \{0\}\to P_1=\ZZ_{\ge 0}$, $\nu_2:\kk[y]/y^2 \setminus \{0\} \to P_2=\{1,x\}$; $h_x=s$, $h_1=1$.

More generally, let $W=<s_i,i\in I|s_i^2=1, (s_is_j)^{m_{ij}}=1>$ be a Coxeter group and $V=\oplus \kk \alpha_i$ be its reflection representation with a basis $\alpha_i$ so that the action is given by $$s_i(\alpha_j)=\alpha_j-a_{ij}\alpha_i$$
where $A$ is the corresponding Cartan matrix.

Then ${\mathcal H}_W$ is generated by $x_i,\alpha_i$ subject to 
$$x_i^2=0, x_i\alpha_j=s_i(\alpha_j)x_i-a_{ij}$$
and the braid relation
$$\underbrace{x_ix_jx_i\cdots}_{m_{ij}}=\underbrace{x_jx_ix_j\cdots}_{m_{ij}}$$

It is easy to see that $\kk W$ embeds into ${\mathcal H}_W$ via $s_i\mapsto \alpha_ix_i+1$.

This embedding extends to an embeddings $S(V)\rtimes \kk W \hookrightarrow {\mathcal H}_W$ and ${\mathcal H}_W\hookrightarrow Frac(S(V))\rtimes \kk W$ and

$${\mathcal H}_W\cong (\kk W)_0\otimes S(V)$$
as a vector space
where $(\kk W)_0=<x_i>$ is the nil-Coxeter algebra.

${\mathcal H}_W$ admits a quotient $\underline {\mathcal H}_W$ by the ideal in $S(V)$ generated by $W$-invariants so that 
$$\underline {\mathcal H}_W\cong (\kk W)_0\otimes A_W$$

It is proved in \cite{Demazure}, \cite{Brundan} that  if $|W|=N$, then the algebra ${\mathcal H}_W$ is isomorphic to $Mat_N(S(V)^W)$, hence $\underline {\mathcal H}_W\cong Mat_N(\kk)$. Applying here Example~\ref{monoid_matrix} and Corollary \ref{tensor}, we obtain an injective valuation on ${\mathcal H}_W$.

We also expect that similarly to ${\mathcal H}_{S_2}$, ${\mathcal H}_W$ is a Takeuchi product of $S(V)$ and the semigroup algebra of $W_{\bf 0}$ over some Hopf algebra $H$ so that the conditions in the end of Remark \ref{rem:Takeuchi} hold.

\end{example}

\begin{example}\label{Galois} [Galois extensions] Let $\mathbb{K}$ be a finite Galois extension of $\kk$ and let $G=Gal(\mathbb{K}/\kk)$. Then the assignments $g\otimes a\mapsto g\circ L_a$ for all $g\in G$, $a\in \mathbb{K}$ define an isomorphism of algebras $\mathbb{K}\rtimes \kk G\widetilde \to End_\kk(\mathbb{K})$ (where $L_a:\mathbb{K}\to \mathbb{K}$ is the multiplication by $a\in \mathbb{K})$, this again follows from \cite{Demazure}, \cite{Brundan}. In particular, any choice of basis $\kk$-basis $\{b_1,\ldots,b_n\}$ of $\mathbb{K}$ canonically identifies the algebra
$\mathbb{K}\rtimes \kk G$ with $Mat_n(\kk)$.  Similarly to Example~\ref{Hecke} one can  produce an injective valuation on $\mathbb{K}\rtimes \kk G$.
\end{example}

The following result gives a large class of noncommutative injective valuations.

\begin{corollary} In the assumptions of Theorems \ref{th:string valuations general},  \ref{th:generalized Feigin} in addition
suppose that:

$\bullet$ $\kk_0$ is a subfield of $\kk$ and $q_{k\ell}$ generate a free abelian subgroup $\Gamma\subset \kk^\times$ such that $\Gamma\cap \kk_0=\{1\}$, 

$\bullet$ $A'$ is a $\kk_0$-subalgebra of $A$  containing all $q_{k\ell}$ such that  that the restriction of $\Phi_\partial$ to $A'$ is injective and $\Phi_\partial(A')\subset \kk_0 P_{\bf q}$, where $P_{\bf q}$ is a multiplicative sub-semigroup of $A_{\bf q}$ generated by $t_1,\ldots,t_m$ and by $\Gamma$ (by the suppositions it holds $\kk_0 P_{\bf q}\subset A_{\bf q}$).

Then the assignments 
$x \mapsto \nu_{P_{\bf q}}(\Phi_\partial(x))$ is an injective valuation $A'\setminus \{0\}\to P_{\bf q}$ (here $\nu_{P_{\bf q}}:\kk_0 P_{\bf q}\setminus\{0\} \to P_{\bf q}$ is the tautological valuation of $\kk_0 P_{\bf q}$).

\end{corollary}

One can verify the following proposition.

\begin{proposition}\label{monomial_partial}
For a commutative $\kk$-algebra $A$ let a set $P$ of monomials in elements $a_1,\dots,a_n \in A$ form a $\kk$-basis in $A$, and $P$ be a partial semigroup (in particular, $P$ is endowed with a linear order).  
For an element
$$a=\sum _{u\in P} \alpha_u u\in A\setminus \{0\},\, \alpha_u\in \kk^\times$$
\noindent define $\nu(a):=\max\{u\}$ where $\max$ ranges over $u$ from the latter sum. Then $\nu:A\setminus \{0\} \twoheadrightarrow P$ is an injective valuation onto $P$.
Moreover, $P$ is adapted with respect to $\nu$.
\end{proposition}

\subsection{Injective valuations onto coideal partial semigroups via tropical geometry and adapted bases}

\begin{proposition}\label{partial_filtration}
i) Let $\nu:A\setminus \{0\} \to P$ be a valuation in a partial semigroup $P$. For $u\in P$ denote a $\kk$-linear space $A_{\preceq u}:=\{a\in A\setminus \{0\}\, :\, \nu(a)\preceq u\}\cup \{0\}$ (see Definition~\ref{valuation_partial}~i), ii)). Then $A_{\preceq u}A_{\preceq v} \subset A_{\preceq u\circ v}$, provided that $u\circ v \in P$. Thus the family $\{A_{\preceq u}\, :\, u\in P\}$ forms a filtration of $A$.

ii) Let $P$ be a well-ordered partial semigroup and $\{A_u\, :\, u\in P\}$ be a filtration of an algebra $A$. For $a\in A\setminus \{0\}$ setting $\nu(a)$ to be the minimal $u\in P$ such that $a\in A_u$, defines a valuation $\nu:A\setminus \{0\} \to P$.
\end{proposition}

\begin{problem}
Describe all possible orderings on $M_k$.    
\end{problem}

\begin{example}
Let us apply the construction from Proposition~\ref{pr:partial homomorphism} to the symmetric group $P:=S_k$ and $Q:=M_k$ taking as $\nu$ the valuation from Example~\ref{monoid_matrix_algebra}~ii). We consider the standard representation of $S_k$ in $GL_k$. This provides a partial homomorphism $f:S_k \to M_k$. Then following Remark~\ref{admissible} one obtains a partial semigroup $R_k:= (S_k)_{S_f}$ and a homomorphism from $R_k$ to $M_k$. One can explicitly describe $R_k$ as follows. Two permutations $p,q\in M_k$ are composable in $R_k$ iff $p(1)=1$ taking into account that $\nu(p)=(1,p^{-1}(1))$. 
\end{example}

In case of a coideal partial semigroup $P$ the following construction allows one to obtain a stronger property of filtrations.

\begin{definition}\label{valuation_coideal}
For an algebra $A$ we say that $\nu:A\setminus \{0\} \twoheadrightarrow P$ is a {\it valuation onto a coideal partial semigroup} $P\subset M$ if in addition to Definition~\ref{valuation_partial} for any elements $a,b\in A\setminus \{0\}$ an inequality $\nu(ab)\preceq \nu(a)\circ \nu(b)\in M$ holds, provided that $ab\neq 0$.
\end{definition}

Recall (cf. the definition prior to Remark~\ref{monoid_archimedian}) that an order $\prec$ on a semigroup $M$ is {\it archimedian} if for any $u\in M$ the set all elements of $M$ less than $u$ is finite.
 
\begin{proposition}
Let $\nu:A\setminus \{0\} \twoheadrightarrow P$ be a valuation on a $\kk$-algebra onto a coideal partial semigroup $P\subset M$. We assume that $M$ is endowed with an archimedian order $\prec$. In this case for
the filtration: $A_{\preceq u}$
it holds
$$A_{\preceq u} A_{\preceq v} \subset A_{\max\{P\ni w \preceq u\circ v\}}$$
\noindent owing to Definition~\ref{valuation_coideal}. Observe that the latter maximum exists since $\prec$ is archimedian.
\end{proposition}

One can also denote $A_u:= A_{\preceq u}/A_{\prec u}$ and the natural projection $p_u:A_{\preceq u} \twoheadrightarrow A_u$. The following remark extends Remark~\ref{length} to coideal partial semigroups.

\begin{remark}
Let $A$ be a (not necessary commutative) $\kk$-algebra and $\nu:A\twoheadrightarrow P\subset M$ be an injective valuation onto a coideal partial semigroup $P$. We assume that $M$ is endowed with a linear order $\prec$ and a function $f:M\to \ZZ_{\ge 0}$ such that $c_1\prec c_2$ implies that $f(c_1)\le f(c_2)$, and $f(c_1+c_2)\le f(c_1)+f(c_2)$ for $c_1,c_2\in M$, moreover the set $C_n:=\{c\in M\, :\, f(c)\le n\}$ is finite for any $n\in \ZZ_{\ge 0}$. Note that the latter implies that the order $\prec$ is archimedian. Then the $\kk$-subspaces $A_n:=\{a\in A\setminus \{0\}\, :\, f(\nu(a))\le n\} \cup \{0\},\, n\in\ZZ_{\ge 0}$ provide a filtration of $A$ such that $\dim (A_n)=|C_n|$.     
\end{remark}

\begin{proposition}\label{direct}
When $\nu:A\setminus \{0\} \twoheadrightarrow P$ is a valuation onto a coideal partial semigroup $P\subset M$,  one can define a graded associated algebra $A:=\bigoplus_{u\in P} A_u$ as follows. Let $u,v\in P, c\in A_u, d\in A_v$. If  $u+v\in P, a\in  A_{\preceq u}, b\in A_{\preceq v}$ such that $p_u(a)=c, p_v(b)=d$ then we define the product $cd:= p_{u+v}(ab)\in A_{u+v}$. It holds $cd\neq 0$. Otherwise, if $u+v\notin P$ then we define $cd:=0$.
\end{proposition}
    
\begin{proof} The correctness of the definition of the product $cd$ and that $cd\neq 0$ in case when $u+v\in P$ follows from Definitions~\ref{valuation_partial}~iii), ~\ref{valuation_coideal}. The associativity of $A$ can be verified taking into account Definition~\ref{monoid_partial}. 

The theorem is proved. \end{proof}

\vspace{2mm}

Sometimes we consider valuations of  algebras $A$ over domains $R$ which are free as $R$-modules. Namely, let $R$ be a subring of a field $\kk$, we can view any $\kk$-algebra $\hat A$ as an $R$-algebra. Then for an   $R$-subalgebra $A$, any $\kk$-valuation of $\hat A$ is the $R$-valuation of $A$. 

And conversely,  any $R$-valuation of $A$ extends to a $\kk$-valuation of $\kk\otimes_R A$. 

\begin{example}\label{Laurent}
Consider a polynomial algebra 
$\kk[x_1,\dots,x_n]$ endowed with a (non-injective when $n\ge 2$) valuation $\deg$ to $\ZZ_{\ge 0}$. Denote by $A_i, i\ge 0$ the linear span of all Laurent monomials in $x_1,\dots,x_n$ of degree $i$. Then the algebra $A:=\bigoplus _{i\ge 0}A_i$ is endowed with the injective $A_0$-valuation $\deg$ (see Definition~\ref{def_injective}). As an adapted basis of $\mathcal A$ one can take $\{x_1^i, i\ge 0\}$. 
\end{example}

One can generalize the construction from Example~\ref{Laurent} which provides a way to produce injective valuations. 

\begin{proposition}
Let $A=\bigoplus _{c\in P} A_c$ be a graded commutative domain, where $(P,+)$ is an ordered monoid in which  its neutral element $0$ is minimal. One can treat $\nu_0: A\setminus \{0\} \twoheadrightarrow P$ as a (possibly non-injective) valuation (cf. Proposition~\ref{direct}). 

For $d\in P$ denote by $A_d$ the set of all fractions of the form $f/g$, where $f\in A_{d_1}, g\in A_{d_2}, d=d_1-d_2$. The graded algebra $A:=\bigoplus_{d\in P} A_d$ (being a subalgebra of the quotient field of $A$) admits the natural valuation $\nu:A\setminus \{0\} \twoheadrightarrow P$. Then the $A_0$-valuation $\nu$ is injective. As an adapted basis one can take a set $\{b_d\ :\ b_d\in A_d\}$.
\end{proposition}

\begin{remark}
One can take a smaller ring $A'\subset A$ satisfying the same properties as $A$. Let $\nu_0$ be $R$-valuation where a subalgebra $R\subset A$ (e.g. one can take $R=A_0$). Consider homogeneous generators ${\bf B}$ of $A$ as $R$-module, i.e. for each $b\in \bf B$ there is (a unique) $d\in P$ such that $b\in A_d$, in other words, $\nu_0(b)=d$. It holds $A_d=R\cdot \{b\in {\bf B}\ :\ b\in A_d\}$. Then one can take $A'_d$ to consist of all fractions of the form $f/g, f\in A_{d_1}, g\in A_{d_2}, d=d_1-d_2$, where $g$ is a product of some elements from $\bf B$. Define the algebra $A':=\bigoplus_{d\in P} A'_d$. As an adapted basis of $A'$ with respect to $\nu$ for each $d\in P$ one can take an element $b_d\in \bf B$ such that $\nu_0(b_d)(=\nu(b_d))=d$.
\end{remark}

\vspace{2mm}

The following proposition generalizes Theorem~\ref{quotient} to partial semigroups. We utilize the notations from Theorem~\ref{quotient}.

\begin{proposition}\label{quotient_partial}
Let $A=\kk[x_1,\dots,x_n]/I$ be an algebra and $I_{trop}\subset I$ be a $(n-d)$-dimensional subideal satisfying the following properties. Assume that there exists a common $(n-d)$-dimensional rational plane $H\subset \RR^n$ of the tropical variety $Trop (I_{trop})$ such that $H$ is prop and $I_{trop}$ is saturated with respect to $H$. Then there exists a coideal partial monoid $P\subset \ZZ_{\ge 0}^n/H_{\ZZ}$ and an injective valuation $\nu:A\setminus \{0\}\twoheadrightarrow P$.  A linear order on $P$ is induced by a linear order on $\ZZ_{\ge 0}^n/H_{\ZZ}$ which in its turn, is determined by a hyperplane from $ETrop(I_{trop})$.   
\end{proposition}

\begin{proof} Apply Theorem~\ref{quotient} to the ideal $I_{trop}$ and obtain an injective valuation 
$$\nu_0:(\kk[x_1,\dots,x_n]/I_{trop}) \twoheadrightarrow \ZZ_{\ge 0}^n/H_{\ZZ}.$$
\noindent Then apply Theorem~\ref{graded_partial} to $\nu_0$ which results in $\nu$. Observe that $P=(\ZZ_{\ge 0}^n/H_{\ZZ}) \setminus \nu_0((I/I_{trop})\setminus \{0\})$.
\end{proof}

\vspace{2mm}

The following proposition is inverse to Proposition~\ref{monomial_partial} and extends Theorem~\ref{th:basis} to valuations onto coideal partial monoids.

\begin{proposition}\label{basis_partial}
Let $A$ be a commutative $\kk$-algebra and $\nu:A\setminus \{0\} \twoheadrightarrow P\subset M$ be an injective valuation onto a finitely-generated coideal partial commutative monoid $P$, where the monoid $M$ is endowed with a linear well-ordering $\prec$. Assume that $c_1,\dots,c_s \in P$ is a  family of generators of $P$. Take $a_1,\dots,a_s \in A$ such that $\nu(a_i)=c_i, 1\le i\le s$. Then $A=\kk[a_1,\dots,a_s]/I$ for a suitable ideal $I\subset \kk[a_1,\dots,a_s]$. Consider a linear ordering $\vartriangleleft$ on monomials in $a_1,\dots,a_s$ such that $a_1^{i_1}\cdots a_s^{i_s} \vartriangleleft a_1^{j_1}\cdots a_s^{j_s}$ if either  $M\ni i_1\nu(a_1)\circ \cdots \circ i_s\nu(a_s) \prec j_1\nu(a_1)\circ \cdots \circ j_s\nu(a_s)\in M$ or $i_1\nu(a_1)\circ \cdots \circ i_s\nu(a_s) = j_1\nu(a_1)\circ \cdots \circ j_s\nu(a_s)$ and the monomial $a_1^{i_1}\cdots a_s^{i_s}$ is less than $a_1^{j_1}\cdots a_s^{j_s}$ in deglex. Then the monomials in $a_1,\dots,a_s$ belonging to the complement of the monomials ideal $J$ of leading monomials of the Gr\"obner basis of $I$ (relatively to $\vartriangleleft$), constitute an adapted basis of $A$ with respect to $\nu$.   
\end{proposition}

\begin{proof} For a monomial $a:=a_1^{i_1}\cdots a_s^{i_s}$ such that $i_1\nu(a_1)\circ \cdots \circ i_s\nu(a_s) \in M\setminus P$ it holds $\nu(a) \prec i_1\nu(a_1)\circ \cdots \circ i_s\nu(a_s)$ due to Definitions~\ref{valuation_partial}~iii), ~\ref{valuation_coideal}. Therefore $J$ contains all monomials $a_1^{i_1}\cdots a_s^{i_s}$ for which $i_1\nu(a_1)\circ \cdots \circ i_s\nu(a_s) \in M\setminus P$.

On the other hand, among all monomials $a_1^{j_1}\cdots a_s^{j_s}$ with a fixed valuation $v:= j_1\nu(a_1)\circ \cdots \circ j_s\nu(a_s)\in P$ all these monomials belong to $J$ except of a single one being minimal in deglex since $\nu$ is injective (cf. the proof of Theorem~\ref{th:basis}~ii) and remark~\ref{minimal}). Denote the latter monomial by $a_v$. Then the set $\{a_v\, :\, v\in P\}$ constitutes an adapted basis of $A$ with respect to $\nu$. 

The proposition is proved. \end{proof}

\vspace{2mm}

\begin{definition}\label{def_rank}
For a commutative partial monoid $P$ we define its rank $rk(P)$ to be the maximal number of elements $c_1,\dots, c_r\in P$ such that all the elements $i_1c_1\circ \cdots \circ i_rc_r\in P, i_1,\dots,i_r \ge 0$ are pairwise distinct. In this case we call  elements $c_1,\dots, c_r$ independent. 
\end{definition}

The following extends Corollary~\ref{valuation_rank} to coideal partial monoids.

\begin{proposition}\label{rank_partial}
Let $A$ be a commutative $\kk$-algebra and $\nu:A\setminus \{0\} \twoheadrightarrow P\subset M$ be an injective valuation onto a finitely-generated coideal partial commutative monoid, where monoid $M$ is endowed with a linear well-ordering.  Then $\dim (A) = rk (P)$.    
\end{proposition}

\begin{proof} Denote $r:=rk(P)$ and let $c_1,\dots, c_r \in P$ be independent. Take $a_1,\dots, a_r \in A$ for which $\nu(a_i)=c_i, 1\le i\le r$. Then monomials $a_1^{i_1}\cdots a_r^{i_r}, i_1,\dots, i_r\ge 0$ are $\kk$-linearly independent, hence $d:=\dim(A)\ge r$.

Conversely, among monomials belonging to the complement of $J$ (see Proposition~\ref{basis_partial}) there are monomials $b_1,\dots, b_d$ such that all monomials in $b_1,\dots,b_d$ belong to the complement of $J$ taking into account the property of the Gr\"obner basis (cf. the proof of Corollary~\ref{valuation_rank}). Then $\nu(b_1), \dots, \nu(b_d) \in P$ are independent due to Proposition~\ref{basis_partial}, hence $r\ge d$. 

The proposition is proved. \end{proof} 

\subsection{Sub-multiplicative maps and Jordan-H\"older bijections for algebras}

We say that a map ${\bf K}:P\to P'$ of ordered partial semigroups is  {\it sub-multiplicative} if it satisfies the following:

${\bf K}(u\circ v)\preceq {\bf K}(u)\circ {\bf K}(v)$
whenever $u\circ v$ and ${\bf K}(u)\circ {\bf K}(v)$ are defined in $P$ and in $P'$, respectively.

The following justifies this definition.

\begin{theorem}
    
\label{sub-multiplicative}
Let $\nu:A\setminus \{0\} \to P, \nu':A\setminus \{0\} \to P'$ be a pair of injective valuations of an algebra $A$ to partial semigroups $P$ and $P'$, respectively. Then JH-bijection ${\bf K}_{\nu',\nu}$ from $\nu(A\setminus \{0\})$ to $\nu'(A\setminus \{0\})$ is sub-multiplicative.
\end{theorem}

\begin{proof} For any elements $u,v\in P$ for which $u\circ v$ and ${\bf K}(u)\circ {\bf K}(v)$  are defined take $a,b\in A\setminus \{0\}$ such that $\nu(a)=u, \nu(b)=v$ and $\nu'(a)={\bf K}(u), \nu'(b)={\bf K}(v)$. Then $\nu(a b)=u\circ v$ and ${\bf K}(u\circ v)\preceq \nu'(a b)={\bf K}(u)\circ {\bf K}(v)$.
\end{proof}

The following result asserts that sub-multiplicative maps are frequently homomorphisms.
\begin{theorem} 
\label{th:lifting of homomorphisms}
In assumptions on (entire or partial) semigroups $P$ and $P'$ as in 
Lemma~\ref{le:crossed product of monoids}, 
let ${\bf K}$ be a map $P:=\underline P\times_\chi \Gamma\to P':=\underline P'\times_{\chi'} \Gamma$ commuting with the multiplication by elements of $\Gamma$. Then

(a) There is a unique map $\underline {\bf K}:\underline P\to \underline P'$ such that ${\bf K}(c,1)=(\underline {\bf K}(c),\underline  \gamma(c))$ for all $c\in \underline P$ where $\underline \gamma$ is a (unique) suitable map $P\to \Gamma$. 

(b) Suppose that ${\bf K}$ is a sub-multiplicative bijection and ${\bf K}^{-1}$ is sub-multiplicative as well. Suppose also that 
 $\underline \gamma(\underline  P)=\{1\}$.
Then $\underline {\bf K}$ and $\underline {\bf K}^{-1}$ are also sub-multiplicative bijections and
$$\chi'(\underline {\bf K}(c),\underline {\bf K}(d))=\chi(c,d) $$
whenever $c\circ d$ is defined and $\underline {\bf K}(c\circ d)=\underline {\bf K}(c)\circ \underline {\bf K}(d)$. 

(c) Consider partial semigroup structure $\widetilde P$ on $P$ and $\widetilde P'$ on $P'$ such that $\underline {\bf K}$ becomes an isomorphism of semigroups $\underline P\widetilde \to \underline P'$, then  ${\bf K}$ becomes an isomorphism of partial semigroups $\widetilde P\to \widetilde P'$.

\end{theorem}

\begin{proof} Part (a) is immediate because of the canonical projection $P'\twoheadrightarrow \underline P'$ and the embedding $\underline P\hookrightarrow 
P$, $c\mapsto (c,1)$, which send any  map ${\bf K}:P \to P'$ to a well-defined map $\underline {\bf K}:\underline  P\to \underline P'$.

Applying ${\bf K}$ and ${\bf K}^{-1}$  to the multiplication tables
$(c,1)\circ (d,1)=(c\circ d,\chi(c,d))$, $c,d\in \underline P$,  $(c',1)\circ (d',1)=(c'\circ d',\chi'(c',d'))$, $c',d'\in \underline P'$ respectively, 
we obtain using sub-multiplicativity
$${\bf K}(c\circ d,\chi(c,d))\preceq {\bf K}(c,1)\circ {\bf K}(d,1)\ ,
$$
$${\bf K}^{-1}(c'\circ d',\chi'(c',d'))\preceq {\bf K}^{-1}(c',1)\circ {\bf K}^{-1}(d',1)
\ .$$

Using (a) and $\Gamma$-equivariance of ${\bf K}$, we rewrite the above inequalities as
$$(\underline{\bf K}(c\circ d), \chi(c,d))\preceq (\underline  {\bf K}(c),1)\circ (\underline  {\bf K}(d),1)=(\underline  {\bf K}(c)\circ \underline  {\bf K}(d), \chi'(\underline  {\bf K}(c),\underline  {\bf K}(d)))\ ,
$$
$$(\underline {\bf K}^{-1}(c'\circ d'),\chi'(c',d'))\preceq (\underline {\bf K}^{-1}(c'),1)\circ (\underline {\bf K}^{-1}(d'),1)$$
$$=(\underline {\bf K}^{-1}(c')\circ \underline {\bf K}^{-1}(d'),\chi(\underline {\bf K}^{-1}(c'),\underline {\bf K}^{-1}(d')))\ .
$$

It holds that
$\underline {\bf K}(c)\circ \underline {\bf K}(d)=\underline {\bf K}(c\circ d)$.
Then the first inequality simplifies:
$$\chi(c,d)\preceq \chi'(\underline  {\bf K}(c),\underline  {\bf K}(d))\ .
$$

It holds that
$\underline {\bf K}^{-1}(c')\circ \underline {\bf K}^{-1}(d')=\underline {\bf K}^{-1}(c'\circ d')$.
Then the second inequality simplifies
$$\chi'(c',d')\preceq  \chi(\underline {\bf K}^{-1}(c'),(\underline {\bf K}^{-1}(d'))\ .
$$
Finally, taking $c'=\underline  {\bf K}(c)$, $d'=\underline  {\bf K}(d)$, the latter inequality simplifies

$$\chi'(\underline  {\bf K}(c),\underline  {\bf K}(d))\preceq  \chi(c,d)\ .
$$
Combining two obtained opposite inequalities, we finish the proof of (b).

Part (c) is immediate.
The theorem is proved. \end{proof}

The following is an immediate consequence of Theorem~\ref{th:lifting of homomorphisms}(b)(c) for quantum cones.

\begin{corollary}\label{JH_multiplicativity} Let ${\bf K}:P\to P'$ be a sub-multiplicative bijection of quantum cones
lifting the underlying bijection  $\underline {\bf K}:\underline P\to \underline P'$  of the corresponding  abelian semigroups $\underline P$ and $ \underline P'$. Suppose that ${\bf K}^{-1}$ is also sub-multiplicative.
Then in the notation \eqref{eq:Lambda-multiplication} it holds
\begin{equation}
\label{eq:g-equivariance}
g_{\underline {\bf K}(a),\underline {\bf K}(a')}=g_{a,a'}
\end{equation}
for all $a,a'\in \underline P$ such that $\underline {\bf K}(a+a')=\underline {\bf K}(a)+\underline {\bf K}(a')$.
In particular, ${\bf K}$ is an isomorphism of partial monoids where
the partial multiplication in $P$ is given by $(a,\gamma)\circ (a',\gamma')$   defined (and equal to $(a+a',\gamma\gamma')$) iff $\underline {\bf K}(a+a')=\underline {\bf K}(a)+\underline {\bf K}(a')$.

 \end{corollary}

\begin{remark} 
\label{rem:linearity domain}
Given a map $\underline {\bf K}:\underline P\to \underline P'$ we say a submonoid $\underline C$ of $\underline P$ is a {\it linearity domain} for $\underline {\bf K}$ if $\underline {\bf K}(a+a')=\underline {\bf K}(a)+\underline {\bf K}(a')$ for all $a,a'\in \underline C$. 
Thus, we can think of \eqref{eq:g-equivariance} as a ``symplectomorphism."  Moreover, the restriction of ${\bf K}$ to a preimage $C\subset P$ of a linearity  domain $\underline C$ extends to an isomorphism of groups  $<C>\widetilde \to <{\bf K}(C)>$, where $<\bullet>$ denotes the group generated by a given submonoid. In particular, if $C$ is of the same dimension as $P$, then $<C>=<P>$ and ${\bf K}(C)=<P'>$ and  we have an isomorphism of groups $<P>\widetilde \to <{\bf K}(P)>$.
    
\end{remark}

 \begin{remark} It is natural to expect that the converse also holds, that is,  \eqref{eq:g-equivariance} implies that $a$ and $a'$ belong to the same linearity domain.
\end{remark}

\begin{example}
Consider a partial semigroup $P$ endowed with two different orders. They provide two injective valuations $\nu_1, \nu_2: \kk P\setminus \{0\} \twoheadrightarrow P$. Then $\{[u]\, :\, u\in P\}\subset \kk P$ (Remark ~\ref{monoidal_partial}) is a common adapted basis for $\nu_1, \nu_2$, and the JH-bijection ${\bf K}_{\nu_1,\nu_2}$ is identity map $Id_P$.
\end{example} 

\begin{proposition}
Let $\nu:A\setminus \{0\} \to P,\ \nu':A\setminus \{0\} \to P'$ be injective valuations of an algebra $A$ to partial semigroups $P, P'$, respectively. If JH-bijection ${\bf K}:= {\bf K}_{\nu', \nu}$ is monotone (i.e. $p_1\preceq_P p_2$ implies ${\bf K}(p_1)\preceq_{P'} {\bf K}(p_2)$) then $\bf K$ is a homomorphism (thereby, an isomorphism) of partial semigroups.    
\end{proposition}

\begin{proof}
Take $a,b\in A\setminus \{0\}$ such that $\nu(a)\circ_P \nu(b) \in P$ is defined. Due to Proposition~\ref{sub-multiplicative}
 it holds
 $${\bf K}(\nu(a)\circ_P \nu(b)) \preceq_{P'} {\bf K} (\nu(a)) \circ_{P'} {\bf K} (\nu(b)),$$
 \noindent provided that the right-hand side is defined. Therefore, the monotonicity of ${\bf K}^{-1}$ and Proposition~\ref{sub-multiplicative} imply that
 $$\nu(a)\circ_P \nu(b) \preceq_P {\bf K}^{-1} ({\bf K} (\nu(a)) \circ_{P'} {\bf K} (\nu(b))) \preceq_P \nu(a)\circ_P \nu(b)\ .$$ 
 The proposition is proved.     
\end{proof} 
 
\begin{proposition}\label{two_semigroups}
    Let $P$ be a set with two (partial) operations $(a,b)\mapsto a\circ b$ and $(a,b)\mapsto a\bullet b$ so that $P$ has partial semigroup structures $P_{\circ}$ and $P_{\bullet}$ respectively.  Define a new operation 
$$ab:=a\circ b+a\bullet b$$
on the vector space $\kk P$
(with the convention if a summand is not defined it is replaced by zero) and denote this algebra by $A_{\circ,\bullet}$.

(a) $A_{\circ,\bullet}$ is associative iff
$\circ$ and $\bullet$ are mutually associative:
$$(a\circ b)\bullet c=a\bullet (b\circ c),~(a\bullet b)\circ c=a\circ(b\bullet c)$$
for all $a,b,c\in P$. We say that $(a\circ b)\bullet c$ is defined if both $a\circ b$ and $(a\circ b)\bullet c$ are defined (here we assume that $(a\circ b)\bullet c$ and $a\bullet (b\circ c)$ are defined or not defined simultaneously
(as well as $(a\bullet b)\circ c$ and $a\circ(b\bullet c)$).
\vspace{2mm}

(b) Suppose additionally that both $P_\circ$ and $P_\bullet$ are ordered  with $\preceq^\circ$ and $\preceq^\bullet$, respectively, and it holds
$$a\circ b\preceq^\bullet a\bullet b\preceq^\circ a\circ b,$$
\noindent provided that $a\circ b$ and $a\bullet b$ are defined (cf. Proposition~\ref{sub-multiplicative}).
Then the assignment $[a]\mapsto a$
define injective valuations $\nu_\circ$ and $\nu_\bullet$ on $A_{\circ,\bullet}$ to $P_{\circ}$ and to $P_{\bullet}$, respectively
Moreover, identity map $P\mapsto P$
is the corresponding JH-bijection, and $[P]\subset A_{\circ,\bullet}$ is a common adapted basis of valuations $\nu_\circ$ and $\nu_\bullet$.
    
\end{proposition}

\begin{proof} Prove (a). Indeed,
$$(ab)c=(a\circ b+a\bullet b)c=(a\circ b)c+(a\bullet b)c$$
$$=(a\circ b)\circ c+(a\circ b)\bullet c+(a\bullet b)\circ c+(a\bullet b)\bullet c$$
On the other hand, 
$$a(bc)=a(b\circ c+b\bullet c) $$
$$=a\circ (b\circ c)+a\bullet (b\circ c)+a\circ(b\bullet c)+a\bullet(b\bullet c)$$

This gives associativity because  $(a\circ b)\bullet c=a\bullet (b\circ c)$ and $(a\bullet b)\circ c=a\circ(b\bullet c)$. \vspace{2mm}

(b) We claim that $\nu_{\bullet}$ is a valuation. Take $a,b\in P\subset A_{\circ,\bullet}$ such that $a\bullet b$ is defined. If $a\circ b \prec^{\bullet} a\bullet b$ (or $a\circ b$ is not defined) then $\nu_{\bullet} (ab) = a\bullet b = \nu_{\bullet}(a) \bullet \nu_{\bullet} (b)$. Otherwise, if $a\circ b = a\bullet b$ then $ab=2a\bullet b$, and again we get that $\nu_{\bullet}(ab)=a\bullet b = \nu_{\bullet}(a) \bullet \nu_{\bullet} (b)$. The claim is proved.

In a similar manner we establish that $\nu_{\circ}$ is a valuation as well. \end{proof}

We mention that an issue of whether a sum of two associative products form again an associative product similarly to (a) is widely studied (see, e.g. \cite{Sokolov}), while not in the context of semigroup algebras.      

\begin{remark}\label{basis_semigroup}
Let $A$ be a $\kk$-algebra with a basis $B$ equipped with a linear order $\prec$. For $b_1,b_2\in B$ define $b_1\circ b_2\in B$ to be the highest (with respect to $\prec$) element in the decomposition in $B$ of $b_1b_2$, provided that $b_1b_2\neq 0$, otherwise $b_1\circ b_2$ is not defined. Assume that the following properties are fulfilled:

i) if $b_1\prec b_2$ then $b_1\circ b_0\preceq b_2\circ b_0$, provided that $b_1\circ b_0, b_2\circ b_0$ are defined (respectively, $b_0\circ b_1\preceq b_0\circ b_2$, provided that $b_0\circ b_1, b_0\circ b_2$ are defined);

ii) $(b_1\circ b_2)\circ b_3=b_1\circ (b_2\circ b_3)$, moreover, $b_1\circ b_2, (b_1\circ b_2)\circ b_3$ are defined iff $b_2\circ b_3, b_1\circ (b_2\circ b_3)$ are defined, $b_1,b_2,b_3\in B$.

Then $(B,\circ)$ is a partial semigroup. For any $a\in A\setminus \{0\}$ consider its decomposition $a=\lambda b+\cdots, \lambda\in \kk^\times$ in $B$ where $b\in B$ is the highest element of $B$ in this decomposition, we define $\nu(a):=b$. Then $\nu:A\setminus \{0\}\twoheadrightarrow B$ is an injective valuation.

Conversely, having an injective valuation  $\nu:A\setminus \{0\}\twoheadrightarrow P$ and an adapted basis $B$ one defines (as above) the (partial) operation $\circ$ on $B$ such that the partial semigroup $(B,\circ)$ is isomorphic to $P$.
\end{remark}

Note that for two injective valuations $\nu, \nu'$ on $A\setminus \{0\}$ the images $\nu(A\setminus \{0\})$ and $\nu'(A\setminus \{0\})$ are not necessarily isomorphic as (partial) semigroups. We will demonstrate this  Remark~\ref{cusp_JH} below we give here more examples in Section \ref{subsec:Examples Jordan-Holder bijections}.

\begin{theorem}\label{JH_lifting}
    
Let $\nu:A\setminus\{0\} \to P$ and $\nu':A\setminus\{0\}\to P'$ be injective valuations satisfying assumptions of Proposition \ref{pr:lifting of valuation}. 
Then ${\bf K}_{\nu',\nu} : C_\nu\widetilde \to C_{\nu'}$ is well-defined, unique, and commutes with multiplication by elements of $\Gamma$ (here we identify $\Gamma$ with its respective images $\nu(\Gamma)$ and $\nu'(\Gamma)$).

\end{theorem} 

\begin{proof} The assumptions of the theorem and Corollary \ref{cor:injective definition} imply an existence of a $\mathbb {K}$-linear basis $\underline {\bf B}$ of $\mathbb {K}\otimes A$ adapted to both $\underline \nu$ and $\underline \nu'$. 

The argument from the proof of Proposition \ref{pr:lifting of valuation} implies that ${\bf B}=\Gamma\cdot\underline {\bf B}$ is a basis of $A$ adapted to both $\nu$ and $\nu'$.

This proves the existence of ${\bf K}_{\nu',\nu}$. Its $\Gamma$-equivariance follows.

The uniqueness follows from Proposition~\ref{pr:common}.
    
\end{proof}

\section{Injective valuations to entire semigroups}\label{three}

In this section we consider (more familiar) valuations of algebras in (entire) semigroups (rather than in partial semigroups as in section~\ref{four}), e.g., each such algebra is a (commutative or noncommutative) domain.

\subsection{Valuations of domains into semigroups}

Let $C$ be a semigroup endowed with a linear ordering $<$ compatible with the semigroup operation $+$ (not necessary commutative). For a $\kk$-algebra $A$ its valuation we define as a mapping $\nu:A\setminus \{0\} \to C$ such that
$$\nu(\alpha a)=\nu(a), \nu(a+a_0)\le \max\{\nu(a), \nu(a_0)\}, \nu(aa_0)=\nu(a)+\nu(a_0),$$
$$a,a_0,a+a_0\in A\setminus \{0\}, \alpha \in \kk^\times.$$
\noindent Denote $C_{\nu}:=\nu(A\setminus \{0\})$.
An example is provided by a semigroup algebra $\kk C$ with a valuation (see Proposition~\ref{monomial_partial}) 
$$\nu(\alpha_1c_1+\cdots +\alpha_kc_k):=\max\{c_1,\dots,c_k\}, c_1,\dots,c_k\in C, \alpha_1,\dots,\alpha_k\in \kk^\times.$$

Let $C:=\langle c_1,\dots,c_n\rangle$ be a free semigroup generated by $c_1,\dots,c_n$. One can define a linear ordering on $C$ as follows (see Lemma~\ref{deglex}): $c_{i_1}\cdots c_{i_m} < c_{j_1}\cdots c_{j_s}, 1\le i_1,\dots,i_m,j_1,\dots,j_s\le n$ iff either $m<s$ or $m=s$ and the vector $(i_1,\dots,i_m)$ is less than the vector $(j_1,\dots,j_m)$ with respect to lex. Note that this provides a well-ordering on $C$.

We say that a valuation $\nu$ is {\it injective} if there exists a $\kk$-basis $\{a_c\, :\, c\in C_{\nu}\}$ of $A$, where $\nu(a_c)=c, c\in C_{\nu}$
(such a basis we call {\it adapted} with respect to $\nu$). 
Then $\nu$ has one-dimensional leaves (\cite{Kaveh}).
Observe that $$a_{c_1}a_{c_2}=\alpha (c_1,c_2) a_{c_1+c_2}+\sum_{c<c_1+c_2}\alpha_c c$$
\noindent for suitable $\alpha (c_1,c_2) \in \kk^times, \alpha_c \in \kk$. Observe that due to the associativity in $A$ the following relations are fulfilled:
$$\alpha(c_1,c_2)\alpha(c_1+c_2,c_3)=\alpha(c_2,c_3)\alpha(c_1,c_2+c_3).$$
For example, $\{c\in C\}$ is an adapted basis of $\kk C$ with respect to the tautological valuation (see Definition~\ref{monoidal_partial}).

More generally, for an arbitrary valuation $\nu$ on $A$ we say that $\{a_i \in A\}_i$ is an adapted basis \cite{Kaveh-Manon} (with respect to $\nu$) if for any $a=\sum_j \alpha_j a_j\in A\setminus \{0\},\, \alpha_j \in \kk^\times$ it holds $\nu (a)=\max_j \{\nu(a_j)\}$. In particular, if $A$ has a Khovanskii basis \cite{Kaveh-Manon} then one can produce relying on it an adapted basis.

\begin{theorem}\label{filtration}
Let $A$ be a $\kk$-algebra, $\nu:A\setminus \{0\}\twoheadrightarrow C_{\nu}$ be a mapping onto a linearly ordered semigroup $C_{\nu}$ such that 
$\nu(\alpha a)=\nu(a), \nu(a+a_0)\le \max\{\nu(a), \nu(a_0)\}, a,a_0,a+a_0\in A\setminus \{0\}, \alpha \in \kk^\times$. Denote 
$$A_c:=\{a\in A\setminus \{0\}\, :\, \nu(a)\le c\} \cup \{0\}$$
\noindent and $G_c:=A_c/A_{<c}$. Consider an associated graded algebra $G:=\bigoplus_{c\in C_{\nu}}G_c$.

i) $\nu$ is a valuation iff $G$ is a domain. In this case $\nu_0(g)=c$ for $g\in G_c^\setminus\{0\}$ defines a valuation on $G$.

ii) Let $C_{\nu}$ be well-ordered and $\nu$ be a valuation. Then $\nu$ is an injective valuation iff $\dim_{\kk}(G_c)=1$ for any $c\in C_{\nu}$.

iii) Let $\nu$ be an injective valuation and ${\mathcal C}=\{c_i\}\subset C_{\nu}$ be a set of generators of a well-ordered $C_{\nu}$. Then $A$ has an adapted basis of the form $\{a_{i_1}\cdots a_{i_k}\, :\, (i_1,\dots,i_k)\in S\}$ for an appropriate set $S$, where $\nu(a_i)=c_i \in {\mathcal{C}}$ and $C_{\nu}=\{c_{i_1}\cdots c_{i_k}\, :\, (i_1,\dots,i_k)\in S\}$. Note that $\mathcal{C}$ can be infinite.

iv) For a valuation $\nu$ on $A$ and a well-ordered $C$ there is an adapted basis of $A$.
\end{theorem}

\begin{proof}
i) Let $\nu$ be a valuation. Denote by $p_c:A_c\twoheadrightarrow G_c$ the projection. For any $g\in G_c\setminus \{0\}, g_0\in G_{c_0}\setminus \{0\}$ take $a\in p_c^{-1}(g), a_0\in p_{c_0}^{-1}(g_0)$. Then $\nu(a)=c, \nu(a_0)=c_0$. Since $\nu(aa_0)=\nu(a)+\nu(a_0)$, it holds $aa_0\notin A_{<(c+c_0)}$. Therefore $gg_0=p_{c+c_0}(aa_0)\neq 0$, i.e. $G$ is a domain.

In a similar manner one can verify the inverse statement. \vspace{2mm}

ii) Let $\dim(G_c)=1$ for any $c\in C_{\nu}$. For each $c\in C_{\nu}$ pick $a_c\in A_c$ such that $\nu(a_c)=c$. We claim that the elements $\{a_c\, :\, c\in C_{\nu}\}$ constitute an adapted basis of $A$ with respect to $\nu$ (this implies the injectivity of $\nu$). 

Clearly, the elements $\{a_c\, :\, c\in C_{\nu}\}$ are linearly independent. For any element $a\in A\setminus \{0\}$ with $\nu(a)=c_0$ there exists (and unique) $\alpha \in \kk^\times$ such that $\nu(a-\alpha a_{c_0})<c_0$ since $\dim (G_{c_0})=1$. Applying a similar argument to $a-\alpha a_{c_0}$ (in place of $a$), unless  $a-\alpha a_{c_0}=0$, and continuing in this way, we arrive eventually at a decomposition of $a$ in a linear combination of the elements from $\{a_c\, :\, c\in C_{\nu}\}$, taking into account that $C_{\nu}$ is well-ordered. The claim is proved.

In a similar manner one can verify the inverse statement. \vspace{2mm}

iii) follows from ii). \vspace{2mm}

iv) Choose a basis $\{b_{c,i}\, :\, i\in I_c\}$ of $G_c$ and elements $a_{c,i} \in A_c$ such that $p_c(a_{c,i})=b_{c,i}$. We claim that ${\bf A}:=\{a_{c,i}\, :\, c\in C, i\in I_c\}$ constitute an adapted basis of $A$.

Indeed, consider $a=\sum_{c,j} \alpha_{c,j} a_{c,j} \in A\setminus \{0\},\, \alpha_{c,j} \in \kk^\times$. Denote a subsum $e:=\sum_l \alpha_{c_0,l} a_{c_0,l}$ which ranges over all $l$ such that $c_0:=\nu(a_{c_0,l})=\max_{c,j} \{\nu(a_{c,j})\}$. Then $\nu(e)=c_0$ owing to the choice of $a_{c,i}$. Hence $\nu(a)=c_0$. In particular, the elements of ${\bf A}$ are independent.

Similarly to the proof above of ii) one can express any element of $A$ as a linear combination of elements from ${\bf A}$ which proves the claim.

The theorem is proved.     
\end{proof} 

\begin{remark}
If $\nu$ is an injective valuation on an algebra $A$ over an radically closed field onto a well-ordered finitely-generated monoid $C$ 
then one can treat $A$ as a deformation of $\kk C$ (see Proposition~\ref{isomorphism} below).    
\end{remark}

Now we describe a construction which starting with a valuation on an algebra, produces a valuation on its quotient (cf. Theorem~\ref{graded_partial}). Let $A=\bigoplus_{c\in C} A_c$
be a domain over a field $\kk$ graded by an ordered monoid 
$C$. For  $a\in A\setminus \{0\}$ denote by $lt(a)\in A_{c_0}$ the leading term of $a$ for suitable $c_0\in C$, i.e. $a-lt(a)\in \bigoplus_{c<c_0} A_c$. Note that $\nu_0(a):=c_0$ defines a valuation on $A\setminus \{0\}$ (not necessary injective). For an ideal $J\subset A$ denote by $lt(J)\subset A$ the homogeneous ideal generated by $lt(f)$ for $f\in J$.

\begin{theorem}\label{graded}
Let $A=\bigoplus_{c\in C} A_c$
be a domain over a field $\kk$ graded by an ordered monoid $C$. 
For an ideal $J\subset A$ one can define a mapping $\nu$ on the algebra $(A/J)\setminus \{0\}$ filtered by $C$ as follows. For $g\in (A/J)\setminus \{0\}$ denote 
$$\nu(g):= \min\{\nu_0(g+J)\}\in C.$$

i) $\nu(\alpha g_1)=\nu(g_1),  \nu(g_1+g_2)\le \max\{\nu(g_1), \nu(g_2)\};$

ii) $\nu(g_1g_2)=\nu(g_1)+\nu(g_2)$ for any $g_1,g_2\in (A/J)\setminus \{0\}$ iff the ideal $lt(J)\subset A$ is prime;

iii) $\nu$ is injective iff $C$ is well-ordered and $\dim_{\kk} (A_c/(lt(J)\cap A_c))=1$ for each $c\in C$.

Thus, when the conditions in ii), iii) are satisfied, $\nu$ is an injective well-ordered valuation of $(A/J)\setminus \{0\}$.
    \end{theorem}

\begin{proof}  i) is straightforward.

One can verify the following lemma.

\begin{lemma}\label{leading_unique}
Assume that for $g\in (A/J)\setminus \{0\}$ it holds $\nu(g)=\nu_0(g+f_0)=c_0, f_0\in J$. Then $lt(g+f_0)\notin lt(J)$. In addition, for any $c>c_0$ and $lt(g+f)\in A_c, f\in J$ it holds $lt(g+f)\in lt(J)$.
    \end{lemma}

ii) Let $lt(J)$ be prime, and $\nu(g_1)=\nu_0(g_1+f_1), \nu(g_2)=\nu_0(g_2+f_2)$ for appropriate $f_1,f_2\in J$. It holds $lt(g_1+f_1)lt(g_2+f_2)\notin lt(J)$ since $lt(J)$ is prime and employing Lemma~\ref{leading_unique}. Therefore $\nu(g_1g_2)=\nu(g_1)+\nu(g_2)$ again due to Lemma~\ref{leading_unique}.

One can prove ii) in the opposite direction in a similar way. \vspace{2mm}

iii) Let $\dim (A_c/(lt(J)\cap A_c))=1$ for any $c\in C$. Then for every $g_1,g_2\in (A/J)\setminus \{0\}$ such that $\nu(g_1)=\nu(g_2)=c, lt(g_1+f_1),lt(g_2+f_2)\in A_c$, we have $lt(g_1+f_1)-\alpha \cdot lt(g_2+f_2)\in lt(J)$ for a suitable $\alpha \in \kk$. Hence there exists $f\in J, lt(f)=lt(g_1+f_1)-\alpha \cdot lt(g_2+f_2)$ for which $\nu_0((g_1+f_1)-\alpha \cdot (g_2+f_2)-f)<c$ that establishes the injectivity of $\nu$, taking into account that $C$ is well-ordered (cf. the proof of Theorem~\ref{filtration}~ii)).

One can prove iii) in the opposite direction in a similar way. 

The theorem is proved.     
\end{proof} 

\vspace{2mm}

\begin{remark}
Assume that $A=\kk\langle x_1,\dots,x_n\rangle /J$ for a prime ideal $J\subset \kk\langle x_1,\dots,x_n\rangle$ such that $x_i\notin J, 1\le i\le n$, and $\nu$ is a valuation (not necessary injective) on $A\setminus \{0\}$. Then one can define $\nu_0(x_i):=\nu(x_i), 1\le i\le n$ which provides a grading on $\kk\langle x_1,\dots,x_n \rangle$. Now if we apply the construction from Theorem~\ref{graded} to the latter graded algebra $\kk\langle x_1,\dots,x_n \rangle$ and to the ideal $J$, we arrive at the initial valuation $\nu$ on $A\setminus \{0\}$.
\end{remark}

Let $\nu:A\setminus \{0\}\to C$ be a well-ordered injective valuation of an algebra $A$.

Given an ideal $J$ in $A$, we say that a generating set $B$ of $J$ is a $\nu$-Gr\"obner basis of $J$ if $(AbA,b\in B)$ is a $\nu$-ensemble, that is, 
$$\nu(J\setminus \{0\})=\bigcup\limits_{b\in B} C_\nu\cdot \nu(b)\cdot C_\nu$$

\subsection{Examples of injective valuations on algebras of dimension 2}

\begin{example}\label{skew}
Consider the following injective valuation on the ring $A:=\kk[x,y] \setminus \{0\}$. One can uniquely represent an arbitrary polynomial $f\in A$  as $f=g(y,y^3-x^2)+xh(y,y^3-x^2)$ for some polynomials $g,h$. Define $\nu$ and its adapted basis as follows:
$$\nu(y^k(y^3-x^2)^l):=(2k,l),\, \nu(xy^k(y^3-x^2)^l):=(2k+3,l),\, k,l\ge 0.$$
\noindent Therefore, $\nu(f)=\max\{\nu(g(y,y^3-x^2)),\nu(xh(y,y^3-x^2))\}$.
The valuation monoid is $\{(u,v)\in \ZZ_{\ge 0}^2\, :\, u\neq 1\}$. We consider its linear ordering with respect to deglex, say, with $u$ being higher than $v$.
Thus, $\nu$ is not induced by a minimal generating set of $\kk[x,y]$.
\end{example}

One can straightforwardly verify the following proposition.

\begin{proposition}\label{basis_partition}
Let $\nu$ be an injective valuation $\nu$ on an algebra $A\setminus \{0\}$ with a finitely generated valuation in a well-ordered semigroup $C$. Consider a partition of $C$ according to \cite{K}. Namely, each element of the partition has a form $c+D$ where $c\in C$ and a semigroup $D\subseteq C$ is isomorphic to $\ZZ_{\ge 0}^k$ for some $k$ with basis vectors $c_1,\dots,c_k\in D$. Let $a,a_1,\dots,a_k\in A$ be such that $\nu(a)=c,\nu(a_1)=c_1,\dots,\nu(a_k)=c_k$. Then the elements $aa_1^{i_1}\cdots a_k^{i_k},\, i_1,\dots,i_k\in \ZZ_{\ge 0}$ for all the elements of the partition of $C$ form an adapted basis of $A$ with respect to $\nu$.  
\end{proposition}

Let $A$ be a finitely generated $\kk$-algebra of dimension $d$. Let $\nu : A\setminus \{0\} \to C$ be a valuation on $A\setminus \{0\}$ and $C$ be a well-ordered semigroup of a rank $r$.

For each $c\in C$ pick an arbitrary element $a_c\in A$ such that $\nu(a_c)=c$. Then the elements $\{a_c\, :\, c\in C\}$ are $\kk$-linearly independent. Therefore, $r\le d$. Indeed, otherwise take linearly independent $c_0,\dots,c_d\in C$ (in Groth\'endieck group of $C$), then all the monomials in the elements $a_{c_0},\dots, a_{c_d}$ are linearly independent. The obtained contradiction justifies the inequality $r\le d$. Note that for the latter inequality we did not use the injectivity of $\nu$.

Obviously, one can yield a well-ordered injective valuation on $\kk[x_1,\dots,x_d]\setminus \{0\}$ (in a unique manner) by means of assigning linearly independent vectors $\nu(x_1),\dots,\nu(x_d)\in \ZZ_{\ge 0}^d$ and defining a well-ordering on $\ZZ_{\ge 0}^d$.

Below we produce a different family of well-ordered injective valuations of rank 2 on the polynomial ring $\kk[x,y]\setminus \{0\}$ generalizing Example~\ref{skew} in which $\nu(x),\nu(y)$ are linearly dependent.

\begin{proposition}\label{partition}
Let $f=x^n+\sum_{0\le i<n} f_ix^i,\, f_i\in \kk[y]$ be a polynomial such that $m:=deg_y (f_0)$ is relatively prime with $n$, and $mi+ndeg_y (f_i)\le mn$ for $0\le i\le n$. Then there is a well-ordered injective valuation $\nu : (\kk[x,y]\setminus \{0\}) \to \ZZ_{\ge 0}^2$ defined as follows on its adapted basis:
\begin{equation}\label{4}
\nu(x^iy^kf^l):= (mi+nk,l),\, 0\le i<n,\, 0\le k,l.
\end{equation}
\end{proposition}

\begin{proof} We observe that $\kk[x,y]$ is a finite $\kk[f,y]$-module with a basis $1,x,\dots,x^{n-1}$ with an irreducible monic polynomial $f(x,y)-f$ defining $x^n$. This justifies that in \eqref{4} we have a basis of $\kk[x,y]$. The right-hand sides of \eqref{4} are pairwise distinct due to relative primality of $m,n$.

To verify the multiplicativity of $\nu$ note that $mi+ndeg_y (f_i)<mn$ for $0<i<n$, hence 
$$\nu (x^j)+\nu (x^{n-j})=(mj+m(n-j),0)=\nu(y^m)=\nu(\sum_{0\le i<n} f_i(y)x^i-f)=\nu(x^n)\ .$$ 
\end{proof}
\vspace{2mm}

One can extend this construction.

\begin{corollary}\label{finite_module}
Let a ring $B$ be a finite $A$-module with an integral basis $1,x,\dots,x^{n-1}$. Let $\nu$ be a well-ordered injective valuation on $A\setminus \{0\}$ with a valuation semigroup $C\subseteq \ZZ_{\ge 0}^d$. Assume that $x^n$ satisfies a polynomial $f=x^n+\sum_{0\le i<n} f_ix^i$ where $f_i\in A,\, 0\le i<n$ such that $$\frac{i\nu(f_0)}{n} \notin G(C),\, 0<i<n, \, n\nu(f_i)<(n-i)\nu(f_0),\, 0\le i\le n$$ 
\noindent where $G(C)$ denotes Groth\'endieck group of $C$.
Then one can uniquely extend $\nu$ to a well-ordered injective valuation $\nu_1$ on $B\setminus \{0\}$ such that $\nu_1(x)=\nu(f_0)/n$. Clearly, the valuation semigroup of $\nu_1$ has the same rank as of $\nu$.   
\end{corollary}

\begin{remark}\label{simplicial}
Let $\nu:\kk[x,y]\setminus \{0\} \to C$
be a well-ordered injective valuation. When the values $\nu(x),\nu(y)$ are independent, the semigroup $C$ is isomorphic to $\ZZ_{\ge 0}^2$, while  in Proposition~\ref{partition} the semigroup of the produced valuation consists of $n$ copies of (shifted) $\ZZ_{\ge 0}^2$.
\end{remark} 

\begin{example}
In Corollary~\ref{finite_module}  we have provided a construction of an extension of a domain with an injective valuation. In the course of this construction the Groth\'endieck group of the valuation monoid is also extended. Now we give an example of an extension of a domain with an injective valuation when the Groth\'endieck group of monoids does not change.

Let a domain $A_0:=\kk[x,y]$ and $\nu$ be its valuation onto $\ZZ_{\ge 0}^2$ such that $\nu(x)=(1,0), \nu(y)=(0,1)$ (one can take an arbitrary linear well-ordering on $\ZZ_{\ge 0}^2$). Consider polynomials $a,b\in A_0$ such that the leading monomial (with respect to $\nu$) of $a$ equals $x^k$ for some $k\ge 1$, while the leading monomial of $b$ equals $y^l$ for some $l\ge 1$. Denote $A:=A_0[b/a]\subset \kk(x,y)$. Therefore, the extension of $\nu$ on $A\setminus \{0\}$ is inherited uniquely from $\nu$. Observe that $\nu(A\setminus \{0\})\subset \ZZ^2$ is well-ordered since $l\ge 1$. The following set forms an adapted basis of $A$:
$$\{x^iy^j\, :\ i,j\ge 0\} \bigsqcup \{(b/a)^sx^iy^j\, :\, s\ge 1, 0\le i<k, 0\le j\}.$$
\noindent Indeed, this set spans $A$. On the other hand, $\nu(x^iy^j)=(i,j),\,  \nu(b/a)^sx^iy^j=(-sk+i, sl+j)$, and these values are pairwise distinct for different $i,j,s$.
\end{example}

\begin{example}\label{skew_extension}
Consider an injective homomorphism $\kk[x,y]\hookrightarrow \kk[x-y^{3/2}, y^{1/2}]$ and an injective well-ordered valuation $\nu_1$ on the latter algebra defined by $\nu_1(x-y^{3/2}):=(-3,1),\, \nu_1(y^{1/2}):=(1,0)$. Then $\nu_1(x-y^{3/2}), \nu_1(y^{1/2})$ are linearly independent (cf. Remark~\ref{simplicial}). One can verify that the restriction of $\nu_1$ to $\kk[x,y]\setminus \{0\}$ coincides with $\nu$.    
\end{example}

\subsection{Injective well-ordered valuations on varieties based on tropical geometry}\label{subsection_3.4}

In the sequel we provide a realization of the construction from Theorem~\ref{graded}. Let $I\subseteq \kk[X_1,\dots,X_n]$ be a prime ideal where $\kk$ is a field 
of zero characteristic. Our purpose is to construct injective well-ordered valuations on the quotient ring $A\setminus \{0\}=(\kk[X_1,\dots,X_n]/I)\setminus \{0\}$ which are induced from the tautological valuation $\nu_0(X_1^{j_1}\cdots X_n^{j_n}):=(j_1,\dots,j_n)\in \ZZ_{\ge 0}^n$ on $\kk[X_1,\dots,X_n]\setminus \{0\}$. Note that to determine $\nu_0$ completely, one has to fix also a linear ordering on $\ZZ_{\ge 0}^n$.

For the sake of convenience we need to describe linear orders $\prec$ on monomials $X^J=X_1^{j_1}\cdots X_n^{j_n}$ compatible with the product, i.e. $X^J\prec X^K$ implies $X^{J+L}\prec X^{K+L}$, in a different language than in section~\ref{tame_valuations}. To this end, we introduce an infinitesimal $\varepsilon$, i.e. $0<\varepsilon <y$ for any $0<y\in \RR$. Then $\RR[\varepsilon]$ is an ordered ring. Assign weights $0\neq w_i=w(X_i)\in \RR_{\ge 0}[\varepsilon], 1\le i\le n$. This induces a linear (non-strict) order on monomials $X^J$ according to the value of $w_1j_1+\cdots +w_nj_n\in \RR_{\ge 0}[\varepsilon]$. This determines a well-ordering, in other words, there does not exist a strictly decreasing infinite sequence of monomials. It is proved in \cite{R} and in Theorem 9 \cite{Kh} (in a different language) that any linear order on monomials can be obtained in the described manner. For instance, for two variables lex corresponds to the vector of weights $(1, \varepsilon)$, and deglex corresponds to $(1+\varepsilon, 1)$. 

Generalizing Lemma~\ref{deglex}, one can produce well orderings on the free monoid $\langle X_1,\dots, X_n\rangle$ as follows. For weights $0\neq w_i\in \RR_{\ge 0}[\varepsilon], 1\le i\le n$ define a (non-strict) well ordering on words $X\in \langle X_1,\dots, X_n\rangle$ according to the value of $w_1u_1+\cdots + w_nu_n$, where $u_i$ denotes the number of occurrences of $X_i$ in the word $X$. To make this ordering strict, one can in addition, lexicographically order all the words with a given value of $w_1u_1+\cdots + w_nu_n$ (with respect to $X_1\succ \cdots \succ X_n$).

\begin{definition}\label{common}
Denote $d:=\dim A$. Consider the tropical variety $T:=Trop(I)\subseteq \RR^n$ \cite{MS}. 
One can view each element of $T$ as a hyperplane in $\RR^n$ which supports from above Newton polytope $N(f)\subset \RR^n$ at least at two points (thus, at least at an edge) for every $f\in I$. In such a case we say that this edge is located on the roof of $N(f)$. Then $T$ is equidimensional of dimension $d$ \cite{MS} being a finite union of polyhedra each of dimension $d$. Every polyhedron corresponds to a union of hyperplanes containing a (unique) common subplane of dimension $n-d$ which is dual to the polyhedron (we call these subplanes {\bf common} for the tropical variety $T$). Every such common subplane $H\subset \RR^n$ is supporting to $N(f)$ for any $f\in I$ and is definable by linear equations with rational coefficients.

We extend $Trop(I)$ considering $$ETrop(I):= Trop(I)\bigotimes_{\RR} \RR[\varepsilon]\subset (\RR[\varepsilon])^n$$
\noindent where $ETrop(I)$ satisfies the same linear inequalities as $Trop(I)$. Thus, one can view $Etrop(I)$ still as a finite union of polyhedra. Each hyperplane from $ETrop(I)$ contains $H\bigotimes_{\RR} \RR[\varepsilon]$ for some common subplane $H$ of $Trop(I)$ and supports $N(f)\bigotimes_{\RR} \RR[\varepsilon]$ at least at two points.
\end{definition}

We call a subplane $H$ {\bf prop} if $H_0\cap \RR_{\ge 0}^n = \{X_{l_1}=\cdots =X_{l_m}=0\} \cap \RR_{\ge 0}^n$ for suitable $1\le l_1,\dots,l_m\le n$ where $H_0$ is parallel to $H$ and contains the origin $(0,\dots,0)$.

Let us fix a common subplane $H$ for the time being. We say that the ideal $I$ is {\bf saturated (with respect to $H$)} if for any pair of integer points $u,v\in \ZZ_{\ge 0}^n$ such that $v-u\in H$ there exists a polynomial $f\in I$ whose Newton polytope $N(f)$ possesses an edge $(u,v)$ on its roof. Below (see Theorem~\ref{quotient}) under  
the condition of saturation we obtain an injective valuation, so this condition is stronger than the property that an
initial ideal corresponding to $H$ is prime in $\kk[X_1,\dots,X_n]$ (cf. Theorem~\ref{graded}~ii) and \cite{Kaveh-Manon}).

\begin{remark}\label{finite-condition}
In fact, one can reduce the condition of saturation to a finite number of conditions. Indeed, consider a semigroup 
$$G:=\{(u,v)\, :\, u,v\in \ZZ_{\ge 0}^n,\, u-v\in H\}\subset \ZZ_{\ge 0}^{2n}.$$
\noindent Due to Gordan's lemma \cite{DT} $G$ is finitely generated. Among its generators select all $(u,v)$ such that $u\neq v$. Denote a vector $(w_1',\dots,w_n')=: u-v$ and a vector $w:=(w_1,\dots,w_n)=:(w_1',\dots,w_n')/GCD(w_1',\dots,w_n')$. Introduce points
\begin{equation}\label{3}
u_0:=(\max\{w_1,0\},\dots,\max\{w_n,0\}),\, v_0:=(\max\{-w_1,0\},\dots,\max\{-w_n,0\}) \in \ZZ_{\ge 0}^n.
\end{equation}
Then $u_0-v_0=w$. 

One can verify that it suffices for the saturation to impose for all the constructed pairs of points $u_0,v_0$ \eqref{3} the existence of a polynomial $f\in I$ such that $N(f)$ has an edge $(u_0,v_0)$ on its roof. 
\end{remark}

From now on we assume that the subplane $H$ is prop and $I$ is saturated (with respect to $H$). Our aim is to produce a valuation $\nu:=\nu_H$ on $A\setminus \{0\}$. Denote $H_{\ZZ}:= H\cap \ZZ^n$. The following construction of a valuation is similar to \cite{Kaveh-Manon}.

Consider an epimorphism $\varphi : \ZZ^n \to \ZZ^n/H_{\ZZ}$. Then the image $C:=\varphi (\ZZ^n_{\ge 0})$ is a semigroup cone (since $H$ is prop). The valuation $\nu$ under production will have $C$ as its valuation cone. Choose some linear ordering $<$ on $C$ for definiteness by fixing a prop hyperplane determined by a vector $(w_1,\dots, w_n)\in (\RR_{\ge 0}[\varepsilon])^n$ from $ETrop(I)$ which contains $H\bigotimes_{\RR} \RR[\varepsilon]$. 

Take $0\neq a\in A$. Assume that there exists $f\in a+I$ such that its Newton polytope $N(f)$ contains no edge in $H$. Then there is a unique vertex $v$ of $N(f)$ with the maximal value of the ordering of $\varphi(v)\in C$. Put $\nu(a):=\varphi(v)$. 

Let us establish the correctness of this definition. If otherwise, for some $f_1\in a+I$ its Newton polytope $N(f_1)$ has a unique vertex $v_1$ with the maximal value of the ordering of $\varphi(v_1)$, then $\varphi(v)=\varphi(v_1)$ 
taking into account that $f-f_1\in I$.

Next we show that for any $0\neq a\in A$ there exists $f\in a+I$ for which $N(f)$ contains no edge in $H$. Indeed, take $f\in a+I$ such that the vertices $v$ of $N(f)$ with the maximal value of the ordering of $\varphi(v)\in C$ are minimal among all $f\in a+I$. If $u$ is another vertex of $N(f)$ for which $\varphi(v)=\varphi(u)$, i.~e. an interval $(u,v)$ lies in $H$, then due to the saturation condition there exists $g\in I$ whose Newton polytope $N(g)$ contains an edge $(u,v)$ on its roof. Therefore, for a suitable $\alpha \in \kk$ the support of the polynomial $f+\alpha g$ does not contain $u$. Continuing in this way, we arrive eventually to a polynomial $f_1\in a+I$ such that its Newton polytope $N(f_1)$ contains a single vertex $w_0$ with the maximal ordering of $\varphi(w_0)\in C$ greater than the orderings of $\varphi(w)$ for all other vertices $w$ of $N(f_1)$. Clearly, $\varphi(w_0)=\varphi(v)$ due to the choice of $f$ satisfying the minimality property. 

Observe that we have proved at the same time that one can equivalently define
\begin{equation}\label{2}
\nu(a)=\min_{f\in a+I} \max_{v\in N(f)} \{\varphi(v)\}    
\end{equation}
where $v\in N(f)$ means that $v$ is a vertex of $N(f)$.

Thus, the valuation $\nu$ on $A\setminus \{0\}$ is defined correctly. If $0\neq a_1,a_2\in A$ then take polynomials $f_1\in a_1+I,\, f_2\in a_2+I$ such that Newton polytope $N(f_1)$ (respectively, $N(f_2)$) contains a unique vertex $v_1$ (respectively, $v_2$) such that $\varphi(v_1)$ (respectively, $\varphi(v_2)$) has a greater ordering than $\varphi(w)$ for all other vertices $w$ of $N(f_1)$ (respectively, $N(f_2)$). Then $\nu(a_1+a_2)\le \max\{\varphi(v_1),\varphi(v_2)\}=\max\{\nu(a_1),\nu(a_2)\}$ because of \eqref{2}. In addition, for a polynomial $f_1f_2\in a_1a_2+I$ its Newton polytope $N(f_1f_2)$ contains a unique vertex $v_1+v_2$ such that $\varphi(v_1+v_2)\in C$ has a greater ordering than all other vertices of $N(f_1f_2)$, hence $\nu(a_1a_2)=\varphi(v_1+v_2)=\varphi(v_1)+\varphi(v_2)$.

Now we verify the injectivity of $\nu$. Let $\nu(a_1)=\nu(a_2)$ for $0\neq a_1,a_2\in A$. Take $f_1\in a_1+I,\, f_2\in a_2+I$ with vertices $v_1\in N(f_1),\, v_2\in N(f_2)$ as above. Thus, $\varphi(v_1)=\varphi(v_2)$. Therefore, there exists $g\in I$ such that its Newton polytope $N(g)$ contains an edge $(v_1,v_2)$ on its roof due to the saturation condition. Hence, for any vertex $w\in N(f_1+\alpha f_2+ \beta g)$ we have $\varphi(w)<\varphi(v_1)=\nu(a_1)$ for appropriate $\alpha, \beta \in \kk$. Thus, $\nu(a_1+\alpha a_2)<\nu(a_1)$, see \eqref{2}, which justifies the injectivity of $\nu$.

We summarize the proved above in the following theorem.

\begin{theorem}\label{quotient}
Let $A=\kk[X_1,\dots,X_n]/I$ be a domain of dimension $d$. Let $H\subset \RR^n$ be one of a finite number of (rationally definable) common subplanes of dimension $n-d$ dual to a (highest dimensional) polyhedron of dimension $d$ of the tropical variety  $Trop (I)\subset \RR^n$ (see Definition~\ref{common}). Assume that $H$ is prop and $I$ is saturated with respect to $H$. Consider a natural epimorphism
$\varphi : (\RR[\varepsilon])^n \to (\RR[\varepsilon])^n/(H\bigotimes_{\RR} \RR[\varepsilon])$.
\noindent Fix a prop hyperplane from $ETrop(I)$ which contains $H\bigotimes_{\RR} \RR[\varepsilon]$, it determines a linear order on $\varphi(\ZZ_{\ge 0}^n)$.
Then \eqref{2} defines a well-ordered injective valuation $\nu$ on $A\setminus \{0\}$ having a valuation cone $\varphi(\ZZ_{\ge 0}^n)$.
\end{theorem}

\begin{remark}
The constructions of injective valuations on $\kk[x,y]\setminus \{0\}$ from Example~\ref{skew} and Proposition~\ref{partition} are particular cases of Theorem~\ref{quotient} when one represents $\kk[x,y]\simeq \kk[x_1,\dots, x_n]/I$ for suitable ideals $I\subset \kk[x_1,\dots, x_n]$.   
\end{remark}

Let $(M,\cdot)$ be a (not necessary commutative) monoid. We say that an equivalence relation $\sim$ is admissible if $u\sim v$ implies $wu\sim wv, uw\sim vw$ for any $u,v,w \in M$. Then one can define a quotient monoid $M/\sim$ on equivalence classes. A linear order $\prec$ on equivalence classes $U\prec V$ (or on $M/\sim$) is defined as $u\prec v$ for any $u\in U, v\in V$, we require that this linear order on $M/\sim$ is correct. The latter linear order
is admissible if $U\prec V$ implies $UW\prec VW, WU\prec WV$ (cf. Definition~\ref{monoid_partial}). Below we consider only admissible equivalence relations and linear orders.

Denote by $M:= \langle a_1,\dots,a_s\rangle$ the free monoid generated by $a_1,\dots,a_s$. Let $A:=\kk\langle a_1,\dots,a_s\rangle / I$ be a (not necessary commutative) algebra where $I\subset \kk\langle a_1,\dots,a_s\rangle$ is an ideal. We say that an equivalence relation $\sim$ on $M$ and a linear order $\prec$ on $M/\sim$ are {\bf compatible with $I$} if for any element 
\begin{equation}\label{99}
f=\sum_{u\in supp(f)\subset M} \alpha_u u \in I,\, \alpha_u\in \kk^\times    
\end{equation}
there are elements $u_1,u_2 \in supp(f)$ such that $u_1\sim u_2$ and for every $u\in supp(f)$ it holds $u\preceq u_1$. One can treat this concept as a generalization of the tropical variety of $I$ to the non-commutative case.

We say that {\bf $I$ is saturated with respect to $\sim, \prec$} if for any pair $u_1\sim u_2, u_1\neq u_2$ there exists $f$ of the form \eqref{99} such that $u_1,u_2\in supp(f)$ and for every $u\in supp(f), u\neq u_1, u_2$ it holds $u\prec u_1$. Similarly to the proof of Theorem~\ref{quotient}
one can verify the following proposition.

\begin{proposition}
Let $A:=\kk\langle a_1,\dots, a_s \rangle /I$ be an algebra, $\sim$ be an admissible equivalence relation on the free monoid $M:=\langle a_1,\dots, a_s \rangle$, and $\prec$ be an admissible well order on $M/\sim$. Assume that $\sim, \prec$ are compatible with $I$, and $I$ is saturated with respect to $\sim, \prec$. Then there is an injective valuation $\nu:A\setminus \{0\} \twoheadrightarrow M/\sim$ defined as follows: for $f\in A\setminus \{0\}$ put $\nu(f)$ as the minimal equivalence class $U_0$ such that 
$$f=\alpha_{u_0} u_0 +\sum_{u\in M} \alpha_u u,\, \alpha_{u_0},\alpha_u \in \kk^\times,$$
\noindent where $u_0 \in U_0$ and $u\prec u_0$ for all $u$ (cf. \eqref{2}). In addition, one can pick an adapted basis among monomials from $M$.
\end{proposition} 

\subsection{Injective well-ordered valuations on algebraic curves}

\begin{remark}
Assume that the field $\kk$ is algebraically closed and $A=\kk[X_1,\dots,X_n]/I$ is a domain. Let $\nu: A\setminus \{0\} \to C$ be a valuation and let $\psi:C \to \QQ$ be a homomorphism preserving the order. There exist Puiseux series $(x_1,\dots,x_n)\in \kk((\varepsilon ^{1/\infty}))^n$ satisfying $I$ of the form $\alpha_0 \varepsilon^{s_0/q} + \alpha_1 \varepsilon^{s_1/q}+\cdots \in \kk((\varepsilon ^{1/\infty}))^n$ where integers $s_0>s_1>\cdots$ decrease, such that $ord (x_i)=\psi \circ \nu (X_i), 1\le i\le n$
\cite{MS}.     
\end{remark}

\begin{proposition}\label{curve_valuation}
Let a valuation $\nu$ on an irreducible curve  $A\setminus \{0\}:=\kk[x,y]/(f)\setminus \{0\}, f\in \kk[x,y]$ fulfill Theorem~\ref{quotient}, i.e. $\nu$ is injective and $\nu(a)\ge 0$ for any $a\in A\setminus \{0\}$, where $\kk$ is algebraically closed. Then there exists an injective homomorphism $\eta: A \hookrightarrow \kk((\varepsilon^{1/\infty}))$ such that for every $a\in A\setminus \{0\}$ it holds $\nu(a)=ord(\eta (a))$, provided that $\nu(x)=1$ for normalization.      
\end{proposition}

\begin{proof}
According to Theorem~\ref{quotient} Newton polygon $N_f$ has an edge with endpoints $(p,0), (0,q)$ for relatively prime $p,q\ge 1$. Let $p\le q$ for definiteness. Then the equation $f(x,y)=0$ has a Puiseux series solution of the form $y(x)=\alpha x^{p/q}+\dots$ where $\alpha \in \kk^\times$ and the terms in dots contain powers of $x$ less than $p/q$.

One can define $\eta (x):= \varepsilon, \eta(y):= y(\varepsilon)$. Then $\eta$ is injective since $f$ is irreducible. The monomials ${\bf B}:= \{x^iy^j : 0\le i<\infty, 0\le j<q\}$ constitute a basis of $A$. The orders $ord (\eta (x^iy^j))=i+jp/q=\nu(x^iy^j)$ are pairwise distinct for the monomials from $\bf B$. 

The proposition is proved.
\end{proof} 

\vspace{2mm}

One can prove a certain converse statement to Proposition~\ref{curve_valuation}.

\begin{remark}
Let for a polynomial $f\in \kk[x,y]$ where $\kk$ is algebraically closed, its Newton polygon $N_f\subset \RR^2$ is not of the shape from Proposition~\ref{curve_valuation}, i.e. $N_f$ does not contain an edge with vertices $(p,0), (0,q)$ with relatively prime $p, q$. Then the imbedding $\eta :A:=\kk[x,y]/(f)\hookrightarrow \kk((x^{1/\infty}))$ into the field of Puiseux series induces a valuation $\nu :A\setminus \{0\} \to \QQ$ by a formula
$\nu(a):=ord (\eta (a))$, being not an injective well-ordered.     
\end{remark}

Any automorphism $\varphi$ of $\kk[x,y]$ produces an injective valuation on the algebra $\kk[x,y]/(f\circ \varphi)\setminus \{0\}$.

Consider an algebra $A:= \kk[x,y]/(f)$ of a curve where $f$ is irreducible. Let $A\hookrightarrow \kk((x^{1/\infty}))$ be an injective homomorphism into the field of Newton-Puiseux series. We investigate when this  induces an injective well-ordered valuation $\nu(=ord)$ on $A\setminus \{0\}$. W.l.o.g. one can suppose that $ord(x)=1$ and $f=y^d +f_1$ is normalized, i.e. $deg_y (f_1)<d$.  

\begin{lemma}\label{puiseux_imbedding}
Let $M\subset \kk((x^{1/\infty}))$ be a free $\kk[x]$-module of a rank $d$. Then $M\setminus \{0\}$ admits a $\kk[x]$-basis $s_1,\dots,s_d$ such that $ord(s_1),\dots,ord(s_d)$ are non-negative and $ord(s_i)-ord(s_j)\notin \ZZ$ for each pair $1\le i\neq j\le d$ iff for any $s\in M\setminus \{0\}$ it holds $ord(s)\ge 0$.  \end{lemma}

\begin{proof}
In one direction the lemma is evident, so assume that $ord(s)\ge 0$ for any $s\in M\setminus \{0\}$. Let $p_1,\dots,p_d \in M$ be a $\kk[x]$-basis of $M$. If $ord(p_i)-ord(p_j)\in \ZZ_{\ge 0}$ for some $1\le i\neq j\le d$ and $ord(p_i-\alpha x^{ord (p_i)}) < ord (p_i), \, ord (p_j-\beta x^{ord (p_j)} < ord (p_j)$ for suitable $\alpha, \beta \in \kk^\times$, one can replace $p_j$ by $p_j':= p_j-(\beta/\alpha) x^{ord (p_i)-ord (p_j)} p_i$. Clearly, $ord (p_j')<ord (p_j)$. Continuing in this way, we arrive to a required basis $s_1,\dots,s_d$.
\end{proof}

\begin{remark}\label{algorithm}
Let $f=y^d+f_1\in \ZZ[x,y]$ be normalized. Assume that the bit-sizes of the integer coefficients of $f$ do not exceed $L$. Here we agree that the field $\kk=\overline{\QQ}$. 

For a root $Y\in \kk((x^{1/\infty}))$ of $f$ consider a free $\kk[x]$-module $M\subset \kk((x^{1/\infty}))$ with a basis $1,Y,\dots,Y^{d-1}$. Then the algorithm designed in the proof of Lemma~\ref{puiseux_imbedding} either yields a basis $s_1,\dots, s_d$ of $M$ such that $ord (s_i)\ge 0$ and $ord(s_i)-ord(s_j)\notin \ZZ$ for every pair $1\le i\neq j\le d$ or the algorithm discovers an element $s\in M$ such that $ord (s)<0$.

The complexity of the algorithm is polynomial in $d, deg_x (f), L$. It follows from the polynomial complexity bound for developing Newton-Puiseux series \cite{Chistov}.
\end{remark}

Now we are able to summarize the obtained above in the following corollary.

\begin{corollary}\label{puiseux_adapted}
Let $A=\kk[x,y]/(f)$ be an algebra of an irreducible curve. Let $Y\in \kk((x^{1/\infty}))$ be a root of $f$ in the field of Newton-Puiseux series. Denote by $M\subset \kk((x^{1/\infty}))$ the $\kk[x]$-module generated by $1,Y,\dots, Y^{d-1}$. The valuation $ord$ on $A\setminus \{0\}$ induced by means of an injective homomorphism $A\hookrightarrow \kk((x^{1/\infty}))$ where $y\to Y$, is injective and well-ordered iff for any $s\in M\setminus \{0\}$ it holds $ord (s)\ge 0$ (agreeing $ord(x)=1$).

In the case of $f\in \ZZ[x,y]$ and $\kk= \overline{\QQ}$ there is an algorithm which either yields an adapted (with respect to $ord$) $\kk[x]$-basis of $A$ or discovers an element $s\in M\setminus \{0\}$ such that $ord (s)<0$. 
\end{corollary}

\begin{remark}
Due to Lemma~\ref{puiseux_imbedding} an adapted basis yielded in Corollary~\ref{puiseux_adapted} has a form $\{s_ix^j\, :\, 1\le i\le d, 0\le j\}$ for appropriate elements $s_i\in M, 1\le i\le d$.  \end{remark}

\begin{example}
Let $f:=(y^2-x)^3-8x^2$. The Newton-Puiseux expansion of its root is $Y=x^{1/2}+x^{1/6}+\cdots$. Denote $a:=y^2-x$. 
Newton polygon $N_f$ has an edge with the endpoints $(3,0),\, (0,6)$. Therefore, it does not fulfill the conditions of Theorem~\ref{quotient}. Nevertheless, the algebra $A:=\overline{\QQ}[x,a]\setminus \{0\}$ admits an injective well-ordered valuation $ord$ with an adapted basis of a form
$$x^j,\, yx^j,\, ax^j,\, ayx^j,\, a^2x^j,\, a^2yx^j,\, j\ge 0$$
\noindent due to Corollary~\ref{puiseux_adapted}.
It holds $ord(y)=1/2, ord(a)=2/3, ord(ay)=7/6, ord(a^2)=4/3, ord(a^2y)=11/6$.
\end{example}

Now we proceed to a proof of a converse statement to Corollary~\ref{puiseux_adapted}: if an algebra $A:=\kk[x,y]/(f)\setminus \{0\}$ of an irreducible curve admits an injective well-ordered valuation $\nu$, then $\nu$ is inherited from an injective homomorphism $A\hookrightarrow \kk((x^{1/\infty}))$ under which $y$ is mapped to a root of $f$ (and in addition, $\nu$ does not depend on a choice of a root). We agree that $\nu (x)=1$. Denote $f=y^d+f_1, deg_y (f_1)<d$.

We will repeatedly make use of the following easy observation. Let $a=\sum_{0\le i<d} \alpha_i y^i\in A$ and $g(a)=0$ for a suitable polynomial $g\in \kk[x,z], deg_z (g)\le d$. Then the value $\nu (a)$ is among the slopes of the edges of Newton polygon $N_g$. 

First, we recall some properties of Newton-Puiseux expansions of the roots in $\kk((x^{1/\infty}))$ of $f$ (see e.g. \cite{Walker}). There is a partition of the roots of $f$ into classes of cardinalities $d_1,\dots,d_k$ where $d_1+\cdots d_k=d$. For each class of a cardinality $d_i$ every root from this class has a form
\begin{eqnarray}\label{37}
Y=\sum_{j\ge 0} \beta_j x^{p_j/d_i} \in \kk((x^{1/\infty}))    
\end{eqnarray}
where integers $p_0>p_1>\cdots$ decrease. Moreover, all the roots from this class are exhausted by Newton-Puiseux series
\begin{eqnarray}\label{38}
\sum_{j\ge 0} \beta_j \omega^{p_j}x^{p_j/d_i}    
\end{eqnarray}
where $\omega$ ranges over the roots of unity of the degree $d_i$. In the process of Newton-Puiseux expanding of $Y$ for any intermediate current polynomial $h\in \kk[x,y]$ for the slope $p/q\in \QQ$ of each edge of Newton polygon $N_h$ it holds $q|d_i$.

\begin{lemma}\label{slope_injective}
If an algebra $A=\kk[x,y]/(f)\setminus \{0\}$ of a curve admits an injective well-ordered valuation $\nu$ then the roots of $f$ in the field of Newton-Puiseux series constitute a single class. 
\end{lemma}

\begin{proof}
Denote by $d_1,\dots,d_k$ the cardinalities of the classes of the roots of $f$. Consider an element $a=\sum_{0\le i<d} \alpha_i y^i \in A\setminus \{0\}$. Let $g(a)=0$ for an appropriate polynomial $g\in \kk[x,y],\, deg_y (h)\le d$. Then $g(\sum_{0\le i<d} \alpha_i Y^i)=0$ for any root $Y\in \kk((x^{1/\infty}))$ of $f$. Therefore, for the slope $p/q$ of every edge of Newton polygon $N_g$ it holds $q|d_l$ for suitable $1\le l\le k$.

Hence the values of $\nu$ on $A\setminus \{0\}$ are contained in a set 
$$\ZZ_{\ge 0}/d_1 \cup \cdots \cup \ZZ_{\ge 0}/d_k.$$
\noindent Here we use that $\nu$ is well-ordered, so non-negative on $A\setminus \{0\}$. Denote by $L_N\subset A$ for an integer $N\ge 0$ the $\kk$-linear space  with a basis $y^ix^j\, :\, 0\le i<d, 0\le j<N$. Then $\dim (L_N)=Nd$. On the other hand, $\nu$ attains on $L_N$ the values from a set
$$\{0,\dots,N+const\} \cup \bigcup_{1\le l\le k, 1\le p<d_l} (\{0,\dots,N+const\}+p/d_l).$$
\noindent The cardinality of the latter set does not exceed $(N+const)(d-k+1)$. Thus, if $k\ge 2$ then the valuation $\nu$ attains on $L_N$ less than $\dim (L_N)$ values, which contradicts to the injectivity of $\nu$. This completes the proof of the lemma.
\end{proof} 

For any $a=\sum_{0\le i< d} \alpha_i y^i \in A\setminus \{0\}$ due to Lemma~\ref{slope_injective} we have
$$\sum_{0\le i< d} \alpha_i Y^i=\gamma x^{p/q}+\cdots \in \kk((x^{1/\infty}))$$
\noindent where $Y$ is a root of $f$ \eqref{37}, $p/q$ is the leading exponent of Newton-Puiseux expansion, and $\gamma \in \kk^\times,\, p\in \ZZ$. Let $g(a)=0$ for a polynomial $g\in \kk[x,y], deg_y (g)\le d$. All the roots of $g$ have an expansion of the form $\gamma \omega^p x^{p/d}+\cdots$, where $\omega$ ranges over the roots of unity of the degree $d$. Hence Newton polygon $N_g$ has a unique edge with the slope $p/d$, thus $\nu(a)=p/d$.

Summarizing, we have established the following theorem.

\begin{theorem}\label{curve_injective}
If an algebra $A=\kk[x,y]/(f)\setminus \{0\}$ of an irreducible curve admits an injective well-ordered valuation $\nu$ then $\nu$ is inherited from the valuation $ord$ on $\kk((x^{1/\infty}))$ by means of an injective homomorphism $A\hookrightarrow \kk((x^{1/\infty}))$ where $y$ is mapped to a root $Y\in \kk((x^{1/\infty}))$ of $f$. The value of $\nu$ does not depend on a choice of a root.
\end{theorem}

\begin{remark}
Corollary~\ref{puiseux_adapted} and Theorem~\ref{curve_injective} together describe all the injective  well-ordered valuations on a curve, and moreover, provide an algorithm to yield all such valuations.   
\end{remark}

\begin{example}
We provide a complete description when an algebra $A=\kk[x,y]/(g)$ where $g$ is a quadratic polynomial, admits an injective well-ordered valuation $\nu$. 

First, if $g=xy+px+qy+t$ then either $\nu(x)=0$ or $\nu(y)=0$ (cf. Theorem~\ref{quotient}).
In both cases we get a contradiction with the injectivity of $\nu$.

Now we assume that $g=x^2+exy+by^2+px+qy+t$ and either $b\neq 0$ or $e\neq 0$. Then $\nu(x)=\nu(y)$ (unless $b=0, e\neq 0$ when one should consider in addition, another possibility $\nu(x)=0$, which contradicts to the injectivity, cf. above). Therefore, due to the injectivity, there exists $\alpha \in \kk$ such that for $u:=x+\alpha y$ it holds $\nu(u)<\nu(y)$. Substituting $u-\alpha y$ for $x$ in $g$, we deduce that $\alpha^2 -\alpha e+b=0$ (being the coefficient at the highest monomial $y^2$ in $g$) and $2\alpha -e=0$ (being the coefficient at the next highest monomial $uy$ in $g$). Hence $\alpha=e/2$ and $e^2-4b=0$ (being the discriminant of the highest form of $g$). Thus, $g=u^2+pu+(q-ep/2)y+t$.

If $q-ep/2\neq 0$, we fall in the conditions of Theorem~\ref{quotient}, therefore $A$ admits an injective well-ordered valuation $\nu$, and the monomials in $u$ constitute an adapted basis of $A$ with respect to $\nu$. By the same token this arguments covers also the case $b=e=0$.

Else if $q-ep/2=0$, we have $g=u^2+pu+t$, hence $\nu(u)=0$ which contradicts to the injectivity of $\nu$ (cf. above).

Thus, $A$ admits an injective well-ordered valuation iff (the discriminant of the highest form of $g$) $e^2-4b=0$, while $q-ep/2\neq 0$.

\end{example}

Consider a domain $A=\kk[x,y]/(g)$ where $g\in \kk[x,y]$. We study necessary conditions when $A\setminus \{0\}$ admits an injective well-ordered valuation $\nu$ (cf. the sufficient conditions from Theorem~\ref{quotient}). There exists an edge $e$ of the roof of Newton polygon ${\mathcal N}(g)$ such that for any points $(i,j), (k,l)$ from the edge $e$ it holds $\nu(x^iy^j)=\nu(x^ky^l)$. In this case we say that $\nu$ {\it goes along the edge $e$}.

\begin{proposition}
Let $\nu$ be an injective well-ordered valuation on $\kk[x,y]/(g) \setminus \{0\}$ which goes along an edge of ${\mathcal N}(g)$ being parallel to the line $\{x=-y\}$, and $\deg (g)>1$. Then the discriminant of the leading homogeneous form of $g$ vanishes.    \end{proposition}

\begin{proof}
We have $\nu(x)=\nu(y)$ because $\nu$ goes along the edge parallel to the line $\{x=-y\}$. The injectivity implies the existence of $0\neq \alpha \in \kk$ such that for $z:= x-\alpha y \in A$ it holds $\nu (z)<\nu(x)$. Since $d:=\deg (g) >1$ the element $z\notin \kk$, hence $\nu (z)>0$. 

Denote by $h(x,y):= b_0 x^d + b_1x^{d-1}y +\cdots + b_d y^d$ the leading homogeneous form of $g$, where $b_0,\dots,b_d \in \kk$. Replace $x$ in $g$ by $z+\alpha y$ and the resulting polynomial denote by ${\tilde g} \in \kk[z,y]$. In ${\tilde g}$ the monomial $y^d$ has the higher valuation than the other monomials. Therefore, the coefficient in ${\tilde g}$ at this monomial, which equals $h(\alpha,1)$, vanishes. The monomial $zy^{d-1}$ has the higher valuation than the other monomials in ${\tilde g}$ (except of the monomial $y^d$). Therefore, the coefficient in ${\tilde g}$ at the monomial  $zy^{d-1}$ which equals  the derivative $h_x (\alpha, 1)$, vanishes as well. Since $h$ and its derivative have a common root, its discriminant vanishes.
\end{proof}

\subsection{Adapted bases in domains with injective well-ordered valuations}

\begin{remark}
In case when $A=\kk[X_1,\dots,X_n]/(g)$ is a ring of regular functions on an irreducible hypersurface, we consider an edge of Newton polytope $N(g)$ with the endpoints $u=(u_1,\dots,u_n),\, v=(v_1,\dots, v_n)\in \ZZ_{\ge 0}^n$. Denote by $H$ the line passing through $u,v$. The principal ideal $(g)$ is saturated with respect to $H$ iff $\min\{u_i,v_i\}=0,\, 1\le i\le n$ and in addition, $u_1,\dots,u_n,v_1,\dots, v_n$ have no nontrivial common divisor, cf. Remark~\ref{finite-condition} and \eqref{3}. Moreover, $H$ is prop iff either $0\neq u,v$ or one of vectors $u,v$ equals $0$ and the other one has a single non-zero coordinate equal $1$. When $H$ is prop and $I$ is saturated with respect to $H$, there exists a well-ordered injective valuation $\nu$ on $A\setminus \{0\}$ with a valuation cone $\varphi(\ZZ_{\ge 0}^n)\subset \ZZ^n/H_{\ZZ}$ according to Theorem~\ref{quotient}. 

Observe that in this way one can obtain a well-ordered valuation $\nu$ on $\kk[x,y]\setminus \{0\}\simeq (\kk[x,y,z]/(z-y^3+x^2))\setminus \{0\}$ produced in Example~\ref{skew} (see also Proposition~\ref{partition}). Indeed, Newton polytope of the polynomial $f:=z-y^3+x^2$ is a triangle. As  $H$ we take the line passing through the edge $(2,0,0), (0,3,0)$. The principal ideal $(f)$ is saturated with respect to $H$ (cf. Proposition~\ref{curve_valuation} and Example~\ref{cusp}). Therefore, Theorem~\ref{quotient} provides just the valuation $\nu$ as in Example~\ref{skew}. 
\end{remark}

\begin{example}\label{cusp}
Let $g\in X^3+Y^2+{\mathcal L}\{1,\, Y,\, X,\, XY,\, X^2\}$ where $\mathcal L$ denote the linear span. The domain $A:=\kk[X,Y]/(g)$ defines a curve. Then the line $H=\{2X+3Y=0\}$ and $\varphi :\ZZ^2 \to \ZZ$ is given by $\varphi(i,j)=2i+3j$, the valuation cone $\varphi(\ZZ_{\ge 0}^2)=\ZZ_{\ge 0}\setminus \{1\}$. The valuation $\nu(X^{i_0}Y^{j_0}+{\mathcal L}\{X^iY^j\, :\, 2i+3j<2i_0+3j_0\})=2i_0+3j_0$ on $A\setminus \{0\}$ is well-ordered and injective.
\end{example}

\begin{theorem} 
\label{th:basis}
Let $A$ be a $\kk$-algebra. 

i) Then for any finite set of its generators $x_1,\ldots,x_m$ there is a finite set of vectors $S\subset \ZZ_{\ge 0}^m$ such that all monomials $x^w$, $w\in \ZZ_{\ge 0}^m$ for which holds  $(w-S)\cap \ZZ_{\ge 0}^m=\emptyset$ form a basis $\bf B$ of $A$;

ii) let $\nu :A\setminus \{0\} \twoheadrightarrow C$ be an injective valuation onto a monoid $C$ generated by $c_1,\dots,c_m$ endowed with a linear well ordering $\prec$. Let $a_1,\dots,a_m \in A$ be such that $\nu(a_i)=c_i,\, 1\le i\le m$. Similarly to i) there exists a finite set $S$ of monomials in $a_1,\dots,a_m$ such that $\bf B$ consisting of monomials off the monomial ideal generated by $S$, form an adapted basis of $A$ with respect to $\nu$. 
\end{theorem}

\begin{proof} i) Choose a finite  presentation $A=\kk[x_1,\ldots,x_m]/J$ and fix an injective linear weight function $q:\ZZ_{\ge 0}^m\to \RR$
inducing a well-ordering on $\ZZ_{\ge 0}^m$ and being compatible with the addition: if $q(v_1)<q(v_2)$ then $q(v_1+v)<q(v_2+v)$ for any $v,v_1,v_2\in \ZZ_{\ge 0}^m$. In particular, one can take $q(u_1,\dots,u_m)=\alpha_1u_1+\cdots+ \alpha_mu_m$ for $0<\alpha_1,\dots,\alpha_m \in \RR$ being $\QQ$-linearly independent.

Take a Gr\"obner basis of $J$ (with respect to the ordering $q$). In each element $a=\sum_i \beta_i x^{v_i},\, \beta_i\in \kk$ of the basis choose $v_{i_0}$ with the biggest value of $q(v_{i_0})$ among $q(v_i)$. We call $v_{i_0}:=lev(a)$ the leading exponent vector of $a$. Put $S$ to consist of the leading exponent vectors of all the elements of the basis.

First, we verify that the elements of ${\bf B}:=\{x^w\, :\, w\notin S+\ZZ_{\ge 0}^m\}$ are $\kk$-linearly independent in $A$. Indeed, otherwise let $$\sum_j \gamma_j x^{w_j} \in J,\, \gamma_j\in \kk, w_j\notin S+\ZZ_{\ge 0}^m.$$
\noindent This contradicts to the property of Gr\"obner bases that the monomial ideal $S+\ZZ_{\ge 0}^m$ coincides with the ideal of the leading monomials of all the elements of $J$.

Now we show that any element of the form $x^v,\, v\in \ZZ_{\ge 0}^m$ is a $\kk$-linear combination of the elements of $\bf B$. If $x^v \notin {\bf B}$ then, again due to the property of Gr\"obner bases, there exists an element $a_0\in J$ such that its leading monomial coincides with $x^v$. Consider a linear combination $x^v+\alpha a_0$ for an appropriate (unique) $\alpha \in \kk$ for which $q(lev(x^v+\alpha a_0))<q(v)$. Then we continue in a similar way, taking the biggest monomial in $x^v+\alpha a_0$ which does not belong to $\bf B$, provided that it does exist. This process terminates due to the well-ordering with respect to $q$. i) is proved. \vspace{2mm}

ii) Again pick positive reals $\alpha_1,\dots,\alpha_m$ being $\QQ$-linearly independent. Introduce a well-ordering $q$ on the monomials in $a_1,\dots,a_m$ as follows. We say that $q(a_1^{j_1}\cdots a_m^{j_m})< q(a_1^{i_1}\cdots a_m^{i_m})$ iff either $\nu(a_1^{j_1}\cdots a_m^{j_m}) \prec \nu(a_1^{i_1}\cdots a_m^{i_m})$ or $\nu(a_1^{j_1}\cdots a_m^{j_m}) = \nu(a_1^{i_1}\cdots a_m^{i_m})$ and $\alpha_1 j_1+\cdots +\alpha_m j_m < \alpha_1 i_1+\cdots +\alpha_m i_m$. 

The elements $a_1,\dots,a_m$ are generators of $A$ since $\nu$ is injective and $\prec$ is well-ordered. Therefore, $A=\kk[a_1,\dots,a_m]/J$ for certain ideal $J$.
Consider a Gr\"obner basis of $J$ with respect to $q$. 

We claim that the basis $\bf B$ of $A$ (consisting of some monomials in $a_1,\dots,a_m$) produced in i), is adapted with respect to $\nu$. Suppose the contrary. Let $\nu(a_1^{i_1}\cdots a_m^{i_m}) = \nu(a_1^{j_1}\cdots a_m^{j_m})$ for two different monomials from the basis $\bf B$. Let $\alpha_1 i_1+\cdots+\alpha_m i_m > \alpha_1 j_1+\cdots+\alpha_m j_m$ for definiteness. There exists (and unique) $\beta \in \kk$ for which 
$\nu(a_1^{i_1}\cdots a_m^{i_m} +\beta a_1^{j_1}\cdots a_m^{j_m}) \prec \nu (a_1^{i_1}\cdots a_m^{i_m})$ holds, because $\nu$ is injective. There exists an element $a_1^{l_1}\cdots a_m^{l_m} \in {\bf B}$ such  that $\nu(a_1^{i_1}\cdots a_m^{i_m} +\beta a_1^{j_1}\cdots a_m^{j_m}) = \nu (a_1^{l_1}\cdots a_m^{l_m})$. We continue the process this way. Due to well-ordering of $\nu$ the process terminates, and we arrive at an element of the form
\begin{eqnarray}\label{9}
a_1^{i_1}\cdots a_m^{i_m} +\beta a_1^{j_1}\cdots a_m^{j_m} + \sum_K \beta_K a^K \in J 
\end{eqnarray}
for appropriate $\beta_K \in \kk$, where for all the monomials from the latter sum in \eqref{9} it holds $\nu(a^K)\prec \nu(a_1^{i_1}\cdots a_m^{i_m})$. Thus, $a_1^{i_1}\cdots a_m^{i_m}$ is the highest (with respect to $q$) monomial in \eqref{9}, hence $a_1^{i_1}\cdots a_m^{i_m} \notin {\bf B}$ due to the construction of $\bf B$ in i). The obtained contradiction proves the claim and ii).     
\end{proof}

\begin{remark}
The elements $a_1,\dots,a_m$ produced in the proof of Theorem~\ref{th:basis}~ii) constitute a Khovanskii basis of $A$ \cite{Kaveh-Manon}.
\end{remark}

\begin{remark}
i) The proof of Theorem~\ref{th:basis} provides an inverse to the construction from Theorem~\ref{graded}. Namely, let $\nu$ be an injective well-ordered valuation on $A\setminus \{0\}$ with a valuation semigroup $C$, and a set of generators $a_1,\dots,a_m$ of $A$ be produced as in the proof of Theorem~\ref{th:basis}. One can represent the polynomial algebra $\kk[a_1,\dots,a_m]=\bigoplus _{c\in C} D_c$ as a graded domain where a $\kk$-basis of $D_c$ consists of all the monomials $p=a_1^{i_1}\cdots a_m^{i_m}$ such that $\nu(p)=c$.  Then we fall in the conditions of Theorem~\ref{graded}. \vspace{2mm}

ii) Theorem~\ref{th:basis} implies that one can view $A$ as a deformation of $\kk C$.

\end{remark}

\begin{definition} 
\label{def:prevaluation}
Given a basis ${\bf B}$ of an algebra $A$ we say that a map $\nu:{\bf B}\to \ZZ^m$ is a ${\bf B}$-prevaluation of $A$ if 

$\bullet$ $\nu(bb')=\nu(b)+\nu(b')$ for any $b,b'\in {\bf B}$ such that $bb'\in {\bf B}$.

$\bullet$ $\nu({\bf B})$ is a submonoid in $\ZZ^m$.

\end{definition}

If $\nu$ is injective, then, clearly, the basis ${\bf B}$ is naturally labeled by the monoid $\nu({\bf B})$. It is also clear that if $\nu:A\setminus \{0\}\to \ZZ^m$ is a valuation, then $\nu|_{\bf B}$ is a ${\bf B}$-prevaluation of $A$.

The following are immediate

\begin{lemma}
\label{le:JH prevaluation}
If $\nu,\nu'$ are injective ${\bf B}$-prevaluations, then the assignments $a\to \nu(\nu^{-1}(b))$ define a bijection ${\bf K}_{\nu,\nu'}:\nu({\bf B})\widetilde \to \nu'({\bf B})$ (we refer to it as a generalized JB bijection).
    
\end{lemma}

\begin{problem} Suppose that $\kk$ is a ring and $A$ is a finitely generated and finitely presented commutative algebra over $\kk$. If $A$ is a free $\kk$-module, does it admit a standard basis ${\bf B}$ (i.e., as in Theorem~\ref{th:basis}~i)? 
    
\end{problem}

\begin{problem} Using an adapted basis ${\bf B}$ of Theorem \ref{th:basis}~i), we can define a multivariate Hilbert series of $A$ by
$$Hilb(A)=\sum_{b\in {\bf B}} b$$
By definition, this is a rational function with denominator being the product of $(1-x_i)$.

Therefore, we can define a multivariate Hilbert polynomial of $A$ as the ``numerator" of $Hilb(A)$. 
The question is whether this definition 
gives more information about $Gr~A$ and $A$ than the ordinary Hilbert series $Hilb(A,t)$.

\end{problem}

\begin{remark}\label{minimal}
The adapted basis $\bf B$ produced in Theorem~\ref{th:basis}~ii) consists of the following elements: for each $c\in C$ take the monomial $M$ in $a_1,\dots,a_m$ being minimal (with respect to $f$) among the monomials for which $\nu(M)=c$ holds. 

Another description is that $\bf B$ consists of all the monomials being $\kk$-linearly independent (in $A$) from less (with respect to $f$) monomials. For any monomial $M_0 \in \bf B$ consider the next (with respect to $f$) monomial $M_1 \in \bf B$. Then for any monomial $M$ such that $f(M_0)\le f(M)<f(M_1)$ it holds $\nu (M)=\nu(M_0)$.
\end{remark}

\begin{corollary}\label{valuation_rank}
Let $A$ be a commutative $\kk$-algebra, $\nu:A\setminus \{0\} 
\twoheadrightarrow C$ be an injective valuation onto a finitely-generated monoid $C$ of rank $r$ (see Definition~\ref{def_rank}) endowed with a linear well ordering. Then $r= d:=\dim(A)$.  
\end{corollary}

\begin{proof} First we show that $r\le d$. Pick independent elements $c_1,\dots, c_r\in C$ and $a_1,\dots,a_r \in A$ such that $\nu(a_i)=c_i, 1\le i\le r$. Then all monomials in $a_1,\dots,a_r$ have pairwise distinct valuations $\nu$, therefore $a_1,\dots,a_r$ are algebraically independent, thus $r\le d$. Now we prove the opposite inequality.

Due to Theorem~\ref{th:basis}  $\bf B$ is the complement of a monomial ideal generated by the leading monomials of Gr\"obner basis of the ideal $J$ in the representation $A=\kk[a_1,\dots,a_m]/J$. Therefore, there exist $1\le l_1<\cdots < l_d\le m$ such that all the monomials in $a_{l_1},\dots,a_{l_d}$ belong to $\bf B$, see Proposition 3 in Chapter 9.1 and Proposition 4 in Chapter 9.3 \cite{C}. Hence the elements $\nu(a_{l_1}),\dots,\nu(a_{l_d})$ are independent in $C$, taking into account that the basis $\bf B$ is adapted to $\nu$ due to Theorem~\ref{th:basis} ii). Thus, $r\ge d$.      
\end{proof}

\begin{remark}
Assume that $C$ is a (not necessary commutative) monoid generated by $c_1,\dots, c_r$. We call the length $|c|$ of $c\in C$ the minimal length of words in $c_1,\dots, c_r$ equal $c$. Let $C$ be endowed with a linear well-ordering $\prec$ compatible with the length, i.e. $|c_0|<|c|, c_0,c \in C$ implies $c_0\prec c$. For example, the ordering described prior to Theorem~\ref{filtration} of the free monoid is compatible with the length.

Consider an algebra $A$ having an injective valuation $\nu :A \twoheadrightarrow C$, and pick elements $a_1,\dots, a_r \in A$ such that $\nu(a_i)=c_i, 1\le i\le r$. For each $c\in C$ choose a monomial $a_c$ in $a_1,\dots, a_r$ for which $\nu(a_c)=c$. Then $\{a_c\, :\, c\in C\}$ form an adapted basis of $A$ (and $a_1,\dots, a_r$ form a Khovanskii basis of $A$). Then the linear subspaces $A_k:=\{a\in A\, :\, |\nu(a)|\le k\}, k\ge 0$ constitute a filtration of $A$, and $\dim A_k$ coincides with the cardinality of the set $C_k:=\{c\in C\, :\, |c|\le k\}$, moreover $\nu(A_k)=C_k$. We recall that in the commutative case the latter cardinality grows polynomially in $k$ (being a Hilbert polynomial, see e.g. \cite{K}). 
\end{remark}

The following remark is inverse to Theorem~\ref{th:basis}~ii) and to Remark~\ref{minimal}. According to Theorem~\ref{th:basis}~ii) and to Remark~\ref{minimal} every injective well-ordered valuation on an algebra can be obtained as described in the remark.

\begin{remark}
Let $a_1,\dots,a_m$ be generators of a commutative $\kk$-algebra $A$ endowed with a linear order $f$ on monomials in $a_1,\dots,a_m$ (compatible with the product). 

Consider the family $\bf B$ of all the monomials being $\kk$-linearly independent in $A$ from the less ones (with respect to $f$). Then $\bf B$ forms a basis of $A$. For each $M_1,M_2 \in \bf B$ denote by $h(M_1,M_2)(=h(M_2,M_1)) \in \bf B$ the leading monomial in the $\kk$-linear expansion of the product $M_1M_2$ in $\bf B$. Assume that for any pair of monomials $M_0,M_1 \in \bf B$ fulfilling $f(M_0)<f(M_1)$ it holds $f(h(M_0,M_2))<f(h(M_1,M_2))$. Then one can introduce a monoid $C$ being in a bijective correspondence with $\bf B$ determined by the monoid operation $h$ and the linear ordering $f$.

This induces also an injective well-ordered valuation $\nu :A\setminus \{0\} \twoheadrightarrow C$ defined by the leading monomial from $\bf B$ in the $\kk$-linear expansion. Then $\bf B$ is an adapted basis of $\nu$.

One can reorder the monomials in $a_1,\dots,a_m$ as follows to make the new ordering $\lhd$ similar to the one in Theorem~\ref{th:basis}~ii) and in Remark~\ref{minimal}. We say that for a pair of monomials it holds $a_1^{i_1}\cdots a_m^{i_m} \lhd a_1^{j_1}\cdots a_m^{j_m}$ if either $f(\nu (a_1^{i_1}\cdots a_m^{i_m})) < f(\nu (a_1^{j_1}\cdots a_m^{j_m}))$ or $f(\nu (a_1^{i_1}\cdots a_m^{i_m})) = f(\nu (a_1^{j_1}\cdots a_m^{j_m}))$ and $f(a_1^{i_1}\cdots a_m^{i_m}) < f(a_1^{j_1}\cdots a_m^{j_m})$. Then the construction from Theorem~\ref{th:basis}~ii) applied to $\lhd$ produces the same basis $\bf B$ which now satisfies the properties from Remark~\ref{minimal}.
\end{remark}

Observe that the valuation produced in Theorem~\ref{quotient} fulfills the conditions of Theorem~\ref{th:basis} ii) and of Remark~\ref{minimal}. In particular, $\nu$ admits an adapted basis of monomials in $X_1,\dots,X_n$. The monomials with equal values of $\nu$ lie in the planes parallel to $H$.

\subsection{Injective valuations,  filtrations and deformations}

Now one can establish an inverse statement to Theorem~\ref{quotient}.

\begin{theorem}\label{inverse}
Let $A$ be a commutative domain of dimension $d$ endowed with an injective well-ordered valuation $\nu$ onto a finitely-generated monoid. Then there exist a Khovanskii basis  $X_1,\dots,X_n$ of $A$ such that $A=\kk[X_1,\dots,X_n]/I$, and $\nu$ is obtained as in Theorem~\ref{quotient}. 

In other words, there is a prop subplane $H\subset \RR^n$ of dimension $n-d$, being a common subplane for the tropical variety $Trop (I)\subset \RR^n$. Moreover, the ideal $I$ is saturated with respect to $H$. There exists a hyperplane $Q\in ETrop(I)\subset (\RR[\varepsilon])^n$ which contains the subplane $H\bigotimes_{\RR} \RR[\varepsilon]$, and $Q$ is determined by a suitable vector $(w_1,\dots,w_n)\in (\RR_{\ge 0}[\varepsilon])^n$.
Then $\nu$ is defined by \eqref{2}, the valuation monoid $\nu(A\setminus \{0\}) = \varphi (\ZZ_{\ge 0}^n)$ where $\varphi : \RR^n \twoheadrightarrow \RR^n /H$, and the linear order on $\varphi(\ZZ_{\ge 0}^n)\ni \varphi(i_1,\dots,i_n)$ is determined by the value of $w_1i_1+\cdots +w_ni_n$.
\end{theorem}

\begin{proof}
Applying Theorem~\ref{th:basis} one can find generators $X_1,\dots,X_n$ of $A$ such that the valuation monoid $C:=\nu(A\setminus \{0\})$ equals the set of values $\nu(M)$ over all the monomials $M=X_1^{i_1}\cdots X_n^{i_n}$ in $X_1,\dots,X_n$. Then $A=\kk[X_1,\dots,X_n]/I$ for an appropriate ideal $I$. 

Due to \cite{R} there exist elements $w_1,\dots,w_n \in \RR_{\ge 0}[\varepsilon]$ such that the linear order in $C$ of  $\nu(X_1^{i_1}\cdots X_n^{i_n})$ coincides with the order of the values of  $w_1i_1+\cdots+w_ni_n$ in the semi-ring $\RR_{\ge 0}[\varepsilon]$.

Denote by $H\subset \RR^n$ a plane being the linear span of all the vectors of the form $(l_1,\dots,l_n)-(j_1,\dots,j_n)\in \ZZ^n$ where $w_1l_1+\cdots + w_nl_n = w_1j_1+\cdots + w_nj_n$, the latter is equivalent to $\nu(X_1^{l_1}\cdots X_n^{l_n})=
\nu(X_1^{j_1}\cdots X_n^{j_n})$. Due to Theorem~\ref{th:basis} the hyperplane $Q$ contains $H\bigotimes_{\RR} \RR(\varepsilon)$ and supports the Newton polytope $N(g)\bigotimes_{\RR} \RR[\varepsilon]$ for any $g\in I$. Theorem~\ref{th:basis} also implies that $\dim(H)=n-d$. Hence $H$ is a common subplane of $Trop(I)$. In addition, $H$ is prop since $w_1,\dots,w_n \in \RR_{\ge 0}[\varepsilon]$.

Take two arbitrary points $(l_1,\dots,l_n),(j_1,\dots,j_n) \in \ZZ_{\ge 0}^n$ such that $(l_1,\dots,l_n)-(j_1,\dots,j_n)\in H$. Then $w_1l_1+\cdots +w_nl_n= w_1j_1+\cdots +w_nj_n$. Therefore $\nu (X_1^{l_1}\cdots X_n^{l_n})=\nu (X_1^{j_1}\cdots X_n^{j_n})$, and due to injectivity and well-ordering of $\nu$ there exist $\beta \in \kk^\times$ and $g_1=\sum_S \gamma_S X^S \in \kk[X_1,\dots,X_n]$ such that $\nu(X^S)<\nu(X_1^{j_1}\cdots X_n^{j_n})$ for every $X^S$ occurring in 
$g_1$, and it holds $X_1^{l_1}\cdots X_n^{l_n} - \beta X_1^{j_1}\cdots X_n^{j_n} -g_1 \in I$. Hence $I$ is saturated with respect to $H$ (cf. Theorem~\ref{quotient}).

Finally, for any $a\in A\setminus\{0\}$ one can uniquely express $a=\sum_{b\in B_0} \alpha_b b$ in a basis ${\bf B}\supset B_0\ni b$ produced in Theorem~\ref{th:basis}, $\alpha_b \in \kk^\times$. Then $\nu(a)=\max_{b\in B_0} \{\nu(b)\}$, thus $\nu$ is defined by \eqref{2}.

\end{proof}

We say that the linear order $\prec$ defined by $w_1,\dots,w_n\in \RR_{\ge 0}[\varepsilon]$  is {\it archimedian} if among $w_1,\dots, w_n$ there are no infinitesimals. This is equivalent to that for any pair of monomials $m_1, m_2\neq 1$ there exists an integer $N$ such that $m_1\prec m_2^N$. The linear order on $\ZZ_{\ge 0}^n$ is archimedian iff this order is isomorphic to $\ZZ_{\ge 0}$. For instance, deglex is archimedian, while lex is not.

More generally, we say that a linear order $\prec$ on a commutative monoid $C$ is archimedian if for any elements $1\neq c_1, c_2\in C$ there exists an integer $N$ such that $c_2 \prec Nc_1$. Note that if for any $c\in C$ there is at most a finite number of elements $c_0\in C$ such that $c_0\prec c$ then $C$ is archimedian and well-ordered.
Conversely, if a commutative monoid  $C$ is finitely-generated and $\prec$ is an archimedian linear order on $C$ then for any $c\in C$ there is at most a finite number of elements of $C$ less than $c$. In particular, in this case $C$ is well-ordered. For a not necessary commutative monoid $C$ we also say that a linear order $\prec$ on it is archimedian if for any $c\in C$ there is at most a finite number of elements of $C$ less than $c$. 

Let $c_1,\dots, c_k\in C$ be a set of generators of a commutative monoid $C$. Due to \cite{R} there exist elements $w_1,\dots, w_k\in \RR_{\ge 0}[\varepsilon]$ not being infinitesimals such that 
$$\nu(j_1c_1+\cdots +j_kc_k)\preceq \nu(l_1c_1+\cdots +l_kc_k) \Leftrightarrow (w_1j_1+\cdots +w_kj_k \le w_1l_1+\cdots +w_kl_k)$$
\noindent for any $j_1,\dots, j_k, l_1,\dots, l_k \in \ZZ_{\ge 0}$. Define a function $W:C\to \RR_{\ge 0}[\varepsilon]$ as follows:
$$W(j_1c_1+\cdots +j_kc_k):=w_1j_1+\cdots +w_kj_k.$$

\begin{remark}\label{monoid_archimedian}
Let $A$ be a commutative domain endowed with a  valuation (not necessary injective) onto a finitely-generated monoid $C$ with an archimedian linear order $\prec$. 
For each $s\in \ZZ_{\ge 0}$ consider the set
$$A_s:=\{a\in A\setminus\{0\}\, :\, W(\nu(a))\le s\} \cup \{0\}.$$
\noindent The sequence $A_0\subset A_1\subset \cdots$ constitutes a filtration of $A$. Observe that $\dim(A_s)$ is finite since $\prec$ is archimedian. 
\end{remark}

\begin{remark}\label{length}
Now let $A$ be a (not necessary commutative) $\kk$-algebra endowed with an injective valuation $\nu$ to a (not necessary commutative) monoid $C$ with a linear order $\prec$. Assume also that there is a function $f:C \to \ZZ_{\ge 0}$ such that $(c_1\preceq c_2) \Rightarrow (f(c_1)\le f(c_2)),\, c_1,c_2\in C$ and $f(c_1+c_2)\le f(c_1)+f(c_2)$ satisfying the property that 
the set $C_n:= \{c\in C\, :\, f(c)\le n\}$ is finite for any $n\in \ZZ_{\ge 0}$. Note that the latter implies that the order $\prec$ is archimedian. Then the subspaces $A_n:=\{a\in A\, :\, f(\nu(a))\le n\}\cup \{0\},\, n\in \ZZ_{\ge 0}$ provide a filtration of $A$ (cf. Remark~\ref{monoid_archimedian}) such that $\dim A_n=|C_n|$.      
\end{remark}

Now we assume that the field $\kk$
is radically closed (i.e.
each root of an arbitrary degree of any element of $\kk$ also belongs to $\kk$). Let $A$ be a $d$-dimensional 
$\kk$-algebra with an injective valuation $\nu:A\setminus \{0\}\twoheadrightarrow C$ where $C$ is a finitely-generated monoid endowed with a linear well-ordering. Applying Theorem~\ref{th:basis} construct a Khovanskii basis $x_1,\dots,x_n \in A$, then $A=\kk[x_1,\dots,x_n]/I$ for a suitable ideal $I\subset \kk[x_1,\dots,x_n]$.

Denote by $S\subset \kk[x_1,\dots,x_n]$ a binomial ideal generated by  elements of the form 
\begin{eqnarray}\label{10}
s:=\alpha x_1^{i_1}\cdots x_n^{i_n}- \beta x_1^{j_1}\cdots x_n^{j_n},\, \alpha,\beta \in \kk^\times \end{eqnarray}
such that $\nu(x_1^{i_1}\cdots x_n^{i_n})=\nu(x_1^{j_1}\cdots x_n^{j_n})$,
and there exists an element $g\in I$  which
is the sum of $s$ and of monomials in $x_1,\dots,x_n$ having valuation $\nu$ less than  $\nu(x_1^{i_1}\cdots x_n^{i_n})$.
Denote a binomial algebra $M:=\kk[x_1,\dots,x_n]/S$.
Observe that for any pair of vectors $(i_1,\dots,i_n), (j_1,\dots,j_n)\in \ZZ_{\ge 0}^n$ there exist $s$ and $g$ as above iff $\nu(x_1^{i_1}\cdots x_n^{i_n})
=\nu(x_1^{j_1}\cdots x_n^{j_n})$ due to Theorem~\ref{th:basis}. In addition, $\nu(x_1^{i_1}\cdots x_n^{i_n})
=\nu(x_1^{j_1}\cdots x_n^{j_n})$ is equivalent to that $(i_1,\dots,i_n)-(j_1,\dots,j_n)\in H$ where an $(n-d)$-dimensional subplane $H$ was constructed in the proof of Theorem~\ref{inverse}.

\begin{proposition}\label{isomorphism}
(cf. \cite{Kaveh}). Let a $\kk$-algebra $A$ have an injective valuation $\nu:A\setminus\{0\}\twoheadrightarrow C$ where a finitely-generated monoid $C$ is endowed with a linear well-ordering. Assume that $\kk$ is radically closed. Then both the associated graded algebra $gr A:=\bigoplus _{c\in C} A_{\le c}/A_{<c}$ and the binomial algebra $M$ are 
isomorphic to the monoidal algebra $\kk C$.    
\end{proposition}

\begin{proof}
There exist $\gamma_1,\dots,\gamma_n\in \kk \setminus \{0\}$ such that the mapping $$\mu :x_1^{i_1}\cdots x_n^{i_n} \to \gamma_1^{i_1}\cdots \gamma_n^{i_n} \cdot \nu(x_1^{i_1}\cdots x_n^{i_n})$$
\noindent provides an isomorphism of $M$ and $\kk C$ because of Theorem~\ref{th:basis}, taking into account that $\kk$ is radically closed. 
In other words, if $g\in I$ equals the sum of a binomial \eqref{10} and of monomials with the valuation $\nu$ less than $\nu(x_1^{i_1}\cdots x_n^{i_n})=\nu(x_1^{j_1}\cdots x_n^{j_n})$ then $\mu (\alpha x_1^{i_1}\cdots x_n^{i_n})=\mu (\beta x_1^{j_1}\cdots x_n^{j_n})$, i.e. $\alpha\gamma_1^{i_1}\cdots \gamma_n^{i_n}=\beta \gamma_1^{j_1}\cdots \gamma_n^{j_n}$.

Any element $a\in A\setminus \{0\}$ with the valuation $\nu(a)=c$ can be represented uniquely as a $\kk$-linear combination of elements of a basis $\bf B$ constructed in Theorem~\ref{th:basis}, among which $\alpha x_1^{i_1}\cdots x_n^{i_n}, \alpha \in \kk^\times, x_1^{i_1}\cdots x_n^{i_n} \in {\bf B}$ has the maximal valuation $\nu(x_1^{i_1}\cdots x_n^{i_n})=c$. 
We define a mapping $\sigma (a):=\mu (\alpha x_1^{i_1}\cdots x_n^{i_n}) \in \kk C$. Then $\sigma$ defines a correct mapping on $gr A$.

To verify that $\sigma$ is a homomorphism on $gr A$
take monomials $u:=x_1^{i_1}\cdots x_n^{i_n}, v:=x_1^{k_1}\cdots x_n^{k_n}\in A$. Due to Theorem~\ref{th:basis} there exists a unique monomial $x_1^{l_1}\cdots x_n^{l_n}\in {\bf B}$ such that $\nu(x_1^{l_1}\cdots x_n^{l_n})=\nu(uv)$. Hence there exists $\alpha\in \kk^\times$ such that $\nu(\alpha x_1^{l_1}\cdots x_n^{l_n} - uv)<\nu(uv)$ due to Theorem~\ref{th:basis}. Therefore
$$\sigma(uv)=\mu(\alpha x_1^{l_1}\cdots x_n^{l_n})=\alpha \gamma_1^{l_1}\cdots \gamma_n^{l_n} \nu(uv)=\gamma_1^{i_1+k_1}\cdots \gamma_n^{i_n+k_n} \nu(u)\nu(v)=\sigma(u) \sigma(v)$$
\noindent (we use the product notation for the monoid operation). 

Finally, one can check that $\sigma$ is an isomorphism. 
\end{proof}

\begin{corollary}
Let $A$ be a $\kk$-domain of a dimension $d$ over a radically closed field $\kk$, and $\nu$ be an injective valuation of $A\setminus\{0\}$ onto a well-ordered finitely-generated monoid. Then the variety $Spec (gr (A, \nu))$ is toric of dimension $d$.    
\end{corollary}

\begin{example}
 Consider a domain $A:=\kk[x,y]/(x^6-2y^4-1)$. Take a common line $H$ from the tropical variety $Trop (x^6-2y^4-1)$ defined by the equation $2i+3j=6$ (cf.  Theorem~\ref{quotient}) and a corresponding map $\nu : \{x^iy^j\, :\, i,j\ge 0\} \to \ZZ_{\ge 0}^2$ for which $\nu(x)=2, \nu(y)=3$. We obtain that the graded algebra $gr (A, \nu)$ has a zero divisor iff the polynomial $x^6-2y^4$ is reducible over $\kk$. In the latter case $\nu$ does not provide a valuation on $A\setminus \{0\}$ (see Theorem~\ref{filtration}), and the variety $Spec (gr(A, \nu))$ is reducible. Otherwise, if $x^6-2y^4$ is irreducible over $\kk$ then $\nu$ provides a valuation on $A$ being not injective since $\nu(x^3)=\nu(y^2)=6$, so $\dim (A_{\nu\le 6}/A_{\nu<6}) =2$.   
\end{example}

The following is a generalization of the well-known commutative result 
on localization of valuations.

We say that a submonoid $Q$  of a  monoid $P$ is Ore if
for any $p\in P$, $q\in Q$ there exist $p',p''\in P$, $q,q''\in Q$ such that $q'p=p'q$ and $pq''=qp''$.

\begin{proposition}\label{ore} Let $P$ be an ordered  monoid and $\nu:A\setminus \{0\}\to P$ a valuation. Let $M$ be an  Ore submonoid of the multiplicative semigroup of $A$, and denote $Q=\nu(M)$ (hence $Q$ is an Ore submonoid of $P_\nu$). Then  the monoid $P_\nu\cdot Q^{-1}$ is  ordered. The assignments $\nu(am^{-1}):=\nu(a)\nu(m)^{-1}$ define a valuation $\nu_M:A[M^{-1}]\setminus \{0\}\to P_\nu\cdot Q^{-1}$.
    \end{proposition}

\begin{proof} 
Let $\prec$ be an order on $P$. One can define an order $\trianglelefteq$ on $P_\nu\cdot Q^{-1}$ as follows: $pq^{-1}\trianglelefteq q_0^{-1}p_0$ iff $q_0p\preceq p_0q$. 

Recall that for any $b\in A$ and $n\in M$ there exist $a\in A$ and $m\in M$ such that $n^{-1}b=am^{-1}$ (equivalently, $na=bm$). Then

$$\nu(n^{-1}b)=\nu(am^{-1})=\nu(a)\nu(m)^{-1}=\nu(n)^{-1}\nu(b)$$ 

since $\nu(n)\nu(a)=\nu(b)\nu(m)$.

Now 
$u=n^{-1}b$, $v=am^{-1}$. Then $u+v=n^{-1}(bm+na)m^{-1}$ and assume w.l.o.g. that $\nu(bm+na)\preceq \nu(na)$. It holds
$$\nu(u+v)=\nu(n)^{-1}\nu(bm+na)\nu(m)^{-1}\trianglelefteq \nu(n)^{-1}\nu(na)\nu(m)^{-1}=\nu(n)^{-1}\nu(n)\nu(a)\nu(m)^{-1}=\nu(v).$$

\end{proof}

\begin{example} 
We give an example of a non-commutative algebra admitting an injective valuation onto $\ZZ_{\ge 0}^2$ (endowed with say, lex ordering) such that its graded algebra (with respect to the valuation) is non-commutative (unlike Proposition~\ref{isomorphism}). Denote by $A_q$ the $\kk[q^{\frac{1}{2}},q^{-\frac{1}{2}}]$-algebra generated by $x,y$ satisfying a relation $yx=qxy$ (see Example~\ref{ex:quantum plane}), i.e., $A_q$ is the quantum plane. 

Denote by $M_q\subset A_q$ the set $q^{\ZZ_{\ge 0}} x^{\ZZ_{\ge 0}}y^{\ZZ_{\ge 0}}$. Clearly, this is a submonoid of $A_q^\times$ because
$$(q^r x^m y^n)(q^{r'}x^{m'}y^{n'})=q^{r+r'+m'n}x^{m+m'}y^{n+n'}$$
 It is easy to see that $A_q=\kk M_q$  and $M_q$ is isomorphic to a submonoid a group of all $3\times 3$ unipotent matrices over $\frac{1}{2}\ZZ_{\ge 0}$ via $q^r x^m y^n\mapsto 
 \begin{pmatrix}
  1 & n &r\\
  0 & 1 & m\\
  0 & 0 & 1
 \end{pmatrix}$.

Applying Proposition \ref{ore}, we see that $M_q$ is naturally (lexicographically on triples $(m,n,r)$) ordered and obtain the tautological valuation of $A_q$ to $M_q$. Clearly, the assignments $(m,n,r)\mapsto (m,n)$ define a surjective ordered homomorphism of monoids $f:M_q\to \ZZ_{\ge 0}^2$. In view of Proposition \ref{pr:valuation after homomorphism}, this defines a  valuation $\nu:A_q\setminus \{0\}\to \ZZ_{\ge 0}^2$ via $\nu(q^{\frac{1}{2}})=0$,  $\nu(x)=(0,1),\, \nu(y)=(1,0)$.

This valuation is non-injective  if $A_q$ is viewed as a $\kk$-algebra and it becomes injective if $A_q$ is viewed as a $\kk[q^{\frac{1}{2}},q^{-\frac{1}{2}}]$-algebra with an adapted basis $x^{\ZZ_{\ge 0}}y^{\ZZ_{\ge 0}}$.

It is well-known (see e.g., \cite[Proposition A.1]{BZq}) that $\mathbb{K}\otimes A_q$ is an Ore domain, where $\mathbb{K}:=\kk(q^{\frac{1}{2}})$.
Applying Proposition~\ref{ore} to any submonoid $M\subset A_q$, we obtain a valuation $\nu_M: A_q [M^{-1}]\setminus \{0\} \to \Gamma_q$, where $\Gamma_q$ is the (Heisenberg) group generated by $M_q$.

 Recall that Ore extension $R_{\varphi,\partial}[y]$ of a ring $R$ is a flat deformation of $R[y]$
given by an endomorphism $\varphi$ and a $\varphi$-derivation $\partial$ via 
$$yr=\varphi(r)y+\partial(r)$$
In particular, if $R=\mathbb{K}[x]$, $\partial=0$ and $\varphi(x)=qx$, $\varphi(q)=q$ then $R_{\varphi,\partial}[y]=\mathbb{K}\otimes A_q$ (cf. Proposition~\ref{skew_quantum}).

On the other hand, if $\varphi=Id_R$, and  $\partial$ is a derivation of $R$, then $R_{\varphi,\partial}[y]$ is a generalized Weyl algebra:
$[y,r]:=yr-ry=\partial(r)$
for all $r\in R$.

Suppose that $R$ is a $\mathbb{K}$-algebra and $\nu_R:R\setminus \{0\}\to P$ is a valuation over $\mathbb{K}$.
Suppose also that $\varphi_0$ is an ordered automorphism of $P$ such that $\nu_R(\varphi(r))=\varphi_0(\nu_R(r))$. 
Let
$\hat P:= P \times_\varphi \ZZ_{\ge 0}$ be an Ore extension of $P$, i.e., $(p,k)\circ (p',k')=(p\circ \varphi_0^k(p'),k+k')$.
Clearly, $\hat P$ is lexicographically ordered (with $\ZZ_{\ge 0}$ higher than $P$), and
the assignments $ry^k\mapsto (\nu_R(r),k)$ extend $\nu_R$ to a valuation $\nu:R_{\varphi,\partial}[y]\setminus \{0\}\to \hat P$.
This valuation is, clearly, injective if $\nu_R$ was injective. If ${\bf B}\subset R$ is an adapted basis with respect to $\nu_R$ then ${\bf B}\times y^{\ZZ_{\ge 0}}\subset R_{\varphi,\partial}$ is an adapted basis with respect to $\nu$.

Note that assignments $(p,k)\mapsto k$ define an ordered (surjective) homomorphism $f:\hat P\to \ZZ_{\ge 0}$. In view of Proposition \ref{pr:valuation after homomorphism}, this defines a (non-injective) valuation of $R_{\varphi,\partial}[y]$ to $\ZZ_{\ge 0}$ via $ry^k\mapsto k$.

In particular, the Weyl $A_1$ algebra generated by $x,y$ satisfying a relation $xy=yx+1$ can be viewed as an Ore extension of $R=\kk[x]$ with $\varphi=Id_R$, $\partial=\frac{d}{dx}$. Then $\hat P=\ZZ_{\ge 0}^2$ and thus $A_1$
admits an injective valuation to $\ZZ_{\ge 0}^2$ defined  by $\nu(x)=(1,0),\, \nu(y)=(0,1)$ onto $\ZZ_{\ge 0}^2$, and its graded algebra is isomorphic to the polynomial ring $\kk[x,y]$.
\end{example}

Let $A=\kk[x_1,\dots,x_n]/J$ be a $\kk$-algebra with $\dim A=1$ and $\nu:A\setminus \{0\} \twoheadrightarrow C\subset \ZZ_{\ge 0}$ be an injective valuation (cf. Corollary~\ref{valuation_rank}). Then there exist non-negative integers $r_1,\dots,r_n$ such that $\nu(x_1^{i_n}\cdots x_n^{i_n}) \preceq \nu(x_1^{j_1}\cdots x_n^{j_n})$ iff $i_1r_1+\cdots +i_nr_n \le j_1r_1+\cdots +j_nr_n$ (cf. Theorem~\ref{inverse}).

Take $t\in \kk^\times$, and for any polynomial $g\in \kk[x_1,\dots,x_n]$ with a monomial $x_1^{i_n}\cdots x_n^{i_n}$ of the highest valuation $\nu(x_1^{i_n}\cdots x_n^{i_n})=c$ among the monomials of $g$, replace $x_i$ by $x_i t^{-r_i}, 1\le i\le n$. The resulting polynomial has the form $(g_0+tg_1)t^{-i_1r_1-\cdots - i_nr_n}$ where $g_0$ coincides with the sum of the monomials of $g$ having their valuation equal $c$, while every monomial in $g_1$ has the valuation less than $c$. The resulting ideal we denote by $J_t\subset \kk[x_1,\dots,x_n]$ and denote $A_t:=\kk[x_1,\dots,x_n]/J_t$. 

Thus, one can view the family $A_t,\, t\in \kk^\times$ as a deformation of $A_0$ being a binomial algebra  isomorphic to $\kk C$ (see Proposition~\ref{isomorphism}) when the field $\kk$ is radically closed. Summarizing, we have established the following proposition.

\begin{proposition}\label{deformation}
Let $A=\kk[x_1,\dots,x_n]/J$ be a $\kk$-algebra with $\dim A=1$ and $\nu:A\setminus \{0\} \twoheadrightarrow C\subset \ZZ_{\ge 0}$ be an injective valuation. Then there exist non-negative integers $r_1,\dots,r_n$ such that for any $t\in \kk^\times$ the transformation $x_i \to x_i t^{-r_i}, 1\le i\le n$ provides an ideal $J_t\subset \kk[x_1,\dots,x_n]$ and an algebra  $A_t:=\kk[x_1,\dots,x_n]/J_t$. Obviously, $A_1=A$. The associated graded algebra $gr A_t \simeq \kk C, t\in \kk^\times$ (see Proposition~\ref{isomorphism}) when $\kk$ is radically closed. One can view the family $A_t,\, t\in \kk^\times$ as a deformation of $A_0$ being a binomial algebra  isomorphic to $\kk C$ (see Proposition~\ref{isomorphism}) when the field $\kk$ is radically closed.    
\end{proposition}

\begin{conjecture} For any domain $A$ with injective well-ordered valuation $\nu$ there is a family $A_t$ such that $A_1=A$ and $A_0=gr A$ with respect to the filtration on $A$ induced by $\nu$ (as in Proposition~\ref{deformation}).
\end{conjecture}

\begin{problem} Let $E$ be a locally nilpotent derivation of a domain $A$ (see Lemma~\ref{le:nilpotent valuation}). Classify all subalgebras $B$ of $A$ such that $\lambda_E(B\setminus\{0\})\cup \{0\}$ is a subalgebra of $A$.

\end{problem}

\begin{problem} Classify all injective decorated valuations $(\nu,\lambda)$ on $\kk[x_1,\ldots,x_n]$ (see Definition~\ref{decorated valuation}).

\end{problem}

\begin{problem} Given an injective valuation $\nu:\kk[x_1,\ldots,x_m]\setminus \{0\}\to \ZZ_{\ge 0}^m$, is it true that $C_\nu$ is a finitely generated submonoid of $\ZZ_{\ge 0}^m$?

\end{problem}

Given a submonoid $M$ of $\ZZ^m$,  denote   $\overline M:=(\RR_{\ge 0}\otimes M)\cap \ZZ^m$ and refer to it as the {\it saturation } of $M$. By definition, $\overline M$ is a submonoid of $\ZZ^m$ and $M$ is a submonoid of $\overline M$. We say that $M$ is saturated if $\overline M=M$.

\begin{problem} Suppose that $\nu$ is a saturated injective valuation $A\setminus\{0\}\to \ZZ_{\ge 0}^m$. Is it true that $Spec ~A$ is smooth or rational?

\end{problem}

\begin{problem} If $A\subset \kk[x_1,\ldots,x_m]$ and $\dim~A=d<m$. Can $A$ be embedded into $\kk[y_1,\ldots,y_d]$?

\end{problem}

\begin{problem} Describe all injective valuation on $\kk[x_1,\ldots,x_m]$ into $\ZZ_{\ge 0}^m$ whose valuation monoid is finitely generated but not saturated.

\end{problem}

\subsection{Algorithm testing a family of generators of a valuation monoid}

Let $(f_1,\dots,f_m):\kk^n \to \kk^m, m\le n$ be a dominant polynomial map, i.e. the polynomials $f_1,\dots,f_m\in \kk[x_1,\dots,x_n]$ are algebraically independent.

\begin{problem}
 Is the valuation monoid $\nu(\kk[y_1,\dots,y_m]\setminus \{0\})\subset \ZZ_{\ge 0}^n$  finitely-generated?    
\end{problem}

The goal of this subsection is to prove the following proposition.

\begin{proposition}
Let polynomials $f_1,\dots,f_m\in \kk[x_1,\dots,x_n], m\le n$ define an injective homomorphism  $\kk[y_1,\dots,y_m] \hookrightarrow \kk[x_1,\dots,x_n]$. Given a computable injective valuation on $\kk[x_1,\dots,x_n]\setminus \{0\}$ to $\ZZ_{\ge 0}^n$ (e.g., lex or deglex), this provides an injective valuation $\nu:\kk[y_1,\dots,y_m]\setminus \{0\} \to \ZZ_{\ge 0}^n$. 

There is an algorithm which given generators $g_1,\dots,g_p\in \kk[y_1,\dots,y_m]$ of $\kk[y_1,\dots,y_m]$ tests whether the elements
$$\nu(g_1),\dots,\nu(g_p)\in C:=\nu(\kk[y_1,\dots,y_m]\setminus \{0\})\subset \ZZ_{\ge 0}^n$$
\noindent generate the valuation monoid $C$. If not then the algorithm yields an element $c\in C\setminus \ZZ_{\ge 0}\{\nu(g_1),\dots,\nu(g_p)\}$.
\end{proposition}

\begin{proof}
We have $\kk[y_1,\dots,y_m]=\kk[g_1,\dots,g_p]/I$ for a suitable ideal $I$. Consider a (non-strict) linear ordering $\prec$ on monomials in $g_1,\dots,g_p$ according to $\nu$. Denote by $H\subset \RR^p$ a (rational) $(p-m)$-dimensional plane such that $\nu(g_1^{i_1}\cdots g_p^{i_p})=\nu(g_1^{j_1}\cdots g_p^{j_p})$
iff $(i_1-j_1,\dots,i_p-j_p)\in H$ (cf. the proof of Theorem~\ref{th:basis}), in other words, $(i_1,\dots,i_p)\preceq (j_1,\dots,j_p)\preceq (i_1,\dots,i_p)$ (slightly abusing the notations we identify  a monomial $g_1^{i_1}\cdots g_p^{i_p}$ with the vector $(i_1,\dots,i_p)$). Define a linear ordering $\vartriangleleft$ on monomials in $g_1,\dots,g_p$ as follows: $g_1^{i_1}\cdots g_p^{i_p} \vartriangleleft g_1^{j_1}\cdots g_p^{j_p}$ iff either $g_1^{i_1}\cdots g_p^{i_p} \prec g_1^{j_1}\cdots g_p^{j_p}$ or $\nu(g_1^{i_1}\cdots g_p^{i_p})=\nu(g_1^{j_1}\cdots g_p^{j_p})$ and the vector $(i_1,\dots,i_p)$ is less than $(j_1,\dots,j_p)$ in deglex (again cf. the proof of Theorem~\ref{th:basis}).

The algorithm constructs a Gr\"obner basis of $I$ with respect to $\vartriangleleft$. Denote by $G\subset \ZZ_{\ge 0}^p$ the complement to the monomial ideal of leading monomials of the Gr\"obner basis. Then $G$ is a finite (disjoint) union of sets of the form 
$$u+\{(u_1,\dots,u_p)\, :\, u_i \vdash_i 0\}$$
\noindent where each $\vdash_i, 1\le i\le p$ is either $=$ or $\ge$.

Consider a plane $H_0\subset \RR^p$ parallel to $H$ which has a common point with $\ZZ_{\ge 0}^p$. Then $H_0\cap G\neq \emptyset$. Indeed, otherwise for any point $g_1^{i_1}\cdots g_p^{i_p}\in H_0$ we get (taking into account properties of Gr\"obner bases) that 
$$g_1^{i_1}\cdots g_p^{i_p}=\sum_L \alpha_L g^L, \alpha_L\in \kk^\times, \nu(g^L)\prec \nu(g_1^{i_1}\cdots g_p^{i_p})\, {\text{for each}}\, L,$$ 
\noindent which contradicts to the subadditivity of the valuation.

First, assume that for every $H_0$ parallel to $H$ such that $H_0\cap \ZZ_{\ge 0}^p \neq \emptyset$ it holds $|G\cap H_0|=1$. Then we fall in the conditions of Theorem~\ref{quotient}, and thus $\nu(g_1),\dots,\nu(g_p)$ constitute a family of generators of $\nu(\kk[y_1,\dots,y_m]\setminus \{0\})=\ZZ_{\ge 0}^p/H_{\ZZ}$.

Now on the contrary, assume that $|G\cap H_0|\ge 2$ for some $H_0$ parallel to $H$. In this case the algorithm can find a pair of different monomials $g_1^{i_1}\cdots g_p^{i_p}, g_1^{j_1}\cdots g_p^{j_p} \in G\cap H_0$ for some $H_0$ invoking integer linear programming. If the algorithm fails, it means that $|G\cap H_0|=1$ for all $H_0$ (see the previous case). Since the valuation $\nu$ is injective, there exists $\alpha \in \kk^\times$ such that
$$c_0:= \nu(g_1^{i_1}\cdots g_p^{i_p} -\alpha g_1^{j_1}\cdots g_p^{j_p})\prec \nu(g_1^{i_1}\cdots g_p^{i_p})=\nu(g_1^{j_1}\cdots g_p^{j_p}).$$
\noindent If $\ZZ_{\ge 0}^p\ni c_0\notin \ZZ_{\ge 0}\{\nu(g_1),\dots,\nu(g_p)\}$ then $c:=c_0$ meets the requirements of the Proposition.

Otherwise, if $c_0 \in \ZZ_{\ge 0}\{\nu(g_1),\dots,\nu(g_p)\}$ there exists a monomial $g_1^{l_1}\cdots g_p^{l_p}$ such that $\nu(g_1^{l_1}\cdots g_p^{l_p})=c_0$. Again due to the injectivity there exists $\beta \in \kk^\times$ for which it holds 
$$\nu(g_1^{i_1}\cdots g_p^{i_p} -\alpha g_1^{j_1}\cdots g_p^{j_p}-\beta g_1^{l_1}\cdots g_p^{l_p})\prec c_0.$$
\noindent Observe that $g_1^{i_1}\cdots g_p^{i_p} -\alpha g_1^{j_1}\cdots g_p^{j_p}-\beta g_1^{l_1}\cdots g_p^{l_p}\neq 0$, because otherwise this would contradict to that $g_1^{i_1}\cdots g_p^{i_p}, g_1^{j_1}\cdots g_p^{j_p} \in G$. Continuing in a similar way, the algorithm eventually arrives at a required element $c\in C \setminus \ZZ_{\ge 0}\{\nu(g_1),\dots,\nu(g_p)\}$ since $C$ is well-ordered.
\end{proof}

\begin{remark}
 In the proof of the latter Proposition we used $f_1,\dots,f_m$ only to be able to compute $\nu(g)$ for $g\in \kk[y_1,\dots,y_m]$. In fact, one could stick with an arbitrary computable injective valuation.   
\end{remark}

\subsection{The space of injective valuations on a domain}

For a $\kk$-domain $A$ we consider the space $V:=V(A)$ of all injective valuations $\nu:A\setminus \{0\} \to \RR_{\ge 0}$. Given a basis $\bf B$ of $A$ endowed with a linear order $\prec$ one can consider a set $V_{{\bf B}, \prec}$ of mappings $\nu:{\bf B} \to \RR_{\ge 0}$ such that for any $b,b_0\in {\bf B}$ a relation $b\prec b_0$ implies $\nu(b)<\nu(b_0)$, and for any $b_1,b_2\in {\bf B}$ for which 
$$b_1b_2=\alpha_{b_0}b_0+\sum_{b\in \bf B} \alpha_b b,\, \alpha_{b_0}, \alpha_b \in \kk,\, b\prec b_0$$
\noindent it holds $\nu(b_0)=\nu(b_1)+\nu(b_2)$. Then $\nu$ induces an injective valuation on $A$ with an adapted basis $\bf B$: namely, for any $a=\alpha_{b_0}b_0+\sum_{b\in \bf B} \alpha_b b,\, b\prec b_0$ we define $\nu(a):=\nu(b_0)$. One can define a topology on $V$ with open basic sets $V_{{\bf B}, \prec}$ for all ${\bf B}, \prec$.
\vspace{1mm}

{\bf Question}. Is $V$ connected?
\vspace{1mm}

Recall (see the proof of Theorem~\ref{inverse}) that for any valuation $\nu:A\setminus \{0\} \twoheadrightarrow C$ onto a well-ordered finitely-generated monoid $C$ one can find generators $X_1,\dots,X_m$ of $A$ and elements $w_1,\dots,w_m \in \RR_{\ge 0}[\varepsilon]$ such that $\nu(X_1),\dots,\nu(X_m)$ generate $C$, and the order on monomials $i_1\nu(X_1)+\cdots +i_m\nu(X_m)$ is determined by $w_1i_1+\cdots +w_mi_m$. Choosing a basis $\bf B$ among monomials of the form $X_1^{i_1}\cdots X_m^{i_m}$ and defining $\nu(X_1^{i_1}\cdots X_m^{i_m}):=w_1i_1+\cdots +w_mi_m$, one obtains that $\nu \in V$, provided that $w_1,\dots,w_m\in \RR_{\ge 0}$.

\subsection{Injective well-ordered valuations of 2-dimensional algebras}

Let $A:=\kk[x,y,z]/(f)$ be a 2-dimensional algebra where $f\in \kk[x,y,z]$. Consider a valuation $\nu : A\setminus \{0\} \to \ZZ_{\ge 0}^2$ studied in Theorem~\ref{quotient}. Then $\nu$ is induced by an edge $e$ of the Newton polytope $N(f)\subset \RR_{\ge 0}^3$ and a 2-dimensional plane $Q\subset \RR^3$ containing $e$. Note that in the notations of Theorem~\ref{quotient} $H$ is the line passing through $e$, and $Q\in Trop(f)$.  Moreover, the proof of Theorem~\ref{quotient} and Remark~\ref{finite-condition}
imply that if $\nu$ is injective then the endpoints of $e$ (up to a permutation of the coordinates $x,y,z$) are $(p,0,0), (0,q,r)$ where $p,q,r \in \ZZ_{\ge 0}$ have no common divisor.

\begin{lemma}\label{vertex_common}
If two edges $e_1, e_2$ of the Newton polytope $N(f)\subset \RR_{\ge 0}^3$ induce injective valuations then $e_1, e_2$ have a common vertex located on a coordinate line.    
\end{lemma}

\begin{proof}We say that a point $v\in N(f)$ belongs to a {\it roof} of $N(f)$ if on a ray emanating from the origin $(0,0,0)$ and passing through $v$, the latter is the last point from $N(f)$ on the ray. Then $e_1, e_2$ lie on the roof. 

Consider a 2-dimensional plane $Q_1\subset \RR^3$ which contains $e_1$ and supports $N(f)$. The projection of the roof to $Q_1$ by means of the rays contains the projections of $e_1$ and of $e_2$. If $e_1, e_2$ had no common vertex located on a coordinate line then the projections of $e_1, e_2$ would have a common point being internal in  the projection of either $e_1$ or $e_2$. The obtained contradiction completes the proof. 
\end{proof} 

\begin{corollary}\label{edges_injective} All the edges of the Newton polytope $N(f)$ inducing injective valuations either

i) have a common vertex located on a coordinate line or

ii) form a triangle with its vertices on the coordinate lines, and in this case the roof of $N(f)$ consists of this triangle.
    \end{corollary}

\begin{remark}\label{cusp_JH}
i) Observe that in case i) of Corollary~\ref{edges_injective} when all the edges have a common vertex $(p,0,0)$, a common adapted basis of all the injective valuations induced by the edges, is $\{x^iy^jz^k\, :\, 0\le i<p, 0\le j,k<\infty\}$ (cf. the proof of Theorem~\ref{quotient},  Remark~\ref{finite-condition} and Theorem~\ref{th:basis}). 

ii) One can verify that in case ii) of Corollary~\ref{edges_injective} three injective valuations do not
possess a common adapted basis.

For example, let $f:= z+x^2+y^3$. Then $A:=\kk[x,y,z]/(f)\simeq \kk[x,y]$. Denote by $\nu_x$ the injective valuation induced by the edge $(z,y^3)$ with respect to lex ordering in which $y>x$, and an adapted basis $\{y^ix^j\, :\, 0\le i,j<\infty\}$. Denote by $\nu_y$ the injective valuation induced by the edge $(z,x^2)$ with respect to lex ordering in which $x>y$, and an adapted basis $\{x^iy^j\, :\, 0\le i,j<\infty\}$. Finally, denote by $\nu_z$ the injective valuation induced by the edge $(y^3,x^2)$ with respect to lex ordering in which $x,y>z$, and $\nu_z(x)=(3,0), \nu_z(y)=(2,0), \nu_z(z)=(0,1)$. An adapted basis of $\nu_z$ is $\{y^iz^j, xy^iz^j\, :\, 0\le i,j<\infty\}$ (cf. Examples~\ref{skew}, ~\ref{cusp}).

One can compute all three JH-bijections 
(see Theorem~\ref{th:generalized JH}):

$${\bf K}_{\nu_z,\nu_y}(2j,i)=(2i,j), {\bf K}_{\nu_z,\nu_y}(2j+1,i)=(2i+3,j);$$

$${\bf K}_{\nu_z,\nu_x}(3j,i)=(3i,j), {\bf K}_{\nu_z,\nu_x}(3j+1,i)=(3i+2,j), {\bf K}_{\nu_z,\nu_x}(3j+2,i)=(3i+4,j);$$

$${\bf K}_{\nu_x,\nu_y}(j,i)=(i,j).$$
\end{remark}

One can generalize Corollary~\ref{edges_injective} (in one direction) to hypersurfaces of arbitrary dimensions.

\begin{remark}
Let $A:=\kk[x_1,\dots,x_n]/(f)$ where $f\in \kk[x_1,\dots,x_n]$. An edge $e$ of the Newton polytope $N(f)\subset \RR_{\ge 0}^n$ induces an injective valuation $\nu$ on $A\setminus \{0\}$ iff the endpoints of $e$ are $v=(v_1,\dots,v_n), u=(u_1,\dots,u_n)\in \ZZ_{\ge 0}^n$ such that $\min\{v_i,u_i\}=0, 1\le i\le n$, and the integers $\max\{v_i,u_i\}, 1\le i\le n$ have no common divisor (cf. Theorem~\ref{quotient} and Remark~\ref{finite-condition}). Denote by $H\subset \RR^n$ the line containing $e$.

Denote by $\prec$ a linear ordering of the valuation cone (being a subset of $\ZZ^{n-1}$) of $\nu$. Following the proof of Theorem~\ref{th:basis}~ii) one can extend $\prec$ to a linear ordering $q$ on $\ZZ_{\ge 0}^n$ in two different ways according to an ordering (direction) in $H$. Then according to one of these two choices of $q$ either $u$ or $v$ becomes a leading monomial of $f$ with respect to $q$. Any of these two choices of $q$ provides an adapted basis of $A$ with respect to $\nu$ being a complement of a principal monomial ideal in $\kk[x_1,\dots,x_n]$ (see Theorem~\ref{th:basis}).

If edges $e_1,\dots, e_s$ of $N(f)$ induce injective valuations $\nu_1,\dots,\nu_s$, respectively, of $A$ and have a common vertex in $N(f)$ then $\nu_1,\dots,\nu_s$ possess
a common adapted basis being a complement of a principal monomial ideal in $\kk[x_1,\dots,x_n]$. 
\end{remark}

\begin{example}\label{3-dimensional_valuation}
Now we give an example of a pair of injective valuations on a 3-dimensional algebra $A:=\kk[x,y,z,t]/(f:=x^2+y^3+z^5+t^7)$ to $\ZZ_{\ge 0}^3$ endowed with the lexicographical ordering. Each pair of monomials of $f$ provides an injective valuation of $A\setminus \{0\}$ (see Theorem~\ref{quotient} and Proposition~\ref{curve_valuation}). In particular, a pair of monomials $x^2, y^3$ provides an injective valuation $\nu_1$ for which it holds
$$\nu_1(x)=(3,0,0), \nu_1(y)=(2,0,0), \nu_1(z)=(0,1,0), \nu_1(t)=(0,0,1).$$
\noindent In its turn, a pair of monomials $z^5, t^7$ provides an injective valuation $\nu_2$ for which it holds
$$\nu_2(z)=(7,0,0), \nu_2(t)=(5,0,0), \nu_2(x)=(0,1,0), \nu_2(y)=(0,0,1).$$
\noindent Denote $w:= x^2+y^3$. Then $\nu_1, \nu_2$ have a common adapted basis 
$$\{y^it^jw^lx^kz^m\ :\ i,j,l\ge 0, 0\le k\le 1, 0\le m\le 4\},$$
\noindent and $\nu_1(w)=\nu_1(-z^5-t^7)=(0,5,0),\ \nu_2(w)=(0,2,0)$.

One can generalize this construction to a polynomial of the form $f:=\sum_{1\le i\le n} x_i^{q_i}$ where $q_i, 1\le i\le n$ are pairwise relatively prime. Each pair of monomials of $f$ provides an injective valuation of $\kk[x_1\dots,x_n]/(f)$ into $\ZZ_{\ge 0}^{n-1}$.
\end{example}

\subsection{Enumerating injective well-ordered valuations of a hypersurface of a prime degree at a main variable}

\begin{remark}\label{long}
Let a polynomial $f=y^d+f_1\in \kk[y,x_1,\dots,x_n]$ be normalized with respect to $y$, i.e. $deg_y (f_1)<d$ (one can reduce to this situation invoking Noether normalization). Denote $A:=\kk[y,x_1,\dots,x_n]/(f)$. We say that an edge of Newton polytope $N_f\subset \RR^{n+1}$ is {\it long} if its endpoints are $(d,0,\dots,0),\, (0,i_1,\dots,i_n)$ and $GCD(d,i_1,\dots,i_n)=1$. 

Then the domain $A\setminus \{0\}$ admits an injective well-ordered valuation into $\QQ_{\ge 0}^n$ with an adapted basis 
$$\{y^kx_1^{j_1}\cdots x_n^{j_n}\, :\, 0\le k<d, 0\le j_1,\dots,j_n<\infty\}$$
\noindent and 
$$\nu(x_l)=e_l, 1\le l\le n,\, \nu(y)=\frac{i_1e_1+\cdots +i_ne_n}{d}$$
\noindent where $e_l=(0,\dots,0,1,0,\dots,0) \in \ZZ^n, 1\le l\le n$ is an ort vector (cf. Proposition~\ref{partition}).

Observe that for any valuation (not necessary injective or well-ordered) on $A\setminus \{0\}$ for which $\nu(x_l)=e_l, 1\le l\le n$ there exists an edge of $N_f$ with endpoints $(d_1,j_1,\dots,j_n)$,  $(d_2,k_1,\dots,k_n)$ such that
\begin{eqnarray}\label{50}
 \nu(y)=\frac{(k_1-j_1)e_1+\cdots +(k_n-j_n)e_n}{d_1-d_2}.   
\end{eqnarray}
\end{remark}

As a more general setting than in Remark~\ref{long} we assume that a domain $A$ is a finite integral extension of $\kk[x_1,\dots,x_n]$ of a rank $d$. We study injective well-ordered valuations $\nu$ on $A\setminus \{0\}$ into $\QQ_{\ge 0}^n$ for which $\nu(x_l)=e_l, 1\le l\le n$.

The following construction is valid for a valuation $\nu$ not necessary injective or well-ordered. Denote by $G(\nu)\subset \QQ^n/\ZZ^n$ the image $\nu(A\setminus \{0\})/\ZZ^n$. Note that $G(\nu)$ is an abelian semigroup. Moreover, for every element $a\in A\setminus \{0\}$ there exists a polynomial $h\in \kk[z,x_1,\dots,x_n]$ such that $h(a)=0$ and $deg_z (h)\le d$. As in Remark~\ref{long} there is an edge of Newton polytope $N_{h}\subset \RR^{n+1}$ with endpoints $(d_1,j_1,\dots,j_n), (d_2,k_1,\dots,k_n)$ such that
$$\nu(a)=\frac{(k_1-j_1)e_1+\cdots +(k_n-j_n)e_n}{d_1-d_2}$$
\noindent (see \eqref{50}). Hence $(d_1-d_2)\nu(a)$ is the unit element in $G(\nu)$, thus $G(\nu)$ is a group. Moreover,
$$G(\nu)\subset \frac{\ZZ^n}{LCM\{1,\dots, d\}}/\ZZ^n,$$
\noindent in particular, $G(\nu)$ is finite. We call $G(\nu)$ the {\it group of the valuation}.

\begin{lemma}\label{group_valuation}
Let a domain $A$ be a free  $\kk[x_1,\dots,x_n]$-module of a rank $d$. Let $\nu$ be a  well-ordered valuation of $A\setminus \{0\}$ being an extension into $\QQ_{\ge 0}^n$ of a 
valuation on $\kk[x_1,\dots,x_n]\setminus \{0\}$ with $\nu(x_l)=e_l, 1\le l\le n$. Then for the size $s$ of the group of the valuation $G(\nu)$ holds

i)  $s\le d$;

ii) if $\nu$ 
is injective and archimedian then $s=d$.
\end{lemma}

\begin{proof}
i) Let $b_1,\dots,b_d\in A$ be a free $\kk[x_1,\dots,x_n]$-basis of $A$. Pick $t_1,\dots,t_s\in A\setminus \{0\}$ such that $\nu(t_l)-\nu(t_j)\notin \ZZ^n$ for $1\le l\neq j\le s$. Express $t_j=\sum_{1\le i\le d} h_{j,i}b_i$ for appropriate polynomials $h_{j,i}\in \kk[x_1,\dots,x_n]$. Denote 
$$M_N:=\{x_1^{l_1}\cdots x_n^{l_n}\, :\, 0\le l_1+\cdots +l_n\le N\},\, W_N:= t_1M_N+\cdots +t_sM_N$$
\noindent for an integer $N$. Then
$\dim_{\kk} W_N = s|M_N|$ due to the valuation property. On the other hand,
$$W_N\subset b_1M_{N+c}+\cdots +b_dM_{N+c}$$
\noindent for a suitable constant $c\in \ZZ_{\ge 0}$. Therefore, considering sufficiently big $N$, we obtain that $s\le d$. 
\vspace{2mm}

ii) Denote $$V_N:= \{(i_1,\dots,i_n)\in \ZZ_{\ge 0}^n\, :\, \nu(i_1,\dots,i_n)\le N\}$$
\noindent (we identify the archimedian valuation with defining it linear form). Then $|V_N|\sim c_0\cdot N^n$ for an appropriate $0<c_0\in \RR$. Observe that $\nu$ on $\QQ_{\ge 0}^n$ is defined by the same linear form as $\nu$ is.

Denote by 
$g_1,\dots,g_s$ 
the unique representatives of the elements of $G(\nu)$ in the cube $[0,1)^n$. Then 
$$\nu(b_1\cdot V_N+\cdots+b_d\cdot V_N) \subset g_1\cdot V_{N+c} \bigsqcup \cdots \bigsqcup g_s\cdot V_{N+c}$$ 
\noindent for a suitable constant $c$. Since $\nu$ is injective this implies that $s\ge d$ taking into account that $\dim (b_1\cdot V_N+\dots+b_d\cdot V_N)=d|V_N|$.
\end{proof}

\begin{remark}\label{finite_dimension}
One can literally extend Lemma~\ref{group_valuation} to a domain $A\supset \kk[x_1,\dots,x_n]$ such that $\kk(x_1,\dots,x_n)$-dimension of $A\otimes _{\kk[x_1,\dots,x_n]} \kk(x_1,\dots,x_n)$ equals $d$.     
\end{remark}

It would be interesting to clarify whether Lemma~\ref{group_valuation} and Remark~\ref{finite_dimension} are true for not necessary archimedian valuations.

\begin{corollary}\label{prime}
Under the conditions of Lemma~\ref{group_valuation} or Remark~\ref{finite_dimension} in case of a square-free $d$ the group $G(\nu)$ is cyclic of size $d$ and every its generating element has a form $\frac{i_1e_1+\cdots +i_ne_n}{d}$ where $GCD(d,i_1,\dots,i_n)=1$.    
\end{corollary}

Now let $A:=\kk[y,x_1,\dots,x_n]/(f)$ be as in Remark~\ref{long} and $d$ be a prime. Our goal is to design an algorithm which enumerates all injective well-ordered archimedian valuations on $A\setminus \{0\}$ into $\QQ_{\ge 0}^n$ (with $\nu(x_l)=e_l, 1\le l\le n$, cf. Corollary~\ref{prime}). 

The algorithm produces a finite tree $T$ by recursion. Some leaves of $T$ correspond to injective well-ordered valuations on $A\setminus \{0\}$. As a base of recursion a root of $T$ is produced.

As a recursive hypothesis assume that at a vertex $v$ of a depth $s$ of $T$ a constructible set $U_v\subset \kk^s$ and a set of monomials $m_1,\dots, m_s \in \ZZ_{\ge 0}^n$ are produced (we identify monomials in the variables $x_1,\dots,x_n$ with $\ZZ_{\ge 0}^n$). We suppose that $m_s$ does not belong to the monomial ideal generated by $m_1,\dots,m_{s-1}$. In addition, the algorithm produces a polynomial $f_v=y^d+f_{v,1}\in \kk[y,x_1,\dots,x_n,z_1,\dots,z_s]$ (where $deg_y (f_{v,1})< d$) such that for any point $(\alpha_1,\dots, \alpha_s)\in U_v$ it holds
\begin{eqnarray}\label{41}
f_v(y-\alpha_1 m_1 - \cdots - \alpha_s m_s, x_1,\dots, x_n, \alpha_1,\dots, \alpha_s)=0    
\end{eqnarray}
and that Newton polytopes $N_{f_v} \subset \RR^{n+1}$ are the same for all the points $(\alpha_1,\dots, \alpha_s)\in U_v$.

Now we proceed to the description of the recursive step of the algorithm. First assume that Newton polytope $N_{f_v}$ has a long edge with endpoints $(d,0,\dots,0), (0,i_1,\dots,i_n)$, denote by $g(z_1,\dots,z_n)\in \kk[z_1,\dots, z_n]$ the coefficient of $f_v$ at the monomial $x_1^{i_1}\cdots x_n^{i_n}$. The algorithm verifies (invoking linear programming) whether there exists a linear archimedian ordering $\succ$ (compatible with addition) on $\QQ_{\ge 0}^n$ such that $m_1\succ \cdots \succ m_s \succ \frac{i_1e_1+\dots +i_ne_n}{d}$ (otherwise, the algorithm ignores the long edge under consideration). If such an ordering does exist then Remark~\ref{long} provides an injective well-ordered valuation $\nu$ on $A$ such that 
$$\nu (y-\alpha_1 m_1 -\cdots - \alpha_s m_s)=\frac{i_1e_1+\dots +i_ne_n}{d}.$$
\noindent Thus, as an adapted basis of $A$ with respect to $\nu$ one can take 
$$\{(y-\alpha_1 m_1 -\cdots - \alpha_s m_s)^k \cdot x_1^{j_1}\cdots x_n^{j_n}\},\, 0\le k<d, 0\le j_1,\dots, j_n <\infty.$$
\noindent The algorithm produces a vertex being a son of $v$ and a leaf in $T$ which corresponds to $\nu$.

Now we consider a not long edge of $N_{f_v}$ with endpoints $(d_1, j_1,\dots, j_n),\, (d_2, k_1,\dots, k_n)$ (obviously, $0\le d_1\neq d_2\le d$). Denote
$$m_{s+1}:= \frac{(k_1-j_1)e_1+\cdots +(k_n-j_n)e_n}{d_1-d_2}$$
\noindent (cf. \eqref{50}, \eqref{41}), provided that $m_{s+1} \in \ZZ_{\ge 0}^n$. Observe that if $m_{s+1} \notin \ZZ_{\ge 0}^n$ then for no element $a\in A$ it holds $\nu(a)=m_{s+1}$ for an injective well-ordered archimedian valuation $\nu$ because of Corollary~\ref{prime}.

The algorithm verifies (invoking linear programming) whether there exists a linear archimedian ordering $\succ$ on $\ZZ_{\ge 0}^n$ such that $m_1\succ \cdots \succ m_s \succ m_{s+1}$ (otherwise, the algorithm ignores the edge of $N_{f_v}$ under consideration). The latter is necessary because the algorithm looks for an injective well-ordered valuation $\nu$ such that $\nu(y-\alpha_1 m_1 - \cdots - \alpha_s m_s)=\nu(m_{s+1})(=m_{s+1})$. Note that in particular, the existence of a suitable linear ordering $\succ$ implies that $m_{s+1}$ does not belong to the monomial ideal generated by $m_1,\dots,m_s$.

The algorithm calculates a polynomial $g\in \kk[y,x_1,\dots, x_n,z_1,\dots,z_s,z_{s+1}]$ such that
$$g(y-\alpha_1 m_1 -\cdots - \alpha_s m_s -z_{s+1} m_{s+1},x_1,\dots, x_n,\alpha_1,\dots,\alpha_s,z_{s+1})=0$$
\noindent for any point $(\alpha_1,\dots, \alpha_s)\in U_v$ (cf. \eqref{41}). Note that it still holds $g=y^d+g_1$ where $deg_y (g_1)< d$. For different values of $z_{s+1}$ in $\kk$ there is a finite number of possible shapes of Newton polytopes $N_g\subset \RR^{n+1}$ (which are determined by their vertices). For each fixed shape the algorithm produces a vertex $w$ being a son of $v$ in $T$ together with a constructible set $U_w\subset \kk^{s+1}$ assuring the fixed shape. We put a polynomial $f_w:= g$. This completes the description of the recursive step of the algorithm. \vspace{2mm}

Observe that the tree $T$ is finite since along every its path the monomial ideal generated by $m_1,\dots, m_s$ strictly increases and therefore, the path is finite due to Hilbert's Idealbasissatz (also we make use of K\"onig's Lemma). Summarizing, we have established the following proposition.

\begin{proposition}
There is an algorithm which for a polynomial $f=y^d+f_1\in \kk[y,x_1,\dots,x_n]$ with a prime $d$ where $deg_y (f_1)< d$, enumerates all injective well-ordered archimedian valuations on $(\kk[y,x_1,\dots,x_n]/(f))\setminus \{0\}$.    
\end{proposition}

It would be interesting to generalize the latter proposition to arbitrary affine algebras. The next example demonstrates that it is not possible to generalize it directly even to domains being free $\kk[x_1,\dots,x_n]$-modules of composite ranks.

\begin{example}\label{2x2}
A domain $A=\kk[x^{1/2}, y^{1/2}]$ is $4$-dimensional free $\kk[x,y]$-module. Then $\nu(x^{1/2})=e_1/2, \nu(y^{1/2})=e_2/2$ and the group $G(\nu)$ is isomorphic to $\ZZ_2 \times \ZZ_2$. Note that an element $z:=x^{1/2}+y^{1/2}$ has a minimal polynomial of degree 4, namely $z^4-2(x+y)z^2+x^2+y^2-4xy=0$ whose Newton polytope has no long edge (cf. Remark~\ref{long}).     
\end{example}

\begin{remark}
Let a domain $A$ be a 
$\kk[x_1,\dots,x_n]$-module, 
admitting an injective well-ordered valuation 
$\nu$ satisfying the conditions of Remark~\ref{finite_dimension} such that $d$ is square-free. Due to Corollary~\ref{prime} the group $G(\nu)$ is cyclic of the size $d$. Pick $y\in A$ such that $\nu(y)/\ZZ^n$ is a generator of $G(\nu)$. Denote by $f\in \kk[x_1,\dots,x_n,y]$ the minimal polynomial of $y$ over $\kk[x_1,\dots,x_n]$. Then $deg_y f=d$ since Newton polytope of $f$ contains an edge with a denominator equal to $d$, and $f=qy^d+\cdots, q\in \kk[x_1,\dots,x_n]$.
The domain $A_0:=\kk[x_1,\dots,x_n,qy]$ is a free $\kk[x_1,\dots,x_n]$-module with a basis $1, qy,\dots, (qy)^{d-1}$, and Newton polytope of the minimal polynomial $q^{d-1}f$ of $qy$ contains a long edge. According to Remark~\ref{long} this long edge provides on $A_0\setminus \{0\}$ the valuation coinciding with the restriction of $\nu$.

Furthermore, this restriction is extended uniquely to $A\setminus \{0\}$. Namely, for any $a\in A\setminus \{0\}$ there exists $p\in \kk[x_1,\dots,x_n]$ such that $pa\in A_0\setminus \{0\}$, therefore $\nu(a)=\nu(pa)-\nu(p)$.
\end{remark}

\subsection{Convexity of the extended Jordan-H\"older bijection for   valuations in an archimedian monoid}

Let $\nu:A\setminus \{0\} \to \ZZ_{\ge 0}^n$ be an injective  valuation in a monoid with archimedian linear order (cf. Remark~\ref{monoid_archimedian}). 
We say that $\nu$ is {\bf finitary} if the ordering $\prec$ on vectors $v=(v_1,\dots,v_n)\in \ZZ_{\ge 0}^n$ is determined by a linear function $\alpha(v):=\alpha_1 v_1+\cdots+\alpha_n v_n$ where positive reals $\alpha_1,\dots,\alpha_n$ are $\QQ$-linearly independent. Then the ordering $\prec$ is isomorphic to $\ZZ_{\ge 0}$, i.e. the monoid $\ZZ_{\ge 0}^n$ is archimedian. 
For instance, $lex$ is not finitary. 
In particular, for any $c\in \nu (A\setminus \{0\})$ the $\kk$-vector space $A_{\le c}:= \{a\in A\setminus \{0\}\, :\, \nu(a)\preceq c\} \cup \{0\}$ is finite-dimensional.

Let $\nu_0 : A\setminus \{0\} \to \ZZ_{\ge 0}^n$ be another injective valuation (not necessary archimedian). Denote the valuation cones $C:=\nu(A\setminus \{0\}),\, C_0:=\nu_0(A\setminus \{0\})$. Consider a convex cone $C_0^{(\QQ)}:= C_0 \otimes \QQ_{\ge 0}$ and denote
$$S_0:= C_0^{(\QQ)} \cap \{(v_1,\dots,v_n)\in \QQ_{\ge 0}^n\, :\, v_1+\cdots +v_n=1\}.$$

The following map 
$${\bf K}(c_0):= \min_{\prec} \{\nu (\nu_0^{-1}(c_0))\} \in C$$
\noindent is a bijection ${\bf K} :C_0 \widetilde \to C$ called generalized JH-bijection (see Theorem~\ref{th:generalized JH}).
Denote a function ${\overline {\bf K}} := \alpha \circ {\bf K} : C_0 \to \RR_{\ge 0}$. One can define ${\overline {\bf K}}$ on $S_0$ (so, on rational points) as follows. For a point $u:=(u_1,\dots,u_n)\in S_0$ we have ${\overline {\bf K}}((p+q)u) \le {\overline {\bf K}}(pu) + {\overline {\bf K}} (qu)$, provided that $pu,qu \in C_0$, since the subadditivity ${\overline {\bf K}}(w_1+w_2) \le {\overline {\bf K}}(w_1)+{\overline {\bf K}}(w_2)$ holds for any elements $w_1,w_2 \in C_0$. Therefore, due to Fekete's subadditivity lemma \cite{S} there exists the limit
$${\underline {\bf K}}(u):= \lim_{p\to \infty,\, pu\in C_0}  
\frac{{\overline {\bf K}}(pu)}{p}.$$

\begin{proposition}
Function 
${\underline {\bf K}}:S_0 \to \RR_{\ge 0}$ is convex.
\end{proposition}

\begin{proof}  Consider a convex combination $w=\sum_i \lambda_i w_i$ of points $w,w_i \in S_0$ where $0< \lambda_i \in \QQ,\, \sum_i \lambda_i =1$ for all $i$. Then for a suitable $0<s\in \ZZ$ it holds $sw,sw_i \in C_0$ for all $i$. Denote by $q$ the common denominator of all $\lambda_i$. Hence ${\overline {\bf K}}(pqsw)\le \sum_i {\overline {\bf K}}(pqs\lambda_iw_i)$ for any $p\in \ZZ_{\ge 0}$ because of the subadditivity of ${\overline {\bf K}}$. Dividing the both sides of the latter inequality by $pqs$ and tending $p$ to infinity, we conclude that ${\underline {\bf K}}(w)\le \sum_i \lambda_i {\underline {\bf K}}(w_i)$, which completes the proof.

The proposition is proved. 
\end{proof} 

\begin{remark}
We have taken an ordering $\prec$ to be archimedian since otherwise one can't assure an inequality after tending to a limit. For example, in $lex$ ordering each element of a sequence $(1-1/p,1) \in \QQ_{\ge 0}^2$ is less than $(1,0)$, while their limit against $p$ is not. 
\end{remark}

\begin{proposition}
One can extend ${\underline {\bf K}}$ (from rational points) to real points in the interior $int(S_0\otimes \RR_{\ge 0})$ being a (continuous) convex function.
\end{proposition}

\begin{proof}  ${\underline {\bf K}}$ is locally Lipschitz in $int(S_0)$ (cf. e.g., \cite{NK}), hence it can be (uniquely) extended to a continuous function (moreover, locally Lipschitz) on $int(S_0 \otimes \RR_{\ge 0})$ which is also convex. 

The proposition is proved. \end{proof}

\vspace{2mm}

\subsection{Examples of Jordan-H\"older bijections}
\label{subsec:Examples Jordan-Holder bijections}

\begin{definition}\label{over_JH}
Let $\nu_1, \nu_2: A\setminus \{0\} \to \ZZ_{\ge 0}^l$ be injective valuations of an algebra of $\dim A=l$. Assume that  $\nu_1, \nu_2$ have a common adapted monomial basis
$${\bf B}=\{y_1^{i_1}\cdots y_L^{i_L}\ |\ (i_1,\dots,i_L)\in P\},$$
\noindent where $P\subset \ZZ_{\ge 0}^L$ is a coideal semigroup, $y_1,\dots,y_L\in A\setminus \{0\}$. Then there are unique linear maps $\widehat{\nu_1}, \widehat{\nu_2} : \QQ^L \to \QQ^l$ such that $\nu_i(y_1^{i_1}\cdots y_L^{i_L})=\widehat{\nu_i}(i_1,\dots,i_L),\ i=1,2$. Hence $\widehat{\nu_1}, \widehat{\nu_2}$ are both surjective due to Proposition~\ref{rank_partial}. Therefore there are bijections
$$\widehat{\nu_i}: y^P \widetilde \to C_{\nu_i},\ i=1,2$$
\noindent and it holds 
$${\bf K}_{\nu_2,\nu_1} \circ \widehat{\nu_1} | _{y^P} = \widehat{\nu_2} | _{y^P}.$$
\end{definition}

\begin{example}
\label{ex:mod 2}
Let  $\varphi$ and $\psi$ are injective homomorphisms $\kk[z_1,z_2]\to \kk[t_1,t_2]$ given respectively by:
$\varphi(z_1)=t_1,~\varphi(z_2)=t_2,~\psi(z_1)=t_1,~\psi(z_2)=t_1^2+t_2$.
Clearly, $C_\varphi=C_\psi=\ZZ_{\ge 0}^2$.
One can easily see that the set 
$${\bf B}=\{b^\varepsilon_{\bf d}=z_1^\varepsilon z_2^{d_1}(z_2-z_1^2)^{d_2}\, |\, {\bf d}=(d_1,d_2)\in \ZZ_{\ge 0}^2,~\varepsilon\in \{0,1\}\}$$ 
is a basis of $\kk[z_1,z_2]$ adapted to both $\nu_\varphi$ and $\nu_\psi$ (i.e., is an JH-basis in $\kk[z_1,z_2]$) and
$$\nu_\varphi(b^\varepsilon_{\bf d})=(\varepsilon+2d_2,d_1),\nu_\psi(b^\varepsilon_{\bf d})=(\varepsilon+2d_1,d_2)$$
for all ${\bf d}=(d_1,d_2)\in \ZZ_{\ge 0}^2,~\varepsilon\in \{0,1\}$.
Therefore,
$${\bf K}_{\varphi,\psi}(a_1,a_2)=(2a_2+(a_1)_2,\left\lfloor\frac{a_1}{2}\right\rfloor)$$
for all $a_1,a_2\in \ZZ_{\ge 0}$.
where $(a)_2=a-2\left\lfloor\frac{a}{2}\right\rfloor$ is the parity of $a$. Moreover,  ${\bf K}_{\varphi,\psi}$ is an involution on $\ZZ_{\ge 0}^2$.

\end{example}

\begin{example}
\label{ex:free mod 2}  
Let  $\varphi$ and $\psi$ be  isomorphisms $\kk<z_1,z_2>\to \kk<t_1,t_2>$ given respectively by:
$\varphi(z_1)=t_1,~\varphi(z_2)=t_2,~\psi(z_1)=t_1,~\psi(z_2)=t_1^2+t_2$.
Clearly, $C_\varphi=C_\psi=F_2$, the free semigroup of rank $2$. We assume that $F_2$ is endowed with the linear ordering deglex in which $z_1\succ z_2$ and $t_1\succ t_2$, respectively (cf. Lemma~\ref{deglex}).

Denote $b_0:=z_2, b_1:=z_1z_2-z_2z_1, b_2:=z_1^2-z_2$ and by $P:=\langle b_0, b_1, b_2\rangle \subset \kk\langle z_1, z_2\rangle$ the semigroup generated by $b_0, b_1, b_2$. We claim that ${\bf B}:= P \sqcup P\cdot z_1$ is a basis of $\kk\langle z_1, z_2\rangle$  adapted to both valuations $\nu_{\varphi}, \nu_{\psi}$. 

Indeed, it holds that $\nu_{\varphi}(b_0)=t_2, \nu_{\varphi}(b_1)=t_1t_2, \nu_{\varphi}(b_2)=t_1^2, \nu_{\varphi}(z_1)=t_1$. Moreover, the semigroup $\nu_{\varphi}(P)=\langle t_1^2, t_1t_2, t_2\rangle$ is freely generated by $t_1^2, t_1t_2, t_2$, and it consists of all the words $t_1^{i_1} \cdot t_2^{j_1}\cdots t_1^{i_s} \in \langle t_1, t_2\rangle$ for which $i_s$ is even. Therefore $\nu_{\varphi}$ is injective on $\bf B$ and $\nu_{\varphi}({\bf B})=\langle t_1, t_2\rangle$. Similarly $\nu_{\psi}(b_0)=t_1^2, \nu_{\psi}(b_1)=t_1t_2, \nu_{\psi}(b_2)=t_2, \nu_{\psi}(z_1)=t_1$, also $\nu_{\psi}$ is injective on $\bf B$ and $\nu_{\psi}(P)=\nu_{\varphi}(P),\ \nu_{\psi}({\bf B})=\nu_{\varphi}({\bf B})$.

To justify that $\bf B$ is a basis of $\kk\langle z_1, z_2\rangle$ it suffices to verify that $\bf B$ spans $\kk\langle z_1, z_2\rangle$. We prove (by induction on $\prec$) that any monomial $M\in \langle z_1, z_2\rangle$ belongs to the linear span of $\bf B$. The base of induction for $M=z_1, z_2$ is clear. Let $M$ have the length greater than 1. If $M=z_2M_1$ then we apply the inductive hypothesis to the monomial $M_1$. If $M=z_1z_2M_2$ then $M=(z_1z_2-z_2z_1)M_2+z_2z_1M_2$, and we apply the inductive hypothesis to the monomials $M_2, z_2z_1M_2 \prec M$. Finally, if $M=z_1^2M_3$ then $M=(z_1^2-z_2)M_3+z_2M_3$, and we apply the inductive hypothesis to the monomials $M_3, z_2M_3 \prec M$.

Thus, ${\bf K}_{\psi, \varphi}(t_2)=t_1^2, {\bf K}_{\psi, \varphi}(t_1t_2)=t_1t_2, {\bf K}_{\psi, \varphi}(t_1^2)=t_2, {\bf K}_{\psi, \varphi}(t_1)=t_1$.

\end{example}

\begin{example}\label{2_noncommutative}
Let $\varphi, \psi: \kk\langle x_{11}, x_{12}, x_{21}, x_{22}\rangle \to \kk\langle t_{11}, t_{12}, t_{21}, t_{22}\rangle$ be homomorphisms defined by 
$$\varphi(x_{11})=t_{11}, \varphi(x_{12})=t_{12}, \varphi(x_{21})=t_{21}, \varphi(x_{22})=t_{22};$$ $$\psi(x_{11})=t_{11}t_{12}t_{21}, \psi(x_{12})=t_{11}t_{12}, \psi(x_{21})=t_{11}t_{21}, \psi(x_{22})=t_{11}+t_{22}.$$
 Then the adapted basis of the algebra $\kk\langle x_{11}, x_{12}, x_{21}, x_{22}\rangle$ for both valuations $\nu_{\varphi}, \nu_{\psi}$ is the free semigroup generated by $x_{11}, x_{12}, x_{21}, x_{22}$. The semigroup $C_{\varphi}$ is freely generated by $t_{11}, t_{12}, t_{21}, t_{22}$, the semigroup $C_{\psi}$ is freely generated by $t_{11}t_{12}t_{21}, t_{11}t_{12}, t_{11}t_{21}, t_{11}$. The orders $\succ$ on $C_{\varphi}, C_{\psi}$ are induced by deglex on a free semigroup  (cf. Lemma~\ref{deglex}) with $t_{11} \succ t_{12} \succ t_{21} \succ t_{22}$. JH-bijection ${\bf K}:={\bf K}_{\psi, \varphi}$ is given by
 $${\bf K}_{\psi, \varphi}(t_{11})=t_{11}t_{12}t_{21}, {\bf K}(t_{12})=t_{11}t_{12}, {\bf K}(t_{21})=t_{11}t_{21}, {\bf K}(t_{22})=t_{11}.$$
\end{example}

\begin{example}\label{2.13 non-commutative}
Denote by $\kk\langle t_1,t_2,t_3\rangle$ the free algebra endowed with a well-ordering $\prec$ on monomials defined as follows. If a monomial $m_1$ is shorter than a monomial $m_2$ then $m_1\prec m_2$. Otherwise, if their lengths coincide then $\prec$ is determined by lex with respect to $t_3\prec t_2\prec t_1$.

Consider homomorphisms 
$$\varphi_i: A:=\kk\langle x_1,x_2,x_3\rangle \to \kk\langle t_1,t_2,t_3\rangle, i=1,2;$$
$$\varphi_1(x_1)=t_1+t_3, \varphi_1(x_2)=t_2, \varphi_1(x_3)=t_1t_2;\, \varphi_2(x_1)=t_2, \varphi_2(x_2)=t_1+t_3, \varphi_2(x_3)=t_2t_3.$$
\noindent Denote $m:=x_1x_2-x_3$. We claim that a set ${\bf B}\subset A$ of monomials in $x_1,x_2,x_3,m$ without submonomials 
$x_1x_2$ constitutes a common adapted basis of $A$ with respect to valuations $\nu_{\varphi_1}, \nu_{\varphi_2}$.

First, $\bf B$ spans $A$ since in any monomial in $x_1,x_2,x_3$ one can replace each occurrence of submonomial $x_1x_2$ by $m+x_3$.

Now we observe that given a monomial $b\in \bf B$ in order to compute the leading monomial of $\varphi_1(b) \in \kk\langle t_1,t_2,t_3 \rangle$ one has to replace each occurrence of $x_1,x_2,x_3$ and of $m$ as follows:
$$x_1\to t_1,\, x_2\to t_2,\, x_3\to t_1t_2,\, m\to t_3t_2.$$
\noindent Note that these leading monomials are pairwise distinct for different monomials from $\bf B$ since each such monomial $T$ in $t_1,t_2,t_3$ can be uniquely represented in the following way. Between any pair of adjacent occurrences in $T$ of submonomials of the form either $t_1t_2$ or $t_3t_2$ the submonomial of $T$ has a form $t_2\dots t_2t_1\dots t_1$.
In other words, one can describe the set of leading monomials of $\varphi_1({\bf B})$ as the monoid $C_{\varphi_1}\subset \langle t_1,t_2, t_3\rangle$ generated by $t_1, t_2, t_3t_2$.

This implies that the elements of $\bf B$ are linearly independent, so $\bf B$ is a basis of $A$, in addition that $\varphi_1$ is a monomorphism, and $\bf B$ is an adapted basis with respect to the valuation $\nu_{\varphi_1}$.

In a similar manner, the leading monomial of $\varphi_2(b)$ is obtained by means of the following replacements:
$$x_1\to t_2,\, x_2\to t_1,\, x_3\to t_2t_3,\, m\to t_2t_1.$$
\noindent Again, these leading monomials are distinct for different elements $b\in \bf B$. Any leading monomial $T_2$ can be 
uniquely represented as follows. Between an adjacent occurrences of a pair of submonomials of the form either $t_2t_3$ or $t_2t_1$ the submonomial of $T_2$ coincides with $t_1\dots t_1t_2\dots t_2$. Hence $\varphi_2$ is also a monomorphism. Thus, $\bf B$ is a common adapted basis with respect to both valuations $\nu_{\varphi_1}, \nu_{\varphi_2}$. The set of leading monomials of $\varphi_2 ({\bf B})$ equals the monoid $C_{\varphi_2} \subset \langle t_1,t_2,t_3\rangle$ generated by $t_1,t_2,t_2t_3$.

Note that 
it holds $C_{\varphi_1} \neq C_{\varphi_2}$: for instance, $t_3t_2 \in C_{\varphi_1} \setminus C_{\varphi_2}$, while $t_2t_3 \in C_{\varphi_2} \setminus C_{\varphi_1}$. One obtains the JH-bijection ${\bf K}_{\nu_{\varphi_2}, \nu_{\varphi_1}}$ as follows. In a monomial $T\in C_{\varphi_1}$ (see the notations above) we replace each occurrence of $t_1t_2$ by $t_2t_3$, respectively, each occurrence of $t_3t_2$ by $t_2t_1$, in addition, in a submonomial of the form $t_2\dots t_2t_1\dots t_1$ between a pair of adjacent occurrences of either $t_1t_2$ or $t_3t_2$, we replace $t_2$ by $t_1$ and $t_1$ by $t_2$, thereby we get a submonomial $t_1\dots t_1t_2\dots t_2$. The resulting monomial is ${\bf K}_{\nu_{\varphi_2}, \nu_{\varphi_1}}(T)\in C_{\varphi_2}$. 
\end{example}

\begin{remark}\label{JH_homomophism}
i) Let $\nu : A\setminus \{0\} \to \ZZ_{\ge 0}^d$ be an injective valuation of an algebra $A$ with $\dim A=d$ (cf. Theorem~\ref{th:basis}), and
$${\bf B}:=\{y_1^{i_1}\cdots y_L^{i_L}\ |\ (i_1,\dots, i_L)\in P\subset \ZZ_{\ge 0}^L\}$$
\noindent be a monomial basis of $A$ adapted to $\nu$, where $P$ is a coideal of $\ZZ_{\ge 0}^L$ (see Definition~\ref{over_JH}). Then $\bf B$ is a finite union of (coordinate) subsemigroups of the form 
$${\bf B}_T:=\{y_{j_1}^{q_1}\cdots y_{j_t}^{q_t}\ |\ q_1,\dots, q_t\in \ZZ_{\ge 0}\}$$
\noindent for some subsets $T=\{j_1,\dots,j_t\} \subset \{1,\dots, L\}$, and of some their cosets of the form ${\bf B}_T\cdot y_{r_1}^{s_1}\cdots y_{r_{L-t}}^{s_{L-t}}$, where $\{r_1,\dots,r_{L-t}\}=\{1,\dots,L\}\setminus T,\ s_1,\dots, s_{L-t}\in \ZZ_{\ge 0}$. Then the cone
$$\RR_{\ge 0} \otimes \nu({\bf B})= \bigcup_T \RR_{\ge 0} \otimes \nu({\bf B}_T),$$
\noindent where $T$ ranges over subsets $T\subset \{1,\dots, L\}$ such that a coordinate subsemigroup ${\bf B}_T \subset {\bf B}$. Observe that $\RR_{\ge 0} \otimes \nu({\bf B}_T)$ is a simplicial cone of dimension $t$ with the elements $\nu(y_{j_1}),\dots, \nu(y_{j_t})$ lying on its extremal rays. Moreover, it holds
$$\RR_{\ge 0}\otimes \nu({\bf B}_{T_1}) \bigcap \RR_{\ge 0}\otimes \nu({\bf B}_{T_2})= \RR_{\ge 0}\otimes \nu({\bf B}_{T_1 \cap T_2}).$$

ii) Let $\nu: A\setminus \{0\} \twoheadrightarrow P,\ \nu':A\setminus \{0\} \twoheadrightarrow P'$ be injective valuations of an algebra $A$. Let $\bf B \subset A$ be a basis of $A$ common adapted to $\nu, \nu'$. Assume that $S\subset \bf B$ is a subsemigroup. Then
$${\bf K}_{\nu',\nu}|_{\nu(S)} : \nu(S) \twoheadrightarrow \nu'(S)$$
\noindent is an isomorphism of semigroups.

Now assume that in addition it holds that a coset $bS\subset \bf B$. Then
$${\bf K}_{\nu',\nu}|_{\nu(bS)}=\nu'(b)+{\bf K}_{\nu',\nu}|_{\nu(S)}.$$

Thus, one can extend ${\bf K}:={\bf K}_{\nu',\nu}$ to a piece-wise linear map 
$$\RR_{\ge 0} \otimes {\bf K} : \RR_{\ge 0} \otimes \nu({\bf B}) \to \RR_{\ge 0} \otimes \nu'({\bf B}),$$
\noindent being linear on $\RR_{\ge 0} \otimes \nu({\bf B}_T)$ for each coordinate subsemigroup ${\bf B}_T \subset {\bf B}$.

In  Example~\ref{ex:mod 2}  one can take $S:= \{b^0_{\bf d} |\, {\bf d}=(d_1,d_2)\in \ZZ_{\ge 0}^2\}$, and $z_1S\subset \bf B$ is a coset. In Example~\ref{ex:free mod 2} one can take $S=P$.
In Example~\ref{2_noncommutative}, $\bf K$ is an isomorphism of semigroups, thus one can take $S:=\langle x_{11},x_{12}, x_{21}, x_{22}\rangle$.
In Example~\ref{2.13 non-commutative} one can take $S$ the monoid either of monomials in $x_1, x_3, m$ or of monomials in $x_2, x_3, m$.
\end{remark}

\begin{definition} Let $\nu_\bullet$ and $\nu_\circ$ be injective valuations on an algebra $A$. Suppose that a basis ${\bf B}$ is adapted to both valuations. This turns ${\bf B}$ into  ordered partial semigroups $({\bf B},\circ,\preceq^\circ)$ and $({\bf B},\bullet,\preceq^\bullet)$, see Remark~\ref{basis_semigroup}.

We say that $\nu_\bullet$ and $\nu_\circ$ are polar with respect to ${\bf B}$ if any $b''$ occurring in $bb'\ne 0$ satisfies 
$$b\bullet b'\preceq^\circ b'', b\circ b'\preceq^\bullet b''.$$
\end{definition}

\begin{remark}
i) For the algebra ${\bf K}[x,y,z]/(z+x^2+y^3)$ constructed in Remark~\ref{cusp_JH}, each pair among its injective valuations $\nu_x, \nu_y, \nu_z$ is polar with respect to the produced common adapted basis. For instance, for the common basis $\{y^iz^jx^k\ :\ i,j\ge 0, 0\le k\le 1\}$ of $\nu_y, \nu_z$ it holds $x\cdot x=-y^3-z$, hence $\nu_y(x^2)=\nu_y(z)\succ \nu_y(y^3), \nu_z(x^2)=\nu_z(y^3)\succ \nu_z(z)$. 

In a similar way, for the algebra $\kk[x,y,z,t]/(x^2+y^3+z^5+t^7)$ constructed in Example~\ref{3-dimensional_valuation}, its valuations $\nu_1,\nu_2$ are polar with respect to the produced their common adapted basis $\{y^it^jw^lx^kz^m\ :\ i,j,l\ge 0, 0\le k\le 1, 0\le m\le 4\}$. Indeed, $x\cdot x=w-y^3, z^2\cdot z^3=z \cdot z^4=w-t^7$, and $\nu_1(x^2)=\nu_1(w)\succ \nu_1(y^3), \nu_2(x^2)=\nu_2(y^3)\succ \nu_2(w), \nu_1(z^5)=\nu_1(t^7)\succ \nu_1(w), \nu_2(z^5)=\nu_2(w)\succ \nu_2(t^7)$, therefore the polar condition holds also for the decomposition in the basis of the products $xz^2\cdot xz^3=xz \cdot xz^4=-wt^7+w^2+y^3t^7-y^3w$. \vspace{2mm} 

ii) Observe that the injective valuations produced in Example~\ref{ex:mod 2}  
are polar with respect to the basis $\bf B$. 
 \end{remark}

\subsection{Valuations of quantum groups and their JH-bijections}\label{3.13}

Let $\gg$ be a semisimple or symmetrizable Kac-Moody Lie algebra with the $I\times I$ Cartan matrix $A$ and its symmetrization $C=(c_{ij})=DA$.

Let $U_q(\nn)$ be the quantized enveloping algebra over $\kk:=\QQ(q^{\frac{1}{2}})$ of the nilpotent part $\nn$ of a symmetrizable Kac-Moody algebra $\gg$, it is generated by $E_i$, $i\in I$ over $\kk$ subject to $q$-Serre relations (see e.g., \cite{BG})

\begin{equation}
\label{eq:qserre}
\sum\limits_{r,s\ge 0,\,r+s=1-a_{ij}}\mskip-8mu (-1)^s  E_i^{\langle r\rangle}E_j E_i^{\langle s\rangle}= 0
\end{equation}
for all distinct $i,j\in I$,  where $q_i=q^{d_i}$, 
$E_i^{( k)}:=\big(\prod\limits_{s=1}^k (s)_{q_i}\big)^{-1} E_i^k$ and $(s)_v=\frac{v^s-v^{-s}}{v-v^{-1}}$.

Following \cite[Section 1.2.13]{lu}, define $q$-derivations $\partial_i$ of $U_q(\nn)$  
via $\partial_i(E_j)=\delta_{ij}$
and 
\begin{equation}
\label{eq:partial U_q(n)}
\partial_i(xy)=\partial_i(x)y+K_i(x)\partial_i(y)
\end{equation} 
where $K_i$ is an automorphism of $U_q(\nn)$ given by $K_i(E_j)=q^{-c_{ij}}E_j$.

Denote by $\overline{\cdot}$ the $\QQ$-linear anti-involution on $U_q(\nn)$ such that $\overline E_i=E_i$, $\overline{q^{\frac{1}{2}}}=q^{-\frac{1}{2}}$. 

For  any sequence $\ii=(i_1,\ldots,i_m)$ and any nonzero $x\in U_q(\nn)$ consider a string valuation $\underline \nu_\ii=\nu_{\partial_{i_1},\ldots,\partial_{i_m}}$ in the notation Theorem \ref{th:string valuations general}.

\begin{proposition}
\label{pr:string nuii}
For any sequence $\ii$ the assignments $x\mapsto \underline \nu_\ii(x)$ define a $\kk$-valuation $\underline \nu_\ii :U_q(\nn)\setminus \{0\}\to \ZZ_{\ge 0}^m$.
    
\end{proposition}

\begin{proof} It is known (cf. \cite{EK}, Corollary~3.5) that $U_q(\nn)$ is a flat deformation of $U(\nn)$. The celebrated   Poincar\'e-Birkhoff-Witt (PBW) theorem  states that $gr~U(\gg)=S(\gg)$ for any Lie algebra $\gg$, therefore $U(\nn)$ is a domain. Hence $U_q(\nn)$ is a domain as well. Furthermore, the automorphisms $\varphi_k:=K_{i_k}$ and the $\varphi_k$-derivations $\partial_k:=\partial_{i_k}$ (with a slight abuse of notation) satisfy Theorem \ref{th:string valuations general} because
$$\partial_pK_j =q^{c_{pj}}K_j\partial_p$$
for all $p,j\in I$ (hence $q_{k\ell}=q^{c_{i_k,i_\ell}}$ for $k,\ell=1,\ldots,m$).

Also, it is immediate that the involved derivations are locally nilpotents.
Thus, all hypotheses of Theorem \ref{th:string valuations general} are satisfied and the proposition is proved.
\end{proof}

For any $w\in W$ let $U_q(w)$ be a subalgebra of $U_q(\nn)$ defined in \cite[Equation (1.1)]{BG} (it is referred to as the quantum Schubert cell). For any reduced word $\ii$ of $w$ in \cite[Section 4.1]{BG} we constructed a PBW-basis $X_\ii^{\bf a}$, ${\bf a}\in \ZZ_{\ge 0}^m$ of $U_q(w)$.

\begin{proposition}

\label{pr:PBW valuation nuii}
In the notation of \cite[Section 4.1]{BG} the assignments $X_\ii^{\bf a}\mapsto {\bf a}$
define an injective $\kk$-valuation  $\underline \nu^\ii:U_q(w)\setminus\{0\}\to \ZZ_{\ge 0}^m$ (with respect to the lexicographic order on $\ZZ_{\ge 0}^m$).
    
\end{proposition}

\begin{proof} Recall that the partial order $\prec$ on $\ZZ_{\ge 0}^m$ was introduced in \cite[Section 4.2]{BG}. It is easy to see that ${\bf a}\prec {\bf b}$ implies that ${\bf a}< {\bf b}$, where $<$ denotes the lexicographic order. 

Then   \cite[Corollary 4.14]{BG} implies that $X_\ii^{\bf a}\cdot X_\ii^{\bf b}\in \kk^\times\cdot X_\ii^{{\bf a+b}}+\sum\limits_{{\bf c}<{\bf a+b}}\kk\cdot  X_\ii^{{\bf c}}$ for any ${\bf a},{\bf b}\in \ZZ_{\ge 0}$.

The proposition is proved.
\end{proof}

The rest of the section is devoted to lifting of the aforementioned valuations and the above JH-bijections to the appropriate quantum octants $P_\ii$ and $P^\ii$ (see below).

For any sequence $\ii=(i_1.\ldots,i_m)\in I^m$ denote by
$A_\ii$  an algebra over $\kk=\QQ(q^{ \frac{1}{2}})$ generated by $t_1,\ldots,t_m$  subject to \begin{equation}
\label{eq:quantum plane}   
t_\ell t_k=q^{c_{i_k,i_\ell}}t_kt_\ell
\end{equation}
for $1\le k\le \ell\le m$, where $C=(c_{ij})$ is the symmetrized Cartan matrix of $\gg$. Also denote by $\overline{\cdot}$  a $\QQ$-linear anti-involution on $A_\ii$ such that $\overline t_k=t_k$, $\overline{q^{\frac{1}{2}}}=q^{-\frac{1}{2}}$.

The following is well-known (see e.g., \cite{ber,beru,im,jo}).

\begin{theorem} [Feigin's homomorphism] 
\label{Feigin's homomorphism}
For any sequence $\ii=(i_1,\ldots,i_m)\in I^m$ the assignments 
$$E_i\mapsto \sum_{k:i_k=i} t_k$$
define  a homomorphism $\Phi_\ii:U_q(\nn)\to A_\ii$.

\end{theorem}

The following is an immediate consequence of Theorem \ref{Feigin's homomorphism}.

\begin{corollary} 
\label{cor:equivariant Feigin}
$\Phi_\ii$ is $\overline{\cdot}$-equivariant.
    
\end{corollary}

The following is an immediate consequence of Theorem \ref{th:generalized Feigin} for $A=U_q(\nn)$,  $\partial_k:=\partial_{i_k}$, $q_{kl}=q^{c_{i_k,i_l}}$ (see also \cite[Equation 1.7]{ber}).

\begin{corollary}
\label{cor:explicit Feigin}

$\Phi_\ii(x)=\sum\limits_{a\in \ZZ_{\ge 0}^m} \varepsilon(\partial_{i_m}^{(a_m)} \cdots \partial_{i_1}^{(a_1)}(x))t_1^{a_1}\cdots t_m^{a_m}$ for any $x\in U_q(\nn)$ and any $\ii\in I^m$,
where $\varepsilon:U_q(\nn)\to \kk$ is the counit homomorphism given by $\varepsilon(1)=1$, $\varepsilon(E_i)=0$ for $i\in I$.

\end{corollary}

The following is an immediate consequence of Corollary \ref{cor:explicit Feigin}

\begin{corollary}
\label{cor:leading term Feigin} $\Phi_\ii(x)\in \kk^\times \cdot t^{\underline \nu_\ii(x)}+\text{lex lower terms}$
for any $x\in U_q(\nn)$ such that $\Phi_\ii(x)\ne 0$.

\end{corollary} 
 
The following is the main result of the section.

\begin{theorem} 
\label{th:Feigin}
For any reduced word $\ii$ of any $w\in W$ the restriction of $\Phi_\ii$ to $U_q(w)$ is injective.
  
\end{theorem}

\begin{proof} For any sequence $\ii=(i_1,\ldots,i_m)$ denote by $U_\ii\subset U_q(\nn)$ the $\kk$-linear span of $E_{i_1}^{a_1}\cdots E_{i_m}^{a_m}$, $a_1,\ldots,a_m\in \ZZ_{\ge 0}$.

In the notation of \eqref{eq:partial U_q(n)} and Corollary \ref{cor:explicit Feigin} define a pairing $(\cdot,\cdot)$ 
on $U_q(\nn)$ by
$$(E_{j_1}\cdots E_{j_\ell},x)=\varepsilon(\partial_{j_1}\cdots \partial_{j_\ell}(x))$$
for any  $i_1,\ldots,j_\ell\in I$, $x\in U_q(\nn)$. Clearly, the pairing satisfies
$$(uE_i,x)=(u,\partial_i(x))$$
and
$$(u,xE_i)=(\partial_i^*(u),x)$$ 
where $\partial_i^*$ is determined by $\partial_i^*(E_j)=\delta_{ij}$
and 
\begin{equation*}
\partial_i^*(uv)=\partial_i^*(u)v+K_i^*(u)\partial^*_i(v)
\end{equation*} 
and $K_i^*$ is an automorphism of $U_q(\nn)$ given by $K_i^*(E_j)=q^{c_{ij}}E_j$.

Denote by $U_\ii^\perp$, the orthogonal complement of $U_\ii$.

The following is well-known,  see e.g.,  \cite[Section 4.4]{lu}, \cite{lu2},  \cite[Lemma 0.4]{ber}

\begin{lemma}  (a) $U_\ii^\perp$ is an ideal of $U_q(\nn)$ for any $\ii$.

(b) $U_\ii=U_{\ii'}$ for any reduced words $\ii,\ii'$ of $w$ (so we denote this space by $U_w$).

\end{lemma}

The following is an immediate consequence of Corollary \ref{cor:explicit Feigin}.

\begin{corollary}  $U_{\ii^{op}}^\perp=Ker~\Phi_\ii$ for any $\ii=(i_1,\ldots,i_m)\in I^m$. 
    
\end{corollary}

Denote by $\tilde U_q(\ii):=U_q(\nn)/U_\ii^\perp$. 
Fix $w\in W$.  Then, by the above lemma:

$\bullet$ $\tilde U_q(\ii)$ is an algebra such that $\tilde U_q(\ii)=\tilde U_q(\ii')$ for any reduced words $\ii,\ii'$ of $w$ (so we denote it $\tilde U_q(w)$)

$\bullet$ The assignments $u+U_w^\perp\mapsto \Phi_\ii(u)$ is a well-defined injective homomorphism of algebras 
\begin{equation}
\label{eq:Feigin on quotient}
    \tilde \Phi_\ii:\tilde U_q(w)\hookrightarrow A_\ii
    \end{equation}

We need the following result of Kimura.

\begin{lemma} (\cite[Theorem 5.13]{ki}) 
\label{kimura}
The restriction of the canonical surjective homomorphism $U_q(\nn)\twoheadrightarrow \tilde U_q(w)$ to $U_q(w)$ is an injective homomorphism of algebras $U_q(w)\hookrightarrow \tilde U_q(w)$.
    
\end{lemma}

Combining \eqref{eq:Feigin on quotient} with Lemma \ref{kimura}, we finish the proof of Theorem~\ref{th:Feigin}.
\end{proof}

We abbreviate $\lambda_\ii:=\lambda_{\bf \partial}$, where ${\bf \partial}=(\partial_{i_1},\ldots, \partial_{i_m})$ in the notation of \eqref{eq:partial U_q(n)} and \eqref{eq:lambdaE}.

The following is an immediate consequence of Corollary \ref{pr:string nuii}, Theorem \ref{th:Feigin},  Proposition \ref{pr:injective subalgebra}, and of the grading on $U_q(w)$.

\begin{corollary}
\label{cor:injective restriction w}
In the assumptions of Proposition \ref{pr:string nuii} one has:

(a) the restriction of $\underline \nu_\ii$ to $U_q(w)$ is an injective valuation $U_q(w)\setminus \{0\}\to \ZZ_{\ge 0}^m$.

(b) $\lambda_\ii(U_q(w)\setminus \{0\})=\kk^\times$.
\end{corollary}

Combining this with Proposition \ref{pr:PBW valuation nuii}, we define the respective JH-bijections
\begin{equation}
\label{eq:ii,iip JH}
\underline {\bf K}^+_{\ii',\ii}:= {\bf K}_{\underline\nu^{\ii'},\underline \nu^\ii}:\ZZ_{\ge 0}^m\widetilde \to \ZZ_{\ge 0}^m,
~\underline {\bf K}^-_{\ii',\ii}:={\bf K}_{\underline\nu_{\ii'},\underline\nu_\ii}:\underline  C_\ii \widetilde \to \underline C_{\ii'},~\underline {\bf K}_{\ii',\ii}:= {\bf K}_{\underline\nu_{\ii'},\underline\nu^\ii}:\ZZ_{\ge 0}^m\widetilde \to \underline C_{\ii'}
\end{equation}
for any reduced words $\ii,\ii'$ for $w$, where we abbreviate $\underline C_\ii:=C_{\underline\nu_\ii}:=\underline\nu_\ii(U_q(w)\setminus \{0\})$.

\begin{remark} 
\label{rem:thin thick}
$\underline C_\ii$ can be thought as a ``thin" string cone in the sense that it is much smaller than a ``thick" string cone $\underline \nu_\ii(U_q(\nn)\setminus\{0\})$ if $\gg$ is infinite-dimensional (in the notation Proposition \ref{pr:string nuii}), which is much more complicated than the ``thin" one (see e.g., \cite{NZ,Jo}). At the same time, if $\gg$ is finite-dimensional and  $\ii$ is a reduced word for $w_0$, then $U_q(\nn)=U_q(w_0)$ and the ``thick" and ``thin" cones coincide.
    
\end{remark}

Denote $\AA:=\ZZ[q^{\frac{1}{2}},q^{-\frac{1}{2}}]$
and let $U^A(\nn)\subset U_q(\nn)$ be the $\AA$-subalgebra of $U_q(\nn)$ defined in \cite[Section 2.4]{BG} and denote $U^\AA(w):=U_q(w)\cap U^\AA(\nn)$ in notation of \cite[Section 2.4]{BG}.

By the very construction, $U_q(\nn)=\kk\otimes U^\AA(\nn)$. It follows from \cite[Proposition 40.2.1]{lu} and \cite[Theorem 4.5]{BG} that $U_q(w)=\kk\otimes U^\AA(w)$ as well.

Thus, Theorem \ref{th:Feigin} is equivalent to its $\AA$-weakening.

\begin{corollary} 
\label{th:Feigin A}
For any reduced word $\ii$ of any $w\in W$ the restriction of $\Phi_\ii$ to $U^\AA(w)$ is injective.
  
\end{corollary}

Denote by $P_\ii$ the monoid generated by $t_1,\ldots,t_m$, and invertible central element $q^{\frac{1}{2}}$ subject to \eqref{eq:quantum plane} (it is immediate from Example \ref{ex:quantum plane} that $A_\ii=\kk\otimes \QQ P_\ii$),  Clearly, it is naturally ordered as in Example \ref{ex:quantum plane}, where the order on $\Gamma=q^{\frac{1}{2}\ZZ}\cong \ZZ$ is natural. In the notation of Example \ref{ex:quantum plane} we abbreviate $t^{[a]}:=(a,1)$ in $P_\ii$. By definition, $t^{[a]}=q_{\ii,a}t_1^{a_1}\cdots t_m^{a_m}$, where $q_{\ii,a}=q^{f_\ii(a)}$ and $f_\ii(a)=\frac{1}{2}\sum\limits_{1\le k< \ell\le m} c_{i_k,i_\ell}a_ka_\ell$.
Then
\begin{equation}
\label{eq:Lambda_ii}
t^{[a]}t^{[a']}=q^{\frac{1}{2}\Lambda_\ii(a,a')}t^{[a+a']}
\end{equation}
where $\Lambda_\ii$ is a skew-symmetric form on $\ZZ_{\ge 0}^m$ given by $\Lambda_\ii(e_k,e_\ell)=-c_{i_k,i_\ell}$ for $1\le k<\ell \le m$.

The following is an immediate consequence of Theorem \ref{th:Feigin}, Proposition \ref{pr:injective subalgebra},  and Corollary \ref{cor:equivariant Feigin}.  

\begin{corollary}
\label{cor:nu_ii}
In the assumptions of Theorem \ref{th:Feigin} the assignments $\displaystyle{x\mapsto \nu_{P_\ii}(\Phi_\ii(x))}$ define an injective $\overline{\cdot}$-equivariant valuation $\nu_\ii:U^\AA(w)\setminus\{0\}\to P_\ii$.
    
\end{corollary}

Denote by $P^\ii$ the  monoid generated by $X_1,\ldots,X_m$ and an invertible central elements $q^{ \frac{1}{2}}$ subject to 
$$X_\ell X_k=
q^{(\alpha_\ii^{(k)},\alpha_\ii^{(\ell)})} X_k X_\ell$$
where we abbreviated $\alpha^{(k)}_\ii:=s_{i_1}\cdots s_{i_{k-1}}(\alpha_{i_k})$ as in \cite[Section 4.1]{BG} and $(\cdot,\cdot)$ is the inner product on the root space given by $(\alpha_i,\alpha_j)=c_{ij}$. Clearly $P^\ii$ (similarly to $P_\ii$) is naturally ordered as in Example \ref{ex:quantum plane}, where the order on $\Gamma=q^{\frac{1}{2}\ZZ}\cong \ZZ$ is natural. 
In the notation of Example \ref{ex:quantum plane} (similarly to $P_\ii$) we abbreviate $X^{[d]}:=(d,1)$ in $P^\ii$. By definition, $X^{[d]}=q_{\ii,d}X_1^{d_1}\cdots X_m^{d_m}$, where $q_{\ii,d}=q^{f^\ii(d)}$ and $f^\ii(d)=\frac{1}{2}\sum\limits_{1\le k< \ell\le m} (\alpha_\ii^{(k)},\alpha_\ii^{(\ell)})d_k d_\ell$.
Then
\begin{equation}
\label{eq:Lambda^ii}
X^{[d]}X^{[d']}=q^{\frac{1}{2}\Lambda^\ii(d,d')}X^{[d+d']}\ ,
\end{equation}
where $\Lambda^\ii$ is a skew-symmetric form on $\ZZ_{\ge 0}^m$ given by $\Lambda^\ii(e_k,e_\ell)=-(\alpha_\ii^{(k)},\alpha_\ii^{(\ell)})$ for $1\le k<\ell \le m$.

\begin{theorem} 
\label{th:Qii}
In the assumptions of \cite[Theorem 4.5]{BG} the assignments $X_\ii^a\mapsto X^{[a]}$
define an injective  
 $\overline{\cdot}$-equivariant valuation $\nu^\ii:U^\AA(w)\setminus\{0\}\to P^\ii$.
    
\end{theorem}

\begin{proof} Recall that for any $\ii=(i_1,\ldots,i_m)$ one has a partial order $\prec:=\prec_\ii$ on $\ZZ^m$ defined in \cite[Section 4.2]{BG}. It is immediate that for any $a,a'\in \ZZ_{\ge 0}^m$ the inequality $a\prec_\ii a'$ implies that $a<a'$ in the lexicographic (as well as in the inverse lexicographic) order.

 Theorem \cite[Theorem  4.5]{BG} implies that this defines an $\ZZ_{\ge 0}^m$-filtration on $U_q(w)$ via  $\deg X_\ii^{a}=a$ for all $a\in \ZZ_{\ge 0}^m$.  
 
Clearly, $gr~U_q(w)=\kk\otimes \QQ P^\ii$ by Theorem \cite[Theorem  4.5]{BG}.

Finally, the $\overline{\cdot}$-equivariance assertion follows from Corollary \ref{cor:equivariant Feigin}.  

The theorem is proved. 
\end{proof}

\begin{remark} 
\label{rem:commutation with q}
By definition, all $\nu_\ii$ and $\nu^\ii$ commute with multiplication by  $q^{\frac{1}{2}}$.
    
\end{remark}

\begin{remark} 
\label{rem:beru}
According to \cite[Equation (4.10)]{beru} for any $\ii=(i_1,\ldots,i_m)\in I^m$ there is an element $\psi_\ii\in GL_m(\ZZ)$ such that $\Lambda_\ii(\psi_\ii(d),\psi_\ii(d'))=\Lambda^\ii(d,d')$ for any $d,d'\in \ZZ^m$. 

This implies, in particular, that the groups generated by $P_\ii$ and $P^\ii$ are isomorphic.

This also implies that for any $m$-dimensional linearity domain $\underline C$ (in notation of Remark \ref{rem:linearity domain}) the composition $(\underline {\bf K}_{\ii,\ii}\circ \psi_\ii^{-1})|_{\psi_\ii(\underline C)}$ is an element $g\in GL_m(\ZZ)$ such that $\Lambda_\ii(g(a),g(a'))=\Lambda_\ii(a,a')$ for all $a,a'\in \ZZ^m$, in other words, it is an automorphism of $\Lambda_\ii$, i.e.,  an `integral symplectomorphism." Based on numerous examples (e.g., Examples \ref{ex:quantum A2} and \ref{ex:A3} below), we expect that such a $\underline C$ always exists, and $\underline {\bf K}_{\ii,\ii}$ is a restriction of a piecewise linear continuous bijection $\RR^m\widetilde \to \RR^m$.
    
\end{remark}

Extending \eqref{eq:ii,iip JH} to $U^\AA(q)$, we abbreviate 
\begin{equation}
\label{eq:ii,iip JHq}
{\bf K}^+_{\ii',\ii}:= {\bf K}_{\nu^{\ii'},\nu^\ii}:P^\ii\widetilde \to P^{\ii'},
~{\bf K}^-_{\ii',\ii}:={\bf K}_{\nu_{\ii'},\nu_\ii}: C_\ii \widetilde \to C_{\ii'},~{\bf K}_{\ii',\ii}:= {\bf K}_{\nu_{\ii'},\nu^\ii}:P^\ii\widetilde \to C_{\ii'}
\end{equation}
for any reduced words $\ii,\ii'$ for $w$, where we abbreviate $C_\ii:=C_{\nu_\ii}:=\nu_\ii(U^\AA(w)\setminus \{0\})$, which  is a quantum cone extending $\underline C_\ii$ (according to Remark \ref{rem:commutation with q}, all these bijections commute with multiplication by $q^{\frac{1}{2}}$).

Combining this with Theorem \ref{sub-multiplicative} and Corollary \ref{JH_multiplicativity} we obtain the following spectacular

\begin{corollary} 
\label{cor:Lambda-equivariange}
For any reduced words $\ii$, $\ii'$
 of any $w\in W$ one has

(a) $\Lambda_{\ii'}(a,a')=\Lambda_\ii({\bf K}^-_{\ii,\ii'}(a), {\bf K}^-_{\ii,\ii'}(a'))$ for all $a,a'\in \underline  C_{\ii'}$ such that ${\bf K}^-_{\ii,\ii'}(a+a')={\bf K}^-_{\ii,\ii'}(a)+{\bf K}^-_{\ii,\ii'}(a')$.

(b) $\Lambda^{\ii'}(d,d')=\Lambda^\ii({\bf K}^+_{\ii,\ii'}(d), {\bf K}^+_{\ii,\ii'}(d'))$ for all $d,d'\in\ZZ_{\ge 0}^m$ such that ${\bf K}^+_{\ii,\ii'}(d+d')={\bf K}^+_{\ii,\ii'}(d)+{\bf K}^+_{\ii,\ii'}(d')$.

(c) $\Lambda^{\ii'}(d,d')=\Lambda_\ii({\bf K}_{\ii,\ii'}(d), {\bf K}_{\ii,\ii'}(d'))$ for all $d,d'\in\ZZ_{\ge 0}^m$ such that ${\bf K}_{\ii,\ii'}(d+d')={\bf K}_{\ii,\ii'}(d)+{\bf K}_{\ii,\ii'}(d')$.
    
\end{corollary}

Denote by ${\bf B}^{up}$  the dual canonical basis of $U_q(\nn)$. It is well-known (see e.g., \cite[Propositions 40.2.1 and 41.1.4]{lu}, 
\cite[Section 5.1]{ki}, and  \cite[Corollary 1.3]{BG}) that the intersection ${\bf B}_w:=U_q(w)\cap {\bf B}^{up}$ is a basis of $U_q(w)$ for any $w\in W$.

The following is well-known (see e.g., \cite[Lemma 2.13]{BG}).

\begin{proposition} 
\label{pr:top of basis}
For any  $\ii\in I^m$, $m\ge 1$.

(a) $\lambda_\ii({\bf B})\subset {\bf B}$. 

(b) The assignments $b\mapsto (\underline \nu_\ii(b),\lambda_\ii(b))$ define an injective map ${\bf B}\hookrightarrow \ZZ_{\ge 0}^m\times {\bf B}$.

\end{proposition}

Combining Proposition \ref{pr:top of basis}(a) with Corollary \ref{cor:injective restriction w}(b) we see that $\lambda_\ii({\bf B}_w)=\{1\}$ for any reduced word $\ii$ for $w\in W$ (because ${\bf B}_w\cap \kk^\times=\{1\}$).

In turn, this and Proposition \ref{pr:top of basis}(b) imply the following

\begin{corollary} 
\label{cor:adapted nu_ii}
For all $w\in W$ the dual canonical basis ${\bf B}_w$ of $U^\AA(w)$ is adapted to all valuation $\nu_\ii$,  $\ii$ runs over all reduced words of $W$.
    
\end{corollary}

It turns our that we can replace $\nu_\ii$ with $\nu^\ii$.
\begin{theorem} 
\label{th:adapted nu^ii}
${\bf B}_w$ is  adapted to the valuation $\nu^\ii$,  $\ii$ runs over all reduced words of $W$.
    
\end{theorem}

\begin{proof} We need the following.

\begin{lemma} 
\label{le:triangular basis}
For each $b\in {\bf B}_w$ there exists a unique $d=\nu^\ii(b)\in \ZZ_{\ge 0}^m$ such that 
$$b\in X_\ii^d+\sum_{d'<d}\kk\cdot X_\ii^{d'}$$
where $<$ is the lexicographic order on $\ZZ_{\ge 0}^m$.
\end{lemma}

\begin{proof} It follows from the proof of \cite[Theorem 1.1]{BG} that for each $b\in {\bf B}_w$ there exists a unique $d=d(b)\in \ZZ_{\ge 0}^m$ such that 
$b\in X_\ii^d+\sum\limits_{d'\prec d}\kk\cdot X_\ii^{d'}$, 
where $\prec$ is the partial order on the monoid $\ZZ_{\ge 0}^m$ defined in \cite[Section 4.2]{BG}. As we argued in the proof of Proposition \ref{pr:PBW valuation nuii}, $\prec$ is stronger than $<$, i.e.,  $d\prec d'$ implies $d<d'$ for any $d,d'\in \ZZ_{\ge 0}$. In turn, this implies that $U_q(\nn)_{\prec d}\subset U_q(\nn)_{< d}$. 

The lemma is proved. 
\end{proof}

Thus, $\nu^\ii(b)=\nu^\ii(X_\ii^{d(b)})=X_\ii^{[d(b)]}$, which finishes the proof of the theorem.
\end{proof}

Combining Corollary \ref{cor:adapted nu_ii} with Theorem \ref{th:adapted nu^ii}
we obtain the following immediate
\begin{corollary} For any reduced words $\ii,\ii',\ii''$ of any $w\in W$ one has:
$${\bf K}^-_{\ii,\ii'}\circ {\bf K}^-_{\ii',\ii''}={\bf K}^-_{\ii,\ii''},~{\bf K}^+_{\ii,\ii'}\circ {\bf K}^+_{\ii',\ii''}={\bf K}^+_{\ii,\ii''},~{\bf K}_{\ii,\ii'}\circ {\bf K}_{\ii',\ii''}={\bf K}_{\ii,\ii''}$$
$${\bf K}_{\ii,\ii'}^-\circ {\bf K}_{\ii',\ii'}={\bf K}_{\ii,\ii'},~{\bf K}_{\ii,\ii}\circ {\bf K}_{\ii,\ii'}^+={\bf K}_{\ii,\ii'}\ .$$
\end{corollary}

This, in particular, implies that to recover 
all  ${\bf K}_{\ii,\ii'}$, it suffice to know only ${\bf K}_{\ii,\ii}$ and ${\bf K}^-_{\ii,\ii'}$.

\subsection{Examples of JH-bijections for quantum groups}

\begin{example} 
\label{ex:quantum A2}
Let $A=C=\begin{pmatrix} 2 & -1\\
-1 & 2   
\end{pmatrix}$. Let $\ii=(1,2,1)$, $\ii'=(2,1,2)$,
 
Then $P_\ii=P_{\ii'}$ is generated by 
$t_1,t_2,t_3$  and the central element $q^{\frac{1}{2}}$ subject to the relations
$$t_3t_1=q^2t_1t_3,~t_2t_1=q^{-1}t_1t_2,t_3t_2=q^{-1}t_2t_3.$$

Then $U_q(\nn)=U_q(w_0)$ is a $\kk$-algebra (here $\kk=\QQ(q^{\frac{1}{2}})$)
 generated by $E_1,E_2$ subject to the quantum Serre relations
$$E_i^2E_j-(q+q^{-1})E_iE_jE_i+E_jE_i^2=0$$
for $\{i,j\}=\{1,2\}$.

Denote $E_{ij}:=\frac{q^{\frac{1}{2}}E_iE_j-q^{-\frac{1}{2}}E_jE_i}{q-q^{-1}}$
for $\{i,j\}=\{1,2\}$.

It is also clear that $E_iE_{ij}=qE_{ij}E_i$, $E_{ij}E_j=qE_jE_{ij}$ and $E_{12}E_{21}=E_{21}E_{12}$ and $\overline E_{ij}=E_{ij}$.
Also
\begin{equation}
\label{eq:Eij}
    q^{\frac{1}{2}}E_1E_2=qE_{12}+E_{21},~q^{\frac{1}{2}}E_2E_1=qE_{21}+E_{12}
\end{equation}

Recall that $\AA=\QQ[q^\frac{1}{2},q^{-\frac{1}{2}}]$ and clearly, $U^\AA(w_0)$
is generated by $E_1,E_2,E_{12},E_{21}$.

The monoids $P_{\ii}$ and $P_{\ii'}$ are generated by  $q^{\frac{1}{2}},q^{-\frac{1}{2}},$ $t_1,t_2,t_3$ and $q^{\frac{1}{2}},q^{-\frac{1}{2}},$ $t_1',t_2',t_3'$ respectively
subject to
$$t_2t_1=q^{-1}t_1t_2,~t_3t_2=q^{-1}t_2t_3,t_3t_1=q^2t_1t_3$$
and 
$$t'_2t'_1=q^{-1}t'_1t'_2,~t'_3t'_2=q^{-1}t'_2t'_3,t'_3t'_1=q^2t'_1t'_3$$
respectively
(These are relations \eqref{eq:quantum plane}). Also $A_\ii=\kk \otimes \QQ P_\ii$, $A_{\ii'}=\kk \otimes \QQ P_{\ii'}$ are algebras over $\kk$ generated by $t_1,t_2,t_3$ and $t'_1,t'_2,t'_3$ respectively subject to the above relations.

Then Feigin's homomorphisms $\Phi_\ii:U_q(\nn)\to A_\ii$ and $\Phi_{\ii'}:U_q(\nn)\to A_{\ii'}$ are given by
$$\Phi_\ii(E_1)=t_1+t_3,\Phi_\ii(E_2)=t_2, \Phi_{\ii'}(E_1)=t_2,\Phi_{\ii'}(E_2)=t_1+t_3\ .$$

Both $\Phi_\ii$ and $\Phi_{\ii'}$ commute with the anti-involution  $\overline{\cdot}$
fixing the generators and  such that $\overline {q^{\frac{1}{2}}}=q^{-\frac{1}{2}}$. 
Clearly, $\Phi_\ii(E_{12})=q^{-\frac{1}{2}}t_1t_2$, $\Phi_\ii(E_{21})=q^{-\frac{1}{2}}t_2t_3$, $\Phi_{\ii'}(E_{12})=q^{-\frac{1}{2}}t_2t_3$, 
$\Phi_{\ii'}(E_{21})=q^{-\frac{1}{2}}t_1t_2$.

It is well-known (see e.g.,\cite{BG}) that
$${\bf B}=\{b_{\bf m}:=q^{\frac{1}{2}(m_1-m_2)(m_3-m_4)}E_1^{m_1}E_2^{m_2}E_{12}^{m_3}E_{21}^{m_4}~ m_i\in \ZZ_{\ge 0},~\min(m_1,m_2)=0 \}$$
is the dual canonical basis of $U_q(\nn)$. In particular, each element of ${\bf B}$ is fixed by $\overline{\cdot}$.

Denote $t_\ii^{[a]}:=q^{a_1a_3-\frac{a_2}{2}(a_1+a_3)}t_1^{a_1}t_2^{a_2}t_3^{a_3}$ and $t_{\ii'}^{[a]}:=q^{a_1a_3-\frac{a_2}{2}(a_1+a_3)}t_1^{'a_1}t_2^{'a_2}t_3^{'a_3}$
for any $a=(a_1,a_2,a_3)\in \ZZ_{\ge 0}^3$. Clearly, $\overline {t_\ii^{[a]}}=t_{\ii}^{[a]}$ because both are equal to $(a,1)$ in notation of Example \ref{ex:quantum plane}.

We equip $P_\ii$ and $P_{\ii'}$ with the natural lexicographic order $q^{\frac{1}{2}\ZZ}\prec t_3\prec t_2\prec t_1$ and $q^{\frac{1}{2}\ZZ}\prec t'_3\prec t'_2\prec t'_1$. Then the valuation 
$\nu_\ii$ from Corollary \ref{cor:nu_ii} is given by  
$$\nu_\ii(b_{\bf m})=q^{\frac{1}{2}(m_1-m_2)(m_3-m_4)}t_1^{m_1}t_2^{m_2}(q^{-\frac{1}{2}}t_1t_2)^{m_3}(q^{-\frac{1}{2}}t_2t_3)^{m_4}$$
$$=
q^{a_1a_3-\frac{a_2}{2}(a_1+a_3)}t_1^{a_1}t_2^{a_2}t_3^{a_3}=t_\ii^{[a]}\ .$$
where $a_1=m_1+m_3$, $a_2=m_2+m_3+m_4$,
$a_3=m_4$.

Likewise, the valuation 
$\nu_{\ii'}$ from Corollary \ref{cor:nu_ii} is given by  
$$\nu_{\ii'}(b_{\bf m})=q^{\frac{1}{2}(m_1-m_2)(m_3-m_4)}t_2^{m_1}t_1^{m_2}(q^{-\frac{1}{2}}t_2t_3)^{m_3}(q^{-\frac{1}{2}}t_1t_2)^{m_4}$$
$$=
q^{a'_1a'_3-\frac{a'_2}{2}(a'_1+a'_3)}t_1^{'a'_1}t_2^{'a'_2}t_3^{'a'_3}=t_{\ii'}^{[a']}\ .$$
where $a'_1=m_2+m_4$, $a'_2=m_1+m_3+m_4$,
$a'_3=m_3$.

In particular, ${\bf B}$ is a basis adapted to both $\nu_\ii$ and $\nu_{\ii'}$.

Therefore, the JH-bijection ${\bf K}^-:={\bf K}^-_{\ii',\ii}:\hat C_\ii\widetilde \to \hat C_{\ii'}$ from \eqref{eq:ii,iip JHq} is given by 
$${\bf K}^-(q^r t_\ii^{[a]})=q^r t_{\ii'}^{\underline {\bf K}^-(a)}\ ,$$
where $\underline {\bf K}^-:=\underline {\bf K}^-_{\ii',\ii}$ from \eqref{eq:ii,iip JH} is an involution given by
\begin{equation}
\label{eq:underline K-}    
\underline {\bf K}^-(a_1,a_2,a_3)=(\max(a_3,a_2-a_1),a_1+a_3,\min(a_1,a_2-a_3))
\end{equation}
for any $a\in C_\ii=C_{\ii'}=\{a\in \ZZ_{\ge 0}^3:a_2\ge a_3\}$.

Clearly, the forms $\Lambda_\ii$ and $\Lambda_{\ii'}$ from \eqref{eq:Lambda_ii}    are given respectively by:
$$\Lambda_\ii(a,a')=\Lambda_{\ii'}(a,a')
=a_1a'_2-a_2a'_1+a_2a'_3-a_3a'_2-2a_1a'_3+2a_3a'_1
$$

One can show that 
$\Lambda_{\ii'}(\underline {\bf K}^-(a),\underline {\bf K}^-(a'))=\Lambda_\ii(a,a')$ iff $\underline {\bf K}^-(a+a')=\underline {\bf K}^-(a)+\underline {\bf K}^-(a')$,
i.e., either $a_1+a_3\le a_2$ and $a'_1+a'_3\le a'_2$,  or $a_1+a_3\ge a_2$ and $a'_1+a'_3\ge a'_2$ (cf. Corollary~\ref{cor:Lambda-equivariange}).

Furthermore, we claim that $X_1:=E_1$, $X_2:=E_{21}$, $X_3:=E_2$,  generate the quantum Schubert cell $U_q(w)=U_q(\nn)$ (and $U_q^\AA(w)$) subject to
$$X_2X_1=qX_1X_2, X_3X_2=qX_2X_3, q^{\frac{1}{2}}X_3X_1-q^{-\frac{1}{2}}X_1X_3=(q-q^{-1})X_2\ .$$

Likewise,
 $X'_1:=E_2$, $X'_2=E_{12}$, $X'_3:=E_1$ generate $U_q(w)=U_q(\nn)$ (and $U_q^\AA(w)$) subject to
$$X'_2X'_1=qX'_1X'_2, X'_3X'_2=qX'_2X'_3, q^{\frac{1}{2}}X'_3X'_1-q^{-\frac{1}{2}}X'_1X'_3=(q-q^{-1})X'_2\ .$$

The monoids $P^\ii$ and $P^{\ii'}$ from Theorem \ref{th:Qii} are generated by $X_1,X_2,X_3, q^{\pm \frac{1}{2}}$ and $X'_1,X'_2,X'_3, q^{\pm \frac{1}{2}}$ respectively subject to
$$X_2X_1=qX_1X_2, X_3X_2=qX_2X_3, X_3X_1=q^{-1}X_1X_3$$
$$X'_2X'_1=qX'_1X'_2, X'_3X'_2=qX'_2X'_3, X'_3X'_1=q^{-1}X'_1X'_3\ .$$

$$X_\ii^{[d]}=q^{\frac{1}{2}(d_1d_2-d_1d_3+d_2d_3)}X_1^{d_1}X_2^{d_2}X_3^{d_3},~X_{\ii'}^{[d]}=q^{\frac{1}{2}(d_1d_2-d_1d_3+d_2d_3)}X_1^{'d_1}X_2^{'d_2}X_3^{'d_3}$$
respectively in $P^\ii$ and $P^{\ii'}$. Clearly, $\overline{X_\ii^{[d]}}=X_\ii^{[d]}$ 
because both are equal to $(d,1)$ in the notation of Example \ref{ex:quantum plane} and the valuations $\nu^i$, $i=1,2$ are $\overline{\cdot}$-equivariant.

Then $\nu^\ii(E_{12})=\nu^\ii(q^{-\frac{1}{2}}X_1X_3-q^{-1}X_2)=q^{-\frac{1}{2}}X_1X_3$ by \eqref{eq:Eij} and
$$\nu^\ii(b_{\bf m})=q^{\frac{1}{2}(m_1-m_2)(m_3-m_4)}X_1^{m_1}X_3^{m_2}(q^{-\frac{1}{2}}X_1X_3)^{m_3}X_2^{m_4}$$
$$=q^{\frac{1}{2}(d_1d_2-d_1d_3+d_2d_3)}X_1^{d_1}X_2^{d_2}X_3^{d_3}
=X_\ii^{[d]}\ ,$$
where we abbreviated $d_1:=m_1+m_3$, $d_2:=m_4$, $d_3:=m_2+m_3$.

Likewise, $\nu^{\ii'}(E_{21})=\nu^{\ii'}(q^{-\frac{1}{2}}X'_1X'_3-q^{-1}X'_2)=q^{-\frac{1}{2}}X'_1X'_3$ by \eqref{eq:Eij} and
$$\nu^{\ii'}(b_{\bf m})=q^{\frac{1}{2}(m_1-m_2)(m_3-m_4)}X_3^{'m_1}X_1^{'m_2}X_2^{'m_3}(q^{-\frac{1}{2}}X'_1X'_3)^{m_4}$$
$$=q^{\frac{1}{2}(d'_1d'_2-d'_1d'_3+d'_2d'_3)}X_1^{'d_1}X_2^{'d_2}X_3^{'d_3}
=X_{\ii'}^{[d']}$$
where we abbreviated $d'_1:=m_2+m_4$, $d'_2:=m_3$, $d'_3:=m_1+m_4$.

Therefore, $\bf B$ is also  adapted to both $\nu^\ii$ and $\nu^{\ii'}$, and  the JH-bijection ${\bf K}^+:={\bf K}^+_{\ii',\ii}:\hat \ZZ_{\ge 0}^3\widetilde \to \hat \ZZ_{\ge 0}^3$ from \eqref{eq:ii,iip JHq} is given by 
$${\bf K}^+(q^r X_\ii^{[d]})=q^r X_{\ii'}^{\underline {\bf K}^+(d)}\ ,$$
where $\underline {\bf K}^+:=\underline {\bf K}^+_{\ii',\ii}$ from \eqref{eq:ii,iip JH} is an involution on $\ZZ_{\ge 0}^3$ given by
\begin{equation}
\label{eq:underline K+}
\underline {\bf K}^+(d_1,d_2,d_3)= (d_2+\max(0,d_3-d_1),\min(d_1,d_3),d_2+\max(0,d_1-d_3))
\end{equation}
for any $d\in \ZZ_{\ge 0}^3$.

Clearly, the forms $\Lambda^\ii$ and $\Lambda^{\ii'}$ from \eqref{eq:Lambda^ii}    are given  by:
$$\Lambda^\ii(d,d')=\Lambda^{\ii'}(d,d')
=-d_1d'_2+d_2d'_1-d_2d'_3+d_3d'_2+d_1d'_3-d_3d'_1
$$

One can show that 
$\Lambda^{\ii'}(\underline {\bf K}^+(d),\underline {\bf K}^+(d'))=\Lambda^\ii(d,d')$ iff $\underline {\bf K}^+(d+d')=\underline {\bf K}^+(d)+\underline {\bf K}^+(d')$,
i.e., either $d_1\le d_3$ and $d'_1\le d'_3$,  or either $d_1\ge d_3$ and $d'_1\ge d'_3$ (cf. Corollary~\ref{cor:Lambda-equivariange}(a)).

The JH-bijection ${\bf K}:={\bf K}_{\ii,\ii}:\hat \ZZ_{\ge 0}^3\widetilde \to \hat C_\ii$ from \eqref{eq:ii,iip JHq} is given by 
$${\bf K}(q^r X_\ii^{[d]})=q^r t_\ii^{\underline {\bf K}(d)}\ ,$$
where $\underline {\bf K}:=\underline {\bf K}_{\ii,\ii}$ from \eqref{eq:ii,iip JH} is a linear map on $\ZZ_{\ge 0}^3\to C_{\ii'}$ given by
$$\underline {\bf K}(d_1,d_2,d_3)= (d_1,d_2+d_3,d_2)$$
for any $d\in \ZZ_{\ge 0}^3$.

One can show that
$\Lambda_\ii(\underline {\bf K}(d),\underline {\bf K}(d'))=\Lambda^\ii(d,d')$ for all $d, d'\in \ZZ_{\ge 0}^3$ (cf. Corollary~\ref{cor:Lambda-equivariange}).

\end{example}

\begin{example} 
\label{ex:A3}
Let $A=C=\begin{pmatrix} 2 & -1 & 0\\
-1 & 2 & -1\\
0 & -1 & 2
\end{pmatrix}$, $\ii=(2,1,3,2)$.

In this case, $U_q(\nn)$ is generated by $E_1,E_2,E_3$ subject to $E_3E_1=E_1E_3$ and to
$$E_1^2E_2-(q+q^{-1})E_1E_2E_1+E_2E_1^2=0$$
$$E_2^2E_1-(q+q^{-1})E_2E_1E_2+E_1E_2^2=0$$
$$E_2^2E_3-(q+q^{-1})E_2E_3E_2+E_3E_2^2=0$$
$$E_3^2E_2-(q+q^{-1})E_3E_2E_3+E_2E_3^2=0$$
Expanding Example \ref{ex:quantum A2} denote $E_{ij}:=\frac{q^{\frac{1}{2}}E_iE_j-q^{-\frac{1}{2}}E_jE_i}{q-q^{-1}}$
whenever $|i-j|=1$. 

The monoid $P_\ii$ is generated by  $q^{\frac{1}{2}},q^{-\frac{1}{2}}$ $t_{11},t_{12},t_{21},t_{22}$ subject to
$$t_{12}t_{11}=q^{-1}t_{11}t_{12},~t_{21}t_{11}=q^{-1}t_{11}t_{21},t_{22}t_{12}=q^{-1}t_{12}t_{22}\ ,$$
$$t_{22}t_{21}=q^{-1}t_{21}t_{22},t_{21}t_{12}=t_{12}t_{21},t_{22}t_{11}=q^2t_{11}t_{22}$$
(these are relations \eqref{eq:quantum plane}). Also, $A_\ii=\kk \otimes \QQ P_\ii$ is an algebra over $\kk=\QQ(q^{\frac{1}{2}})$ generated by $t_{ij}$, $i,j=1,2$ subject to the above relations.

Then Feigin's homomorphism $\Phi_\ii:U_q(\nn)\to A_\ii$ is given by
$$\Phi_\ii(E_1)=t_{12},\Phi_\ii(E_3)=t_{21}, \Phi_\ii(E_2)=t_{11}+t_{22}$$

Denote $X_{22}=E_2$, $X_{12}=E_{21}$, $X_{21}=E_{23}$, $X_{11}=\frac{q^{\frac{1}{2}}E_{21}E_3-q^{-\frac{1}{2}}E_3E_{21}}{q-q^{-1}}=\frac{q^{\frac{1}{2}}E_{23}E_1-q^{-\frac{1}{2}}E_1E_{23}}{q-q^{-1}}$.

We claim that the quantum Schubert cell $U_q(w)$, $w=s_2s_1s_3s_2\in W=S_4$ is generated by $X_{ij}$, $i,j=1,2$ and has a presentation 
$$X_{12}X_{11}=qX_{11}X_{12},X_{21}X_{11}=qX_{11}X_{21},X_{22}X_{12}=qX_{12}X_{22},X_{22}X_{21}=qX_{21}X_{22}\ ,$$
$$X_{21}X_{12}=X_{12}X_{21},X_{11}X_{22}-X_{22}X_{11}=(q^{-1}-q)X_{12}X_{21}$$

That is, $U_q(w)$ the algebra of quantum $2\times 2$ matrices (see e.g., \cite[Equation (3)]{Man}).
 
One can easily show
that
$$\Phi_\ii(X_{12})=q^{-1/2}t_{11}t_{12},\Phi_\ii(X_{21})=q^{-1/2}t_{11}t_{21},~\Phi_\ii(X_{11})=q^{-1}t_{11}t_{12}t_{21} $$
Denote $D:=X_{11}X_{22}-q^{-1}X_{12}X_{21}=X_{22}X_{11}-qX_{21}X_{12}$.

It is central in $U_q(w)$, fixed by $\overline{\cdot}$, and easy to see that $\Phi_\ii(D)=q^{-1}t_{11}t_{12}t_{21}t_{22}$.

It is well-known (see e.g.,\cite{BG}) that
$${\bf B}=\{b_{\bf m}:=q^{\frac{(m_1+m_4)(m_2+m_3)}{2}}X_{11}^{m_1}X_{12}^{m_2}X_{21}^{m_3}X_{22}^{m_4}D^{m_5}:\text{all}~ m_i\in \ZZ_{\ge 0},~\min(m_1,m_4)=0 \}$$
is the dual canonical basis of $U_q(w)$. In particular, each element of ${\bf B}$ is fixed by $\overline{\cdot}$.

We equip $P_\ii$ with the natural lexicographic order $q^{\frac{1}{2}\ZZ}\prec t_{22}\prec t_{21}\prec t_{12}\prec t_{11}$. Then the valuation $\nu_\ii$ from Corollary \ref{cor:nu_ii} is given by  
$\nu_\ii(b_{\bf m})=$
$$q^{\frac{(m_1+m_4)(m_2+m_3)}{2}}(q^{-1}t_{11}t_{12}t_{21})^{m_1}(q^{-1/2}t_{11}t_{12})^{m_2}(q^{-1/2}t_{11}t_{21})^{m_3}t_{11}^{m_4}(q^{-1}t_{11}t_{12}t_{21}t_{22})^{m_5}$$
$$=q^{\frac{2a_{11}a_{22}-(a_{11}+a_{22})(a_{12}+a_{21})}{2}}
t_{11}^{a_{11}}t_{12}^{a_{12}}t_{21}^{a_{21}}t_{22}^{a_{22}}=t^{[a]}\ .$$
where $a_{11}:=m_1+m_2+m_3+m_4+m_5$, $a_{12}:=m_1+m_2+m_5$,
$a_{21}:=m_1+m_3+m_5$, $a_{22}=m_5$.

The monoid $P^\ii$ from Theorem \ref{th:Qii} is generated by $q^{\pm \frac{1}{2}}$ and $X_{ij}$, $i,j=1,2$ subject to 
$$X_{12}X_{11}=qX_{11}X_{12},X_{21}X_{11}=qX_{11}X_{21},X_{22}X_{12}=qX_{12}X_{22},X_{22}X_{21}=qX_{21}X_{22}\ ,$$
$$X_{21}X_{12}=X_{12}X_{21},X_{11}X_{22}=X_{22}X_{11}$$
Then $\nu^\ii(D)=X_{11}X_{22}$ and 
$$\nu^\ii(b_{\bf m})=q^{\frac{(m_1+m_4)(m_2+m_3)}{2}}X_{11}^{m_1}X_{12}^{m_2}X_{21}^{m_3}X_{22}^{m_4}(X_{11}X_{22})^{m_5}=$$
$$q^{\frac{(m_1+m_4+2m_5)(m_2+m_3)}{2}}X_{11}^{m_1+m_5}X_{12}^{m_2}X_{21}^{m_3}X_{22}^{m_4+m_5}=q^{\frac{(d_{11}+d_{22})(d_{12}+d_{21})}{2}}X_{11}^{d_{11}}X_{12}^{d_{12}}X_{21}^{d_{21}}X_{22}^{d_{22}}=X^{[d]}$$
where we abbreviated $d_{11}:=m_1+m_5$, $d_{12}=m_2$, $d_{21}:=m_3$, $d_{22}:=m_4+m_5$.

Therefore, $\bf B$ is  adapted to both $\nu^\ii$ and $\nu_\ii$, and JH
bijection ${\bf K}:={\bf K}_{\ii,\ii}:P^\ii\widetilde \to P_\ii$ is given by 
$${\bf K}(q^rX^{[d]})= q^rt^{[\underline {\bf K}(d)]}$$
for all $r\in \frac{1}{2}\ZZ$ and $d\in Mat_2(\ZZ_{\ge 0})$, 
where $\underline {\bf K}:Mat_2(\ZZ_{\ge 0})\to P_2$, where $P_2=\{
\begin{pmatrix}
a_{11} & a_{12} \\
a_{21} & a_{22}
\end{pmatrix}\in Mat_2(\ZZ_{\ge 0}) : a_{11}\ge \max(a_{12},a_{21}),\min(a_{12},a_{21})\ge a_{22}\}$, is given by 

\begin{equation}
\label{eq:underline K}
\underline {\bf K}(d)=
\begin{pmatrix}
\max(d_{11},d_{22})+d_{12}+d_{21} & d_{11}+d_{12} \\
d_{11}+d_{21} & \min(d_{11},d_{22})
\end{pmatrix}
\end{equation}
for $d=\begin{pmatrix}
d_{11} & d_{12} \\
d_{21} & d_{22}
\end{pmatrix}\in Mat_2(\ZZ_{\ge 0})$. 

The inverse ${\bf K}^{-1}:P_\ii\to P^\ii$ is given by
$${\bf K}^{-1}(q^r t^{[a]})= q^r X^{[\underline {\bf K}^{-1}(a)]}$$
where 
$\underline {\bf K}^{-1}(a)
=\begin{pmatrix}
\max(a_{22},a_{12}+a_{21}-a_{11}) & \min(a_{11}-a_{21},a_{12}-a_{22}) \\
\min(a_{11}-a_{12},a_{21}-a_{22}) & \max(a_{22},a_{11}+2a_{22}-a_{12}-a_{21})
\end{pmatrix}$ for 
$a=\begin{pmatrix}
a_{11} & a_{12} \\
a_{21} & a_{22}
\end{pmatrix}\in P_2$.

Clearly, the forms $\Lambda_\ii$ and $\Lambda^\ii$ from \eqref{eq:Lambda_ii} and \eqref{eq:Lambda^ii}  are given respectively by:
$$\Lambda_\ii(a,a')
=(a_{11}-a_{22})(a'_{1,2}+a'_{21})- (a_{12}+a_{21})(a'_{11}-a'_{22})-2a_{1,1}a'_{22} + 2a_{22}a'_{11}\ ,$$
$$\Lambda^\ii(d,d')=(d_{22}-d_{11})(d'_{12}+d'_{21})+(d_{12}+d_{21})(d'_{11}-d'_{22})\ .$$

One can show that 
$\Lambda_\ii(\underline {\bf K}(d),\underline {\bf K}(d'))=\Lambda^\ii(d,d')$ iff $\underline {\bf K}(d+d')=\underline {\bf K}(d)+\underline {\bf K}(d')$,
i.e., either $d_{11}\le d_{22}$, $d'_{11}\le d'_{22}$ or $d_{11}\ge d_{22}$, $d'_{11}\ge d'_{22}$ (cf. Corollary~\ref{cor:Lambda-equivariange}).

\end{example}

\section{Appendix: Valuations of vector spaces and Jordan-H\"older bijections}
\label{sec:Results on valuations}

\subsection{Valuations on vector spaces}

\label{subsec:valuations}

Given a vector space $S$
over a field $\kk$ and a totally ordered set
$(C,<)$, following \cite{Kaveh-Khovanskii,Ku,KuKu}, we say that a map $\nu:S\setminus\{0\}\to C$ is a
{\it valuation} if $\nu(\kk^\times \cdot x)=\nu(x)$ for all nonzero
$x\in S$ and
$\nu(x+y)\le \max(\nu(x),\nu(y))$ for $x+y \neq 0$
(this implies that $\nu(x+y)=\max(\nu(x),\nu(y))$ whenever $\nu(x)\ne \nu(y)$).

Denote by $C_\nu$ the image $\nu({S}\setminus \{0\})$.

One can construct  valuations on another vector space (resp. integral domain)  $S'$ by importing a given
valuation $\nu$ on $S$ via any injective $\kk$-linear map $f:S'\hookrightarrow S$ (resp. an injective homomorphism
of $\kk$-algebras). Namely, the composition $\nu\circ f$ is a valuation
on $S'$.

Each valuation $\nu:S\setminus\{0\}\to C$ defines a  filtration $S_{\le}$ of subspaces on  $S$ via
$$S_{\le a}:=\{0\}\cup \{x\in S\setminus \{0\}:\nu(x)\le a\}$$
for $a\in C_\nu$ (if $S$ is an integral domain, this is a filtration on a $\kk$-algebra). We also abbreviate $S_{<a}:=\sum\limits_{a'<a} S_{\le a'}$ and denote $S_a:=S_{\le a}/S_{<a}$ for $a\in C$  ($S_a$ is called in \cite{Kaveh-Khovanskii} the {\it leaf at
$a$}).

Conversely, if $C$ is a well-order, then any increasing filtration $S_{\le a}$, $a\in C$ of $S$  
defines a (well-ordered) valuation $\nu:S\setminus\{0\}\to C$ via
$$\nu(x)=\min\{a\in C:x\in S_{\le a}\}$$
for all $x\in S\setminus \{0\}$.

\begin{definition}
\label{def:adapted}    

We say that  ${\bf B}\subset S$ is {\it adapted} to a valuation $\nu:S\setminus\{0\}\to C$ if $\nu({\bf B})=C_\nu$ and
for each $a\in C$ and nonzero  $v$ in the $\kk$-linear span of ${\bf B}_a=\{b\in {\bf B}\,|\,\nu(b)=a\}$ one has
$$\nu(v)=a\ .$$
(this is slightly different  from \cite{BSch, Kaveh-Manon}).

\end{definition}

\begin{remark}
Such a (necessarily independent) subset ${\bf B}$ of $S$ is called {\it valuation-independent} in \cite{Ku}, \cite{KuKu}, but we prefer terminology of \cite[Definition 2]{Kaveh-Manon}. 
If we  denote $gr~S:=\bigoplus\limits_{a\in C} S_a$, then clearly any adapted subset of $S$ defines a basis of $gr~S$.

\end{remark}

The following is immediate.

\begin{lemma} Let $\nu:S\setminus \{0\}\to C$ be a valuation. Then $\nu$ admits an adapted basis iff $S=\bigoplus\limits_{a\in C_\nu} S_a$ such that $\nu(S_a\setminus \{0\})=\{a\}$ for all $a\in C_\nu$.
    
\end{lemma}

We say that $\nu$ is {\it locally finite} iff there exists an isomorphism $f:S\widetilde \to gr~S =\bigoplus \limits_{a\in C} S_a$ of $\kk$-vector spaces such that $f(S_{\le a})=\bigoplus\limits_{a'\le a} S_{a'}$ for each $a\in C$. 

The following is immediate.

\begin{lemma} 
\label{le:valuation grading}
Let $S$ be a $\kk$-vector space and $\nu:S\setminus\{0\}\to C$ be any valuation. Then: 

(a) For any basis $\underline {\bf B}$ of $gr~S$ such that any $\underline {\bf B}_a:=S_a\cap \underline {\bf B}$ a basis of $S_a$ one can construct an adapted set ${\bf B}$ as follows.
$${\bf B}_a=\iota_a(\underline {\bf B}_a)\ ,$$
where and $\iota_a:S_a\hookrightarrow S_{\le a}$ is any simultaneous splitting of canonical projections $\pi_a:S_{\le a}\twoheadrightarrow S_a$.

(b)  $\nu$ is locally finite iff $S$ admits an adapted basis. 

(c) If $C$ is a well-order, then $\nu$ is locally finite, moreover, any adapted subset of $S$ is a basis.

\end{lemma}

\begin{example} 
\label{ex:non-injective valuation}
Let $S=\kk((t^{-1}))$ be the algebra of all formal Laurent series in $t^{-1}$ over $\kk$. Then setting $\nu(f)=n$ for each $f=\sum\limits_{m=-\infty}^n a_mt^m$ with $a_n\ne 0$  defines a valuation $S\setminus\{0\}\to \ZZ$. This valuation is {\bf not} locally finite, in particular,  the subset ${\bf B}=\{t^m,m\in \ZZ\}$ is adapted to $\nu$, however, it is not a basis of $S$. In fact, there is {\bf no} adapted bases in $S$. 
At the same time, the restriction of $\nu$ to the  subalgebra $S_0=\kk[t,t^{-1}]$ of Laurent polynomials is a locally-finite valuation on $S_0$.
\end{example}

It turns out that we can always propagate valuations to tensor products  without assuming that they are well-ordered.

\begin{proposition}
\label{pr:lex tensor product}
Let $S$ be $\kk$-vector space and  $\nu:S\setminus \{0\}\to C$ be a valuation. Then for any $\kk$-vector space $S'$ we have:

(a) There exists a unique  a valuation $\nu^{S'}:S\otimes S'\setminus\{0\}\to C$ such that
\begin{equation}
\label{eq:lex tensor product}
\nu^{S'}(x\otimes y)=\nu(x)
\end{equation}
for all $x\in S\setminus\{0\}$, $y\in S'\setminus\{0\}$ so that the associated filtration on $S\otimes S'$ is
$$(S\otimes S')_{\le a}=S_{\le a}\otimes S'$$
for $a\in C$. 

(b) For any valuation  $\nu':S'\setminus\{0\}\to C'$ there exists a unique valuation $\nu\otimes \nu':S\otimes S'\setminus \{0\}\to C\times C'$ such that   
\begin{equation}
\label{eq:lex tensor product injective}
(\nu\otimes \nu')(x\otimes y)=(\nu(x),\nu'(y))
\end{equation}
for $x\in S\setminus 0$, $y\in S'\setminus 0$ 
(where we equip $C\times C'$ with the
lexicographic ordering, i.e., $(a,a')<(\tilde a,\tilde a')$ whenever
either $a<\tilde a$ or $\tilde a=a$ and $a'<\tilde a'$) so that the associated filtration on $S\otimes S'$ is
$$(S\otimes S')_{\le (a,a')}=S_{\le a}\otimes S'_{\le a'}+S_{<a}\otimes S'$$
for $(a,a')\in C\times C'$. 
\end{proposition}

We prove Proposition \ref{pr:lex tensor product} in Section \ref{subsec:Proofs of results of Section 4}.

By definition, $C_{\nu^{S'}}=C_\nu$,  $\nu^{S'}(s\otimes (S'\setminus \{0\}))=\nu(s)$ for any $s\in S\setminus \{0\}$ and $\nu^{S'}((S\setminus \{0\})\otimes s')=C_\nu$ for any $s'\in S'\setminus \{0\}$ in Proposition \ref{pr:lex tensor product}(a).

\subsection{Decorated valuations}

\label{subsec:String valuations}

In this section we generalize valuations by taking into account their leading coefficients. 

\begin{definition}\label{decorated valuation} Given $\kk$-vector spaces $S, S'$ and a valuation $\nu:S\setminus \{0\}\to C$, we say a map $\lambda:S\setminus \{0\}\to S'\setminus \{0\}$  is  a {\it leading coefficient} of $\nu$ if $\lambda(c x)=c\lambda(x)$ for all $c\in \kk$, $x\in S\setminus \{0\}$ and 
$
\lambda(x+y)=\begin{cases} 
\lambda(x)+\lambda(y) &\text{if $\nu(x)=\nu(y)=\nu(x+y)$}\\
\lambda(x) &\text{if $\nu(x)>\nu(y)$}\\
\lambda(y) &\text{if $\nu(x)<\nu(y)$}
\end{cases}
$
for any $x,y\in S$ such that $x+y\ne 0$. Sometimes we will refer to the pair $(\nu,\lambda)$ 
as a {\it decorated valuation}.

\end{definition}

If both $S$ and $S'$ are $\kk$-algebras, we say that a leading coefficient $\lambda$ of $\nu$ is {\it multiplicative} if
\begin{equation}
\label{eq:homomorphism of multiplicative monoids}
\lambda(x y)=\lambda(x)\lambda(y)
\end{equation}
for all $x,y\in S\setminus \{0\}$ (i.e.,  $\lambda$ is a homomorphism of multiplicative semigroups).

The following are immediate. 

\begin{lemma} 
\label{le:decorated valuation projections}
For any decorated valuation $(\nu,\lambda)$ on a vector space $S$ one has  

(a) $$\lambda(x_1+\cdots +x_r)=\sum\limits_{\substack{j\in [1,r]:\
\nu(x_j)=\max (\nu(x_1),\ldots,\nu(x_r))}}\lambda(x_j) 
$$
\noindent whenever it holds $x_1+\cdots +x_r\ne 0$ and $\nu(x_1+\cdots +x_r)=\max (\nu(x_1),\ldots,\nu(x_r))$.

(b) For any subspace $S_0$ of $S$ the restriction of $(\nu,\lambda)$ to $S_0\setminus \{0\}$ is a decorated valuation on $S_0$.

(c) For any injective linear map $f:S'\hookrightarrow S''$ the pair $(\nu,f\circ \lambda)$ is a decorated valuation on $S$.  
\end{lemma}

\begin{lemma}
\label{le:lex product valuation}
 
Let $\nu:S\setminus \{0\}\to C$ and $\nu':S'\setminus \{0\}\to C'$ be valuations and $\lambda: S\setminus \{0\}\to S'\setminus \{0\}$ be a leading coefficient of $\nu$. Then the assignments $x\mapsto (\nu(x),\nu'(\lambda(x))$ define a valuation $\nu\times_\lambda \nu': S\setminus \{0\}\to C\times C'$, with the lexicographic order on $C\times C'$.

\end{lemma}

The following immediate result gives various characterizations of decorated valuations in terms of adapted bases.

\begin{lemma} 
\label{le:decorated valuation tensor}
Let $S$ and $S'$ be $\kk$-vector spaces, $\nu:S\setminus\{0\}\to C$ be a valuation, and ${\bf B}$ be a basis of $S$ adapted to $\nu$. Then 

(a) Any map $f:{\bf B}\to S'\setminus \{0\}$ uniquely extends to a leading coefficient $\lambda=\lambda_{{\bf B},f}:S\setminus \{0\}\to S'\setminus \{0\}$.

(b) The assignments $b\otimes s'\mapsto s'$, $b\in {\bf B}$, $s'\in S'$ define a leading coefficient   $\lambda^{\bf B}:S\otimes S'\setminus \{0\}\to  S'$ of the valuation $\nu^{S'}$ (in the notation of Proposition \ref{pr:lex tensor product}(a)).

(c) For any leading coefficient $\lambda:S\setminus \{0\}\to S'\setminus \{0\}$ of $\nu$ there exists a unique injective linear map $\delta=\delta_{{\bf B},\lambda}:S\hookrightarrow S\otimes S'$ such that $\lambda=\lambda^{\bf B}\circ \delta$ (in fact, $\delta$ is given by $\delta(b)=b\otimes \lambda(b)$ for all $b\in {\bf B}$).
\end{lemma}

The following  provides an example of decorated valuations ``in nature.''
\begin{lemma} 
\label{le:nilpotent valuation} Let $\kk$ be of characteristic $0$, $S$ be a $\kk$-vector space, and $E$ be a locally nilpotent linear map $S\to S$, i.e., for each nonzero $x\ne 0$ there is a unique number $\nu_E(x)\in \ZZ_{\ge 0}$ such that $E^{\nu_E(x)}(x)\ne 0$ and $E^{\nu_E(x)+1}(x)=0$. Then

(a) The assignments $x\mapsto \nu_E(x)$ define a valuation $\nu_E: S\setminus \{0\}\to \ZZ_{\ge 0}$.

(b)  The assignments $x\mapsto E^{(\nu_E(x))}(x)$ define the leading coefficient $\lambda_E:S\setminus \{0\}\to S\setminus\{0\}$, where we abbreviate $E^{(n)}:=c_nE^n$, the $c$-modified $n$-th power (here $c=(c_n)_{n\ge 0}$ is any sequence in $\kk^\times$).

(c) If $S$ is an integral domain over $\kk$ and $E$ is a locally nilpotent derivation of $S$, then $\nu_E$ is an additive valuation on $S$ and $\lambda_E$ is its multiplicative leading coefficient.
\end{lemma}

\begin{remark} 
\label{rem:factorials}
If $S$ is an algebra and $E$ is a derivation, we usually take $c_n=\frac{1}{n!}$.
    
\end{remark}

More generally,  let ${\bf E}=(E_1,\ldots,E_m)$ be a family of locally nilpotent linear maps $S\to S$. Define the map $\lambda_{\bf E}:S\setminus\{0\}\to S\setminus \{0\}$ by 
\begin{equation}
\label{eq:lambdaE}
\lambda_{\bf E}:=\lambda_{E_m}\circ \cdots \circ \lambda_{E_1}\ ,
\end{equation}
where $\lambda_E:S\setminus\{0\}\to S\setminus \{0\}$ is
as in Lemma \ref{le:nilpotent valuation} (with the convention $\lambda_\emptyset=Id_{S\setminus \{0\}}$).

Then define the map $\nu_{\bf E}:S\setminus\{0\}\to \ZZ_{\ge 0}^m$  by  
\begin{equation}
\label{eq:string valuation vector space}
\nu_{\bf E}(x)=(a_1,\ldots,a_m)\in \ZZ_{\ge 0}^m\ ,
\end{equation}
where 
$a_k=\nu_{E_k}(\lambda_{(E_1,\ldots,E_{k-1})}(x))$ 
for $k\in [m]$ 
(actually, $a_1=\nu_{E_1}(x)$).

The following is a generalization of Lemma \ref{le:nilpotent valuation}.
\begin{corollary} 
\label{cor:nilpotent valuation and leading coefficient} 

Let $\kk$ be of characteristic $0$, $S$ be a $\kk$-vector space. Then for any family ${\bf E}=(E_1,\ldots,E_m)$ of locally nilpotent linear maps $S\to S$ one has:

(a) $\nu_{\bf E}:S\setminus \{0\}\to \ZZ_{\ge 0}^m$ is a valuation and $\lambda_{\bf E}:S\setminus \{0\}\to S\setminus\{0\}$ is its leading coefficient.

(b) If $S$ is an integral domain over $\kk$,  each $E_k$ is a locally nilpotent derivation of $S$,  then $\nu_{\bf E}$ is a valuation of the algebra $S$. Moreover, if  $c_n=\frac{1}{n!}$ for $n\ge 0$ (see Remark \ref{rem:factorials}) 
then $\lambda_{\bf E}$ is multiplicative.

\end{corollary}

\begin{remark} The decorated valuations $(\nu_{\bf E},\lambda_{\bf E})$  generalize {\it string} valuations and their leading coefficients introduced by Andrei Zelevinsky and the first author in \cite{bz}.

\end{remark} 

\begin{remark} In fact, all valuations $\nu_{\bf E}$ factor (and thus can be defined recursively) as in Lemma \ref{le:lex product valuation}: 
$\nu_{(E_1,\ldots,E_m)}=\nu_{(E_1,\ldots,E_k)}\times_{\lambda_{(E_1,\ldots,E_k)}}\nu_{(E_{k+1},\ldots,E_m)}$ for any $k\in [1,m-1]$.

\end{remark}

\subsection{Injective valuations}

\label{subsec:injective valuations}

Our main focus is on the class of what we call {\it injective valuations}, i.e., locally finite valuations such that
$S_a=S_{\le a}/S_{<a}$ is one-dimensional for each $a\in C_\nu$ (such valuations were called {\it valuations with one-dimensional leaves} in \cite{Kaveh}). Note, however, that the valuation on $\kk((t^{-1}))$ in Example \ref{ex:non-injective valuation} has one-dimensional leaves, but is not locally finite, hence not injective. 

The following is an  immediate consequence of Lemma \ref{le:valuation grading}(c).

\begin{corollary}
\label{cor:injective definition}
A well-ordered valuation $\nu:S\setminus\{0\}\to C$ is injective iff there exists a basis ${\bf B}$ of $S$ such that the restriction of $\nu$ to ${\bf B}$ is an injective map ${\bf B}\hookrightarrow C$.
\end{corollary}
 
As in Section \ref{subsec:valuations}, we refer to any basis ${\bf B}$ satisfying Corollary \ref{cor:injective definition} as {\it adapted} to  $\nu$ and denote by ${\bf A}_\nu$ the set of all bases of $S$ adapted to $\nu$  (in \cite{Ku},
\cite{KuKu} each $\bf B\in {\bf A}_\nu$ is referred to as a {\it valuation basis}).

One can easily show that for any basis ${\bf B}$ adapted to (an injective valuation) $\nu$ one has $\nu({\bf B})=C_\nu$ and $S_{\le a}=\bigoplus\limits_{b\in {\bf B}:\nu(b)\le a}\kk\cdot b$, $S_{< a}=\bigoplus\limits_{b\in {\bf B}:\nu(b)< a}\kk\cdot b$ for all $a\in C_\nu$.

The following result establishes a convenient criterion of injectivity of a valuation.

\begin{proposition} [Euclidean property]
\label{pr:injective} The following are equivalent for a given
well-ordered valuation $\nu:S\setminus \{0\}\to C$.

(a) $\nu$ is injective.

(b)  For any non-zero $x,y\in S$ such that $\nu(x)=\nu(y)$ and $x\notin \kk \cdot y$ there exists (a unique)
$c \in \kk^\times$ such that  $\nu(x-c y)<\nu(x)$.

\end{proposition}

We prove Proposition \ref{pr:injective} in Section \ref{subsec:Proofs of results of Section 4}.

Proposition~\ref{pr:injective} is well-known for finite-dimensional $S$ (see e.g., \cite{Kaveh-Khovanskii}),
for infinite-dimensional $S$ we could not find it in the literature.

\begin{corollary}
\label{adapted_arbitrarily}
For a given
well-ordered injective valuation $\nu:S\setminus \{0\}\to C$  any $\nu$-adapted set is an (adapted) basis of $S$.  
\end{corollary}

\begin{remark} 
\label{rem:adapted is not basis}
We demonstrate that the conclusion of Corollary~\ref{adapted_arbitrarily} is not valid without the assumption of well-orderness. Consider a space $S$ with a basis $\{e_i\ :\ 0\le i<\infty\}$ and an injective valuation $\nu : S\setminus \{0\} \to \ZZ_{\le 0}$ such that $\nu(e_i)=-i$. Then a set $R:= \{e_i+e_{i+1}\ :\ 0\le i< \infty\}$ is adapted, while it is not a basis of $S$: for instance, $e_1$ does not belong to the span of $R$.
\end{remark}

We can build new injective valuations out of
existing ones by the following immediate consequence of the injectivity criterion in Proposition \ref{pr:injective}(b).

\begin{corollary}
\label{cor:new valuations} 
Let $\nu:S\setminus \{0\}\to C$ be a well-ordered injective valuation. Then for any subspace
$\underline S$ of $S$ the restriction $\underline \nu:=\nu|_{\underline S\setminus \{0\}}$ is an injective valuation on $\underline S$.

\end{corollary} 

\begin{remark} 
\label{rem: subspace injective}
It is interesting whether an analog of Corollary \ref{cor:new valuations} holds without assumption of well-orderness of $C$.
    
\end{remark}
 
Given a valuation $\nu:S\setminus\{0\}\to C$, we say that a family $S_i$, $I\in I$ of subspaces of $S$ is  {\it $\nu$-compatible} if $\nu(\bigcap\limits_{i\in I} S_i\setminus \{0\})=\bigcap\limits_{i\in I} \nu(S_i\setminus \{0\})$ (clearly, for any $\nu$ the left hand side is always a subset of the right hand side).

\begin{proposition} 
\label{pr:ensamble}
Suppose that a valuation $\nu$ on a space $S$ is well-ordered injective.
Then a family of subspaces $\{S_i, i\in I\}$ of $S$ is  $\nu$-compatible iff there exists an adapted with respect to $\nu$ basis ${\bf B}$ in $S$ such that ${\bf B}\cap S_i$ is a basis in $S_i$ for each $i\in I$. In addition, in this case ${\bf B}\cap \bigcap_{j\in J}S_j$ is a basis in $\bigcap_{j\in J}S_j$ for each $J\subseteq I$ and
$\nu((\sum_{j\in J}S_j)\setminus \{0\})=\bigcup_{j\in J} \nu(S_j\setminus \{0\})$ for every subset $J\subseteq I$. 
    
\end{proposition}
We prove Proposition \ref{pr:ensamble} in Section \ref{subsec:Proofs of results of Section 4}.

\begin{remark}
If $\{S_i, i\in I\}$  is a $\nu$-compatible family, $|I|<\infty, \dim(S_i)<\infty, i\in I$ then for any subset $J\subseteq I$ it holds
$$\dim(\sum_{j\in J}S_j)= \sum_{L\subseteq J}(-1)^{|L|+1} \dim(\bigcap_{l\in L} S_l).$$
\end{remark}

We can also construct injective valuations on the quotients as follows.

\begin{proposition} 
\label{pr:surjective nu}
Let  $S$ be a $\kk$-vector space  and let $\nu:S\setminus\{0\}\to C$ is an (injective) valuation for some well-order $C$. Then for any  subspace  $J\subset S$ the assignments
$$\nu'(v+J):=\min \{\nu(v+J)\}$$
for all non-zero $v+J\in S/J$ define  an (injective) valuation  $\nu':S/J\setminus\{0\}\to C$.  
\end{proposition} 

\begin{remark} Note, however, that if  $S$ is a commutative integral domain, $J$ a prime ideal,  $C$ a monoid and $\nu(ab)=\nu(a)+\nu(b)$ for all $a,b\in S\setminus\{0\}$ in Proposition \ref{pr:surjective nu}, then 
$\nu'(ab+J)\le \nu'(a+J)+\nu'(b+J)$
for all $a,b\in S\setminus J$ because of the inequality $\min\{\nu(X\cdot Y)\}\le \min\{\nu(X)\}+\min\{\nu(Y)\}$ for any subsets $X,Y\subset S\setminus \{0\}$ (here $X\cdot Y$ is the $\kk$-linear span of $\{xy\,|\,x\in X,y\in Y\}$).
\end{remark}

It turns out that any valuation can be assembled out of injective ones as follows.

\begin{proposition} 
\label{pr:injective restriction} 
Let $S$ be a  $\kk$-vector space and $\nu:S\setminus \{0\}\to C$ be a locally finite valuation (see Section \ref{subsec:valuations}). Then there are $\kk$-vector spaces $\underline S$ and $S'$, an injective valuation $\underline \nu:\underline S\setminus \{0\}\to C$ and a $\kk$-linear embedding 
$\jj:S\hookrightarrow \underline S\otimes S'$ such that $C_{\underline \nu}=C_\nu$ and
$$\nu=\underline \nu^{S'}\circ\jj$$
in the notation \eqref{eq:lex tensor product}.

We prove Proposition \ref{pr:injective restriction} in Section \ref{subsec:Proofs of results of Section 4}.

\end{proposition}

\begin{proposition}\label{commutative_injective} In the notations of Corollary~\ref{cor:nilpotent valuation and leading coefficient}  assume that 
\begin{eqnarray}\label{999}
E_i\circ E_j=p_{ij} E_j\circ E_i
\end{eqnarray}
for some $p_{ij}\in \kk^\times,\ 1\le i,j\le m$ and that
\begin{eqnarray}\label{98}
\bigcap_{1\le i\le m} \ker E_i=\kk.
\end{eqnarray}
Then the valuation $\nu_{\bf E}$ is injective. 
    
\end{proposition}

\begin{proof}
Take an adapted basis ${\bf B}\subset S$ to $\nu_{\bf E}$ with respect to lex ordering $\prec$ in $\ZZ_{\ge 0}^m \supset \nu_{\bf E} (S)$ (see Lemma~\ref{le:valuation grading}). We claim that the basis $\bf B$ is injective. Suppose the contrary. Then for a pair of different elements $b_1, b_2 \in \bf B$ it holds $\nu_{\bf E} (b_1)=\nu_{\bf E} (b_2):=(i_1,\dots,i_m)$. Denote $\beta_1:= E_m^{i_m}\circ \cdots \circ E_1^{i_1} (b_1)$. Then
$$E_j(\beta_1)=\gamma E_m^{i_m}\circ \cdots \circ E_{j+1}^{i_{j+1}}\circ E_j^{i_j+1}\circ E_{j-1}^{i_{j-1}}\circ \cdots E_1^{i_1} (b_1)=0,\ 1\le j\le m$$
\noindent for suitable $\gamma \in \kk^\times$ due to \eqref{999}. Therefore $\beta_1 \in \kk^\times$ because of \eqref{98}. Similarly, $\beta_2:=E_m^{i_m}\circ \cdots \circ E_1^{i_1} (b_2)\in \kk^\times$. Hence $E_m^{i_m}\circ \cdots \circ E_1^{i_1} (\beta_2 b_1 - \beta_1 b_2)=0$, i.e. $\nu_{\bf E} (\beta_2 b_1 - \beta_1 b_2) \prec (i_1,\dots,i_m)$ which contradicts to that $\bf B$ is an adapted basis.
\end{proof}

\subsection{Jordan-H\"older bijections}
For any  valuations $\nu,\nu':S\setminus \{0\}\to C$ such that $\nu'$ is   well-ordered   
define a map ${\bf K}_{\nu',\nu}: C_\nu\to C_{\nu'}$ by
\begin{equation}
\label{eq:Knu'nu} {\bf K}_{\nu',\nu}(a)=\min\{\nu'(\nu^{-1}(a))\}
\end{equation}
for all $a\in C_\nu$.

Our first result  provides an ``industry" for establishing combinatorial bijections.

\begin{theorem}
\label{th:generalized JH} For any well-ordered injective valuations $\nu$ and $\nu'$ on $S$  the maps ${\bf K}_{\nu',\nu}:C_\nu\to C_{\nu'}$ and
${\bf K}_{\nu,\nu'}:C_{\nu'}\to C_\nu$ are well-defined and
mutually inverse bijections. Moreover, there exists a basis ${\bf B}_{\nu,\nu'}$ of $S$ adapted to both $\nu$ and $\nu'$ such that
${\bf K}_{\nu',\nu}(\nu(b))=\nu'(b)$ for all $b\in {\bf B}_{\nu,\nu'}$. 
\end{theorem}

We prove Theorem \ref{th:generalized JH} in Section \ref{subsec:Proofs of results of Section 4}.

We refer to ${\bf K}_{\nu',\nu}$
as a {\it  Jordan-H\"older
bijection (and abbreviate JH-bijection)} and call any  basis ${\bf B}_{\nu,\nu'}$  as an {\it JH-basis}.

\begin{remark}\label{flag} In fact,  Theorem \ref{th:generalized JH}  generalizes well-known facts that any two complete flags in $\kk^n$ have a canonical relative position $w$, which is a permutation of $\{1,\ldots,n\}$, and admit a common basis. Namely, an injective valuation $\nu:S\setminus\{0\}\to C$ defines a complete flag $\FF_\nu$ indexed by $C_\nu$ via $(\FF_\nu)_{\le a}=\{v\in S\setminus\{0\}:\nu(v)\le a\}$, $a\in C_\nu$ (see Sections \ref{subsec:valuations} and \ref{subsec:injective valuations} for details). Conversely, any complete flag $\FF$ on $S$ is of the form $\FF_\nu$. If the indexing sets for flags $\FF_\nu$ and $\FF_{\nu'}$ are well-ordered, then Theorem \ref{th:generalized JH}  asserts that there exist a canonical relative position ${\bf K}_{\nu',\nu}$ of $\FF_\nu$ and $\FF_{\nu'}$ and a common (JH) basis.  This can be also reformulated in terms of generalized Jordan-H\"older correspondence developed by Abels in 1991, see, e.g., Section 2.3 of \cite{BGW}. 

\end{remark}

The following result is an immediate consequence of Theorem \ref{th:generalized JH}.
\begin{corollary}
\label{cor:biinjective} In the assumptions of Theorem \ref{th:generalized JH}
the set  ${\bf A}_\nu\cap {\bf A}_{\nu'}$ is nonempty.
\end{corollary}

The following result is a reverse of Theorem \ref{th:generalized JH}, however, we do not assume that valuations are well-ordered.

\begin{proposition}
\label{pr:common} Let $\nu$ and $\nu'$ be (not necessarily well-ordered) injective valuations on $S$ such that ${\bf A}_\nu\cap {\bf A}_{\nu'}$ is
nonempty. Then the assignments \eqref{eq:Knu'nu}
define a bijection ${\bf K}_{\nu',\nu}:C_\nu\widetilde \to C_{\nu'}$ so that $\nu'(b)={\bf K}_{\nu',\nu}(\nu(b))$ for any
 ${\bf B}\in {\bf A}_\nu\cap {\bf A}_{\nu'}$ and all $b\in {\bf B}$.
\end{proposition}

We prove Proposition \ref{pr:common} in Section \ref{subsec:Proofs of results of Section 4}.

\begin{example} Let $S=\kk[t]$ and let 
$\nu, \nu':S\setminus\{0\} \to -\ZZ_{\ge
0}$ be valuations given by 
$$\nu(t^k)=\nu'(t^k+1)=-k$$
for $k\in \ZZ_{\ge 0}$.
These valuations are obviously injective, and are adapted respectively to the bases ${\bf B}=\{t^k,k\in \ZZ_{\ge 0}\}$, 
${\bf B}'=\{1+t^k,k\in \ZZ_{\ge 0}\}$,  however, $\nu({\bf B}')=\nu'({\bf B})=\{0\}$.

Denote  ${\bf B}''=\{t^k-t^{k+1},k\in \ZZ_{\ge 0}\}$. Clearly, $\nu({\bf B}'')=\nu'({\bf B}'')=-\ZZ_{\ge 0}$
because 
$$\nu(t^k-t^{k+1})=-k,~\nu'(t^k-t^{k+1})=\max(\nu'(1+t^k),\nu'(1+t^{k+1}))=-k$$
for $k\in \ZZ_{\ge 0}$. 
However ${\bf B}''$ is not a basis of $S$, moreover, ${\bf A}_\nu\cap {\bf A}_{\nu'}=\emptyset$ 
In particular, Proposition~\ref{pr:common} and Theorem~\ref{th:generalized JH} are not
applicable to $(\nu,\nu')$ (note that $-\ZZ_{\ge 0}$ endowed with the natural order is
not well-ordered) and thus  ${\bf K}_{\nu',\nu}$ is undefined. 
\end{example}

In some cases, we can obtain injective valuations by utilizing leading coefficients of valuations on their ambient spaces (see Section \ref{subsec:String valuations}).

\begin{proposition} 
\label{pr:injective leading}
Let $(\nu, \lambda)$ be a decorated valuation where $\nu:S\setminus \{0\}\to C$ is a well-ordered valuation and $\lambda:S\setminus \{0\}\to S'\setminus \{0\}$ is its leading coefficient. Let $S_0$ be a subspace of $S$ such that $\lambda(S_0)=\kk^\times \cdot s'$ for some $s'\in S'$. Then the restriction of $\nu$ to $S_0$ is an injective valuation on $S_0$.

\end{proposition}

\begin{proof} Without loss of generality, we consider the case when $S_0=S$, $S'=\kk$, $s'=1$.  It suffices to verify the condition (b) of Proposition \ref{pr:injective}. Indeed, let $x,y\in S\setminus \{0\}$ be such that $y\notin \kk x$ and $\nu(x)=\nu(y)$. Denote $c:=\frac{\lambda(x)}{\lambda(y)}$. Suppose, by contradiction, that $\nu(x-cy)=\nu(x)$. Then 
$\lambda(x-cy)=\lambda(x)+\lambda(-cy)=\lambda(x)-c\lambda(y)=0$, which is  impossible.

The contradiction finishes the proof.
\end{proof}

We can apply this result to integral domains as follows.
Given a commutative integral domain $B$ over $\kk$ and a subalgebra $A$, denote by $B_A$ the set of all $x\in B$ such that 
$A\cdot x \cap (A\setminus \{0\})$ is nonempty. Clearly, $B_A$ is a subalgebra of $B$ (we will sometimes refer to  it as the {\it localization} of $A$ in $B$).

\begin{theorem} 
\label{th:ore density}
Let $B$ be an integral domain over $\kk$, $M$ be a well-ordered monoid, $(\nu, \lambda)$ be a decorated valuation where $\nu$ is a  valuation $B\setminus \{0\}\to M$, and $\lambda:B\to C$ is its multiplicative leading coefficient (here $C$ is an integral domain over $\kk$). Suppose that  $A$ is a subalgebra of $B$ such that  $\lambda(A\setminus \{0\})= \kk^\times$. Then

(a) $\lambda(B_A\setminus \{0\})= \kk^\times$.

(b) The restriction of $\nu$ to  $B_A$ is an injective additive valuation $B_A\setminus \{0\}\to M$.

\end{theorem}

\begin{proof} Indeed, let $b\in B_A\setminus \{0\}$. That is, $xb=y$ for some $x,y\in A\setminus \{0\}$. Therefore, 
$$\lambda(y)=\lambda(xb)=\lambda(x)\lambda(b)$$
since $\lambda$ is multiplicative. Hence $\lambda(b)\in \kk^\times$ because  $\lambda(x),\lambda(y)\in \kk^\times$ by the assumption. This proves (a). 

Part (b) follows from (a) and Proposition  \ref{pr:injective leading}.

The theorem is proved.
\end{proof}

\subsection{Well-ordered submonoids of \texorpdfstring{$\ZZ^m$}~}
\label{subsec:Well-ordered sub-monoids}

For  $M\subset \ZZ^m$ and $k\in [m-1]$ denote by  $M_k$ the image of $M$ under the standard projection $\ZZ^m\to \ZZ^k$ ($a_1,\ldots,a_m)\mapsto (a_1,\ldots,a_k)$). 

\begin{proposition} 
\label{pr:tame twist} Let $m\ge 1$ and $M\subset \ZZ^m$. Then the following are equivalent:

(a) $M$ is well-ordered with respect to the lexicographic order on $\ZZ^m$. 

(b) For $k=0,\ldots, m-1$ there exist functions  $f_k:M_k\to \ZZ$ such that: 
$$a_1+f_0\ge 0, a_2+f_1(a_1)\ge 0, a_3+f_2(a_1,a_2)\ge 0,\ldots,a_m+f_{m-1}(a_1,\ldots,a_{m-1})\ge 0$$
for all $a=(a_1,\ldots,a_m)\in M$.
\vspace{1mm}

If $M$ is a monoid one can  additionally require in (b) that $f_0=f_1(0)=\cdots = f_{m-1}(0,\dots,0)=0$.

\end{proposition}

\begin{proof} First assume (a). For any $0\le k<m$ fix a point $(a_1\dots, a_k)\in M_k$. There exists an integer $N$ such that $a_{k+1}\ge N$ for any $a_{k+1}$ such that $(a_1,\dots, a_k, a_{k+1})\in M_{k+1}$. We put $f_k(a_1,\dots,a_k):= -N$. \vspace{1mm}

Conversely, assume (b) and that (a) is false. Then there exists an infinite decreasing sequence of elements of $M$. Therefore for a suitable maximal possible $0\le k<m$  all elements of the sequence starting with some point have the same prefix $a_1,\dots, a_k$ for appropriate $(a_1,\dots, a_k) \in M_k$. Since $a_{k+1}+f_k(a_1,\dots, a_k)\ge 0$ we get a contradiction with the maximality of $k$. \vspace{1mm}

When $M$ is a monoid and an element $a:=(0,\dots, 0, a_k,\dots, a_m)\in M$ where $a_k\neq 0$, it holds $a_k>0$ because otherwise $a>2a>3a>\dots$. This implies the last statement of
the proposition.
\end{proof}

\begin{example} Given $r\in \QQ_{>0}$, then $M_r=\{(a_1,a_2)\in \ZZ^2:a_1\ge 0,~a_2+r a_1^2\ge 0\}$ is a well-ordered submonoid of $\ZZ^2$.

\end{example}

We say that $g\in GL_m(\QQ)$  is {\it tame} if $g(e_j)\in e_j+\sum\limits_{i=1}^{j-1}\QQ_{\ge 0}\cdot e_i$ for $j\in [m]$, where $\{e_1,\ldots,e_m\}$ is the standard basis of $\ZZ^m$.

\begin{corollary} 
\label{cor:tame linear twist}
A finitely generated submonoid $M\subset \ZZ^m$  is well-ordered (with respect to the lexicographic order on $\ZZ^m$) iff $M\subset g^{-1}(\ZZ_{\ge 0}^m)$ for some tame $g\in GL_m(\ZZ)$. 

\end{corollary}

\subsection{Tame valuations on the Laurent polynomial ring}\label{tame_valuations}
In this section we will view each  $\RR^n$ as a totally ordered set
with respect to the lexicographic ordering.

We say that a valuation
$\nu:\kk[x_1^{\pm 1},\ldots,x_m^{\pm 1}]\setminus \{0\}\to \RR^n$ is {\it tame}
if it is completely determined by its values $\nu(x_i)=v_i\in
\RR^n$, $i=1,\ldots,n$. Clearly, $\nu$ is tame iff it is of the form
$\nu_{\bf v}$,  ${\bf v}=(v_1,\ldots,v_m)\in (\RR^n)^m$:
$$\nu_{\bf v}(\sum\limits_{d\in \ZZ^m} c_d x^d)=\max\limits_{d\in \ZZ^m:c_d\ne 0} \{d_1v_1 +\cdots d_m v_m\}\ .$$

The following is obvious.

\begin{lemma} A tame valuation $\nu=\nu_{\bf v}$ is injective iff the vectors $v_1,\ldots,v_m$ are linearly independent (in particular, $n\ge m$).

\end{lemma}

Since the monomials form an adapted basis to a tame valuation  one
can apply Proposition~\ref{pr:common} to any pair of tame valuations
on the Laurent polynomial algebra $\kk[x_1^{\pm 1},\ldots,x_m^{\pm 1}]$ and get

\begin{corollary}
Any pair of injective tame valuations on  $\kk[x_1^{\pm 1},\ldots,x_m^{\pm 1}]$ has an adapted basis.
\end{corollary}

A tame valuation provides a total well ordering of the monomials.
\cite{R} and
Theorem 9 from \cite{Kh} state that all the total well orderings of
the monomials are exhausted by the tame ones.

One can view $\bf v$ as $n\times m$ matrix. Then lex ordering
corresponds to the unit matrix, and deglex ordering corresponds to
$m\times m$ matrix with ones on the diagonal and in the first row
with zeroes at the rest entries.

\subsection{Algorithms computing Jordan-H\"older bijections}

Consider a pair $\nu,\nu':S\setminus\{0\} \to (C,<)$ of injective well-ordered valuations.
Assume that there are given algorithms mapping $C_{\nu}$ (respectively,
$C_{\nu'}$) to an adapted basis ${\bf B} \in {\bf A}_{\nu}$
(respectively, ${\bf B'} \in {\bf A}_{\nu'}$). Also assume that
$(C_{\nu},<)$ is isomorphic to $\ZZ_{\ge 0}$ and there is given an
algorithm exhibiting this isomorphism. Note that deglex on the
polynomial ring fulfills the latter feature. Then one can compute the
JH-bijection ${\bf K}_{\nu',\nu}:C_{\nu} \to C'_{\nu'}$ and a common adapted
basis from ${\bf A}_{\nu} \cap {\bf A}_{\nu'}$.

Indeed, for any $a\in C_{\nu}$ the algorithm produces $b_a\in {\bf
B}$ with $\nu(b_a)=a$ and all $b_i\in {\bf B},\, i\in I$ such that
$\nu(b_i)<a$. The algorithm expands each $b_a,b_i,\, i\in I$ in
basis ${\bf B'}$ and an element $b_a+\sum_{i\in I}c_i\cdot b_i$ with
indeterminate coefficients $c_i,\, i\in I$
$$b_a+\sum_{i\in I}c_i\cdot b_i=\sum_{1\le j\le p} A_j\cdot b_j'$$
\noindent for suitable linear functions $A_j,\, 1\le j\le p$ in
$c_i,\, i\in I$ and $b_j'\in {\bf B'},\, 1\le j\le p$. Let
$\nu'(b_1')>'\nu'(b_2')>'\cdots >'\nu'(b_p')$.

Consecutively, for $l=1,2,\dots,p$ the algorithm tests whether a
linear in $c_i,\, i\in I$ system $A_1=A_2=\cdots=A_{l-1}=0$ has a
solution. Consider maximal $l$ satisfying the latter property. Then  ${\bf K}_{\nu',\nu}(a)=\nu'(b_l')$. Pick any solution $c_i,\, i\in I$
of the linear system $A_1=A_2=\cdots=A_{l-1}=0$. then
$$\{b_a+\sum_{i\in I}c_i\cdot b_i\}_{a\in C_{\nu}}$$
constitute a common adapted basis for ${\bf A}_{\nu} \cap {\bf A}_{\nu'}$.

This algorithm computes the JH-bijection ${\bf K}_{\nu',\nu}(a)$ for an
arbitrary input $a\in C_{\nu}$ in a general case of a vector space.
Since in case of a polynomial algebra, the JH-bijection is more rigid than in
general, one is able to design a partial algorithm for computing the JH-bijection
and a common adapted basis for both $\nu,\nu'$ in a finite form (we
call this form a piece-wise monoidal representation), provided that
the partial algorithm terminates. Moreover, the partial algorithm
terminates iff the JH-bijection admits a piece-wise monoidal representation.
Below we assume that $C_{\nu}=\ZZ_{\ge 0}^m$.

We accomplish the algorithm from the beginning of this subsection
for computing the JH-bijection ${\bf K}_{\nu',\nu}(a)$ step by step for
increasing $a\in C_{\nu}$ by recursion. Thus, we assume as a
recursive hypothesis that ${\bf K}_{\nu',\nu}(a)$ is already
computed for all $a<a_0$ for some $a_0$. After each step the result
of the algorithm can be given as the following {\it piecewise
monoidal representation}. Polynomials $f_1,\dots,f_n \in
{\bf K}[x_1,\dots,x_m]$ are given together with a partition of $\RR_{\ge
0}^m$ into simplicial cones generated by vectors $a_1:=\nu
(f_1),\dots,a_n:=\nu(f_n)\in  \ZZ_{\ge 0}^m$. Consider one (with
some dimension $p\le m$) of these cones generated by vectors
$a_{i_0},\dots,a_{i_p}$ and denote by $M \subset \ZZ _{\ge 0}^m$ the
monoid generated by vectors $a_{i_0},\dots,a_{i_p}$. In addition, to
each integer point $a$ from the parallelotope $P=\{\alpha_0\cdot
a_{i_0}+\cdots+\alpha_p \cdot a_{i_p}:\, 0\le
\alpha_0,\dots,\alpha_p<1\}\subset \RR_{\ge 0}^m$ generated by
$a_{i_0},\dots,a_{i_p}$ is attached a polynomial $f_a\in \kk
[x_1,\dots,x_m]$ with $\nu (f_a)=a$. Then the monoid of all integer
points from $(M\otimes \RR_{\ge 0}) \cap \ZZ_{\ge 0}^m$ is a
disjoint union of shifted monoids $M+a$ for all $a\in P\cap 
\ZZ_{\ge 0}^m$.

These data determine a basis $\bf B$ of $\kk[x_1,\dots,x_m]$
adapted for $\nu$. Namely, for any point $v=c_0\cdot a_{i_0}+\cdots
c_p\cdot a_{i_p}+ a \in M+a$ where $c_0,\dots,c_p\in \ZZ_{\ge 0}$
put $b_v:=f_{i_0}^{c_0}\cdots f_{i_p}^{c_p}\cdot f_a \in {\bf B}$,
hence $\nu (b_v)=v$. Also we define map ${\bf K}:C_\nu \to
C_{\nu'}$ by ${\bf K}(v):=\nu'(b_v)=c_0\cdot \nu'(f_{i_0})+\cdots+
 c_p\cdot \nu'(f_{i_p})+\nu' (f_a)$. Thereby, ${\bf K}$ is linear on
each shifted monoid $M+a$. Clearly, ${\bf K}_{\nu',\nu}\le' {\bf K}$
holds point-wise.

Now we produce an algorithmic criterion whether the partial
algorithm terminates at the current step of recursion. It terminates
iff for every pair of distinct points $v,v_0\in \ZZ_{\ge 0}^m$ it
holds $\nu'(b_v)\neq \nu'(b_{v_0})$. The latter condition is
equivalent to non-solvability of a suitable integer programming
problem. If the partial algorithm terminates then $\bf B$ is a
common adapted basis for both $\nu, \nu'$ and ${\bf K}={\bf
K}_{\nu',\nu}$ (see Proposition~\ref{pr:common}).

Otherwise, if the partial algorithm does not terminate at the
current recursive step, the algorithm described at the beginning of
this subsection accomplishes the next step for computing the JH-bijection at a
greater (wrt the ordering $<$ on $C_{\nu}$) point. Assume (for the
sake of simplicity) that the algorithm at this step computes just
${\bf K}_{\nu',\nu}(a_0)$ and $f_0\in \kk[x_1,\dots,x_m]$ satisfying
${\bf K}_{\nu',\nu}(c_0) <' {\bf K} (a_0)$ such that $\nu(f_0)=a_0$
and $\nu'(f_0)={\bf K}_{\nu',\nu}(a_0)$. Then at the current
recursive step the partial algorithm adds $f_0$ to $f_1,\dots,f_n$.

Let $a_0$ belong to a $p$-dimensional simplicial cone $T$ generated
by vectors $a_{i_0},\dots,a_{i_p}$ for some $p\le m$. The partition
of $T$ into simplicial cones $T_j,\, 0\le j\le p$ generated by
$a_{i_0},\ldots,a_{i_{j-1}},a_0,a_{i_{j+1}},\dots,a_{i_p}$ induces
the partition of $\ZZ_{\ge 0}^m$ into the  union of shifted monoids
(we keep from the previous recursive step the partitions of all the
cones not containing $a_0$), and thereby, we get a piecewise
monoidal representation after the current recursive step. To define
(the modified after the current recursive step) ${\bf K}':C_{\nu}\to
C_{\nu'}$ on a shifted monoid $M_j'+a$ where monoid $M_j'$ is
generated by
$a_{i_0},\dots,a_{i_{j-1}},a_0,a_{i_{j+1}},\dots,a_{i_p}$, and an
integer point $s$ belongs to the parallelotope generated by the same
vectors $a_{i_0},\dots,a_{i_{j-1}},a_0,a_{i_{j+1}},\dots,a_{i_p}$,
we take polynomial $f_a:=b_a\in {\bf B}$ constructed at the previous
recursive step.

This completes the description of a piecewise monoidal
representation of ${\bf K}'$ at the current recursive step and the
design of the partial algorithm.

\begin{proposition}
The designed partial  algorithm terminates and in this case
yields a piece-wise monoidal representation of ${\bf
K}_{\nu,\nu'}$ (provided that $C_{\nu}=\ZZ_{\ge 0}^m$) together with
a common adapted basis for both $\nu,\nu'$ iff ${\bf K}_{\nu',\nu}$
admits a piece-wise monoidal representation (in particular, ${\bf
K}_{\nu',\nu}$ is linear on each of the shifted monoids from the
representation).
\end{proposition}

{\bf Proof}. We have already shown that if the designed partial
algorithm terminates then ${\bf K}={\bf K}_{\nu',\nu}$.

Conversely, suppose that ${\bf K}_{\nu',\nu}$ admits a piece-wise
monoidal representation with vectors $a_1,\dots,a_n\in C_{\nu}$.
After that the partial algorithm computes ${\bf K}_{\nu',\nu}(a)$
for $a\in \{a_1,\ldots,a_n\}$ and for all integer points $a$
belonging to the parallelotopes from the latter piece-wise monoidal
representation generated by vectors $a_1,\ldots,a_n$,
the resulting ${\bf K}\le' {\bf K}_{\nu',\nu}$ since ${\bf K}$ is
determined by these values ${\bf K}_{\nu',\nu}(a)$.
On the other hand, always holds ${\bf K}\ge'{\bf K}_{\nu',\nu}$, therefore ${\bf K}={\bf K}_{\nu',\nu}$
and the Proposition is proved. \endproof

It would be interesting to understand, whether the designed partial
algorithm always terminates when say, $\nu$ is deglex valuation and
$\nu'=\nu_{\varphi}$ for any injective homomorphism.
$\tau:\kk[x_1,\dots,x_m]\to \kk[x_1,\dots,x_m]$?

\subsection{Proofs of results of Section \ref{sec:Results on valuations}}
\label{subsec:Proofs of results of Section 4}
$ ~~$

\noindent {\bf Proof of Proposition} \ref{pr:lex tensor product}. Prove (a). Let $S$ and $S'$ be $\kk$-vector spaces,  
for any nonzero $z\in S\otimes S'$ 
denote by $V(z)\subset S$ the smallest (by inclusion) subspace of $S$ such that 
$z\in V(z)\otimes S'$.

The following is obvious. 

\begin{lemma} 
\label{le:V(z)}
For  each nonzero $z\in S\otimes S'$ one has:

(a) $V(z)\ne 0$ and  $V(\kk^\times \cdot z)=V(z)$,

(b) $V(z+z')\subseteq V(z)+V(z')$ for any nonzero $z'\in S\otimes S'\setminus\{0,-z\}$.

(c) $V(z)$ is finite dimensional, moreover, it is the
$\kk$-linear span of $\{x_1,\ldots,x_m\}$ for any expansion  
\begin{equation}
\label{eq:m(z)}
z=x_1\otimes y_1+\ldots x_m\otimes y_m
\end{equation}
with minimal possible $m$ (such an $m$ was called {\it rank} of $z$ in \cite{Gri}).

\end{lemma}

Furthermore, given a valuation $\nu:S\setminus\{0\}\to C$. Then, clearly, for  any finite-dimensional subspace $S_0\subset S$ the set   
$$\{\nu(x)\,|\,x\in S_0\setminus\{0\}\}$$
is a finite subset of $C$; denote by $\nu(S_0)$ its maximal element.

Furthermore, in the notation of Lemma \ref{le:V(z)}, for each nonzero  $z\in S\otimes S'$, denote 
\begin{equation}
\label{eq:nu{S'}}
\nu^{S'}(z):=\nu(V(z))\ .
\end{equation}

Clearly, $V(x\otimes y)=\kk \cdot x$ for any nonzero $x\in S$, $y\in S'$, hence $\nu^{S'}(x\otimes y)=\nu(x)$, as in \eqref{eq:lex tensor product}. This and Lemma \ref{le:V(z)} imply that the assignment $z\mapsto \nu^{S'}(z)$ is the desired valuation on $S\otimes S'$. 
This finishes the proof of Proposition \ref{pr:lex tensor product}(a).

Prove (b). For any nonzero $z\in S\otimes S'$ denote by $\underline V(z)$ the smallest (by inclusion) subspace of $S_{\le a}/S_{<a}$, $a=\nu^{S'}(z)$ such that
$$z+S_{<a}\in \underline V(z)\otimes S'$$ 
in $(S_{\le a}/S_{<a})\otimes S'=(S_{\le a}\otimes S')/(S_{<a}\otimes S')$ (that is, 
$\underline V(z)$ is the image of $V(z)$ under the canonical projection $S_{\le a}\twoheadrightarrow S_{\le a}/S_{<a}$).
Furthermore, denote by $V'(z)$ the smallest (by inclusion) subspace of $S'$ such that 
$$z+S_{< a}\in \underline V(z)\otimes V'(z)$$ in 
$(S_{\le a}/S_{<a})\otimes S'$, where $a=\nu^{S'}(z)$.

The following is obvious.

\begin{lemma} 
\label{le:V'(z)}For each nonzero $z\in S\otimes S'$, the subspace $V'(z)$ is finite-dimensional. Moreover, $\dim \underline V(z)=\dim V'(z)$ and for any expansion \eqref{eq:m(z)} with smallest possible $m$, one has
$$\underline V(z)=\bigoplus_{i\in [1,m]:\nu(x_i)=a} \kk \cdot (x_i+V(z)_{< a}),~V'(z)=\bigoplus_{i\in [1,m]:\nu(x_i)=a} \kk \cdot y_i\ .$$ 

\end{lemma}

Furthermore for each nonzero  $z\in S\otimes S'$, denote 
\begin{equation}
\label{eq:nu tensor nu'}
(\nu\otimes \nu')(z):=(\nu^{S'}(z),\nu'(V'(z)))\ .
\end{equation}

Clearly, $V'(x\otimes y)=\kk \cdot y$ for any nonzero $x\in S$, $y\in S'$. Since $\nu^{S'}(x\otimes y)$,  we obtain $(\nu\otimes \nu')(x\otimes y)=(\nu(x),\nu'(y))$, as in \eqref{eq:lex tensor product injective}. This and Lemma \ref{le:V'(z)} imply that the assignment $z\mapsto (\nu\otimes \nu')(z)$ is the desired valuation on $S\otimes S'$. 
This finishes the proof of Proposition \ref{pr:lex tensor product}(b).

The proposition is proved.
\endproof

\noindent {\bf Proof of Proposition~\ref{pr:injective}}.
Prove (a)$=>$(b). Indeed, let $x,y\in S\setminus \{0\}$ with
$a=\nu(x)=\nu(y)$. Then $S_{\le a}=\kk x+S_{<a}= \kk y+S_{<a}$ which implies
that $x-cy\in S_{<a}$ for some (unique) nonzero scalar $c\in \kk$.
This proves the implication (a)$=>$(b).

Prove (b)$=>$(a).
Choose  ${\bf B}\in \tilde {\bf A}_\nu$ in the notation of Section \ref{subsec:valuations}. By Lemma \ref{le:valuation grading}(c), this is a basis of $S$ such that the restriction of $\nu$ to ${\bf B}$ is a surjective map ${\bf B}\twoheadrightarrow C_\nu$ and ${\bf B}_{<a}=S_{<a}\cap {\bf B}$ is a basis of $S_{<a}$ for all $a\in C_\nu$. It remains to establish injectivity of $\nu|_{\bf B}$, which we will do by contradiction. Suppose $b,b'\in {\bf B}$ be such that $b\ne b'$, $\nu(b)=\nu(b')$. Then there exists $c\in \kk^\times$ such that $b'-cb\in S_{<a}$, where $a=\nu(b)$, which implies that $b'$ is a linear combination of elements of ${\bf B}$. The resulting contradiction proves that $\nu|_{\bf B}:{\bf B}\to C_\nu$ is a bijection. In view of Corollary \ref{cor:injective definition}, this proves the
implication (b)$=>$(a).

The proposition is proved. \endproof

\vspace{2mm}

\noindent {\bf Proof of Proposition \ref{pr:ensamble}}.
Let $\bf B$ be an adapted basis of $S$ such that ${\bf B}\cap S_i$ is a basis in $S_i$ for each $i\in I$. Then ${\bf B}\cap S_i=\{b\in {\bf B}\, :\, \nu(b)\in \nu(S_i\setminus \{0\})\}$ (cf. Corollary~\ref{cor:new valuations}). Therefore, if for $b\in \bf B$ it holds $\nu(b)\in \bigcap_{j\in J} \nu(S_j\setminus \{0\})$ for some $J\subseteq I$ then $b\in \bigcap_{j\in J}S_j$, hence $\nu(b)\in \nu(\bigcap_{j\in J}S_j\setminus \{0\})$, this justifies that the family $\{S_i, i\in I\}$ is $\nu$-compatible. Moreover, this implies that ${\bf B}\cap \bigcap_{j\in J}S_j=\{b\in {\bf B}\, :\, \nu(b)\in \nu(\bigcap_{j\in J}S_j\setminus \{0\})\}$ is a basis of $\bigcap_{j\in J}S_j$.

In addition, $\sum_{j\in J}S_j$ is contained in the linear span of the vectors $\{b\in {\bf B}\, :\, \nu(b)\in \bigcup_{j\in J}\nu(S_j\setminus \{0\})\}$. Thus, $\nu((\sum_{j\in J}S_j)\setminus \{0\})\subseteq \bigcup_{j\in J}\nu(S_j\setminus \{0\})$. The opposite inclusion is obvious.

Now conversely, assume that the family $\{S_i, i\in I\}$ is $\nu$-compatible. For each $c\in \nu(S\setminus \{0\})$ there exists a unique subset $J\subseteq I$ such that 
$$c\in \bigcap_{j\in J}\nu(S_j\setminus \{0\}) \setminus \bigcup_{l\notin J}\nu(S_l\setminus \{0\}).$$
\noindent The case $J=\emptyset$ means that $c\notin \bigcup_{i\in I}\nu(S_i\setminus \{0\})$. Due to $\nu$-compatibility there exists a vector $b_c\in \bigcap_{j\in J}S_j$ such that $\nu(b_c)=c$ (the case $J=\emptyset$ means that $b_c\in S$). Observe that for any $c\in \bigcap_{j\in J}\nu(S_j\setminus \{0\})$ it  holds that the vector $b_c\in \bigcap_{j\in J}S_j$. Hence $\{b_c\, :\, c\in \bigcap_{j\in J}\nu(S_j\setminus \{0\})\}$ is a basis of $\bigcap_{j\in J}S_j$ since $\nu$ is injective (cf. Corollary~\ref{cor:new valuations}). Hence the basis ${\bf B}:=\{b_c\, :\, c\in \nu(S\setminus \{0\})\}$ is required.
\endproof

{\bf Proof of Proposition} \ref{pr:injective restriction}. By Lemma \ref{le:valuation grading}(b), ${\bf A}_\nu$ is non-empty, so fix ${\bf B}\in {\bf A}_\nu$. 
Then ${\bf B}_{\le a}:=S_{\le a}\cap {\bf B}=\bigsqcup\limits_{a'\le a} {\bf B}_{a'}$ is a basis in $S_{\le a}$, where 
${\bf B}_{a'}=\{b\in {\bf B}:\nu(b)=a'\}$. 

Furthermore, choose a well-ordering of each ${\bf B}_a$. Denote by ${\bf B}^0\subset {\bf B}$ the set which consists of all minimal (with respect to the chosen well-ordering) elements of all ${\bf B}_a$. By the construction, $|{\bf B}^0\cap {\bf B}_a|=1$ for all $a\in C_\nu$ and the restriction of $\nu$ to ${\bf B}^0$ is a bijection ${\bf B}^0\widetilde \to C_\nu$. 

Using transfinite induction, we repeat this procedure and obtain the following.

\begin{lemma} For each ${\bf B}\in {\bf A}_\nu$ there is a well-ordered set ${\bf I}$ with the minimal element ${\bf 0}$ such that 

$\bullet$ ${\bf B}=\bigsqcup\limits_{\ii\in {\bf I}} {\bf B}^\ii$.

$\bullet$ The restriction of $\nu$ to ${\bf B}^\ii$ is an injective map ${\bf B}^0\hookrightarrow C_\nu$.

$\bullet$ $\nu({\bf B}^\ii)\subset \nu({\bf B}^{\ii'})$ if $\ii'\le \ii$ and  $\nu({\bf B}^0)=C_\nu$

\end{lemma}

Using this, we obtain an injective map $\underline \jj:{\bf B}\hookrightarrow {\bf B}^0\times {\bf I}$ given by
\begin{equation}
\label{eq:iota embedding}
\underline \jj(b)=(b^0,\ii)
\end{equation}
for each $b\in {\bf B}^\ii$, where $b^0$ is the only element of ${\bf B}^0$ such that $\nu(b)=\nu(b^0)$.

Linearizing, we obtain an injective $\kk$-linear map $\jj=\kk\underline \jj:S\to \underline S\otimes S'$, where 
$\underline S=\kk {\bf B}^0\subset S$ and $S'=\kk {\bf I}$. 

By the very construction, the restriction of $\nu$ to $\underline S$ is an injective valuation $\underline \nu:\underline S\setminus \{0\}$. Also, for each nonzero  $x\in S$ written as $x=\sum\limits_{\ii\in {\bf I},b\in {\bf B}^\ii} c_b^\ii b^\ii$ one has
$$\jj(x)=\sum\limits_{\ii\in {\bf I},b\in {\bf B}^\ii} c_b^\ii (b^0,\ii)$$
in the notation \eqref{eq:iota embedding}.

In particular, $\nu(x)=\max\limits_{\ii\in I,b\in {\bf B}^\ii:c_b^\ii\ne 0}\{\nu(b)\}=\max\limits_{\ii\in I,b\in {\bf B}^\ii:c_b^\ii\ne 0}\{\nu(b^0)\}=\underline \nu^{S'}(x)$.

The proposition is proved.
\endproof

\noindent {\bf Proof of Theorem~\ref{th:generalized JH}}. Let $\nu$
and $\nu'$ be injective valuations on $S$.
Fix any $a\in C_{\nu}$. Then choose $x\in S\setminus \{0\}$ such
that $\nu(x)=a$ and $\nu'(x)=\min\{\nu'(\nu^{-1}(a))\}$ and $y\in
S\setminus \{0\}$ such that
$\nu(y)=\min\{\nu(\nu'^{-1}(\nu'(x)))\}$, hence $\nu'(y)=\nu'(x)$.
By definition, $\nu'(x)={\bf K}_{\nu', \nu}(a)$ and $\nu(y)={\bf
K}_{\nu, \nu'} (\nu'(x))={\bf K}_{\nu, \nu'} ({\bf K}_{\nu',\nu}
(a))$. Note that $x\in \nu'^{-1}(\nu'(x))$ hence $\nu(y)\le \nu(x)$.

Using Proposition \ref{pr:injective}(b), choose $c\in \kk$ such that
$\nu'(x-cy)<\nu'(x)$. Thus
$${\bf K}_{\nu',\nu} (a)=\nu'(x)>\nu'(x-cy) \ .$$
Since $\nu'(x)\le \nu'(z)$ for all $z\in S\setminus \{0\}$ with
$\nu(z)=\nu(x)$, this implies that $\nu(x-cy)\ne \nu(x)$. In turn,
this and  the inequality  $\nu(y)\le \nu(x)$ imply that
$\nu(y)=\nu(x)$.

This proves that ${\bf K}_{\nu, \nu'} ({\bf K}_{\nu',\nu} (a))=a$
for all $a\in C_\nu$, i.e., ${\bf K}_{\nu, \nu'}\circ {\bf K}_{\nu',\nu}=Id_{C_\nu}$.
Switching $\nu$ and $\nu'$ in the above
argument, we also obtain ${\bf K}_{\nu', \nu}\circ {\bf K}_{\nu,\nu'}=Id_{C_{\nu'}}$.

This proves the first assertion of the theorem.

Prove the second assertion now.
For each $a\in C_\nu$ denote by $S_a$ the set of all $b\in S\setminus \{0\}$ such that $\nu(b)=a$.
Furthermore,  let $S_a^{min}$ be the set of all $b\in S_a$ such that
$\nu'(b)=\min\{\nu'(b'):b'\in S_a\}$.  Then, well-ordering of $\nu$
implies that $S_a^{min}$ is non-empty and then $\nu'(b)={\bf K}_{\nu', \nu}(a)$ for each $b\in S_a^{min}$ by \eqref{eq:Knu'nu}.
Finally, for each $a\in C_\nu$ choose a single element $b_a\in
S_a^{min}$. Clearly, ${\bf B}=\{b_a:\, a\in C_{\nu}\}$ is adapted to $\nu$ because the restriction of $\nu$ to ${\bf B}$ is a
bijection ${\bf B}\widetilde \to C_\nu$ ($b_a\mapsto \nu(b_a)=a$).
Hence ${\bf B}\in {\bf A}_\nu$ is a basis of $S$ by Lemma \ref{le:valuation grading}(c). Finally, injectivity of ${\bf K}_{\nu',\nu}$
implies that ${\bf B}$ is adapted to $\nu'$ because for any distinct
$a,a'\in C_\nu$ one has $\nu'(b_a)={\bf K}_{\nu',\nu}(a)\ne {\bf
K}_{\nu',\nu}(a')=\nu'(b_{a'})$.
  
The theorem is proved. \endproof

\noindent {\bf Proof of Proposition~\ref{pr:common}}. Let ${\bf B}\in {\bf A}_\nu\cap {\bf A}_{\nu'}$
and for each $d\in C_\nu$ denote by $b_d$ the only element of ${\bf B}$ with $\nu(b_d)=d$.

Clearly, for any $a\in C_\nu$ each $s\in \nu^{-1}(a)$ can be uniquely written as:
$$s=c_a\cdot b_a+\sum_{\tilde a<a} c_{\tilde a}b_{\tilde a}$$
where  $c_a\ne 0$.

Therefore, $\nu'(s)\ge \nu'(b_a)$ for all $s\in \nu^{-1}(a)$, i.e.,
$\nu'(\nu^{-1}(a))\ge \nu'(b_a)$ in $C'$.

On the other hand, $b_a\in v^{-1}(a)$, therefore the minimum in
$\min \{\nu'(\nu^{-1}(a))\}$ (see \eqref{eq:Knu'nu}) is attained and
equals $\nu'(b_a)\in \nu'(\nu^{-1}(a))$, i.~e. ${\bf K}_{\nu',\nu}(a)=\nu'(b_a)$ is defined. The proposition is proved.
\endproof

\end{document}